\documentclass[a4paper,plainchapterheads,yschapters,twoside,truedoublelespace,openright]{iitbthesis}
\usepackage{amsthm}
\usepackage{epsfig}
\usepackage[T1]{fontenc}
\usepackage[utf8]{inputenc}
\usepackage{amssymb}
\usepackage{newtxmath} 
\usepackage{amsmath,epsfig}
\usepackage{graphicx,graphics}
\usepackage{url}

\usepackage{hyperref} 
\usepackage[capitalise]{cleveref}
\usepackage{textcomp}
\usepackage{enumitem}
\usepackage{natbib}
\usepackage{pdflscape}
\renewcommand\bibname{References}
\usepackage{titlesec}
\usepackage{bm,bbm}
\usepackage{float,lscape}
\usepackage{subfig}

\usepackage{epstopdf}
\usepackage{tabu}
\usepackage{multirow}
\usepackage{array}
\usepackage{booktabs}
\usepackage{caption}
\usepackage{lettrine}
\usepackage{lscape}
\usepackage{rotating}
\usepackage{booktabs}
\usepackage{longtable}
\usepackage{anyfontsize, xcolor, nicefrac}
\usepackage{t1enc}
\usepackage{relsize}
\usepackage{placeins}
\usepackage{nomencl}
\usepackage{rotating}
\usepackage{glossaries}
\usepackage{scrlayer}

\usepackage{algorithm2e}

\usepackage[titletoc]{appendix}
\usepackage[nottoc]{tocbibind}
\newtheorem{theorem}{Theorem}[chapter]

\newtheorem{corollary}[theorem]{Corollary}
\newtheorem{definition}[theorem]{Definition}
\newtheorem{example}{Example}

\newtheorem{lemma}[theorem]{Lemma}

\begin{document}

\abovedisplayskip=8mm
\abovedisplayshortskip=8mm
\belowdisplayskip=8mm
\belowdisplayshortskip=8mm



\pagenumbering{gobble}

\title{Limiting behaviour of Branching Processes and Online Social Networks}
\author{Khushboo Agarwal}
\date{2023}

\rollnum{17I190007} 

\iitbdegree{Doctor of Philosophy}

\reporttype{}

\department{Industrial Engineering and Operations Research}

\setguide{Prof. Veeraruna Kavitha}

\maketitle


\begin{dedication}
\large{\textit{Dedicated to my beloved parents and brother.}}
\end{dedication}


\chapter*{}
\thispagestyle{empty}
\vspace{-1in}
\begin{center}
{\Large  {\bf Thesis Approval}}
\end{center}
\vspace*{0.1in} \noindent This thesis entitled {\large \bf Limiting behaviour of Branching Processes and Online Social Networks} by {\bf \large
Khushboo Agarwal} is approved for the degree of  {\large \bf Doctor of Philosophy}.\\\\
\hspace*{4 in} Examiners:\\\\
\hspace*{3.5 in} \ldots\ldots \ldots \ldots \ldots \ldots\ldots
\ldots \ldots \ldots\ldots\\\\
\hspace*{3.5 in} \ldots\ldots \ldots \ldots \ldots \ldots\ldots
\ldots \ldots \ldots\ldots\\\\
\hspace*{3.5 in} \ldots\ldots \ldots \ldots \ldots \ldots\ldots
\ldots \ldots \ldots\ldots\\\\
  \hspace*{3.5in} \ldots\ldots \ldots \ldots \ldots \ldots\ldots
\ldots \ldots \ldots\ldots\\\\
 Supervisor:\hspace*{3.5 in}  Chairperson:\\\\\\
\ldots\ldots \ldots \ldots \ldots \ldots\ldots
\ldots \ldots \ldots\ldots
\hspace*{1.0 in} \ldots\ldots \ldots \ldots \ldots \ldots\ldots
\ldots \ldots \ldots\ldots\\\\\\

\noindent
Date: \ldots\ldots \ldots \ldots\\\\
Place: \ldots\ldots \ldots \ldots


\clearpage
\thispagestyle{empty}

\begin{center}
\Large  {\bf Declaration }
\end{center}
\vspace{-6in}
I declare that this written submission represents my ideas in my own words and where others ideas or words have been included, I have adequately cited and referenced the original sources. I also declare that I have adhered to all principles of academic honesty and integrity and have not misrepresented or fabricated or falsified any idea/data/fact/source in my submission. I understand that any violation of the above will be cause for disciplinary action by the Institute and can also evoke penal action from the sources which have thus not been properly cited or from whom proper permission has not been taken when needed.
\vspace{0.5in}


%
 \begin{table}[h]
 \begin{flushleft}

\vspace{-3.2in} 
 \begin{tabular}{ccccc}
 \rule[5ex]{0pt}{-10ex}&& Date: && \\ 
 \end{tabular}
\end{flushleft}

\vspace{-0.5in} 
\begin{flushright}
 \begin{tabular}{ccccc}
 
 \hline 	\rule[5ex]{0pt}{-10ex}&& Khushboo Agarwal&& \\ 
 \rule[5ex]{0pt}{-10ex}&& Roll No. 17I190007&& \\ \\
 \end{tabular}
\end{flushright}
\end{table}

\pagebreak

\joiningdate{07/07/2017}



\clearpage
\pagenumbering{roman}
\begin{abstract}
  \renewcommand{\thepage}{\roman{page}} \setcounter{page}{1}

The literature considers multi-type Markov branching processes (BPs), where the offspring distribution depends only on the living (current) population. In the thesis, we analyse the total-current population-dependent BPs where the offspring distribution can also depend on the total (dead and living) population. Such a generalization is inspired by the need to accurately model content propagation over online social networks (OSNs). The key question investigated is the time-asymptotic proportion of the populations, which translates to the proportional visibility of the posts on the OSN. We provide the answer using a stochastic approximation technique, which has not been used in the existing BP literature. The analysis is derived using a non-trivial autonomous measurable ODE. Interestingly, we prove the possibility of a new limiting behaviour for the stochastic trajectory, named as hovering around. Such a result is not just new to the theory of BPs but also to the stochastic approximation based literature. 

After analysing the general setup, we explore three new variants of BPs. In the first variant, any living individual of a population can attack and acquire the living individuals of the other population, in addition to producing its offspring. Secondly, the individuals can die due to abnormal circumstances, and not just at the completion of their lifetimes. In another BP, the expected number of offspring decreases as the total-population increases, leading to the saturation of the total-population. 

Such variants aid in analysing unexplored aspects of content propagation over OSNs. We study the competition in advertisement posts for similar products via BP with attack and acquisition. The control of fake-post propagation, while not affecting the sharing of real-post, is carried out using BP with different death types. Further, it is observed that the sharing of any post eventually stops, a phenomenon which we attribute to re-forwarding the post and capture using the saturated BP. 

Lastly, we also designed a participation (mean-field) game where the OSN lures the users with a reward-based scheme to provide their opinion about the actuality of the post (fake or real). The users can be adversarial or exhibit different levels of interest in providing their opinions. We propose an algorithm for the OSN that leads to the desired level of correct identification of posts by the users at Nash equilibrium.

\noindent \textbf{Key words:} Population-dependent Branching Process, Total Population, Proportion, Attack, Stochastic Approximation, Online Social Network, Viral Competing Markets, Re-forwarding, Fake-post detection, Crowd signals

\end{abstract}

\tableofcontents
\listoftables
\listoffigures



%
%



\setlength{\parskip}{2.5mm}
\titlespacing{\chapter}{0cm}{55mm}{10mm}
\titleformat{\chapter}[display]
  {\normalfont\huge\bfseries\centering}
  {\chaptertitlename\ \thechapter}{20pt}{\Huge}
  
  \titlespacing*{\section}
  {0pt}{8mm}{8mm}
  \titlespacing*{\subsection}
  {0pt}{8mm}{8mm}
\pagebreak
\pagenumbering{arabic}

\makeatletter
\def\cleardoublepage{\clearpage\if@twoside \ifodd\c@page\else
	\hbox{}
	\vspace*{\fill}
	\begin{center}
		This page was intentionally left blank.
	\end{center}
	\vspace{\fill}
	\thispagestyle{empty}
	\newpage
	\if@twocolumn\hbox{}\newpage\fi\fi\fi}
\makeatother

\newcommand{\rev}[1]{{\color{blue}#1}}

\newcommand{\revg}[1]{{\color{black}#1}}
\newcommand{\reva}[1]{{\color{black}#1}}
\newcommand{\revr}[1]{{\color{black}#1}}

\newcommand{\ifnonauto}[2]{#1}

\newcommand{\DetailK}[1]{} 

\newcommand{\tcprocess}{total-current population-dependent BP }
\newcommand{\tcprocessnospace}{total-current population-dependent BP}
\newcommand{\hide}[1]{}

\newcommand{\psiL}{\psi^c_{\mbox{\tiny o}}}
\newcommand{\psiaL}{\psi^a_{\mbox{\tiny o}}}

\newcommand{\bL}{\beta^c_{\mbox{\tiny o}}}
\newcommand{\baL}{\beta^a_{\mbox{\tiny o}}}

\newcommand{\tL}{\theta^c_{\mbox{\tiny o}}}
\newcommand{\taL}{\theta^a_{\mbox{\tiny o}}}

\newcommand{\sa}{z}
\newcommand{\Sa}{Z}

\newcommand{\betana}{\beta^o_{\mbox{{\footnotesize na}}}}

\newcommand{\N}{\mathcal{N}}
\newcommand{\I}{I_{\theta/\psi}}
\newcommand{\eop}{\hfill{$\square$}}
\newcommand{\nto}{\nrightarrow}
\newcommand{\Om}{\Phi}
\newcommand{\om}{\phi}
 \newcommand{\ups}{{\mbox{\small ${\Upsilon}$}}}
 \newcommand{\Ups}{{\mathbf \Upsilon}}
\newcommand{\Bin}{{\cal B}}
\newcommand{\tp }{\tau^+}
\newcommand{\tm}{\tau^-}
\newcommand{\up}{\uparrow}
\newcommand{\offs}{\Gamma}
\newcommand{\down}{\hspace{.15mm}\downarrow}
\newcommand{\beq}{\begin{eqnarray*}}
\newcommand{\eeq}{\end{eqnarray*}}
\newcommand{\Cx}{C^x}
\newcommand{\Cy}{C^y}
\newcommand{\cx}{c^x}
\newcommand{\cy}{c^y}
\newcommand{\cM}{{\cal E}}
\newcommand{\bcx}{{\overline c^x}}
\newcommand{\bcy}{{\overline c^y}}
\newcommand{\bax}{{\overline a^x}}
\newcommand{\Ax}{A^x}
\newcommand{\Ay}{A^y}
\newcommand{\ax}{a^x}
\newcommand{\ay}{a^y}
\newcommand{\wm}{\widetilde{e}}
\newcommand{\cA}{{\mathbf A}}
\newcommand{\cR}{{\mathbf S}}
\newcommand{\cRo}{{\mathbf R}_e}
\newcommand{\tc}{\theta^c}
\newcommand{\ta}{\theta^a}
\newcommand{\pc}{\psi^c}
\newcommand{\pa}{\psi^a}
\newcommand{\Tc}{\Theta^c}
\newcommand{\Ta}{\Theta^a}
\newcommand{\Pc}{\Psi^c}
\newcommand{\Pa}{\Psi^a}
\newcommand{\bc}{\beta^c}
\newcommand{\ba}{\beta^a}
\newcommand{\Bc}{\mathrm{B}^c}
\newcommand{\Ba}{\mathrm{B}^a}
\newcommand{\Beta}{\mathrm{B}}
\newcommand{\dist}{ d_{st}}
\newcommand{\st}{{\cal T}}
\newcommand{\bstar}{\beta^*}
\newcommand{\minf}{m^\infty}
\newcommand{\ueps}{\overline{\varepsilon} }
\newcommand{\leps}{\underline{\varepsilon}}
\newcommand{\polya}{\mbox{P\'{o}lya} }
\newcommand{\q}{\mathbf{q}}
\newcommand{\ga}{\mathbf{g}}
\newcommand{\gna}{\bm{\varrho}}

\newcommand{\newbc}{f_{\bc}^\infty(\bc)}
\newcommand{\cS}{\mathbf{D}_b}
\newcommand{\cD}{\mathbf{D}}
\newcommand{\cB}{\mathbf{B}}

\newcommand{\bpam}{e}
\newcommand{\bpaum}{\overline{e}}
\newcommand{\bpalm}{\underline{e}}
\newcommand{\propum}{\overline{m}}
\newcommand{\proplm}{\underline{m}}
\newcommand{\propwm}{m^\infty}
\newcommand{\propm}{m}
\renewcommand\thefootnote{\arabic{footnote}}

\newcommand\overlinebelow[1]{\stackunder[1.2pt]{$#1$}{\rule{1.2ex}{.075ex}}}
\def\lc{\left\lceil}   
\def\rc{\right\rceil}

\newcommand*\widefbox[1]{\fbox{\hspace{0em}#1\hspace{0em}}}
\newcommand{\floor}[1]{\lfloor #1 \rfloor}

\newcommand{\TR}[2]{#2}  

\newcommand{\old}[1]{}
\newcommand{\G}{{\cal G}}
\newcommand{\gzero}{{\mathcal G}_{0}}
\newcommand{\gone}{{\mathcal G}_{1}}
\newcommand{\gtwo}{{\mathcal G}_{2}}
\newcommand{\muzero}{\mu_{0}}
\newcommand{\muone}{\mu_{1}}
\newcommand{\mutwo}{\mu_{2}}
\newcommand{\mua}{\mu_{a}}
\newcommand{\ce}{C_e}
\newcommand{\cp}{Q_p}
\newcommand{\cnp}{Q_{np}}
\newcommand{\bmu}{\bm{\mu}}
\newcommand{\td}{(\theta, \delta)\mbox{-success}}
\renewcommand{\S}{{\cal S}}
\newcommand{\udelta}{\underline{\Delta}}
\newcommand{\na}{\mbox{non-adversarial}}
\newcommand{\w}{\max\left\{\frac{1}{1-\mua}, \frac{1}{\Delta^a\theta_a} \right\}}
\newcommand{\fone}{\frac{\delta_a - \eta^*_{\widetilde{\theta}}\alpha_R}{\delta_a(1-\mua-\eta^*_{\widetilde{\theta}})}}
\newcommand{\xdoubleF}{\frac{1}{1-\alpha_F}\left(1-\mua-\frac{1}{cw\alpha_R \Delta^a}\right)}
\newcommand{\RAI}{{\mathcal R}_{AI}}

\newcommand{\cd}{{\cal D}}

\newpage
\pagebreak
\cleardoublepage
\chapter{Introduction}
Branching processes (BPs) are stochastic processes that model populations' evolution. Since their introduction to study the surname extinction problem, many variants of BPs have been analyzed to understand various exciting problems in multiple domains. This thesis analyses a broad class of BPs that aid in investigating various unexplored aspects of online social networks (OSNs). Therefore, the contribution of this thesis is not limited to the theory of BPs; it also provides interesting insights into content propagation over OSNs. Towards the end, inspired by our findings on the BP-based study of OSNs, we also design a mean-field game among users of the OSN, induced by a reward-based scheme, to nudge users towards correctly identifying the actuality of posts (fake posts as fake and real posts as real) in the presence of adversarial users and other user behaviours.

In the space of BPs, we consider single or two-type BPs, where the dynamics progress in continuous-time and are Markovian. The literature generally assumes the distribution of the number of offspring depends on the current (living) population (for example, \cite{klebaner1989geometric, jagers1997coupling}). We consider that the distribution of the number of offspring depends on the current population and/or total (dead and alive) population. Such an extension is motivated by the need to appropriately capture crucial aspect (namely, re-forwarding) of content propagation over OSNs. Further, in the literature, the current-population dependent mean matrix is assumed to converge to a deterministic mean matrix, leading to a unique limit point, see \cite{klebaner1989geometric, jagers1997coupling}. In our case, the limit of the mean matrix is proportion-dependent and thus can depend on the underlying sample path, possibly leading to multiple limit matrices. Such an assumption requires a different treatment and significantly generalizes the existing models.

Furthermore, in classical literature, the BPs are analyzed in 
super-critical, critical, or sub-critical regimes (see \cite{athreya2004branching}). Limited literature considers BPs transitioning from the super-critical to the sub-critical regime, as population size grows. The authors in \cite{jagers2011population} analyze the BP where the dynamics fluctuate between the two regimes as the current population size fluctuates. While, in \cite{hautphenne2022fluid}, simple total population-dependent birth-death based dynamics are analyzed where the process transitions from super-to-sub critical regime. In this thesis, the two-type variants are analyzed in 
(appropriately defined) super-critical regime. Further, we study a generalized single-type BP that transitions from the super-to-sub critical regime and where the dynamics are not just birth-death type, however, each parent can produce random total population-dependent offspring before dying.

Traditionally, the literature adopts the martingale-based approach to analyze the BPs. We use the ordinary differential equation (ODE) based stochastic-approximation (SA) technique for the following two primary objectives:

\noindent $\bullet$ to derive a deterministic trajectory that approximates the random dynamics over any finite time window - this translates to deriving the approximate deterministic curves for the contents propagating over OSNs.
    
      $\rightarrow$ towards this, we derive an appropriate multi-dimensional first-order autonomous  ODE with a measurable right-hand side. We show that certain normalized trajectories of the embedded chain almost surely converge to the ODE solution uniformly over any finite time window as time progresses.

\noindent $\bullet$ if two population-types are considered, then to derive the time-asymptotic proportion of the populations (which we briefly refer to as `proportion') - for example, this represents the time-asymptotic visibility of the two (competitive or cooperative) contents on the OSNs.

    $\rightarrow$ under finite second-moment conditions, we show that with a certain probability, the limit proportion either converges to the equilibrium points (attractor and saddle points) or infinitely often enters every neighbourhood and exits some neighbourhood of a saddle point of the derived ODE.

In the above, the possible emergence of the latter limiting behaviour, which we named as \textit{hovering around}, is new to both SA and BP-based literature. We do not show that hovering around occurs with positive probability; nonetheless, the possibility of such a new behaviour is exciting and worthy of investigation in future.

We also prescribe and illustrate a procedure to derive the attractor and saddle sets of the derived ODE using a one-dimensional autonomous proportion-dependent ODE with 
(possibly) measurable right-hand side.

Before proceeding further,  we briefly describe how BPs can be used to capture the basic features of content propagation on OSNs. In general, OSNs are usually flooded with a variety of content, which is shared (again) by the recipients and thus may get viral (i.e., the number of copies of the post grows significantly with time). Further, after reading the post, the user most likely loses interest in it forever. Thus, reading the post is analogous to death, while a new share by a user is analogous to offspring. Furthermore, unread and total (read $+$ unread) copies are analogous to the current and total population, respectively. 

\section{Contributions}
We now discuss the significant contributions of this thesis, which are three new total-current population-dependent BPs and their applications in OSNs, and the participation game for fake-post detection on OSNs. We briefly introduce the BPs and highlight their key distinguishing features. We also specify how they contribute to analyzing different aspects of OSNs and their respective important results. Further, motivated by the BP-based insights, we consider the participation game.

\noindent \textbf{1. Branching process with attack (BPA):} 
Unlike prey-predator BPs (\cite{coffey1991galton}), in BPA, \textit{any individual of any population-type can attack the other population type, acquire the attacked individuals, and also produce offspring of its type}. After deriving a thorough analysis of BPA, we analyze \textbf{viral competing markets} on OSNs using that analysis. Generally, there are multiple (commercial) posts on the OSN, many of which might compete with each other. Such competing contents are always at risk of losing their chances. When a user prefers one post over the other, the liked post snatches away (attacks and acquires) the opportunities of the other post depending upon the popularity and/or the freshness of the two contents. One of the exciting results in this direction is that \textit{the post of a less influential content provider can gain more visibility (in the limit) than the post from the more influential content provider if the content of the former appeals more to the users}.

\noindent \textbf{2. BP with unnatural deaths:} The literature majorly considers BPs where individuals die naturally after completing their lifetime. However, due to unfavourable circumstances, their reproductive capacities might be affected, and in fact, they can die in extreme situations. Limited literature models unnatural deaths due to competition and cooperation (see \cite{BPwithinteraction, etheridge2013conditioning, ojeda2020branching}). However, we study a generalized BP, which captures natural death and a variety of unnatural deaths. For such BPs, the above two results are proved.

Using the results of the above-mentioned BP, we design a \textbf{robust control for fake-post propagation over OSNs against adversaries}, while negligibly affecting the authentic/real post propagation --- we model the post propagation process with robust control using a BP with unnatural deaths. Towards this, a warning mechanism based on crowd-signals was proposed in \cite{kapsikar2020controlling}, where all users actively declare the post as real or fake. Here, we consider a more realistic framework where \textit{users exhibit different adversarial or non-cooperative behaviour}: (i) they can independently decide whether to provide their response, (ii) they can choose not to consider the warning signal while providing the response, and (iii) they can be \textit{real-coloring adversaries who deliberately declare any post as real}. In general, adversaries can be smart in declaring the posts opposite to their actuality. However, real-coloring adversaries outnumber the smart ones, as the former are the ones who are not well-informed about the actuality of the posts but still intend to harm the system. At first, we compare and show that the existing warning mechanism significantly under-performs in the presence of adversaries. Then, we design new mechanisms that remarkably perform better than the existing mechanism by cleverly eliminating the influence of the responses of the adversaries.

\noindent \textbf{3. Saturated total-population dependent BP (STP-BP):}
Unlike so far discussed super-critical BPs, we also consider a single-type total population-dependent BP, which \textit{permanently shifts from super-to-sub critical regime as time progresses}. Here, we show that the total population converges and saturates to a limit as time progresses. Further, contrary to the known exponential growth in other existing BP models, the current population grows exponentially initially and then declines to $0$. 

Using STP-BP, we analyze \textbf{saturated viral markets} on the OSN. Note that when the users continually forward an interesting post, it leads to an increase in the \textit{re-forwarding of the post} to some of the previous recipients. Consequently, the effective forwards (after deleting the re-forwards) reduce, eventually leading to the saturation of the total number of copies. Notably, \textit{we obtain deterministic approximate trajectories for the unread and total copies, which depend only on four parameters related to the network characteristics}. Further, we provide expressions for the peak unread copies, maximum outreach and the life span of the post.

\noindent \textbf{4. Single out fake-posts via participation game:} In the robust fake-post detection via BP-based approach, we show that the crowd-based approach can successfully distinguish between real and fake posts despite adversaries in the system if a sufficient fraction of users provide their responses. However, motivating the users to provide their responses is challenging, even more so in the presence of adversarial users (\cite{freire2021fake}). Thus, towards the end of the thesis, we design a \textit{(mean-field) game} where users of the OSN are lured by a \textit{reward-based scheme} to provide the binary (real/fake) signals such that the OSN achieves \textit{$(\theta, \delta)$-level of actuality identification (AI)} - not more than $\delta$ fraction of non-adversarial users incorrectly judge the real post, and at least $\theta$ fraction of non-adversarial users identify the fake-post as fake. An appropriate warning mechanism is proposed to nudge the users such that the \textit{resultant game has at least one Nash Equilibrium (NE) achieving AI}. \textit{We also identify the conditions under which all NEs achieve AI}.

Thus, the thesis contributes towards three different areas - branching processes (BPs), stochastic approximation (SA) and online social networks (OSNs), which we summarize below:

\subsection{Towards BP-literature} We analyze population-dependent BPs whose key differentiating features are total and current population-size dependent offspring, negative offspring (to model attack), proportion dependent functions for the expected number of offspring, even at the limit, unnatural deaths of individuals, and the transition from super-to-sub critical regime. The main focus is to derive the time-asymptotic proportion of the populations and deterministic trajectories that can track the stochastic BP trajectories.

\subsection{Towards SA-literature} The analysis of BPs is derived using the stochastic approximation technique. While deriving their limits, the possibility of a new limiting behaviour, which we name `hovering around', is observed. Such behaviour is new to both BP and SA literature. We also dealt with some discontinuous dynamics.

\subsection{Towards OSNs} Using the derived analysis of BPs, we gain new insights about content propagation over OSNs. We study the effect of competition and re-sharing over the success of post-propagation. Further, we devise new robust mechanisms to control fake-post propagation in the presence of adversaries, using users' responses. To overcome the difficulty of obtaining users' responses, we design an appropriate participation (mean-field) game.

\section{Thesis outline} 
The subject matter of the thesis is presented in the following chapters:
\begin{enumerate}[label = (\roman*)]
    \item In Chapter \ref{ch:basics}, we provide the preliminaries and, importantly, the generalized results for possibly discontinuous ODEs and SA-based schemes that typically arise while dealing with BPs. We also provide the new results that we derive towards the respective domains.
    \item In Chapter \ref{ch:journal1}, we describe and analyze the two-type total-current population dependent BP in continuous-time and Markovian framework. We also study the BP with attack, which facilitates the analysis of viral competing markets over OSNs.
    \item In Chapter \ref{ch:journal2}, we design robust mechanisms (guided by users' responses) to control fake-post propagation over OSNs against real-coloring adversaries. The analysis uses a new BP with both natural and unnatural deaths.
    \item In Chapter \ref{ch:STPBP}, saturated viral markets are analyzed using the newly proposed saturated total population-dependent BP.
    \item At last, in Chapter \ref{ch:MFG}, we design a participation mean-field game to motivate sufficiently many users to provide their responses to ultimately single out fake posts at Nash equilibrium.
    \item In Chapter \ref{ch:summary}, the conclusions are provided and future possibilities of our work are also enlisted. The chapter-wise proofs are provided in Appendices \ref{appendix_journal1}-\ref{appendix_MFG}  at the end of the thesis. 
\end{enumerate}

\chapter{Preliminaries and New results}\label{ch:basics}
The results of this thesis have two flavours: first, contributions towards applications like branching processes (BPs) and online social networks (OSNs), and second, contributions towards commonly used tools like ordinary differential equation (ODE) and stochastic approximation (SA) techniques. In this chapter, we first provide brief understanding of the existing tools and results, and then discuss the new results for ODE and SA in the general forms. The new results pertaining to these techniques are stated in this chapter in such a way that one can readily analyze future applications; the new sophisticated tools can facilitate the analysis of future problems for which existing tools may be insufficient. Along the way, we also prescribe the procedure to analyze the new variants of BPs that we explore in the coming chapters. 

 \section{Branching processes (BPs)}\label{sec_BP}
The branching process (BP) is a stochastic process which was introduced initially by Galton and Watson in $1874$ to study the surnames extinction problem (see \cite{athreya2004branching, harris1963theory}). Since then, the BPs have been used to analyze several phenomena in cell proliferation, genetics, epidemiology, molecular biology, finance, information spreading over online social networks and many more. The basic idea behind the concept is that there is a family tree where each individual has the same probability distribution for the number of offspring. 

To be more precise, consider a population type, and denote it by $x$. Let initially there be one individual in the population. Then, let $\Cx_n$ be the number of living (current) individuals at $n$-th generation, for $n \geq 1$. Each individual lives for one generation. At the end of $n$-th generation, all the individuals die together, and just before dying, assume that $i$-th individual from $n$-th generation produces $\offs_{n, i}$ number of offspring. Thus, the population-size at $(n+1)$-th generation evolves as follows (see \cite{athreya2004branching}):
\begin{align}\label{eqn_pop_evolve_discreteBP}
\Cx_{n+1} = \sum_{i = 1}^{\Cx_n} \offs_{n, i}.
\end{align}
The offspring generated are assumed to be independent and identically distributed (i.i.d.) across generations and individuals. This property leads to the self-similarity of the process starting from any individual.

Let $p_k$ denote the probability that a parent in $n$-th generation produces $k$ offspring in $(n+1)$-th generation such that $\sum_{k = 0}^\infty p_k = 1$. The transition probabilities for the Markov chain are then defined as follows:
\begin{align}\label{eqn_prop_discreteBP}
P(\Cx_{n+1} = k|\Cx_n = i) = 
    \begin{cases} 
      p_k^{*i} &\mbox{ if } i \geq 1, k \geq 0, \\
      \delta_{0i} &\mbox{ if } i = 0, k \geq 0,
   \end{cases}
\end{align}
where $\{p_k^{*i}, k = 0, 1, \dots\}$ is the i-fold convolution of $\{p_k\}_{k \geq 0}$ and $\delta_{ik}$ is Kronecker delta. Since all the transitions occur after one generation and not between two consecutive generations, the underlying process is a discrete-time BP.

\subsection{Classification of BPs}
The dynamics in \eqref{eqn_pop_evolve_discreteBP}-\eqref{eqn_prop_discreteBP} are the simplest. Since then, many more complicated variants of BPs have been studied, which can be majorly classified as follows depending upon: 

\noindent \textbf{Time}: the dynamics can progress in \textit{discrete or continuous-time}. The continuous-time BPs are discussed in sub-section \ref{sub_sec_two_typeBP}.
 
\noindent \textbf{Number of population-types}: the interactions can involve \textit{single or multiple types}.
 
\noindent \textbf{Offspring distribution}: a BP is said to be  \textit{(current) population-dependent BP} if the offspring distribution depends on the number of living individuals (current population-size). Otherwise, the process is called \textit{population-independent BP}. 
 
\noindent \textbf{Lifetime distribution}: if the lifetime of any individual is exponentially distributed, then the BP is called a \textit{Markovian BP}. However, if the probability that any individual dies depends on age, then the resulting BP is an \textit{age-dependent BP}.
 
\noindent \textbf{Mean number of offspring}: BPs are also categorized depending upon the mean number of offspring produced (assumed to be finite) for different types of BPs. Here, we discuss the population-independent BP, while the discussion for the population-dependent variants is deferred to sub-section \ref{sub_sec_two_typeBP}.

Consider a population-independent BP with a single population-type. Say each parent produces $\offs$ number of random offspring (in discrete or continuous framework), as in \eqref{eqn_pop_evolve_discreteBP}. Let $m := E[\offs]$ be the expected number of offspring. Then, the process is in \textit{sub-critical regime} if $m < 1$, \textit{critical regime} if $m = 1$, or in \textit{super-critical regime} if $m > 1$ (see \cite[Chapter I, III]{athreya2004branching} for discrete-time and continuous-time BPs respectively). Our interest lies only in super or sub-critical BPs. 

Consider now the multi-type population-independent discrete-time BP with $d$-types of populations; we shall discuss multi-type population-independent (and dependent) continuous time BPs in sub-section \ref{sub_sec_two_typeBP}. Say $i$-th population type has $z_0^i$ number of individuals at $n = 0$. Let $Z_n^i$ be the population-size of the $i$-th type population at the $n$-th generation. At $(n+1)$-th generation, the population evolves as follows (see \cite{klebaner1989geometric}):
\begin{align}
    Z_{n+1}^j = \sum_{i=1}^d \sum_{\nu = 1}^{Z_n^i} \offs_{ij\nu}^{(n)}, \mbox{ where}
\end{align}$\offs_{ij\nu}^{(n)}$ is the number of $j$-type offspring produced by $\nu$-th parent of $i$-type, where $\nu \in \{1, \dots, Z_n^i\}$. Here, $\offs_{ij\nu}^{(n)}$ are i.i.d. as $\offs_{ij1}^{(n)}$, and $(\offs_{ij1}^{(1)})_{j \in \{1, \dots, d\}}$ are i.i.d. for any $i = 1, \dots, d$. Now, define the mean matrix, $M := [E(\offs_{ij1}^{(1)})]_{i, j \in \{1, \dots, d\}}$.  Then, according to \cite[Chapter V]{athreya2004branching}, the underlying process is said to be a \textit{super(sub)-critical BP} if the largest eigenvalue (in modulus), say $\rho$, of $M$ is strictly larger (smaller) than $1$.

Now, \textit{two of the crucial and commonly asked questions in BPs are about the extinction probability and the growth rate of the populations}. The answer to these questions depends on the criticality parameter. In the single-type population-independent BP, the population gets extinct, i.e., the population size converges to $0$ as time progresses, with probability (w.p.) $1$ in the sub-critical regime. While in the super-critical regime, the population exhibits \underline{\textit{dichotomy}}: the population-size either grows significantly large with non-zero probability or gets extinct (see \cite{athreya2004branching}). The former event is said to be \underline{explosion of the population}. For discrete-time BP, the rate of explosion is $m^n$. 

For the multi-type population-independent BP, there is a further division into:
(i)\textit{ irreducible} and (ii) \textit{decomposable} BP. The process is irreducible if a parent of each type generates offspring of all types. Otherwise, if a parent of $i$-type generates offspring of only $j$-types for $j \geq i$, then the resultant is a decomposable BP. 

Formally, if the mean matrix $M$ is positive regular (irreducible BP) and non-singular, then the process gets extinct w.p. $1$ in the sub-critical regime. Otherwise, the population again exhibits dichotomy: it either explodes at rate $\rho^n$ with non-zero probability or becomes extinct (see \cite{athreya2004branching}). One can refer to \cite{kesten1967limit} for the decomposable BP.

So far, we discussed the population-independent BPs. However, this thesis focuses on the new variants of multi-type population-dependent continuous-time Markov BPs. Next, we briefly explain a known classical variant as a first step towards this.

\subsection{Current population-dependent BP} \label{sub_sec_two_typeBP}
In this sub-section, we shall discuss two-type current population-dependent continuous-time Markovian BP. The dynamics can be easily reduced for the single-type variant and generalised for multi-type BP as well.

Consider two types of populations, denoted by $x$ and $y$. Let $\cx_0, \cy_0$ be their respective initial sizes. Let $\Cx(t)$ and $\Cy(t)$ be the \textit{current population} sizes, i.e., the number of living individuals of $x$ and $y$-type populations respectively at time $t$. Define $\Om(t) := (\Cx(t), \Cy(t))$ as the tuple of population sizes.

The lifetime of any individual of any type is exponentially distributed with parameter $0 < \lambda < \infty$, i.e., \textit{we consider Markovian BPs}, and the process $\Phi$ is a continuous-time jump process. The time instance at which an individual completes its lifetime is referred to as its `death' time.  Consider any $n \geq 1$.  
Let $\tau_n$ be the death time of the $n$-th individual (of any type) dying among the living population; let $\tau_0 := 0$. Let $\Cx_n := \lim_{t \uparrow {\tau}_n} \Cx(t)$ be the current-population size of $x$-type population, just before ${\tau}_n$. Similarly, define $\Cy_n$ and let $S^c_n :=  \Cx_n + \Cy_n$ be the sum current population just before $\tau_n$.

Once the population gets extinct, no births are possible, therefore, any state $\om := (\cx, \cy)$ with $\cx + \cy = 0$ is an absorbing state. Then, $\nu_e := \inf \{n : \Cx_n = 0\}$ represents the epoch at which the extinction occurs, with the usual convention that $\nu_e = \infty$, when $S_n^c > 0$ for all $n$. As is usually done, we extend the embedded process after extinction: define $\Om_n := \Om_{\nu_e}$ and $\tau_{n} :=\tau_{\nu_e}$, for all $n \geq \nu_e$, when $\nu_e < \infty$ (see \cite{athreya2004branching}). Observe here that no two individuals can die at the same time, as for each $n$, $P(\tau_{n+1} - \tau_n > 0) = 1$, since $(\tau_{n+1} - \tau_n)$ is exponentially distributed.

Any offspring is produced only at the death time by an individual (\cite{athreya2004branching, klebaner1989geometric}). No offspring will be produced in between two consecutive death times. Let $\Gamma_{ij}(\om)$ denote the (random) number of $j$-type offspring produced by $i$-type individual at its death time when the population state is given by $\phi$ for $i, j \in \{x, y\}$. 
If an $x$-type parent dies at $\tau_n$, the system for any $t \in [\tau_n, \tau_{n+1})$ (in case $\tau_{n} = \tau_{\nu_e}$, then for all $t \ge  \tau_n$),   can be described as:
\begin{equation}
\begin{aligned}
C^x(t) = C^x_n  + \offs_{xx}(\Om_n) - 1, \mbox{ and }
C^y(t) = C^y_n + \offs_{xy}(\Om_n).
\end{aligned}
\end{equation}
Similar evolution happens when a $y$-type parent dies.
Basically, the sizes of $i$ and $j$-type populations change by $\offs_{ii}(\Om_n)$ and $\offs_{ij}(\Om_n)$ respectively\footnote{For each $i, j$, the distribution of $\offs_{ij}(\Om_n)$ depends on the population sizes $(\Om_n)$, and not on the value of the epoch, $\tau_n$.}, and  the current size (not the total size) of $i$-type reduces by $1$ due to death. Thus, the underlying process is a continuous-time jump process. Also, observe\footnote{Define $\chi_k^x$ be the lifetime of the $k$-th $x$-type individual for $0 \leq k \leq \Cx_n$; similarly define $\chi_k^y$ for $0 \leq k \leq \Cy_n$. Then, $\chi_k^x$ and $\chi_k^y$ are exponentially distributed with parameter $\lambda$. This implies that (recall $\chi_k^x$ are independent):
$
P\left(\chi_n^x = \min\left\{\chi_1^x, \dots, \chi_{\Cx_n}^x, \chi_1^y, \dots, \chi_{\Cy_n}^y\right\}\right) = \frac{\Cx_n}{\Cx_n + \Cy_n} \mbox{ for any } 0 \leq k \leq \Cx_n.
$} that the probability of an $x$-type parent dying at time $\tau_n$, conditioned on $\sigma\{\Cx_n, \Cy_n\}$ is $\frac{\Cx_n}{\Cx_n + \Cy_n} =: B_n^c$. Thus, the probability that a $y$-type parent dies at time $\tau_n$ is $1-\Bc_n$. Observe \textit{$\Bc_n$ is the proportion of $x$-type population among the current population}.

Now, similar to discrete-time variants of BPs, we will discuss the asymptotic behaviour of BPs depending on the mean number of offspring produced. A single-type population-independent BP is said to be in super(sub)-critical regime if $m > 1$ (or $<1$). Here, again analogous to discrete-time variant, the BP exhibits dichotomy, and the population-size explodes exponentially at rate $\lambda(m-1)$ (see \cite[Chapeter III, Section 7, Theorems 1, 2]{athreya2004branching}). 

Consider a $2$-type population-dependent BP with $M(\om) := [E(\offs_{ij}(\om))]_{i, j \in \{x,y\}}$ as the corresponding population-dependent mean matrix. Let $M^\infty$ be the mean matrix of an appropriate super-critical multi-type population-independent BP, i.e., the one which has the largest eigenvalue strictly larger than $1$. Then, the author in \cite{jagers1997coupling} defines that such a BP is in super-critical regime if (recall, $\om = (\cx, \cy)$ is a realisation vector of population-sizes and $s^c = \cx+\cy$):
\begin{align}\label{eqn_super_critical_pop_dep_BP}
    ||M(\om) - M^\infty|| \to 0, \mbox{ as } s^c \to \infty,
\end{align}where the convergence is under the usual topology of matrices.
Similar notion is considered in \cite{klebaner1989geometric} for discrete-time population-dependent BPs, and named as \textit{near-super criticality}.
In any case, it is shown that the limiting behaviour is similar to the limiting population-independent BP.

\subsection{New Variants of BPs}\label{sub_sec_newBP}
Here, we highlight the key features of the new variants of (single or two-type) population-dependent continuous-time Markov BPs considered in this thesis:
\begin{enumerate}
    \item The offspring distribution can depend only on the number of living (current) individuals, or both on current population and the total individuals born so far (total-population); see Chapters \ref{ch:journal1}-\ref{ch:STPBP}.
    \item Each parent of any type produces a non-negative number of offspring of its type. Interestingly, a parent can produce non-negative or negative offspring of other type, depending upon the population-size; see Chapters \ref{ch:journal1}-\ref{ch:STPBP}.
    \item Any individual in a population can die unnaturally, see Chapter \ref{ch:journal2}. Such deaths can be due to competition or cooperation, as considered in \cite{etheridge2013conditioning, ojeda2020branching}, or due to changes in the physical environment (for example, temperature change, natural calamities, invasion of a new virus, etc.). 
    \item Classical BPs assume that the population-dependent mean matrix converges to constant mean matrix, as population-size grows (see \eqref{eqn_super_critical_pop_dep_BP}). We consider proportion dependent limit mean matrix (see \ref{a2} in Chapter \ref{ch:journal1} and \ref{a2_prop} in Chapter \ref{ch:journal2}).
    \item In Chapter \ref{ch:STPBP}, a single-type total population-dependent continuous-time Markov BP is explored where the process transitions from super-critical regime to sub-critical regime as the total population-size grows. Such a transition leads to the saturation of the total population-size.
\end{enumerate}

For all above BPs, our interest lies in deriving the approximate deterministic trajectories of the stochastic BP trajectory and the limits of the BPs. Towards this, unlike the well-known martingale approach (see \cite{athreya2004branching, harris1963theory}), we use the ordinary differential equation (ODE) based stochastic approximation (SA) technique. Next, we discuss existing and new results about the ODEs, which will lay the foundation for the coming discussion on SA based-method in Section \ref{sec_SA}. 

\section{Ordinary differential equations (ODEs)}\label{sec_ODE}
 In this section, we will discuss different types of ordinary differential equations (ODEs). We will study when the solution exists, under what conditions the solution is unique and also the time-asymptotic limits of the solutions of such ODEs. The initial discussion is inspired from \cite{piccinini2012ordinary, filippov2013differential, perko2013differential}. 
 
An \underline{ordinary differential equation} (ODE) is a relation among independent variable $t$, an unknown function $x(t)$ of that variable, and its derivatives. The general form of a $n$-dimensional ODE of first-order is given by:
 \begin{align}
     \dot{x_i} = \frac{d x_i}{dt} = f_i\left(t, x \right) \mbox{ for each } 1 \leq i \leq n,
\end{align}
where each $f_i$ is a real valued function of $(t, x) \in \mathbb{R}^+ \times \mathbb{R}^n$. The function $(x(t))_{t \geq 0}$ which satisfies the above equation is called the \underline{solution of the ODE}. If the function $f$ explicitly depends on time $t$, then the ODE is known as \underline{non-autonomous} ODE; else, the ODE is \underline{autonomous}. Now, at first, we will discuss existing types of ODEs and then present a new result for a special type of ODE that interests us.

\subsection{Existing ODEs}\label{subsec_existingODE}
\noindent \textbf{1. ODE with Lipschitz continuous Right hand sides:}
Consider the autonomous initial value problem (IVP):
\begin{align}\label{eqn_IVP}
    \dot{x} = f(x) \mbox{ and } x(t_0) = x_0, \mbox{ for } (t_0, x_0) \in \mathbb{R}\times \mathbb{R}^n,
\end{align}where the function $f(x) \in C^1(\mathbb{R}^n)$ and is Lipschitz continuous in $\mathbb{R}^n$. 

Then, by \cite[Theorem 3, Chapter 3]{perko2013differential}, a \underline{unique differentiable solution $x(t)$ exists} for all $t \geq 0$. Since the solution exists for all $t \ge 0$, the solution is a global solution. Further, the solution satisfies the following \underline{integral equation}:
\begin{align}\label{eqn_sol_ODE_integralform}
    x(t) = x_0 + \int_{s = t_0}^t f(x(s)) ds.
\end{align}
The reverse statement is also true, i.e., the solution of the above integral equation satisfies the ODE \eqref{eqn_IVP}. By \cite{piccinini2012ordinary}, the solution to \eqref{eqn_sol_ODE_integralform} is obtained by \underline{successive approxima}-\\ \underline{tion} as follows: (i) consider the constant function $x^{(0)}(t) \equiv x_0$, and (ii) define the function $x^{(m)}$ successively as follows:
    $$
         x^{(m)}(t) = x_0 + \int_{s= t_0}^t f\left( x^{(m-1)}(s)\right) ds, \mbox{ for } m \geq 1.
    $$
We will make use of the above approximation to numerically evaluate the ODE solution. 


\noindent \textbf{2. ODE with Continuous Right hand sides}:
Let us now consider the IVP:
\begin{align}\label{eqn_cont_ode}
    \dot{x} = f(t, x)  \mbox{ and } x(t_0) = x_0 , \mbox{ for }  (t_0, x_0) \in [a, b] \times  \mathbb{R}^n ,
\end{align}where $f$ is continuous, and not Lipschitz continuous as above. For such ODEs, the existence of the solution is guaranteed in local sense (see \cite[Chapter 3, Section 1.2]{piccinini2012ordinary}), and in global sense (for all $t \in [a,b]$, see \cite[Chapter 3, Section 1.3]{piccinini2012ordinary}) but it need not be unique. For example, consider the ODE $\dot{x} = \sqrt{|x|}$ with $x(t_0) = 0$; then, clearly $x \equiv 0$ and $x(t) = \frac{(t-t_0)^2}{4} \left(1_{t \geq t_0} - 1_{t \leq t_0}\right)$ are both  solutions for the said ODE, for all $t \geq 0$. 

In case, there are many solutions for the ODE, then each solution  can be bounded as in the  following (see \cite[Chapter 3, Corollary 2]{piccinini2012ordinary}). 
\begin{theorem}[Comparison Result] \label{thrm_comparison}
    Let $f:[a, b] \times \mathbb{R} \to \mathbb{R}$ be a continuous function and let $\alpha, \beta: [a, b] \times \mathbb{R} \to \mathbb{R}$ be two continuous and differentiable functions such that $\alpha < \beta$ and $\dot{\alpha} \leq f(t, \alpha(t)),  \dot{\beta} \geq f(t, \beta(t))$.     Then, every solution $x(t)$ of IVP $\dot{x} = f(t, x),$ $x(a) = x_0$ with $\alpha(a) < x_0 < \beta(a)$ satisfies $\alpha \leq x \leq \beta$ and exists in $[a, b]$. \eop
\end{theorem}
We shall use this comparison result later in Chapters 3, and 5 (see Lemma \ref{lemma_psi_c_general} and Lemma \ref{lemma_existence} respectively). In fact, it is possible to bound all the solutions of the ODE within two functions. It is proved in \cite[Chapter 3, Sub-section 2.2]{piccinini2012ordinary} that there exists two integrals $x = G(t)$ and $x = g(t)$ of the underlying IVP such that any solution $x(t)$ of the IVP can be bounded as:
$$
    g(t) \le x(t) \le G(t).
$$
The solutions $g$ and $G$ are called the \underline{minimal and maximal solutions}. This result of bounding all solutions of IVP under consideration leads to the \underline{Peano Phenomenon} (see \cite[Chapter 3, Sub-section 2.2]{piccinini2012ordinary}). The uniqueness of solutions is guaranteed for such ODEs under restricted conditions, see \cite[Chapter III, Section 3]{piccinini2012ordinary}.

Henceforth, we will consider the following autonomous ODE and derive its stability analysis:
\begin{align}\label{eqn_cont_auto_ODE}
    \dot{x} = f(x) \mbox{ and } x(t_0) = x_0 , \mbox{ for }  (t_0, x_0) \in \mathbb{R} ^+ \times  \mathbb{R}^n,
\end{align}where $f$ is a continuous function.  In particular, we will discuss the time-asymptotic analysis of ODE \eqref{eqn_our_ODE}, and provide several definitions in this regard. We will also re-write the notions in \cite[Chapter 5, Section 2]{piccinini2012ordinary} and \cite[Section 8.4]{hirsch2012differential} in our words.  When $f$ is continuous, one can have many solutions as said above, and textbooks discuss the stability of each of the solution. 
\begin{definition}\label{defn_equi_pt_prelim}
     A set ${\mathbf E} := \{x^*: f( x^*) = 0\}$ is called the set of \underline{equilibrium points} for the ODE \eqref{eqn_cont_auto_ODE}.
\end{definition}
By above, it is clear that if the initial condition $x(0) = x^* \in {\mathbf{E}}$, then the ODE solution remains fixed, i.e., $x(t) = x^*$ for all $t \geq 0$. Next, we see how the ODE solution behaves in a neighborhood of the equilibrium points. Define open ball, $N_\epsilon(\mathbf{A}) := \{x: d(x, \mathbf{A}) < \epsilon\}$ for some finite set $\mathbf{A}$.
\begin{definition}\label{defn_stable_prelim}
A subset ${\mathbf{A}}$ of ${\mathbf E}$ is said to be a \underline{(locally) stable set} for ODE \eqref{eqn_cont_auto_ODE}
if for any $\epsilon > 0$, there exists a $\delta > 0$ such that every solution 
 of the ODE $x(t) \in N_\epsilon(\mathbf{A})$ for every $t > 0$, if initial condition $x(0) \in N_\delta(\mathbf{A})$.
 \end{definition}
\begin{definition}\label{defn_attractors_prelim}
A locally stable set $\mathbf{A}$ is called an \underline{attractor} or \underline{asymptotically stable set} and $\mathbf{D}_\mathbf{A}$ is the \underline{domain of attraction} for ODE \eqref{eqn_cont_auto_ODE}
if every solution $x(t) \to {\mathbf{A}}$ as $t \to \infty$ when $x(0) \in \mathbf{D}_\mathbf{A}$.
\end{definition}
\begin{figure}[http]
    \centering
    \includegraphics[trim = {0cm 0cm 0cm 0cm}, clip, scale = 0.4]{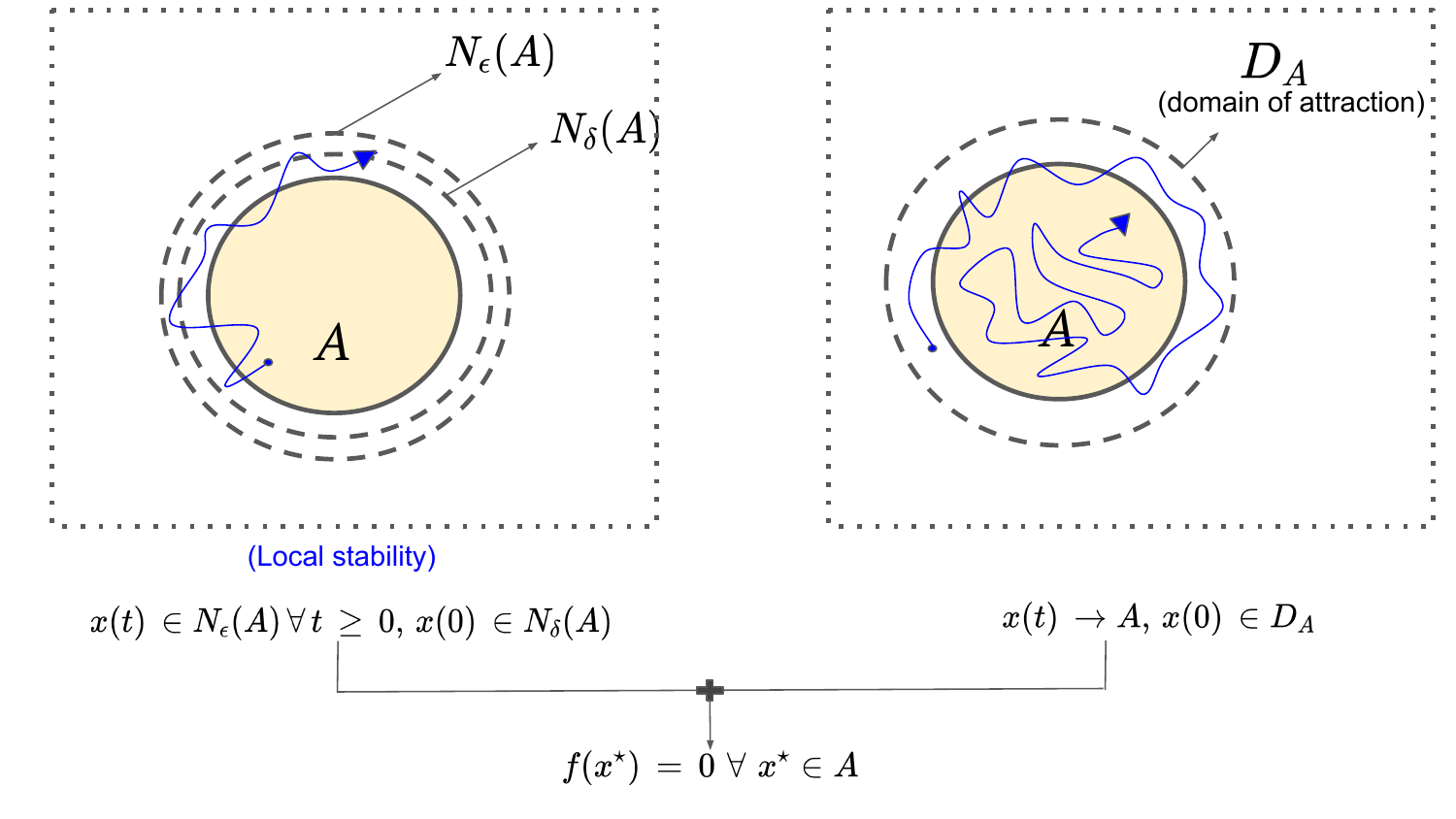}
    \caption{Attractor set}
    \label{fig:attractor}
\end{figure}
Similar definitions hold for the non-autonomous ODEs when $f$ is continuous (see \cite[Chapter 5, Section 2]{piccinini2012ordinary}) or even when $f$ is more general (for example, when $f$ satisfies the Carath\'{e}odory conditions, which we discuss below, see \cite[Section 5]{lasalle1976stability}). Let $\mathbf{A}^\complement$ be the complement of $\mathbf{A}$ and let us define the following:
\begin{definition}\label{defn_saddle_pt_prelim}
A set ${\mathbf{S}} \subset \mathbf{A}^\complement \cap {\mathbf E}$ is \underline{saddle set} for ODE \eqref{eqn_cont_auto_ODE} if there exists ${\mathbf D}_\mathbf{S}$ such that $x(t) \stackrel{t \to \infty}{\longrightarrow} \mathbf{A} $   for some $x(0) \in \mathbf{S}^\complement \cap {\mathbf D}_\mathbf{S}$ and   $x(t) \stackrel{t \to \infty}{\longrightarrow} \mathbf{S} $  for some other $x(0) \in \mathbf{S}^\complement \cap {\mathbf D}_\mathbf{S}$. 
\end{definition}
\begin{figure}[http]
\centering
\begin{minipage}{.45\textwidth}
  \centering
  \includegraphics[trim = {0cm 1cm 1cm 0cm}, clip, scale = 0.38]{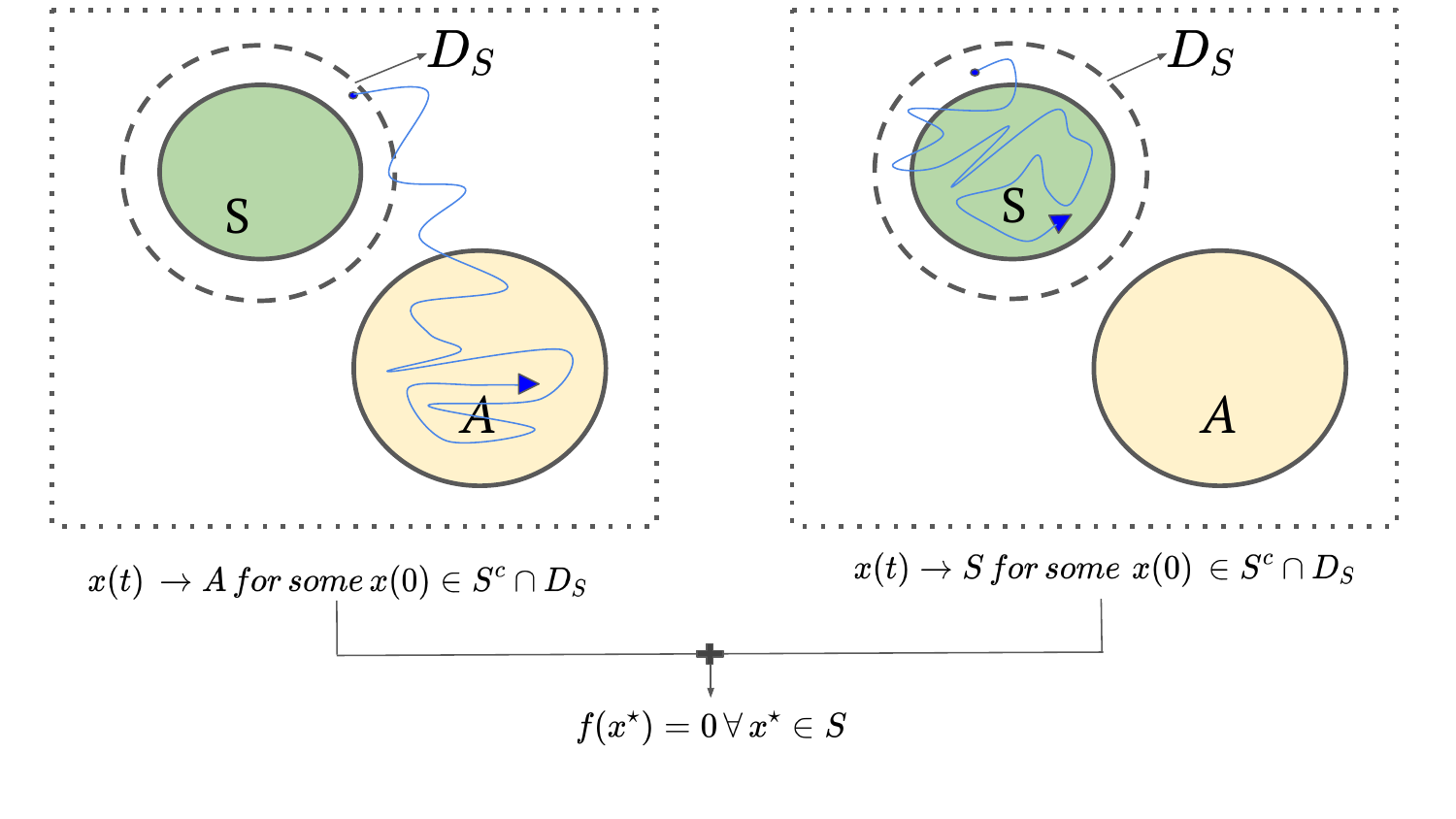}
  \captionof{figure}{Saddle set}
    \label{fig:saddle}
\end{minipage}%
\begin{minipage}{.1\textwidth}
    \hspace{2mm}
\end{minipage}%
\begin{minipage}{.45\textwidth}
  \centering
  \includegraphics[trim = {7cm 1cm 7cm 0cm}, clip, scale = 0.38]{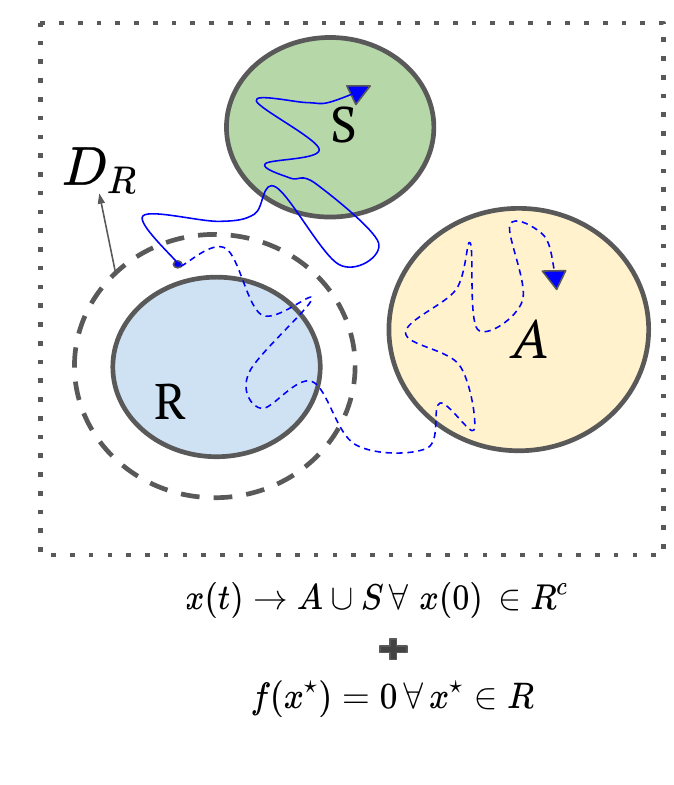}
  \captionof{figure}{Repeller set}
    \label{fig:repeller}
\end{minipage}
\end{figure}

Observe here that when we consider domain of attraction ${\mathbf D}_\mathbf{S}$ for the saddle set ${\mathbf{S}}$, the ODE solution converges to the attractor set for some initial conditions in $\mathbf{S}^\complement \cap {\mathbf D}_\mathbf{S}$, and it converges to the saddle set for the other initial conditions. This notion of domain of attraction is different than that for the attractor set. 
 \begin{definition}\label{defn_repeller_prelim}
A set $\mathbf{R} \subset (\mathbf{A} \cup \mathbf{S})^\complement \cap {\mathbf E}$ is a \underline{repeller set} for the ODE \eqref{eqn_cont_auto_ODE}  if $x(t) \stackrel{t \to \infty}{\longrightarrow} \mathbf{A} \cup \mathbf{S}$  for all $x(0) \in \mathbf{R}^\complement$. 
\end{definition}

In Theorem \ref{thrm_one_dim_ODE}, we will derive the attractor, saddle and repeller sets  for one-dimensional autonomous ODE, in fact, even for the case when right hand side of the ODE is not continuous. Further, in Theorem \ref{thrm_our_ODE}, we will derive these sets for a special $n$-dimensional ODE, for $n \geq 1$.

Next, we move on to discuss ODEs with discontinuous right hand sides. The discontinuity can be in terms of $x$ and or $t$ leading to different solution concepts and limiting properties.

\noindent \textbf{3: Carath\'{e}odory ODE}:
Consider the function $f$ that satisfies the following three properties in the domain $D := \{t \in [t_0, t_0+ a], x\in [x_0-b, x_0+b]\}$ of $(t, x)$-space:
\begin{enumerate}[label = (\roman*)]
    \item the function $f(t, x)$ is defined and continuous in $x$ for almost all $t$,
    \item the function $f(t, x)$ is measurable in $t$ for each $x$, and
    \item $|f(t, x)| \leq m(t)$, where the function $m(t)$ is summable\footnote{the integral $\int_t |m(t)|$ exists and is finite.} (on each finite interval if $t$ is not bounded in the domain $D$).
\end{enumerate}
Then, the \underline{Carath\'{e}odory ODE} is given by:
\begin{align}\label{eqn_cara_ODE}
    \dot{x} = f(t, x), \mbox{ with } x(t_0) = x_0, \mbox{ for } (t_0, x_0) \in D \mbox{ and } f(\cdot, \cdot) \mbox{ satisfy (i)-(iii) as above.}
\end{align}
Next, we define the solution of the above ODE (see \cite[Chapter 1]{filippov2013differential}):
\begin{definition}\label{defn_caraODE_solution}
A function $(x(t))_{t \geq 0}$ defined on an open/closed interval $I$ is said to be a \underline{solution of the Carath\'{e}odory equation} if it is absolutely continuous on each closed interval $[\alpha,\beta] \subset I$, and satisfies the equation \eqref{eqn_cara_ODE} for almost all $t \geq 0$. Or equivalently, if it satisfies the integral equation \eqref{eqn_sol_ODE_integralform} for some $t_0 \in I$.
\end{definition}Then, there exists a (local) solution of the Carath\'{e}odory ODE $\dot{x} = f(t, x)$ with $x(t_0) = x_0$ on a closed interval $[t_0, t_0 + d]$, where $d > 0$ (see \cite[Chapter 1, Theorem 1]{filippov2013differential}). Further, if there exists a summable function $l(t)$ such that for any points $(t, x)$ and $(t, y)$ of the domain~$D$:
$$
    |f(t, x) - f(t, y)| \leq l(t) |x-y|,
$$
then, the \underline{solution is unique} in $D$ (see \cite[Chapter 1, Theorem 2]{filippov2013differential}).

Next, we would like to briefly bring the reader's attention to two more interesting ODEs, which we do not touch upon in the thesis.

\noindent \textbf{4. Discontinuous systems}:  Consider a function $f(t, x)$, where $(t, x) \in \mathbb{R} \times \mathbb{R}^n$. Say that the function $f$ is piecewise continuous in a finite domain $G \subset \mathbb{R}^{n+1}$ if the domain $G$ has (i) a finite number of sub-domains, $G_i$, in each of which the function $f$ is continuous upto the boundary and (ii) a set $M$ of measure zero which consists of boundary points of these sub-domains. 
Then, \cite[Chapter 2]{filippov2013differential} analyses the following ODE:
\begin{align}\label{eqn_filippov_ODE}
    \dot{x} = f(t, x).
\end{align} One simple example of such non-autonomous ODE is (with $M = \{0\} \times \mathbb{R}^n)$:
\[
\dot{x} = 
\begin{cases}
3, &\mbox{if } x < 0\\
1, &\mbox{ if } x = 0,\\
-1, &\mbox{if } x > 0.
\end{cases}
\]
Then, if $x(0) > 0$, the solution is well-defined till $x(t) > 0$, while, if $x(0) < 0$ the solution is again well-defined till $x(t) < 0$. However, in any case as $t$ increases, the solution proceeds towards $x(t) = 0$, but that does not satisfy the above ODE. Therefore, such ODEs are more difficult than Carath\'{e}odory ODE, and hence regular definitions of solution can not be applied. \textit{Basically, one does not know how the solution can be continued} (for example, as $x(t)$ approaches $0$ in the above example).

For such  ODEs, the solution is given via the solution of an appropriate differential inclusion. If at point $(t, x)$ the function $f$ is continuous, then define the set $F(t, x) = \{f(t, x)\}$; else, the set $F(t, x)$ can be defined in different ways to cater to different physical systems. We do not get into the details of how the differential inclusion is defined, but the interested reader can refer to \cite[Chapter 2]{filippov2013differential} for a detailed discussion. Nevertheless, the \underline{solution of \eqref{eqn_filippov_ODE}} \textit{is given by the solution of the following differential inclusion}:
$$
    \dot{x} \in F(t, x),
$$that is, the solution is an absolutely continuous function $x(t)$ defined on an interval or on a segment $I$ for which $x(t)$ satisfies above differential inclusion almost everywhere on $I$. The stability analysis for ODE \eqref{eqn_filippov_ODE} is provided in \cite[Chapter 3]{filippov2013differential}.

\noindent \textbf{5. Asymptotically autonomous ODE}: So far, we studied autonomous and non-autonomous ODEs separately. Interestingly, there are non-autonomous ODEs which become autonomous with time. To be precise, consider the ODE \eqref{eqn_cara_ODE}. Further assume that there exists a continuous function $f^* :\mathbb{R}^n \to \mathbb{R}^n$ such that for all compact $C \subset \mathbb{R}^n$ and all $\epsilon > 0$, there exists $T \geq 0$ which satisfies (see \cite[Assumption (AA)]{logemann2003non}):
$$
\mbox{ess sup}_{t \geq T} |f(t, x) - f^*(x)| < \epsilon, \mbox{ for all } x \in C.
$$
Thus, the function $f(t, x)$ essentially approaches $f^*(x)$ locally uniformly in $x$, as $t\to \infty$.
For such ODEs, the asymptotic (stability) analysis can be derived via the limiting ODE $\dot{x} = f^*(x)$, see \cite[Corollary 4.1]{logemann2003non}. We have presented here a simple case where the limiting system is an ODE, but \cite{logemann2003non} considers the differential inclusion in the limit. 

After providing the required background for ODEs, we will discuss the ODE of our interest and present the main result for the same.

\subsection{Our result}\label{subsec_ourODE}
In this thesis, we shall encounter a specific form of an autonomous ODE with non-linear and (possibly) discontinuous right hand side; we will derive its analysis. Consider the following system of ODE on $\mathbb{R}^n$, for some $n < \infty$:
\begin{align}\label{eqn_our_ODE}
\begin{aligned}
    \dot{x} &= h\big(z\big(x\big)\big) - x \mbox{ with } x(0) = x_0,
\end{aligned}
\end{align}where the functions $h, z$ satisfy the following:
\begin{enumerate}[label=\textbf{A.\arabic*}, ref=\textbf{A.\arabic*}]
    \item Let $z : \mathbb{R}^n \mapsto \mathbb{R}$ be a one-dimensional function of $x \in \mathbb{R}^n$. Further, let $h: \mathbb{R} \mapsto \mathbb{R}^n$ be a measurable function. \label{a1_prelim}
\end{enumerate}
For the above structure of ODE, under certain conditions, one can derive the stability analysis, i.e., the description of attractor ($\mathbf{A}$), saddle ($\mathbf{S}$) and repeller ($\mathbf{R}$) sets and their respective domains of attraction; the definition for these sets are as in Definition \ref{defn_attractors_prelim}, \ref{defn_saddle_pt_prelim}, \ref{defn_repeller_prelim} respectively. We do precisely this in the present subsection. Towards this, we first need to define the solution of the following form (similar to Definition \ref{defn_caraODE_solution}):
\begin{definition}\label{defn_solution_prelim}
A function $(x(t))_{t \geq 0}$ is said to be an \underline{extended solution of ODE \eqref{eqn_our_ODE}} if it is absolutely continuous, and satisfies the equation \eqref{eqn_our_ODE} for almost all $t \geq 0$.
\end{definition}
\textit{Assume that there exists a unique solution $x(\cdot)$ for ODE \eqref{eqn_our_ODE} in the extended sense  over any bounded interval.} Now, we proceed to derive the sets $\mathbf{A}$, $\mathbf{S}$ and $\mathbf{R}$ for the ODE \eqref{eqn_our_ODE}. The main idea is to exploit the structure of ODE \eqref{eqn_our_ODE}, derive the ODE for $z$, and use its asymptotic limits to derive that of the original ODE. In particular, suppose that the ODE for $z(\cdot)$ has the following separable form:
\begin{align}\label{eqn_z_ODE}
    \dot{z} = g_1(x) g_2(z),
\end{align}where the functions $g_1, g_2$ satisfy the following: 
\begin{enumerate}[label=\textbf{A.\arabic*}, ref=\textbf{A.\arabic*}]
\setcounter{enumi}{1}
    \item The functions $g_1: \mathbb{R}^n \mapsto \mathbb{R}$ and $g_2 : \mathbb{R} \mapsto \mathbb{R}$ are measurable. Further assume that the ODE \eqref{eqn_our_ODE} has the following structure: if the initial condition $x_0$ is such that $g_1(x_0) > 0$, then, the corresponding ODE solution satisfies $g_1(x(t)) > g_1(x_0) - \delta > 0$ for all $t\geq 0$ and for any $\delta \in (0, g_1(x_0))$. \label{a2_prelim}
\end{enumerate}
Observe that if $g_1(x) > 0$ for all $x$, then the above condition related to $g_1$ is readily satisfied. Further, if $g_1(x) = 0$ for all $x$, then the ODE \eqref{eqn_our_ODE} is trivially given by $\dot{x} = h(z(x_0)) - x$ whose analysis is straightforward (and the analysis of $z$-ODE is not required); in fact, for the above, all we need is $g_1(x_0) = 0$, i.e., only at the initial condition.

In our case, i.e., for the ODEs approximating the BPs, we do not encounter the first condition. While the second condition is satisfied at the equilibrium point which represents extinction for the BPs. However, the ODEs of this thesis for the rest of the initial conditions indeed satisfy the above assumption \ref{a2_prelim}. 

Further, if the function $g_1(x(t)) > 0$ for all $t \geq 0$ and for some initial condition $x_0$ as in assumption \ref{a2_prelim}, then one may anticipate that $g_1$ does not affect the asymptotic analysis of ODE \eqref{eqn_z_ODE}. We will indeed show that this is true, i.e., for such initial condition $x_0$, the asymptotic analysis of the ODE \eqref{eqn_z_ODE} can be derived by analysing the following one-dimensional ODE:
\begin{align}\label{eqn_one_dim_ODE}
    \dot{z} = g_2(z).
\end{align}
We begin by presenting the asymptotic limits of the above ODE under the following assumption:
\begin{enumerate}[label=\textbf{A.\arabic*}, ref=\textbf{A.\arabic*}]
\setcounter{enumi}{2}
    \item Consider any non-empty interval $[a,b] \subset \mathbb{R}$ such that $g_2(a) \geq 0$ and $g_2(b) \leq 0$. Let ${\cal I}$ be the set of equilibrium points for the ODE \eqref{eqn_one_dim_ODE} in $[a,b]$ and say ${\cal I}= \{z_i^*: 1 \leq i \leq n\}$, for some $1 \leq n < \infty$.  For each $i$, let there exist an open/closed/half-open non-empty interval around $z_i^* \in {\cal I}$, say ${\cal N}_i^*$, such that $\cup_{1\leq i\leq n} {\cal N}_i^* = [a,b]$ and ${\cal N}_i^* \cap {\cal N}_j^* = \emptyset$ for $i\neq j$. Define ${\cal N}_i^- := {\cal N}_i^*\cap [a, z_i^*)$ and 
${\cal N}_i^+ := {\cal N}_i^*\cap (z_i^*, b]$. Let $g_2(x)$ be Lipschitz continuous on ${\cal N}_i^-$ and $ {\cal N}_i^+$ for each $i$. \label{asmp_ode}
\end{enumerate}
\begin{theorem}\label{thrm_one_dim_ODE}
Assume \ref{a1_prelim}-\ref{asmp_ode}. Then, the solution of the ODE \eqref{eqn_one_dim_ODE} exists in extended sense and further, the following are true for the ODE \eqref{eqn_one_dim_ODE}: 
\begin{enumerate}[label=(\roman*)]
    \item if $g_2(z) > 0$ for all $z \in {\cal N}_i^-$, $g_2(z) < 0$ for all $z \in {\cal N}_i^+$, then, $z_i^*$ is an attractor with the domain of attraction of $z_i^*$ as ${\cal N}_i^*$;
    \item   if $g_2(z) < 0$ for all $z \in {\cal N}_i^-$ and $g_2(z) > 0$ for all $z \in {\cal N}_i^+$, then, $z_i^*$ is a repeller;
    \item else if $g_2(z) > 0$ (or $g_2(z) < 0$) for all $z \in {\cal N}_i^-$ and $g_2(z) > 0$ (or $g_2(z) < 0$ respectively) for all $z \in {\cal N}_i^+$, then, $z_i^*$ is a saddle point with the domain of attraction of $z_i^*$ as ${\cal N}_i^*$. 
    \item thus, $\mathbf{A} := \{z_i^*: z_i^* \mbox{ is an attractor}\}$ is the attractor set with $\mathbf{D}_\mathbf{A} = \cup_{\{i : z_i^* \in \mathbf{A}\}} {\cal N}_i^*$, $\mathbf{S} := \{z_i^*: z_i^* \mbox{ is a saddle point}\}$ is the saddle set with $\mathbf{D}_\mathbf{S} = \cup_{\{i : z_i^* \in \mathbf{S}\}} {\cal N}_i^*$ and $\mathbf{R} := \{z_i^*: z_i^* \mbox{ is a repeller}\}$ is the repeller set. \eop
\end{enumerate}
\end{theorem}
The proof of above Theorem follows as in Theorem \ref{thrm_beta_ODE_prop}. Observe that the condition $g_2(a) \geq 0$ and $g_2(b) \leq 0$ in \ref{a3_prelim} ensures that the interval $[a, b]$ is positive invariant for the ODE \eqref{eqn_one_dim_ODE}. Further, this implies that $\mathbf{D}_\mathbf{A} \cup \mathbf{D}_\mathbf{S} = [a,b]$.
\begin{figure}[http]
    \centering
    \includegraphics[trim = {0cm 2cm 2cm 1cm}, clip, scale = 0.4]{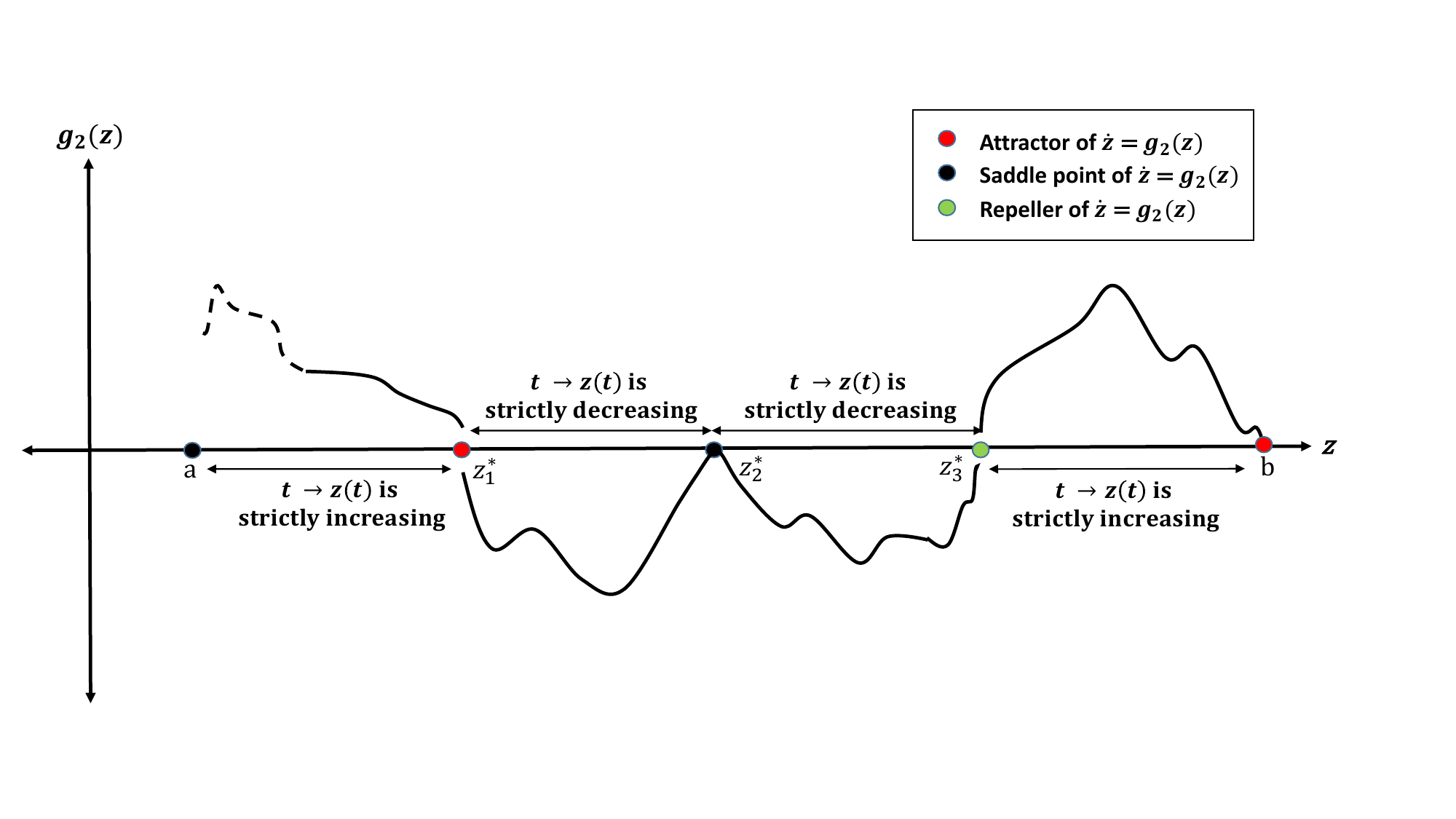}
    \caption{Asymptotic limits for one-dimensional ODE }
    \label{fig:notions_stability}
\end{figure}
An instance of Theorem  \ref{thrm_our_ODE} is presented in Figure \ref{fig:notions_stability} as: say $z(0)$ is in left of $z_1^*$, then since $g_2(z) > 0$, therefore, we show in the proof of above Theorem that $t \mapsto z(t)$ is an increasing function such that $z(t) \to z_1^*$ as $t \to \infty$; while if $z(0) \in (z_1^*, z_2^*)$, then again $z(t) \to z_1^*$, and $z(t)$ moves away from $z_2^*$. This when continued for $z(0)$ in other intervals, it leads to the conclusion that $z_1^*,b$ are attractors, $a, z_2^*$ are the saddle points and $z_3^*$ is the repeller for the ODE \eqref{eqn_one_dim_ODE} (see Definitions \ref{defn_attractors_prelim}-\ref{defn_repeller_prelim}).

In general, observe that since the function $g_2$ is Lipschitz continuous on neighborhoods ${\cal N}_i^-$ and $ {\cal N}_i^+$ for each $i$, therefore, the solution $z(\cdot)$ of the ODE \eqref{eqn_one_dim_ODE} exists in the respective neighborhoods. Further, observe that the function $g_2$ can be continuous, or even dis-continuous at the equilibrium points $\{z_i^* : 1\leq i \leq n\}$. In Theorem \ref{thrm_attractors_beta} of Chapter \ref{ch:journal1}, we consider $g_2$ to be dis-continuous, while in Theorem \ref{thrm_beta_ODE_prop} of Chapter \ref{ch:journal2}, the ODE with $g_2$ as a continuous function is dealt. Here, we are generalising the two results to the case where $g_2$ can be continuous/discontinuous at the equilibrium points and is Lipschitz continuous elsewhere.  

\subsubsection{Asymptotic behaviour of ODE \eqref{eqn_our_ODE}}
Now, by leveraging upon the asymptotic limits of the ODE \eqref{eqn_one_dim_ODE} derived in the above Theorem, we next derive the attractor and saddle sets of the ODE \eqref{eqn_our_ODE} which is the ODE of main interest. However, before proceeding towards the main result of this section, we define a special type of saddle point that facilitates in representing the result.
\begin{definition}\label{defn_q_attractor}
Any $x^*  \in \mathbf{S}$ with $ g_1(x^*) > 0$ is said to be (quasi) \underline{q-attractor} if 
\begin{enumerate}[label=(\roman*)]
    \item for any  $x(0) \in {\mathbb S}(x^*) := \{x: z(x) = z(x^*)\} $, $x(t) \stackrel{t \to \infty}{\longrightarrow} x^*$ exponentially, 
    \item $x(t) \stackrel{t \to \infty}{\longrightarrow} \mathbf{A}$ for other initial conditions.
\end{enumerate}Any $x^*  \in \mathbf{S}$ with $ g_1(x^*) = 0$ is a  q-attractor if the above holds with ${\mathbb S}(x^*) := \{x: g_1(x) = 0\}$.
\end{definition}
Thus, any q-attractor ($x^*$) is a special type of saddle point which exhibits exponential convergence to $\mathbf{S}$ starting from a sub-region (${\mathbb S}(x^*)$) in its neighborhood (see Figure \ref{fig:q_attractor}, with $z(x^*)$ denoted as $z^*$). 
\begin{figure}[http]
    \centering
    \includegraphics[scale = 0.4]{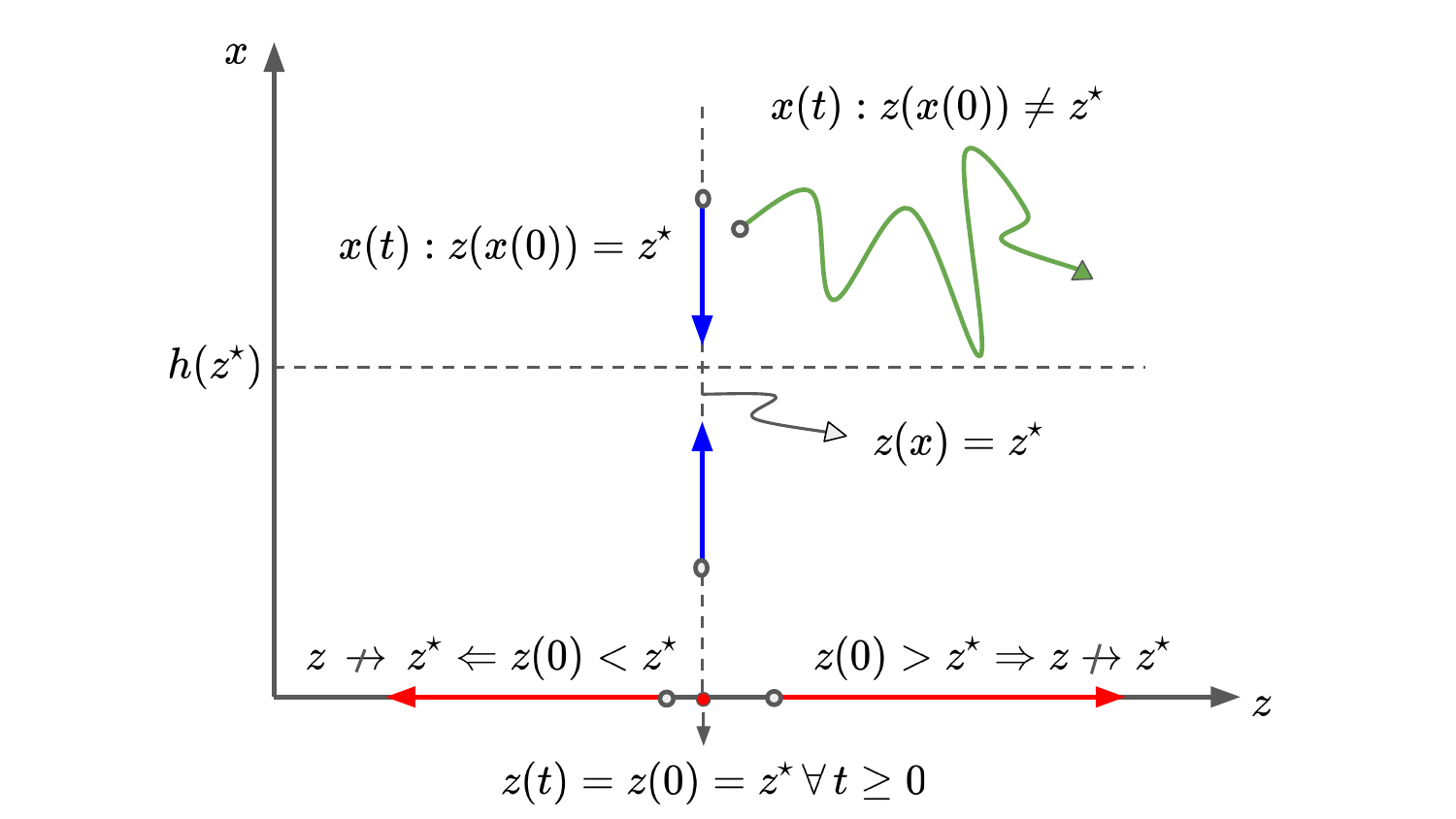}
    \caption{Saddle or repeller point of \eqref{eqn_one_dim_ODE} leads to saddle point of \eqref{eqn_our_ODE}; here $g_1(x^*) > 0$}
    \label{fig:q_attractor}
\end{figure}

Now, one can determine the attractors and q-attractors of the ODE \eqref{eqn_our_ODE} by virtue of the following theorem, proof of which follows as in Theorem \ref{thrm_beta_ODE_prop}:
\begin{theorem}\label{thrm_our_ODE}
Assume \ref{a1_prelim}-\ref{asmp_ode}. Then, there exists a unique extended solution for ODE \eqref{eqn_our_ODE} over any bounded interval. Further, the attractor and the saddle (q-attractor) sets for the ODE \eqref{eqn_our_ODE} are respectively given by:
\begin{align}
\begin{aligned}
\mathbf{A} &:= \{h(z^*): z^* \mbox{ is an attractor for the ODE \eqref{eqn_one_dim_ODE}}\} \mbox{ and}\\
\mathbf{S} &:= \{h(z^*): z^* \mbox{ is a repeller or saddle point for the ODE \eqref{eqn_one_dim_ODE}}\}.
\end{aligned}
\end{align}Furthermore, $\{x \in \mathbb{R}^n : z(x) \in [a,b]\} \subset \mathbf{D}_{\mathbf{A}} \cup \mathbf{D}_{\mathbf{S}}$ is the combined domain of attraction of $\mathbf{A}$ and $\mathbf{S}$ for ODE \eqref{eqn_our_ODE}. \eop
\end{theorem}
This implies that the attractors of ODE \eqref{eqn_one_dim_ODE} provide the attractor set for the ODE \eqref{eqn_our_ODE}, while the repeller and saddle points of the former ODE collectively contribute to the q-attractor set for the latter ODE. This concludes our discussion on ODEs. We will next discuss the SA based result. We would like to mention here that Theorem \ref{thrm_our_ODE} will be instrumental in applying this SA-based result to certain applications like BPs of this thesis.

\section{Stochastic approximation}\label{sec_SA}
The stochastic approximation (SA) based algorithms are recursive stochastic algorithms which were originally introduced by Robins and Monro to find the zero of a real-valued function $\ups \mapsto g(\ups)$, when the function $g(\cdot)$ is not known but noisy observations of $g(\ups)$ are accessible. For a detailed discussion on how it all started, refer to \cite[Chapter 1]{kushner2003stochastic}; for several examples on SA-algorithms in a variety of domains, refer to \cite[Chapters 1-3]{kushner2003stochastic}; for a concise and easy-to-read study on SA-algorithms, refer to \cite{borkar2009stochastic}.

In general, the SA algorithm of our interest takes the following form, where $\Ups_n \in \mathbb{R}^r$ and evolves as follows:
\begin{align}\label{eqn_SA_scheme}
    \Ups_{n+1} = \Ups_n + \epsilon_{n+1} L(\xi_{n+1}, \Ups_n), \mbox{ where} 
\end{align}$L(\xi_{n+1}, \Ups_n)$ denotes the $\mathbb{R}^r$-valued noisy observations at $n$-th iteration, and depends on random variables $\xi_{n+1}$ and the previous iterate $\Ups_n$; further, the step-size sequence satisfies the following assumption:
\begin{enumerate}[label=\textbf{A.\arabic*}, ref=\textbf{A.\arabic*}]
\setcounter{enumi}{3}
    \item $\epsilon_n = 0 \mbox{ for all } n < 0, \epsilon_n \geq 0 \mbox{ for all } n \geq 0, \sum_{n=0}^\infty \epsilon_n = \infty \mbox{ and } \sum_{n=0}^\infty \epsilon_n^2 < \infty$. \label{a3_prelim}
\end{enumerate}
One example of such $\epsilon_n$-sequence is $\epsilon_n = \frac{1}{n}$. Further, assume the following on \eqref{eqn_SA_scheme}:
\begin{enumerate}[label=\textbf{A.\arabic*}, ref=\textbf{A.\arabic*}]
\setcounter{enumi}{4}
    \item $\sup_n E|L(\xi_{n+1}, \Ups_n)|^2 < \infty$. \label{a4_prelim}
    \item There exists a measurable function $g(\cdot)$ of $\ups$ such that:\label{a5_prelim}
    $$
        E[L(\xi_{n+1}, \Ups_n)|\Ups_0, L(\xi_{i+1}, \Ups_i) \mbox{ for } 1 \leq i < n, \Ups_n] = g(\Ups_n).
    $$
\end{enumerate}
We would now like to explain the intuition behind SA-based results. Towards this, for simplicity in explanation consider $\epsilon_{n+i} \approx \epsilon_n$, for all $i \leq N$ for some fixed $N$, and let $\epsilon_n$ be sufficiently small. Then, the iterate $\Ups_{n + N}$ can be written and approximated as follows (see \eqref{eqn_SA_scheme}):
\begin{align*}
    \Ups_{n + N} &= \Ups_n +  \sum_{i=0}^{N-1} \epsilon_{n+i+1} L(\xi_{n+i+1}, \Ups_{n+i}) \\
    &\approx \Ups_n + \epsilon_n \sum_{i=0}^{N-1} L(\xi_{n+i+1}, \Ups_{n}) \\
    &= \Ups_n + (N\epsilon_n) \left(\frac{1}{N} \sum_{i=0}^{N-1} L(\xi_{n+i+1}, \Ups_{n})\right) \\
    &\approx \Ups_n + (N\epsilon_n) g(\Ups_n).
\end{align*}
In the above, the first approximation holds as due to small step-size, $\Ups_n$ does not change much in $N$ steps. Further, as $N$ increases, by strong law of large numbers, under \ref{a5_prelim}, we get the second approximation. Observe that the resultant can be approximated by the solution of the ODE:
\begin{align}\label{eqn_SA_ODE}
    \dot{\ups} = g(\ups), \mbox{ with }  \ups(0) = \Ups_n, 
\end{align}
as
$$
\frac{\Ups_{n+N} - \Ups_n}{N\epsilon_n} \approx g(\Ups_n),
$$
converges to a solution of the ODE \eqref{eqn_SA_ODE} when $N \to \infty$ (and $N\epsilon_n$ decreases to~$0$).

We will show in the following that the above ODE is indeed appropriate to approximate the SA-based scheme in a certain way formalized in the next result.

\subsection{Approximation result over finite-time}
The first result for the SA-based algorithm \eqref{eqn_SA_scheme} provides the approximation over finite-time intervals (proof follows as in Theorem \ref{thrm1}(i)):
\begin{theorem}\label{thrm_finite_time}
  Assume \ref{a3_prelim}-\ref{a5_prelim} for the scheme \eqref{eqn_SA_scheme}. Then, for every $T>0$, almost surely there exists a sub-sequence $(n_l)$ such that:
            $$
            \sup_{k: t_k \in [t_{n_l}, t_{n_l} + T]} d(\Ups_k, \ups(t_k - t_{n_l})) \to 0  \mbox{ as } l \to \infty, \mbox{ where } t_n := \sum_{k=1}^n \epsilon_k \mbox{ and}
            $$
        $\ups(\cdot)$ is the extended solution of ODE \eqref{eqn_SA_ODE} which starts at $\ups(0) =
        \lim_{n_l \to \infty} \Ups_{n_l}$.
\end{theorem}
Thus, with probability $1$, there exists a sub-sequence along which the iterates closely follow the ODE solution when initialised with $\lim_{n_l \to \infty} \Ups_{n_l}$, over any finite time window, as number of iterations increases to $\infty$. In the above, $t_n$ gives the time mapping between the ODE and the stochastic iterates ($\Ups_n$) in terms of the step-size sequence $(\epsilon_k)$. 

The approximation is explained in Figure \ref{fig:finite_time_result}. The red dots represent the SA-iterates, and the solid curves represent the ODE trajectories starting at different iterate values, initialized at values of sub-sequence $(\Ups_{n_l})$. At first, consider an ODE trajectory (see dashed-dotted curve) which starts at $\Ups_n$, for some fixed $n$; it can be easily seen that it poorly approximates $(\Ups_{n_l})$. As $n_l$ increases, the approximation improves (see solid curve and then dashed curve). Further, notice that the gap between the iterates decreases as $n_l$ increases, because of the time mapping ($n \mapsto t_n$) and further because $\epsilon_n \to 0$. 
\begin{figure}[http]
    \centering
    \includegraphics[scale = 0.4]{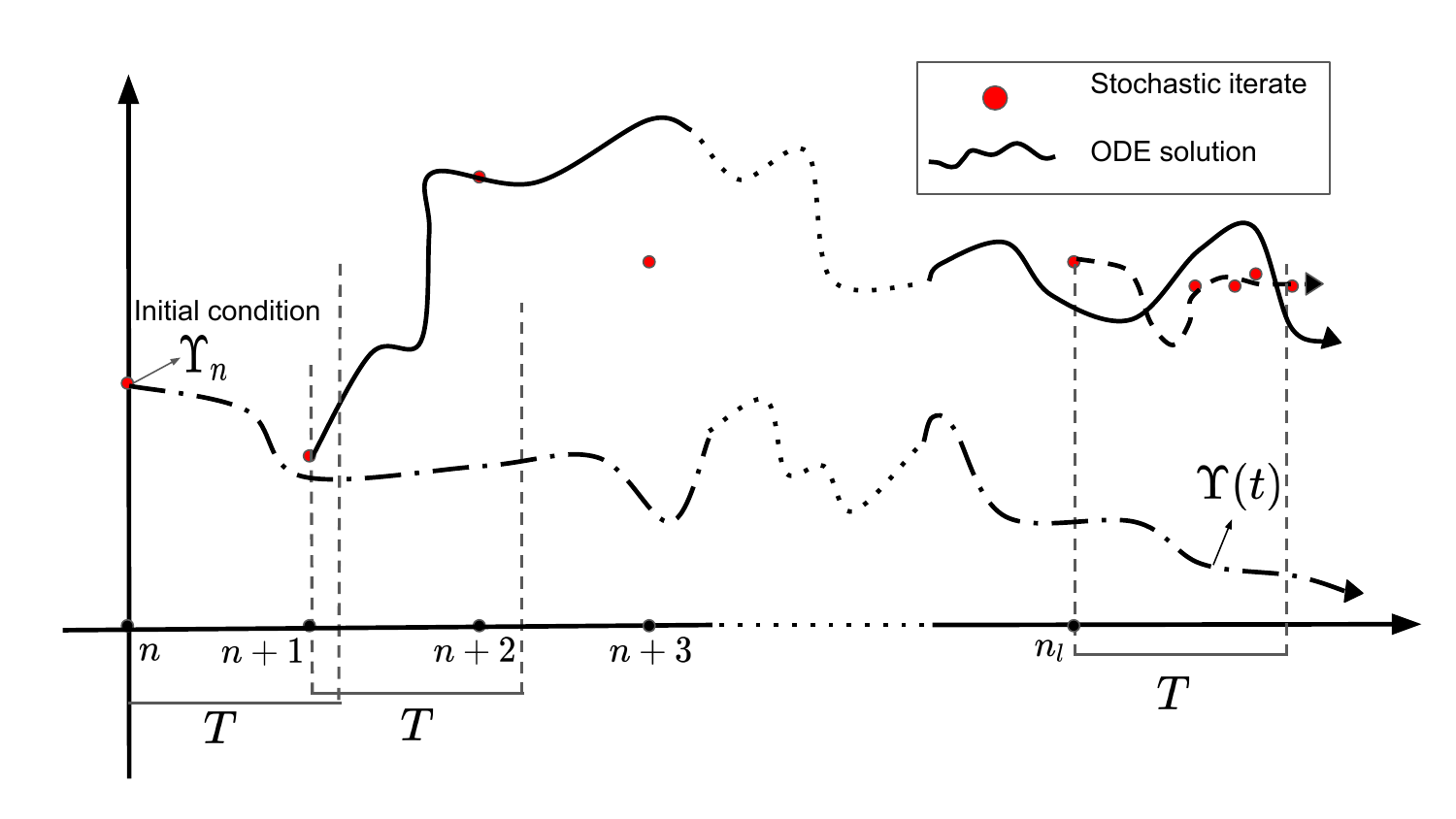}
    \caption{Finite time approximation of SA trajectory}
    \label{fig:finite_time_result}
\end{figure}
\subsection{Asymptotic result - new behaviour `hovering around'}
The second result focuses on the limiting behaviour of the SA-based algorithm. Towards this, we first discuss the existing result \cite[Chapter 5, Theorem 2.1]{kushner2003stochastic}, which is related to us. For the asymptotic result, the authors additionally assume the following:
\begin{enumerate}[label=\textbf{A.\arabic*}, ref=\textbf{A.\arabic*}]
\setcounter{enumi}{6}
    \item Consider the ODE \eqref{eqn_SA_ODE} where $g$ is a continuous function. Let $\mathbf{A} $ be the attractor set, defined as in Definition \ref{defn_attractors}, with $\mathbf{D}$ as the compact subset of domain of attraction. Assume $p := P(\mathcal{V}) > 0$, where\footnote{We say that a sequence of sets $A_n$ is infinitely visited, to be more precise, a sample point $\omega$ visits $A_n$ infinitely often (i.o.) if $\omega \in \cap_n \cup_{k \geq n}A_k$. Basically, for every $n$, there exists a $k > n$, such that $\omega \in A_k$. } $\mathcal{V} :=  \{\omega : \Ups_n(\omega) \in \mathbf{D} \mbox{ i.o.}\}$. \label{a7_prelim}
\end{enumerate}
Then, under \ref{a3_prelim}-\ref{a7_prelim}, the authors prove that $\Ups_n$ converges to the attractor set $\mathbf{A}$ of ODE \eqref{eqn_SA_ODE} w.p. at least $p > 0$. Same result is proved in \cite[Chapter 5, Theorem 2.2]{kushner2003stochastic} even when the function $g(\cdot)$ is measurable in \eqref{eqn_SA_ODE}, under some additional conditions.

The above mentioned results focus on convergence towards the attractor set, given the SA-iterates visit a subset of the corresponding domain of attraction i.o. 
In this thesis, we extend these results where we also consider limiting behaviour around saddle points. 
\begin{enumerate}[label=\textbf{A.\arabic*}, ref=\textbf{A.\arabic*}]
\setcounter{enumi}{7}
    \item (a) Let $\mathbf{A}$ and $\mathbf{S}$ be the attractor and saddle sets as in Definitions \ref{defn_attractors_prelim} and \ref{defn_saddle_pt_prelim} respectively. Let  $\mathbf{D} \subset \mathbf{D}_\mathbf{A} \cup \mathbf{D}_\mathbf{S}$ be a compact subset of combined domain of attraction for $\mathbf{A}$ and $\mathbf{S}$. 
    
    (b) Assume $p := P(\mathcal{V}) > 0$, where $\mathcal{V} :=  \{\omega : \Ups_n(\omega) \in \mathbf{D} \mbox{ i.o.}\}$. \label{a8_prelim}
\end{enumerate}
For SA-based algorithm under \ref{a3_prelim}-\ref{a5_prelim} and \ref{a8_prelim}, we prove that w.p. at least $p$, either $\Ups_n$ converges to $\mathbf{A}\cup\mathbf{S}$ or exhibits an interesting non-convergent, nonetheless some `nearness' behaviour, which we define below:
\begin{definition}\label{defn_hovers_basic}
    The stochastic process $\Ups_n$ is said to  \underline{hover around a set $\mathbf{S}$} if $ \Ups_n \in N_\delta (\mathbf{S})$ i.o. for all $\delta >0$ and $\Ups_n \notin N_{\delta_1} (\mathbf{S}) \mbox{ i.o., for some } \delta_1 >0$.
\end{definition}
Thus, \textit{hovering around depicts a type of the limiting behavior of the stochastic process where the trajectory goes arbitrarily close to the set $\mathbf{S}$ i.o., but also exits a neighbourhood of it i.o.} 
Observe that this new behaviour is different than the `lingering around'\footnote{Such behaviour is observed in \cite{jagers2011population} for branching processes that switch between super-to-sub critical regimes due to 
current population dependency.} behaviour discussed in \cite{jagers2011population}, where the underlying process stays in an $\epsilon$-band around carrying capacity for an exponentially long time, if at all it enters the band (for some $\epsilon > 0$). The notion does not include the phenomenon of entering the band i.o., as in hovering around.

Now, the main result is as follows, proof of which follows as in Theorem \ref{thrm1}(ii):
\begin{theorem}\label{thrm_SA_asym_result}
Assume \ref{a3_prelim}-\ref{a5_prelim} and \ref{a8_prelim}. Define the sets ${\cal C}_1 : =\{\Ups_n \stackrel{n\to\infty}{\to} \mathbf{A} \cup \mathbf{S}\}$ and ${\cal C}_2 := \{ \Ups_n \mbox{ hovers around } \mathbf{S} \}$. Then, $P({\cal C}_1 \cup {\cal C}_2) \geq p$. \eop
\end{theorem}
Thus, with probability at least $p$, $\Ups_n$ has only three limiting behaviours: (i) convergence to attractor set $\mathbf{A}$ or (ii) convergence to saddle set $\mathbf{S}$ or (iii) hovering around the saddle set $\mathbf{S}$. Our result affirms that one of the three events occur w.p. at least $p >0$, but it does not comment on the probability of the individual events.

To the best of our knowledge, the notion of hovering around is new to the literature of SA. Such behaviour is observed as the domains of $\mathbf{A}$ and $ \mathbf{S}$ are close to each other (to be more accurate, $\mathbf{D}_\mathbf{A}$ and the attracting sub-region of $\mathbf{D}_\mathbf{S}$) and the SA trajectory can hop between the two domains due to inherent randomness (see $\xi_i$ in \eqref{eqn_SA_scheme}).
\begin{figure}[http]
\centering
  \centering
  \includegraphics[trim = {0cm 0cm 0cm 0cm}, clip, scale = 0.4]{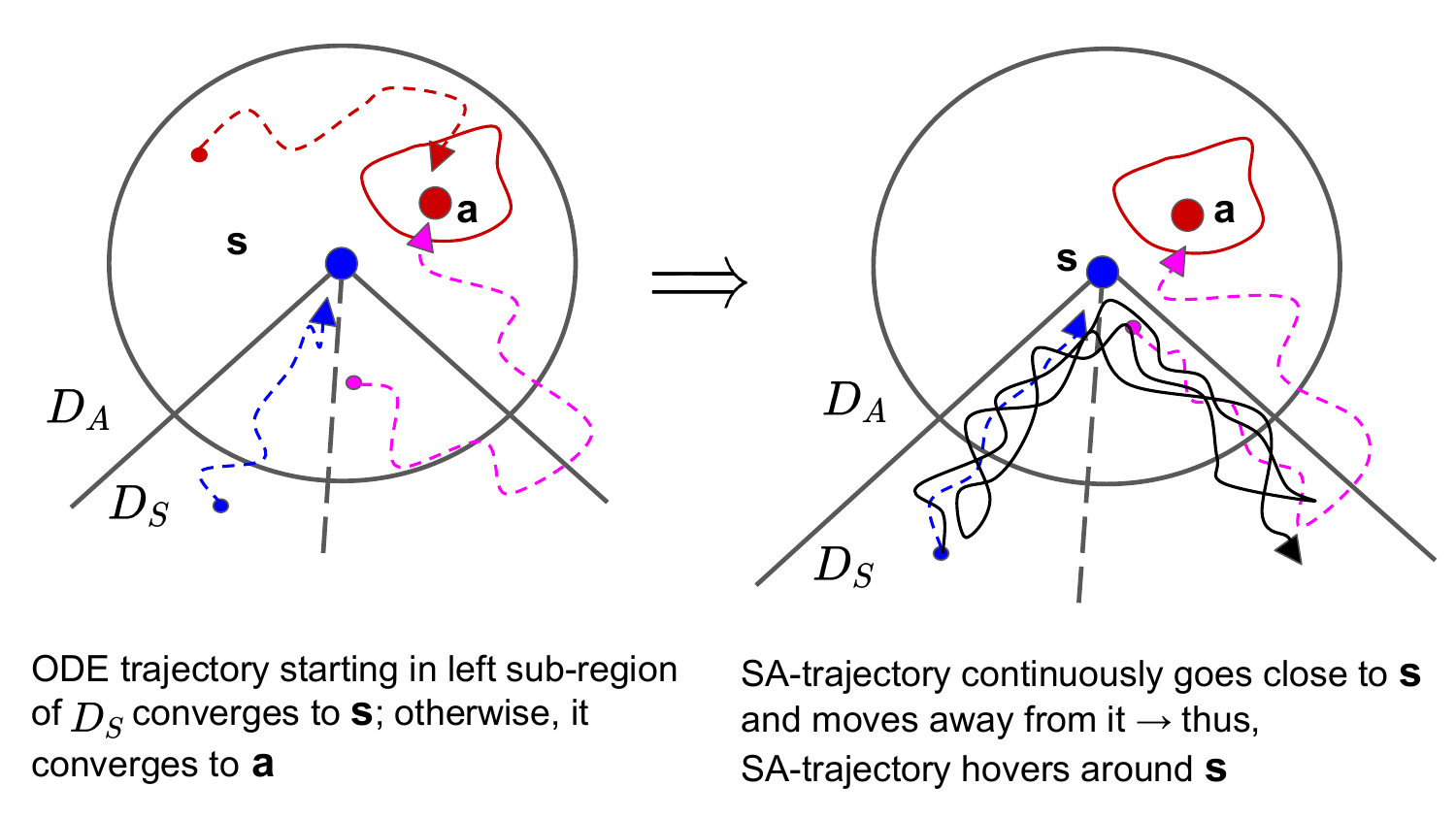}
\caption{hovering around behaviour of SA trajectory ($\mathbf{a} \in \mathbf{A}$ and $\mathbf{s} \in \mathbf{S}$)}
    \label{fig:hovering}
\end{figure}

We pictorially illustrate this behaviour in Figure \ref{fig:hovering} -- the left sub-figure shows that the ODE trajectory converges to saddle point ($\mathbf{s}$) and attractor ($\mathbf{a}$) when the ODE is initialized in left sub-region of $\mathbf{D}_\mathbf{S}$ and $\mathbf{D}_\mathbf{A}$ respectively. More importantly, the ODE initialized in the right sub-region of $\mathbf{D}_\mathbf{S}$ converges to $\mathbf{a}$. Also, note that the  left and right sub-regions of $\mathbf{D}_\mathbf{S}$ are divided by a lower-dimensional line. In right sub-figure, we show the interpolated trajectory for the SA-iterates (briefly called SA trajectory, and shown in black) starting in $\mathbf{D}_\mathbf{S}$; observe that the SA trajectory follows the ODE trajectory initialized at different points (see magenta and blue curves) for finite time-intervals, but then it moves close and away from $\mathbf{s}$ continuously, leading to the hovering around $\mathbf{s}$.

Next, we consider a specific form for the function $g(\cdot)$, which we discussed in sub-section \ref{subsec_ourODE} and will be seen with respect to all SA-schemes related to BPs. In particular,  assume the following:
\begin{enumerate}[label=\textbf{A.\arabic*}, ref=\textbf{A.\arabic*}]
\setcounter{enumi}{8}
    \item The function $g(\ups) = h(z(\ups)) - \ups, \mbox{ where } h, z \mbox{ are as in } \eqref{eqn_our_ODE}, \eqref{eqn_z_ODE} \mbox{ respectively} $ and satisfy \ref{a1_prelim}-\ref{asmp_ode}.\label{ass_g_prelim}
    \item  Assume $P(\{\omega : | \Ups_n(\omega) | \le B \mbox{ i.o.}\}) = 1$, for some $B < \infty$. \label{a10_prelim}
\end{enumerate}
Then, the ODE associated with the SA algorithm \eqref{eqn_SA_scheme} is given by (see \eqref{eqn_SA_ODE}):
\begin{align}\label{eqn_SA_ODE_specific}
    \dot{\ups} = g(\ups) = h(z(\ups)) - \ups, \mbox{ with }  \ups(0) = \Ups_n.
\end{align}
Due to the above structure of the ODE, its  asymptotic limits are given by Theorem \ref{thrm_our_ODE}. Further, the asymptotic behaviour of the SA-scheme is then given by Theorem \ref{thrm_SA_asym_result} as in the following:
\begin{corollary}\label{cor_SA_prelim}
Assume \ref{a3_prelim}-\ref{a5_prelim} and \ref{ass_g_prelim}-\ref{a10_prelim}. Define the sets ${\cal C}_1 : =\{\Ups_n \stackrel{n\to\infty}{\to} \mathbf{A} \cup \mathbf{S}\}$ and ${\cal C}_2 := \{ \Ups_n \mbox{ hovers around } \mathbf{S} \}$. Then, we have (with $a, b$ as in \ref{asmp_ode}):
\begin{align}
\hspace{2cm}
P({\cal C}_1 \cup {\cal C}_2) \geq P(\{\omega: z(\Ups_n(\omega)) \in [a,b] \mbox{ i.o.}\}). \hspace{2cm} \mbox{ \eop} 
\end{align}
\end{corollary}
\begin{proof}
Under \ref{ass_g_prelim} (specifically, under \ref{a1_prelim}-\ref{asmp_ode}), the attractor ($\mathbf{A}$) and saddle ($\mathbf{S}$) sets are given by Theorem \ref{thrm_our_ODE}. Further, the  combined domain of attraction for the SA-scheme \eqref{eqn_SA_scheme}, $  \mathbf{D}_{\mathbf{A}} \cup \mathbf{D}_{\mathbf{S}} \supset  \{\ups \in \mathbb{R}^n : z(\ups) \in [a, b]\}$. Define $\mathbf{D} := \mathbf{D}_{\mathbf{A}} \cup \mathbf{D}_{\mathbf{S}}\cap \{\ups: |\ups| \leq B\}$, where $B$ is given in \ref{a10_prelim}, and observe $\mathbf{D}$ is compact. Then, clearly:
$$
    P(\{\omega: \Ups_n(\omega) \mbox{ visits } \mathbf{D}\mbox{ i.o.} \}) \geq P(\{\omega: z(\Ups_n(\omega)) \in [a,b] \mbox{ i.o.}\}).
$$
Thus, under \ref{a3_prelim}-\ref{a5_prelim}, the corollary follows from Theorem \ref{thrm_SA_asym_result}.
\end{proof}
To conclude, the structure of the approximating ODE as in \eqref{eqn_SA_ODE_specific} provides flexibility to analyze the SA-based algorithms, and therefore, the BPs that we will study in coming chapters. We briefly state the key observations/advantages in the following:
\begin{itemize}
    \item When the approximating ODE \eqref{eqn_SA_ODE_specific} satisfies \ref{a1_prelim} and \ref{a2_prelim}, the analysis of the one-dimensional ODE \eqref{eqn_one_dim_ODE} is sufficient to identify the attractor and saddle sets. In particular, we need to identify the description of the equilibrium points of ODE \eqref{eqn_one_dim_ODE}  as in \ref{asmp_ode}.
    \item The saddle points are in fact the q-attractors, where the underlying ODE trajectory ($\ups$) converges exponentially to $\mathbf{A}$, when started in a sub-region of $\mathbf{D}_{\mathbf{S}}$ and converges to $\mathbf{S}$, when started in the complimentary sub-region.
    \item For BPs, the function $z$ equals $\bc$ which represents the proportion of the current (living) population-sizes of one of the population-types; thus $a = 0$ and $b = 1$ in \ref{asmp_ode}. This also trivially implies that $P(\{\omega: z(\Ups_n(\omega)) \in [a,b] \mbox{ i.o.}\}) = 1$.
    \item Thus, to comment on $P({\cal C}_1 \cup {\cal C}_2)$, it is only left to find a bound on the stochastic iterates, as in \ref{a10_prelim}.
    \item We prove \ref{a10_prelim} for two BPs, namely BP with attack in Chapter \ref{ch:journal1} and BP with unnatural deaths in Chapter \ref{ch:journal2}; the proof is provided using simple strong law of large numbers based arguments on some appropriate bounding sequence. Thus, for the said BPs, $P({\cal C}_1 \cup {\cal C}_2) = 1$ (see Corollary \ref{corollary_BPA} and Theorem \ref{thrm_BP_to_fake} respectively).
\end{itemize}

\section{Summary}
We conclude this chapter by giving the three step procedure to analyse the new BPs considered in this thesis (see sub-section \ref{sub_sec_newBP}):
\begin{enumerate}[label=(\roman*)]
    \item scale the population-sizes of the two types of populations to form an appropriate SA-based iterative scheme,
    \item identify the appropriate ODE that can approximate the above SA-scheme and determine the asymptotic limits of the stochastic iterates using Theorem \ref{thrm_our_ODE} and Theorem \ref{thrm_one_dim_ODE}, and 
    \item the SA iterates corresponding to BP can be approximated over any finite time window by the above ODE, as per Theorem \ref{thrm_finite_time} and the limiting behaviour of SA iterates is given by Corollary \ref{cor_SA_prelim}.
\end{enumerate}
This procedure is followed precisely for BPs introduced and analyzed in Chapters \ref{ch:journal1}-\ref{ch:STPBP}. 
\chapter{Total-current population-dependent BP and Viral competing markets}\label{ch:journal1}

In this chapter, we introduce total-current population-dependent BP\footnote{The work in this chapter has been submitted to a journal.} and analyze the same using the three step procedure discussed in Chapter \ref{ch:basics}. Further, the proofs of some of the generalized results of Chapter \ref{ch:basics} are provided here.  Furthermore, a specific BP, named BP with attack, is discussed; it holds its theoretical relevance in addition to providing insights about the viral competing markets\footnote{An initial study about viral competing markets is in ``Agarwal, Khushboo, and Veeraruna Kavitha. ``Co-virality of competing content over osns?.'' 2021 IFIP Networking Conference (IFIP Networking). IEEE, 2021.'' } on OSNs. Numerical study to validate the theoretical results is also presented towards the end.

\vspace{-7mm}
\section{Introduction}
It is a common practice to study growth patterns and limit proportions for analyzing Markov chains that are predominantly transient, like branching processes (BPs) under the super-critical regime (for example, \cite{athreya2004branching, kesten1967limit}). This chapter  investigates precisely the time-asymptotic proportion of population types for a general class of continuous-time two-type population size-dependent Markov BPs. The offspring depends on the current (living) as well as the total (living and dead) populations, and can also be negative to model attack (removal of offspring of another type). We analyze such \textit{total-current population-dependent BPs}  in what we call  \textit{throughout super-critical regime} - the expected number of offspring produced by any individual is strictly greater than one, for all population sizes. We will refer to the proportion of the current population size (of one of the types) as the proportion and the time-asymptotic proportion as the limit proportion.

The literature mainly considers offspring that depend only on the current population; such models are essential in several biological applications (for example, \cite{klebaner1984population, yakovlev2009relative}). Recently, authors in  \cite{agarwal2022saturated, hautphenne2022fluid} introduced total-population dependent BPs; however, both papers analyze the BPs which shift from the super-to-sub critical regime, while we are interested in throughout-super-critical BPs. To the best of our knowledge, no other work considers such total-population dependency.

The importance of limit proportions is discussed in various papers, for example, \cite{ranbir2019decomposable, jagers1969proportions, klebaner1989geometric,  klein1980multitype} and several others. Further, they are crucial objects for the analysis of many applications. For example, authors in \cite{kapsikar2020controlling} design a warning mechanism robust against fake news propagation, where the control depends on the proportion of posts marked as fake. In \cite{agarwal2021co}, we study the relative visibility of advertisement posts defined in terms of the limit proportion of unread copies of posts shared by competing content providers. The limit proportions in prey-predator BP of \cite{coffey1991galton} denote
the proportions in which preys and predators co-survive (if at all).

To analyze proportions, it is sufficient to study the embedded chain of the underlying BP. This study is derived using stochastic approximation (SA) techniques (e.g., \cite{kushner2003stochastic}); we have previously used such an amalgam of SA-based methods in BPs in \cite{agarwal2021co, agarwal2022saturated, kapsikar2020controlling}. In this chapter, we include a notion of hovering around saddle points and prove that the sets of attractors and saddle points of an autonomous, non-smooth ordinary differential equation (ODE) almost surely describe the limit proportion. In fact, we prove that the limit set of a single-dimensional ODE suffices. We also prove that the ODE solution approximates certain normalized trajectories of the current and total population sizes over any finite time window. 

Previously, SA based approach has been used in the \polya urn (stochastic process closely related to BPs) literature to investigate limit proportions of the balls of a specific colour (see, for example, \cite{athreya1968embedding, higueras2006central, janson2004functional, arthur1987non}). However, the urn-based literature majorly deals with non-extinction scenarios and considers dependency on the current number of balls (not total) in the urn. Further, to the best of our knowledge, no finite time approximation trajectories exist for \polya urn-based models. Furthermore, we also introduce and analyze `BP with attack', where deletion of offspring (attack) from a population type and addition of the same to the other type (acquisition) occurs, in addition to the production of offspring of own type. Thus, this chapter significantly generalizes the models not only in the BP literature but also in the \polya urn literature by including (total and current) population dependency and negative offspring. We provide a more extensive comparison to the existing results in Section \ref{sec_survey}.

\textbf{Organization:} The main result is provided in Section \ref{sec_probdesc_mainresult} and proved in Section \ref{sec_proof_thrm1}. The ODE analysis is derived in Section \ref{procedureAR}, while BP with attack and its application are in Section \ref{sec_BPA}. Section \ref{sec_6} discusses numerical examples for finite time approximation.

\textbf{Notations:} For convenience, we refer the random variable and the corresponding sequence by the same symbol when the context is clear, for example, $\Ups_n$. We abbreviate infinitely often as i.o. and almost surely as a.s. We also use acronyms like BP, SA and ODE defined in the introduction. 
For any function $f$ and time $\tau$, let $f(\tau^-) := \lim_{t \uparrow \tau} f(t)$ and $f(\tau^+) := \lim_{t \downarrow \tau} f(t)$.

\subsection{Problem description}\label{sec_prob}
Consider $x$ and $y$-types of populations, and  
let $\cx_0, \cy_0$ be their respective initial sizes. The lifetime of any individual of any type is exponentially distributed with parameter $0 < \lambda < \infty$ (i.e., \textit{we consider Markovian BPs}). The time instance at which an individual completes its lifetime is referred to as its `death' time.   

Let $\Cx(t), \Cy(t)$ be the \textit{current population} and $\Ax(t), \Ay(t)$ be the \textit{total population} sizes at time $t$. 
Define $\Om(t) := (\Cx(t), \Cy(t), \Ax(t), \Ay(t))$ 
and observe 
$(\Ax(0), \Ay(0)) = (\cx_0, \cy_0)$. Let $\tau$ be the death time of any individual. Let $\offs_{ij}(\Om(\tau^-))$, with  $i, j \in \{x, y\}$, be integer-valued random variables representing $j$-type offspring produced by an $i$-type parent, conditioned on the sigma algebra $\sigma\{\Om(\tau^-)\}$. 
Basically, when $\Om(\tau^-) = \om$, the random offspring are represented by $\offs_{ij}(\om)$ for each $i, j$.
When an individual of $i$-type dies, the sizes of $i$ and $j$-type populations change by $\offs_{ii}(\Om(\tau^-))$ and $\offs_{ij}(\Om(\tau^-))$ respectively\footnote{For each $i, j$, the distribution of $\offs_{ij}(\Om(\tau^-))$ depends on the population size $\Om(\tau^-)$, and not on the value of the epoch, $\tau$.}. Further, the current size (not the total size) of $i$-type reduces by $1$ due to death. The dynamics can then be written as follows, when an $i$-type parent dies, for $i, j \in \{x, y\}$ and $j \neq i$:
\begin{equation}\label{evolve_x_up_time}
\begin{aligned}
C^i(\tau^+) &= C^i(\tau^-)  + \offs_{ii}(\Om(\tau^-)) - 1, \ \ \  A^i(\tau^+) = A^i(\tau^-)  + \offs_{ii}(\Om(\tau^-)), 
\\
C^j(\tau^+) &= C^j(\tau^-) + \offs_{ij}(\Om(\tau^-)), \ \ \ A^j(\tau^+) = A^j(\tau^-) + \offs_{ij}(\Om(\tau^-)).
\end{aligned}
\end{equation}

We consider a significantly generic framework to study \tcprocessnospace, which includes `attack+acquisition' (acquired individuals change their type);  negative (valued) offspring are used to model such attacks. 

In any BP, the expected/mean offspring plays a determining role in the growth of any population. In this chapter, we are keen to analyze the super-critical\footnote{See \cite{athreya1968some, athreya2004branching} for an introduction to super-critical population-independent BPs.} variant of TC-BPs, which we define formally in the next few lines.
Let $\om = (\cx, \cy, \ax, \ay)$ be a realisation of the random vector $\Om$. Let \textit{$m_{ij} (\om) := E[ \offs_{ij} (\om) ]$ for $i, j \in \{x, y\}$ represent the conditional expectation of the number of offspring, conditioned on $\om$}; we refer these as mean functions and $M(\om) := [m_{ij}(\om)]$ as mean matrix. Then, any BP which satisfies $m_{ix}(\om) + m_{iy}(\om) > 1$ for each $\om$ and $i\in\{x,y\}$ is called \underline{throughout-super-critical BP}. We assume the following for the random number of offspring conditioned on $\om$, which also ensures such super-criticality:
\begin{enumerate}[label=\textbf{B.\arabic*}, ref=\textbf{B.\arabic*}]
    \item \label{a1} There exist two integrable random variables $\overline{\offs}$ and $\underline{\offs}$ which bound the random offspring as: $\underline{\offs} \leq \offs_{ix}(\om) + \offs_{iy}(\om) \leq \overline{\offs}$ a.s., for each $\om$. Also,  $E[\overline{\offs}^2] < \infty$ and $E[\underline{\offs}] > 1$. Further, $\offs_{ii}(\om) \geq 0$ a.s., for each $i, \om$.
\end{enumerate}
Like the population-independent counterparts, the \tcprocess satisfying \ref{a1} also exhibits \textit{dichotomy}: the sum current population, $S^c(t) := \Cx(t) + \Cy(t)$ either explodes (i.e., $S^c(t) \to \infty$ as $t \to \infty$) exponentially at a rate at least $\lambda(E[\underline{\offs}]-1)$ or gets extinct ($S^c(t) = 0$ for all $t \geq t_e$ where $t_e < \infty$) a.s., by Lemma \ref{lemma_sum_pop} in Appendix \ref{appendix_prelim}. Now, our aim is two-fold: (i) to evaluate the limit proportion, $\lim_{t \to \infty} \frac{\Cx(t)}{\Cx(t) + \Cy(t)}$ in non-extinction paths, and (ii) to derive the deterministic approximate trajectories for the underlying BP.

\section{Main result}\label{sec_probdesc_mainresult}
 
When one considers a process which explodes with time, like a typical BP, it is a common practice to scale the process appropriately such that the scaled process converges to a  finite limit; this enables the asymptotic study of the rate of explosion,   proportions of various components of the process, etc. Further, since we are primarily interested in studying limit proportion, it suffices to analyze the embedded process (discrete-time chain defined at death instances). It is important to observe here that such an embedded process is very different from a corresponding BP in discrete-time.

Consider $n \geq 1$. Let $\tau_n$ be the time at which $n$-th individual dies. Let $\Om_n := (\Cx_n, \Cy_n, \Ax_n, \Ay_n)$ be the individual (current and total) populations and $S^c_n :=  \Cx_n + \Cy_n$ be the sum current population, both immediately after ${\tau}_n$, e.g., $\Cx_n = \Cx({\tau}_n^+)$. The current population can get extinct, and thus let $\nu_e := \inf \{n : S_n^c = 0\}$ be the extinction epoch,  
with the usual convention that $\nu_e = \infty$, when $S_n^c > 0$ for all $n$. \textit{For the sake of completion, define $\Om_n := \Om_{\nu_e}$ and $\tau_{n} :=\tau_{\nu_e}$, for all $n \geq \nu_e$, when $\nu_e < \infty$.}

Analogous to $S_n^c$, define the sum total population, $S^a_n := \Ax_n + \Ay_n$. Define the ratios:
\begin{align}\label{eqn_ratios}
\begin{aligned}
    \Ups_n &:= \left(\Pc_n, \Tc_n, \Pa_n, \Ta_n\right), \mbox{ where}\\
    \Pc_n &:= \frac{S^c_{n}}{n}, \Tc_n := \frac{\Cx_{n}}{n}, \Pa_n := \frac{S^a_{n}}{n} \mbox{ and } \Ta_n := \frac{\Ax_{n}}{n}, \mbox{ for } n\geq 1, \mbox{ with}
    \end{aligned}
\end{align}
$\Ups_0 := (\cx_0 + \cy_0, \cx_0, \cx_0 + \cy_0, \cx_0)$.  \textit{Define $\Bc_n := \Tc_n/\Pc_n  = \Cx_n/S_n^c$ as the proportion of $x$-type population among current population}; observe that conditioned on $\Om_n$, the probability of $x$-type individual dying before others is given by $\Bc_n$ in Markovian BPs. Let $\ups := (\pc, \tc, \pa, \ta)$ be a realisation of $\Ups$ and $\bc = \cx/(\cx+\cy) = \tc/\pc$ be a realisation of $\Bc$. 

In the literature, it has been a common practice to assume that the mean matrix converges to a constant matrix for studying (current) population-dependent BPs (\cite{jagers1997coupling, klebaner1984population, klebaner1989geometric}) and they assume convergence at a certain rate (as in \ref{a2} given below). \textit{We extend such work by allowing our total-current population-dependent mean functions, $m_{ij}(\om)$, to converge to proportion-dependent mean functions, $\minf_{ij}(\bc(\ups))$ (which can further be discontinuous)}, while still using the similar convergence criterion. In other words, the limit mean matrix in our case can depend on the proportion.
\begin{enumerate}[label=\textbf{B.\arabic*}, ref=\textbf{B.\arabic*}]
\setcounter{enumi}{1}
    \item Define $\bc(\ups) := \tc/\pc = \cx/s^c$. As sum current population, $s^c \to \infty$:  \label{a2}
    \begin{align*}
    |m_{ij}(\om) - \minf_{ij}(\bc(\ups))| \leq \frac{1}{(s^c)^\alpha}, \mbox{ for each }i, j \in \{x, y\}, \mbox{ for some } \alpha \geq 1.
    \end{align*}
\end{enumerate}
Under \ref{a1}-\ref{a2}, we analyze the ratios $\Ups_n$ using SA techniques; specifically, using the solutions of the following ODE:
\begin{align}\label{eqn_ODE}
\dot{\ups} &= \ga(\ups) = \mathbf{h}(\bc)1_{\{\pc > 0\}} - \ups, \mbox{ where } \mathbf{h}(\bc) := (h_{\psi}^{c},  h_{\theta}^{c},  h_{\psi}^{a},  h_{\theta}^{a}), \mbox{ with} \nonumber \\
        h_{\psi}^{c}(\bc) &= \bc \bigg(\minf_{xx}(\bc) +     \minf_{xy}(\bc)\bigg) + (1-\bc)\bigg(\minf_{yy}(\bc) + \minf_{yx}(\bc)\bigg) - 1,  \nonumber\\
        h_{\theta}^{c}(\bc)  &= \bc \bigg(\minf_{xx}(\bc) - 1\bigg) + (1-\bc) \minf_{yx}(\bc), \\
        h_{\psi}^{a}( \bc) &=  \bc \bigg(\minf_{xx}(\bc) + \minf_{xy}(\bc) \bigg) + (1-\bc)\bigg(\minf_{yy}(\bc) + \minf_{yx}(\bc) \bigg)   \mbox{ \normalsize  and}  \nonumber  \\
        h_{\theta}^{a}( \bc)  &= \bc \minf_{xx}(\bc) + (1-\bc) \minf_{yx}(\bc).  \nonumber
\end{align}Given that the above ODE is autonomous and non-smooth (the right hand side is discontinuous),  we next assume the existence of the  unique solution in extended sense (the definition is same as in Definition \ref{defn_solution_prelim}, but is re-written here for the ease of reading):
\begin{definition}\label{defn_solution}
A function $\ups(\cdot)$ is said to be an \underline{extended solution of ODE \eqref{eqn_ODE}} if it is absolutely continuous, and satisfies the equation \eqref{eqn_ODE} for almost all $t \geq 0$.
\end{definition}
\begin{enumerate}[label=\textbf{B.\arabic*}, ref=\textbf{B.\arabic*}]
\setcounter{enumi}{2}
    \item There exists a unique solution $\ups(\cdot)$ for ODE \eqref{eqn_ODE} in the extended sense  over any bounded interval. \ifnonauto{}{{\color{red}Further, assume that the solution is continuous w.r.t. initial conditions, except $0$ such that  for each $T < \infty$, 
    \begin{align}\label{eqn_continuity}
        \sup_{0 \leq t \leq T} d(\ups_{(n)}(t), \ups_{(\infty)}(t)) \to 0, \mbox{ if } \ups_{(n)}(0) \to \ups_{(\infty)}(0), \mbox{ as }n \to \infty,
    \end{align}
    where $\ups_{(n)}(\cdot)$ is the unique solution with initial condition $\ups_{(n)}(0)$, for all $n \leq \infty$.}} \label{a3}
\end{enumerate}
Assumption \ref{a3} is immediately satisfied by standard results in ODEs if $\minf_{ij}(\cdot)$ are Lipschitz continuous and if there was no indicator, $1_{\{\pc > 0\}}$ (see \cite[Theorem 1, sub-section 1.4]{piccinini2012ordinary}). We prove the same for ODE \eqref{eqn_ODE} also when  $\minf_{ij}(\cdot)$ are discontinuous and under certain conditions in Theorem \ref{thrm_attractors_beta} in Section \ref{procedureAR}; such discontinuous functions are typical for BPs with attack.

For systems modelling the BPs,  the following subset of the domain is relevant:
\begin{align}\label{invariant_set}
    \cD_I &:= \{\ups \in (\mathbb{R}^+)^4: \tc \leq \pc \leq \pa \mbox{ and } \ta \leq \pa\}.
\end{align}Therefore, we will be interested in initial conditions $\ups(0) \in \cD_I$ for the ODE \eqref{eqn_ODE}.

Next, we recall the definitions of asymptotically stable and  saddle points for autonomous ODE (see \cite{piccinini2012ordinary}), that facilitates the desired a.s. convergence of ratios $(\Ups_n)$ - 
some of the definitions are stated differently to suit our purpose and  can also be applied for the cases with generalised solutions of ODE. These definitions are exactly the same as in Chapter \ref{ch:basics}, but are re-written here for the ease of reading.
\begin{definition}\label{defn_equi_pt}
    A set $\mathbf{E} := \{\ups: \mathbf{g}(\ups) = 0\}$ is called the set of \underline{equilibrium points} for the ODE \eqref{eqn_ODE}.
\end{definition}
Define open ball, $N_\epsilon(\cA) := \{x: d(x, \cA) < \epsilon\}$ for some finite set $\cA$.
\begin{definition}\label{defn_stable}
A subset $\cA$ of $\mathbf{E}$ is said to be a \underline{(locally) stable set} for ODE \eqref{eqn_ODE}
if for any $\epsilon > 0$, there exists a $\delta > 0$ such that every solution 
 of the ODE $\ups(t) \in N_\epsilon(\cA)$ for every $t > 0$, if initial condition $\ups(0) \in N_\delta(\cA)$.
 \end{definition}
\begin{definition}\label{defn_attractors}
A subset $\cA$ of the locally stable set is called an \underline{attractor} or \underline{asymptotically} \underline{stable set} and $\cD_\cA \subset \cD_I$ is the \underline{domain of attraction} for ODE \eqref{eqn_ODE}
if every solution $\ups(t) \to \cA$ as $t \to \infty$ when $\ups(0) \in \cD_\cA$.
\end{definition}
Let $\cA^\complement$ be the complement of $\cA$.
\begin{definition}\label{defn_saddle}
A set ${\cR} \subset \cA^\complement \cap \mathbf{E}$ is said to be \underline{saddle set} if 
there exists ${\mathbf D}_S$ such that $d(\ups(t) , \cA) \stackrel{t \to \infty}{\longrightarrow} 0 $   for some $\ups(0) \in \cR^\complement \cap {\mathbf D}_S$ and   $d(\ups(t), \cR) \stackrel{t \to \infty}{\longrightarrow} 0 $  for some other $\ups(0) \in \cR^\complement \cap {\mathbf D}_S$. 
\end{definition}
Next, we focus on special types of saddle points which are attracted exponentially to $\cR$ along a particular affine sub-space, and to $\cA$ in the remaining space. Such saddle points are facilitated by the virtue of ODE structure in \eqref{eqn_ODE}.
\begin{definition}\label{defn_q_as}
Any non-zero $\ups^*  \in \cR$  is said to be (quasi) \underline{q-attractor} if (i) for any  $\ups(0) \in {\mathbb S}(\ups^*) := \{\bc(\ups) = \bc(\ups^*)\}$, $\ups(t) \stackrel{t \to \infty}{\longrightarrow} \ups^*$ exponentially, and (ii)  $\ups(t) \stackrel{t \to \infty}{\longrightarrow} \cA$ for other initial conditions. Further, if $\ups^* = \mathbf{0} \in \cR$, it is called \underline{q-attractor} if the above happens with  $ {\mathbb S}(\ups^*) := \{\pc = 0\}$.
\end{definition}

By virtue of ODE structure in \eqref{eqn_ODE}, we will see that the saddle points in our case are q-attractors defined in Definition \ref{defn_q_as} (see Theorem \ref{thrm_attractors_beta} of Section \ref{procedureAR}). Finally, consider the following subset of $\cD_I$, which represents the combined domain of attraction towards $\cA\cup\cR$ (attractors and saddle points):
\begin{align}\label{eqn_domain of attraction}
    \cD &:= (\cD_\cA \cup \cD_\cR) \cap \cD_I  = \{\ups \in \cD_I: \ups(t) \to {\cA}\cup \cR \mbox{ as } t \to \infty, \mbox{ if } \ups(0) = \ups\}.
\end{align}
Thus, if the ODE starts in $\cD$, it converges asymptotically to $\cA\cup\cR$. The main result is:  when BP ($\Ups_n$) visits some compact subset of $\cD$ i.o., then either $\Ups_n$ converges asymptotically to $\cA\cup\cR$ or hovers around $\cR$ (notion defined below).
\begin{definition}\label{defn_hovers}
    The stochastic process $\Ups_n$ is said to  \underline{hover around a set} $\cR$ if $ \Ups_n \in N_\delta (\cR)$ i.o., for all $\delta >0 \mbox{ and  } \Ups_n \notin N_{\delta_1} (\cR) \mbox{ i.o., for some } \delta_1 >0$.
\end{definition}
\textit{Hovering around depicts a type of the limiting behavior of the stochastic process where the trajectory goes arbitrarily close to the set $\cR$ i.o., but still comes out of a neighbourhood of it i.o.}
Contrary to the existing results, our SA based Theorem \ref{thrm1} given below proves the possibility of above behavior as well as convergence to the saddle set ($\cR$). We require an extra assumption and the proof is deferred to the next section.
\begin{enumerate}[label=\textbf{B.\arabic*}, ref=\textbf{B.\arabic*}]
\setcounter{enumi}{3}
    \item (a) Let $\cA\cap\cD_I$ be the attractor set  as in Definition \ref{defn_attractors}. Let each $\ups \in \cR\cap\cD_I$ be the q-attractor as in Definition \ref{defn_q_as}. Consider $\cD$ as in \eqref{eqn_domain of attraction} and let  $\cS := \cD \cap \{\pa \leq b\}$, for some $b > 0$, be a compact subset of combined domain of attraction.
    
    (b) Assume $p_{b} := P(\mathcal{V}) > 0$, where $\mathcal{V} :=  \{\omega : \Ups_n(\omega) \in \cS \mbox{ i.o.}\}$. \label{a4}
\end{enumerate}

\begin{theorem}\label{thrm1}
  Under \ref{a1}-\ref{a3}, we have:
  \begin{enumerate}[label=(\roman*)]
        \item For every $T>0$, a.s.  there exists a sub-sequence $(n_l)$ such that:
            $$
            \sup_{k: t_k \in [t_{n_l}, t_{n_l} + T]} d(\Ups_k, \ups(t_k - t_{n_l})) \to 0  \mbox{ as } l \to \infty, \mbox{ where } t_n := \sum_{k=1}^n \frac{1}{k} \mbox{ and}
            $$
        $\ups(\cdot)$ is the extended solution of ODE \eqref{eqn_ODE} which starts at $\ups(0) =
        \lim_{n_l \to \infty} \Ups_{n_l}$.
        \item Further, assume \ref{a4}. Then, $P({\cal C}_1 \cup {\cal C}_2) \geq p_b$, where
        \begin{align*}
                    \hspace{-10mm}
        \begin{aligned}
                 \hspace{7mm}{\cal C}_1 &: =\{\Ups_n \to (\cA \cup \cR)\cap \cD_I \mbox{ as } n \to \infty\}, \mbox{ and }
                {\cal C}_2 := \{ \Ups_n \mbox{ hovers around } \cR \}. \hspace{5mm}  \mbox{ \eop}
            \end{aligned}
        \end{align*}
  \end{enumerate}
\end{theorem}
Thus, the BP either converges to attractor/saddle set or it can hover around a saddle point, with combined probability at least $p_b > 0$; in fact, the saddle points are q-attractors defined in Definition \ref{defn_q_attractor}. The above result is a specific case of Theorem \ref{thrm_finite_time} and Corollary \ref{cor_SA_prelim} when applied to SA-based scheme corresponding to BPs. We will show that \ref{a1}-\ref{a4} are satisfied for BP with attack in Section \ref{sec_BPA}, with $p_b = 1$, i.e, the above results are true a.s.

\subsection{Significance of Theorem \ref{thrm1}} 

\noindent \textbf{BP trajectories -} Theorem \ref{thrm1}(i) provides \textit{a novel approach for studying the asymptotic trajectory of the BPs using ODE solution}.
Consider the solution of ODE \eqref{eqn_ODE} initialised with $\lim_{n_l \to \infty} \Ups_{n_l}$. Then, the BP $\Ups_k$ is  close to ODE solution $\ups(t_k-t_{n_l})$ at all transition epochs, $k$ with $t_k \in [t_{n_l}, t_{n_l} + T]$. This approximation improves as $n_l$ increases. \textit{The result is true a.s., for all $T < \infty$, independent of $p_b$ and only requires \ref{a1}-\ref{a3}.}  

We suggest a better finite-time approximation using a non-autonomous ODE in Section \ref{sec_6}, inspired by \cite{agarwal2022saturated} where saturated  total population-dependent BP is studied.

\vspace{1mm}
\noindent \textbf{Limit proportion -}
Theorem \ref{thrm1}(ii) provides an alternate approach to derive limit behaviour via the attractors or saddle points (q-attractors) of ODE \eqref{eqn_ODE}.

In \textit{extinction paths}, where both populations get extinct, $\Ups_n \to \mathbf{0}$  as $n \to \infty$, say with probability $p_e > 0$. Thus, extinction paths are in the set $\mathcal{V}$ of \ref{a4}. While in survival paths, the BP either converges or hovers around $\left(\cA \cup (\cR - \{\mathbf{0}\})\right)\cap \cD_I$, with probability at least $p_b - p_e$. 
As an example of convergence to saddle point, the vector $\mathbf{0}$ is a saddle point of ODE \eqref{eqn_ODE} (shown in the proof of Theorem \ref{thrm_attractors_beta}) and is also a limit of the BP in extinction paths.

\vspace{1mm}
\noindent \textbf{Population independent to population dependent BPs -}
One can analyze any general BP with limit mean matrix (say) $M^\infty$ using a population-independent BP with mean matrix as $M^\infty$. The knowledge about limits of the  latter BP can be useful in deriving ODE limits and, thus, the limits  of the former BP. One still needs to show that the  former BP  visits the domain of attraction i.o., given latter  visits the same i.o.

\vspace{1mm}
\noindent \textbf{Limitation -} By Theorem \ref{thrm1}, one can not comment on the individual probability of $\Ups_n$ converging to a particular limit in $\left(\cA \cup \cR\right)\cap \cD_I$ or the likelihood of hovering around. Further, $\Bc_n \to \{0, 1\}$ does not always imply the extinction of $x$ or $y$-type population; however, in BP with attack, this is true (see the discussion at the end of Appendix \ref{appendix_prelim}).

\section{Proof of Theorem \ref{thrm1}}\label{sec_proof_thrm1}

From equation \eqref{evolve_x_up_time}, the embedded process immediately after $n$-th death, when the death is for example of an $x$-type individual, is given by:
\begin{equation}\label{evolve_x_up_gen}
\begin{aligned}
\Cx_{n} &= \Cx_{n-1}  + \offs_{xx, n}(\Om_{n-1}) - 1, \ \ \  \Ax_{n} = \Ax_{n-1}  + \offs_{xx, n}(\Om_{n-1}), \\
\Cy_{n} &= \Cy_{n-1} + \offs_{xy, n}(\Om_{n-1}),  \ \ \ \Ay_{n} = \Ay_{n-1} + \offs_{xy, n}(\Om_{n-1}).
\end{aligned}
\end{equation}
To begin with, we make an important observation to derive an appropriate SA-based scheme which represents the above dynamics and also to prove a boundedness assumption for ratios $\Ups_n$ required for most SA-based studies. 

\noindent \textbf{Key idea:}
Consider a BP with population-independent and positive offspring, i.e., in \ref{a1}, assume $\offs_{ix}(\om) + \offs_{iy}(\om) = \overline{\offs}$ for all $\om$ and all $i \in \{x, y\}$. Let $\overline{\Pi}_n$ represent the sample mean formed by the sequence of   offspring plus the initial population size, i.e., 
\begin{align}\label{eqn_overline_S_n}
    \overline{\Pi}_n = \frac{1}{n}\left(\sum_{k=1}^n  \overline{\offs}_k + s_0^a \right).
\end{align}
By strong law of large numbers, $\overline{\Pi}_n \to \overline{m} := E[\overline{\offs}_1]$ a.s. 
For this special case, ${\Psi}^a_n = \overline{\Pi}_{n}1_{n < \nu_e} + \nu_e \overline{\Pi}_{\nu_e}/n 1_{n \geq \nu_e}$ (see \eqref{evolve_x_up_gen} and recall ${\Psi}^a_n = (\Ax_n+\Ay_n)/n$); hence ${\Psi}^a_n $ converges either to $0$ (in extinction paths, i.e., $\nu_e < \infty$) or to ${\overline m}$ (in survival paths); $\Pc_n$ respectively converges to $0$ or $\overline{m}-1$. This observation actually completes the  proof  for this special case with $\cA = \{(0, 0), (\overline{m}-1, \overline{m})\}$, further when single population (say $x$-type) is considered. It is well known that the sample mean \eqref{eqn_overline_S_n} \textit{can be written as a SA-based scheme} and in \eqref{eqn_SA_pi} given below, we will see that this is true even for the general case. Further, clearly, \eqref{eqn_overline_S_n} becomes an upper bound for all components of $\Ups_n$, which \textit{helps in bounding $\Ups_n$  uniformly in $n$ and a.s.} (see \eqref{Eqn_XnYnSnetc_general} given below), again under~\ref{a1}.  

Analogous to $\overline{\Pi}_n$ as in \eqref{eqn_overline_S_n}, one can construct a lower bounding sequence using $\underline{\offs}$ of \ref{a1}; this provides a uniform positive lower bound for $\Pc_n$, which will help the proof. 


{\bf Proof:}
For any $n \geq 1$, let $\Pi_n$ represent the sample mean formed by the sequence of (possibly $\om$-dependent) offspring plus the initial population size till $\nu_e$, i.e., (recall, $\Phi_{k-1}$ is the population-size vector) 
\begin{align}\label{eqn_pi_n}
     \Pi_n &= \frac{1}{n}\left(\sum_{k=1}^{\min\{n, \nu_e\}} \left( H_k\offs_{x, k}(\Om_{k-1}) + \overline{H}_k \offs_{y, k}(\Om_{k-1})  \right) + s_0^a \right), \mbox{ where } \nonumber  \\
    \offs_{x, k} &:= \offs_{xx, k} + \offs_{xy, k}, \ \  
    \offs_{y, k} := \offs_{yy, k} + \offs_{yx, k},  
\end{align}and $H_k= 1-\overline{H}_k$ is the indicator that an $x$-type individual dies at $k$-th epoch. It is easy to observe that ${\Pi}_n$  can be re-written as (observe that $\nu_e$ also equals $\inf\{n: \Pc_n = 0\}$)
\begin{align}\label{eqn_SA_pi}
\Pi_n =  \Pi_{n-1} + \frac{1}{n} \bigg[ \left(H_n \offs_{x, n}(\Om_{n-1}) + \overline{H}_n \offs_{y, n}(\Om_{n-1}) \right)1_{\Pc_{n-1} > 0} - \Pi_{n-1} \bigg].
\end{align}

In fact, $\Pi_n = \Pa_n$ for all $n\geq 1$, and so the above iterative equation represents $\Pa_n$. Similarly, other ratios in $\Ups_n$ can be re-written as (see \eqref{evolve_x_up_gen}, \eqref{eqn_pi_n}, \eqref{eqn_SA_pi}): 
\begin{align}\label{eq_stoch_approx_scheme_general}
\begin{aligned}
\Ups_n &= \Ups_{n-1} + \frac{1}{n}\mathbf{L}_{n}, \mbox{ where } \mathbf{L}_{n} := (L_{n}^{\psi, c}, L_{n}^{\theta, c}, L_{n}^{\psi, a}, L_{n}^{\theta, a})^t, \mbox{ with}\\
L_{n}^{\psi, c} &:=  \left\{ H_n\bigg(\offs_{x, n}(\Om_{n-1})-1\bigg) + \overline{H}_n \bigg(\offs_{y, n}(\Om_{n-1})-1\bigg)\right\} 1_{\Pc_{n-1} > 0}   -  \Pc_{n-1}, \\
L_{n}^{\theta, c} &:= \left\{H_{n}\bigg(\offs_{xx, n}(\Om_{n-1}) - 1\bigg) + \overline{H}_{n} \offs_{yx, n}(\Om_{n-1})\right\}1_{\Pc_{n-1} > 0} - \Tc_{n-1},  \\
L_{n}^{\psi, a} &:=  \bigg\{ H_n \offs_{x, n}(\Om_{n-1}) + \overline{H}_n \offs_{y, n}(\Om_{n-1}) \bigg\}1_{\Pc_{n-1} > 0}  - \Pa_{n-1}, \mbox{ and}\\
L_{n}^{\theta, a} &:= \bigg\{H_{n}\offs_{xx, n}(\Om_{n-1}) + \overline{H}_{n} \offs_{yx, n}(\Om_{n-1})\bigg\}1_{\Pc_{n-1} > 0}  - \Ta_{n-1}. 
\end{aligned}
\end{align}


The proof of part (i) has two major steps: (a) to construct a sequence of piece-wise constant interpolated trajectories for almost all sample-paths; (b) to prove that the designed trajectories are equicontinuous in extended sense\footnote{\label{footnote_defn_equi}\begin{definition}
  {\bf Equicontinuous in extended sense (\cite[Equation (2.2), pp. 102]{kushner2003stochastic})):} \label{defn_equi} Suppose that for each $n$, $f_n(\cdot)$
is an $\mathbb{R}^r$-valued measurable function on $(-\infty,\infty)$ and $(f_n(0))$ is bounded.
Also suppose that for each $T$ and $\epsilon > 0$, there is a $\delta > 0$ such that
\begin{align}\label{eqn_footnote}
\limsup_n \sup_{0\leq t-s\leq \delta, |t| \leq T} |f_n(t) - f_n(s)| \leq \epsilon. 
\end{align}
 Then the sequence $(f_n(\cdot))$ is said to be equicontinuous in the extended sense.
\end{definition}}. These steps are majorly as in \cite[Theorems 2.1-2.2]{kushner2003stochastic}, but for the changes required for measurable $\mathbf{\ga}(\cdot)$.

For many steps of the proof, we will work only with $\tc$-component of the vector $\ups$, when the proof for the remaining components goes through in exactly similar manner.

Let $\Ups^n(\cdot) := (\Psi^{n, c}(\cdot), \Theta^{n, c}(\cdot), \Psi^{n, a}(\cdot), \Theta^{n, a}(\cdot))$  be  the constant piece-wise interpolated trajectory defined as below (see \eqref{eq_stoch_approx_scheme_general}, and recall $t_n = \sum_{i=1}^n \epsilon_{i-1}$, where $\epsilon_{i-1} = \frac{1}{i}$):
\begin{align}\label{eqn_interpolated_traj_1}
    \Theta^{n, c}(t) := \Tc_n + \sum_{i=n}^{\eta(t_n + t)-1} \epsilon_i L_i^{\theta, c}, \mbox{ for all } t \ge 0,
\end{align}
$\Psi^{n, c}(t), \Psi^{n, a}(t)$ and $\Theta^{n, a}(t)$ are defined analogously. Towards proving equicontinuity, we first consider upper-boundedness of $\Ups^n(0) = \Ups_n$, as the iterates are trivially lower bounded by $0$. The claim is immediately true by strong law of large numbers a.s., to be more precise on the set $\{\overline{\Pi}_n  \to \overline{m}\}$,  because of the following  observation (see \eqref{eqn_overline_S_n}-\eqref{eq_stoch_approx_scheme_general}):
\begin{equation}\label{Eqn_XnYnSnetc_general}
     \Pc_n \leq \Pa_n \mbox{ and } \Tc_n \leq \Ta_n \le \Pa_n = \Pi_n \leq  \overline{\Pi}_n   \mbox{ for all } n, 
\end{equation}
as then for any sample path $\omega \in \{\overline{\Pi}_n  \to \overline{m}\}$ and   $\epsilon > 0$, there exists a $n_\epsilon(\omega) < \infty $,
\begin{align}\label{eqn_bounded_iterates}
\begin{aligned}
\hspace{6mm}
   \sup_n  \max\{\Theta^{n, c}(0), \Psi^{n, c}(0), \Theta^{n, a}(0), \Psi^{n, a}(0)\}  \leq \hspace{16mm} &    \\
   &\hspace{-8cm}\max\left \{ \max_{n \le n_\epsilon(\omega)\}}  \max\{\Theta^{n, c}(0), \Psi^{n, c}(0), \Theta^{n, a}(0), \Psi^{n, a}(0)\}, \   \overline{m} + \epsilon \right \}. 
   \end{aligned}
\end{align}
%
%
Towards the second part of equicontinuity (see \eqref{eqn_footnote} in footnote \ref{footnote_defn_equi}), the interpolated trajectory for $\Theta^{n,c}(\cdot)$ in \eqref{eqn_interpolated_traj_1} can be re-written in `almost integral form', for any $t \geq 0$:
\begin{align}\label{eqn_diff_term1}
\begin{aligned}
     \Theta^{n, c}(t) &:= \Tc_n + \int_0^t \rho_\theta^c(\Ups^n(s), s) ds + {\cal E}_1^{n, c}(t), \mbox{ with the difference term, } \\
     {\cal E}_1^{n, c}(t) &:= \sum_{i=n}^{\eta(t_n + t)-1} \epsilon_i L_i^{\theta, c} - \int_0^t \rho_\theta^c(\Ups^n(s), s) ds, \mbox{ where}
\end{aligned}
\end{align}
$\gna = (\rho_\psi^c, \rho_\theta^c, \rho_\psi^a, \rho_\theta^a)$ is the conditional expectation, $E[\mathbf{L}_n|\mathcal{F}_n] =: \gna(\Ups_n, t_n)$, with respect to the sigma algebra, ${\cal F}_n  := \sigma\{\Om_k : 1 \leq k < n \}$, and equals (see \eqref{eq_stoch_approx_scheme_general}):

 \vspace{-10mm}
{\small
\begin{align}\label{eqn_nonauto_g}
\rho_\psi^c (\ups, t) &:=      \left\{ \bc \bigg(m_{xx}(\om) +     m_{xy}(\om)\bigg) + (1-\bc)\bigg(m_{yy}(\om) + m_{yx}(\om)\bigg) - 1 \right\} 1_{\{\pc > 0\}}- \pc, \nonumber \\
\rho_\theta^c (\ups, t)  &: = \left\{ \bc \bigg(m_{xx}(\om) - 1\bigg) + (1-\bc) m_{yx}(\om)\right\} 1_{\{\pc > 0\}} - \tc,  \\
\rho_\psi^a (\ups, t) &:=      \left\{ \bc \bigg(m_{xx}(\om) + m_{xy}(\om)\bigg) + (1-\bc)\bigg(m_{yy}(\om) + m_{yx}(\om)\bigg)\right\} 1_{\{\pc > 0\}} - \pa,   \mbox{ \normalsize  and,}  \nonumber \\
\rho_\theta^a (\ups, t)  &: = \bigg\{\bc m_{xx}(\om) + (1-\bc) m_{yx}(\om)\bigg\} 1_{\{\pc > 0\}} - \ta,   \mbox{ \normalsize with, }   \eta(t) := \max\left  \{ n: t_n \le t \right \},  \nonumber\\
\om &= \om(\ups, t) := 
  \big(\tc \eta(t),  \ (\pc-\tc)\eta(t),\ \ta \eta(t), \  (\pa-\ta)\eta(t)\big). \nonumber
\end{align}}

We further re-write the interpolated trajectory using the autonomous ODE \eqref{eqn_ODE}:
\begin{align}\label{eqn_diff_term2}
\begin{aligned}
     \Theta^{n, c}(t) &:= \Tc_n + \int_0^t g_\theta^c(\Ups^n(s)) ds + {\cal E}_1^{n, c}(t) + {\cal E}_2^{n, c}(t), \mbox{ where}
\end{aligned}
\end{align}
$$
{\cal E}_2^{n, c}(t) := \int_0^t \rho_\theta^c(\Ups^n(s), s) ds - \int_0^t g_\theta^c(\Ups^n(s)) ds.
$$
In Appendix \ref{appendix_B}, we show that ${\cal E}_1^{n, c}(t) + {\cal E}_2^{n, c}(t)$ converges uniformly to $0$, as $n \to \infty$, over any finite time window and further show:

\begin{lemma}\label{lemma_equi_cont_thrm1}
The sequence $(\Ups^n(\cdot))$ is equicontinuous in extended sense a.s. \eop
\end{lemma}
Now, consider the set $N$ of all sample paths for which $(\Ups^n(\cdot))$ is not equicontinuous - by Lemma \ref{lemma_equi_cont_thrm1}, $P(N) = 0$ (see proof of above Lemma for precise definition of $N$). Then, by extended version of Arzela-Ascoli Theorem \cite[section 4, Theorem 2.2, pp. 127]{kushner2003stochastic}, there exists a sub-sequence $(\Ups^{n_m}(\omega, \cdot))$ which converges to some continuous limit, call $\ups(\omega, \cdot)$, uniformly on each bounded interval for $\omega \notin N$ such that: 
\begin{align}\label{eqn_ups_inf}
\ups(t) = \lim_{n_m \to \infty} \ups_{n_m}(\omega) + \int_0^t \ga (\ups(s)) ds.
\end{align}
Thus, for every $\epsilon > 0$ and $T > 0$, there exists $n(\omega, \epsilon, T)$ such that: 
\begin{align}\label{eqn_dist_scheme_ODE_}
\sup_{l \in L} d({\Ups}_l , \ups(t_l - t_{n_m})) \leq \epsilon/2 \mbox{ for all } n_m \geq n(\omega, \epsilon, T), 
\end{align}
where $L:= \{l :  t_{n_m} \leq t_l \leq T + t_{n_m}\}$; observe for any $l \in L$, ${\Ups}^{n_m}(t) = {\Ups}_l$ if $t = t_l - t_{n_m}$. Now, we are left to show that $\ups(\cdot)$ in \eqref{eqn_ups_inf}, the solution of the fixed point equation (of the integral operator), is the extended solution of ODE \eqref{eqn_ODE} starting at $\ups(0) = \lim_{n_l \to \infty} \Ups_{n_l}$, i.e., 
$$
    \lim_{h \to 0} \frac{\ups(t+h) - \ups(t)}{h} = \ga(\ups(t)) = \frac{d \ups(t)}{d t} \mbox{ for almost all } t.
$$
One can easily show that the function $\ga \circ \ups$ is locally integrable, and thus, by 
\cite[Theorem 3.21]{folland1999real}, the claim holds. 
This completes part (i).

For part (ii), under \ref{a4}, the proof is again inspired from \cite{kushner2003stochastic} and \cite[Theorem 2.3.1, pp. 39]{kushner2012stochastic}, even when the solution of ODE \eqref{eqn_ODE} is in extended sense, not the classical one. Further major difference in the proof is to include the arguments required to prove the event of hovering around $\cR$. We complete this proof in Appendix \ref{proof_thrm1}. \eop

\section{Derivation of $\cA, \cR$ via analysis of proportion ODE}\label{procedureAR}

Under \ref{a2}, $\om$-dependent mean functions converge to just $\bc$-dependent mean functions, and thus, one may anticipate that the analysis of $ \bc(\ups(t))  = \bc(t)$ plays a crucial role. In fact, we claim and prove that the time limits of $\bc$, obtained from the following  limit ODE for $\bc$ (derived using \eqref{eqn_ODE}), leads to the required analysis:

\vspace{-6mm}
\begin{align}\label{eqn_beta_ODE}
\begin{aligned}
\dot{\bc} &= \frac{1}{\pc} g_\beta(\bc)1_{\{\pc > 0\}},\mbox{ where}\\ 
g_\beta(\bc) &:= - \bc \minf_{xy}(\bc) + (1-\bc) \minf_{yx}(\bc)\\
    &\hspace{15mm}+ \bc(1-\bc) \bigg\{\minf_{xx}(\bc) + \minf_{xy}(\bc) - \big(\minf_{yx}(\bc) +\minf_{yy}(\bc)\big)  \bigg\}.
    \end{aligned}
\end{align}
From above, $g_\beta$ depends only on $\bc$, thus, \textit{one might expect that the asymptotic analysis of $\bc$ is independent of other components of $\ups$}. We will see that this is indeed true, and in fact, asymptotic analysis of all components of $\ups$ can be derived using $g_\beta$.
In this regard, we define the following:
\begin{definition}
    Any point $\bstar \in [0,1]$ is  (projected) \underline{p-stable} if $\mathbf{h}(\bstar)$ is an attractor for ODE \eqref{eqn_ODE}; a $\bstar$ is called \underline{p-saddle} if $\mathbf{h}(\bstar)$ is a saddle point, more specifically, q-attractor defined in Definition \ref{defn_q_attractor}.
\end{definition}
\begin{figure}[htbp]
    \centering
    \includegraphics[trim = {2.8cm 0.1cm 3cm 0cm}, clip, scale = 0.4]{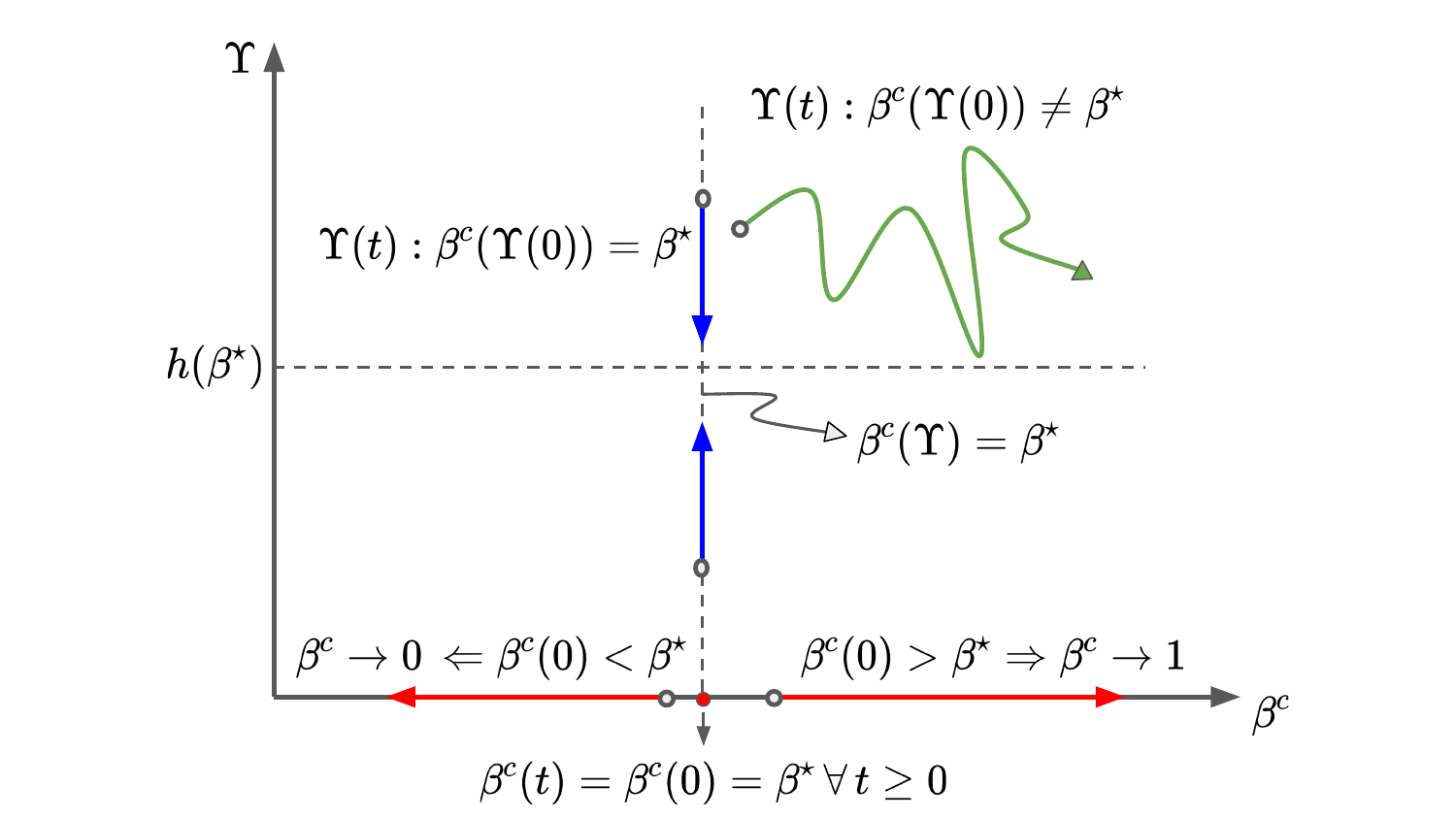}
    \caption{Repeller of \eqref{eqn_beta_ode_simple} leads to saddle point of \eqref{eqn_ODE}}\label{fig_saddle}
\end{figure}
Under certain conditions, we will show that the attractors of the following one-dimensional ODE are p-stable, while the repellers are p-saddle:
\begin{align}\label{eqn_beta_ode_simple}
    \dot{\bc} &= g_\beta(\bc).
\end{align}
When $\bstar$ is a repeller of \eqref{eqn_beta_ode_simple}, we have $g_\beta(\bstar) = 0$. Thus, when ODE \eqref{eqn_ODE} is initialised with $\bc(\ups(0)) = \bstar$, the ODE solution may remain in affine sub-space $\{\bc(\ups) = \bstar\}$ and may converge to $\mathbf{h}(\bstar)$ (see Figure \ref{fig_saddle}). 
While if $\bc(\ups(0)) \neq \bstar$,  one might expect the solution of ODE \eqref{eqn_ODE} to repel away from $\mathbf{h}(\bstar)$, by definition of repeller. These observations indicate that $\bstar$ should be p-saddle and we precisely prove the same in our \textit{second important result} below. This result is instrumental in deriving $\cA$ and $\cR$ using the limit set of ODE \eqref{eqn_beta_ode_simple}; see Appendix \ref{proof_thrm2} for the proof. 
\begin{theorem}\label{thrm_attractors_beta}
Consider the interval $[0,1]$ such that $g_\beta(0) \geq 0$ and $g_\beta(1) \leq 0$. Let ${\cal I} = \{x_i^* : 1 \leq i \leq n\}$ be the set of dis-continuities with $1 \le n < \infty$ and ${\cal J} :=  \{y_i^* : 1 \leq i \leq m\}  \subset {\cal I}^\complement$ be the set of points with $m < \infty$  (${\cal J}$ is empty when $m = 0$) such that:

    \noindent  (a) $g_\beta(x) = 0$ for each $x \in {\cal I} \cup {\cal J}$, i.e., ${\cal I} \cup {\cal J}$ is the set of equilibrium points for \eqref{eqn_beta_ode_simple},
    
    \noindent (b) for each $1\leq i \leq n$, there exists an open/closed/half-open non-empty interval around $x_i^* \in {\cal I}$, say ${\cal N}_i^*$, such that 
    
    (i) $\cup_{1\leq i\leq n} {\cal N}_i^* = [0,1]- {\cal J}$ and  ${\cal N}_i^* \cap {\cal N}_j^* = \emptyset$ for $i\neq j$,
    
    (ii) $g_\beta(\beta) > 0$ for all $\beta \in  {\cal N}^{-}_i:= {\cal N}^*_i\cap[0, x_i^*)$, $g_\beta$ is Lipschitz continuous on $ {\cal N}^{-}_i$,

    (iii) $g_\beta(\beta) < 0$ for all $\beta \in {\cal N}^{+}_i:
    = {\cal N}^*_i \cap (x_i^*, 1]$, $g_\beta$ is Lipschitz continuous on ${\cal N}^{+}_i$.

\noindent  Then, ODE \eqref{eqn_ODE} satisfies \ref{a3}. Further, the set ${\cal I}$ is an attractor for \eqref{eqn_beta_ode_simple} and p-stable for  \eqref{eqn_ODE}; also, ${\cal J}$ is the set of repellers for \eqref{eqn_beta_ode_simple} and p-saddle for \eqref{eqn_ODE}. 
Furthermore,  $\cA := \{\mathbf{h}(x_i^*) : x_i^* \in {\cal I}\}$ is the attractor set, $\cR := \{\mathbf{h}(y_i^*) : y_i^* \in {\cal J}\}\cup \{\mathbf{0}\}$ is the saddle set  in $\cD_I$ and entire $\cD_I$ is the combined domain of attraction for  \eqref{eqn_ODE}. \eop
\end{theorem}
We believe that the above Theorem can be extended for $g_\beta$ which is continuous, by standard ODE results, and we precisely do so in the next Chapter (see Theorem \ref{thrm_beta_ODE_prop}). Observe that the above Theorem is a special case of Theorem \ref{thrm_one_dim_ODE} and Theorem \ref{thrm_our_ODE}, when the right hand side of ODE \eqref{eqn_one_dim_ODE} is discontinuous at the equilibrium points, with $a = 0$ and $b = 1$. Here, we required $g_\beta$ to be discontinuous for BP with attack (see assumption \ref{k2} in Section \ref{sec_BPA}), and thus the hypothesis of Theorem \ref{thrm_attractors_beta}. The last part of the Theorem asserts that the p-stable/p-saddle points are the only attractors/saddle points of ODE \eqref{eqn_ODE}, other than $\mathbf{0} \in \cR$.

\section{Related work}\label{sec_survey}
There is a vast literature related to BPs, however, we simply discuss few relevant strands related to our work. 

Irreducible population-dependent BP with discrete and continuous-time framework are considered in \cite{klebaner1989geometric, jagers1997coupling} respectively; they do not consider total population dependent offspring; further, the population-dependent mean matrix converges to a constant mean matrix, but we support proportion-dependent mean matrix in the limit.  In \cite{athreya1968some}, authors consider continuous-time, but population-independent, irreducible BPs. 

In \cite{coffey1991galton}, the prey-predator BP is analyzed in discrete-time setting and co-survival conditions are identified, but the limit proportion is not derived; they also do not consider population-dependency.  In Section \ref{sec_BPA}, we consider a continuous-time population-dependent BP with double-sided attack and acquisition. One can also analyze the continuous time population-dependent variant of prey-predator BP using our results. 

In \cite{agarwal2021co}, we introduce BP with attack and provide limit proportion for the case with population-independent and symmetric offspring, i.e., with $m_{xx}(\om) = m_{yy}(\om) = m$, for all $\om$, for some $m > 1$. We significantly generalize by considering total population-dependency and symmetric/asymmetric offspring. We analyze a particular case of proportion dependent BP (offspring depend on the proportion of the populations) along with other co-authors in \cite{kapsikar2020controlling}. Our results cover the model in \cite{kapsikar2020controlling} and can also be used to generalize their result which will be a part of our future work.

\textbf{\polya urn models:} In \cite{athreya1968embedding}, it is shown that the \polya urn models can be embedded into a continuous-time population-independent BP. Thus, the asymptotic analysis of the continuous-time BPs can be derived using the corresponding analysis of the \polya urn models. However, our work differs from the \polya urn literature (\cite{higueras2006central, janson2004functional, arthur1987non}) for reasons mentioned in the introduction - they neither consider total population-dependency nor commonly deal with extinction (non-replacement) scenarios; recall, in BPs, extinction occurs with non-zero probability, even in the super-critical regime. In \cite{janson2004functional}, which is an exception, the possibility of extinction is considered, but they do not consider population-dependency.

In \cite{higueras2006central}, authors analyze the urn model with the removal of balls of other colours (not the chosen one) - same as a negative offspring in BP with attack of Section \ref{sec_BPA}. However, they assume a unique attractor for ODE and a constant number of additions (offspring) to the urn. We again have a significant generalization with a random number of offspring and where the random trajectory of the BP with attack can converge 
 to or hover around one of the attractors/saddle points of ODE (see Corollary \ref{corollary_BPA}).

\section{Branching Process with Attack} \label{sec_BPA}
Consider a BP with two population types, say $x$ and $y$. Each individual of any type lives for a random time, $\tau \sim exp(\lambda)$, where $\lambda \in (0, \infty)$. It produces a random number of offspring before dying. The BP also includes  attack and acquisition by rival types. 
    
To be precise, an individual of (say) $x$-type produces $\xi_{xx}(\Om(\tau^-))$ offspring of its type. Further, it attacks/removes $\xi_{xy}(C^y(\tau^-))$ individuals of $y$-type population; naturally, the attacked population can not exceed the population available to be attacked at $\tau^-$, hence $\xi_{xy}(C^y(\tau^-)) \leq C^y(\tau^-)$ a.s.; note that the number of attacks do not depend on the size of the attacking population. The attacked individuals are then deleted from the $y$-population, and acquired by (i.e., added to) the $x$-population.  Thus, for example, when a $x$-type individual dies, the current populations change as follows:
    
    \vspace{-0.5cm}
    {\small
    \begin{align*}
    \Cx(\tau^+) &= \Cx(\tau^-) + \xi_{xx}(\Om(\tau^-)) + \xi_{xy}(C^y(\tau^-)) - 1,  \mbox{ and }
    \Cy(\tau^+) = \Cy(\tau^-) - \xi_{xy}(C^y(\tau^-)).
    \end{align*}}The total and $y$-population also evolve similarly. 
    We call such a BP as \textit{Branching Process with Attack}.  The dynamics in  \eqref{evolve_x_up_time} capture this BP, when for each $i, j$:
    \begin{align}\label{eqn_dynamics_BPA}
    \offs_{ii}(\Om(\tau^-)) := \xi_{ii}(\Om(\tau^-)) + \xi_{ij}(C^j(\tau^-)), \mbox{ \normalsize and } \offs_{ij}(\Om(\tau^-)) := - \xi_{ij}(C^j(\tau^-)).
    \end{align}
Next, we assume:
\begin{enumerate}[label=\textbf{K.\arabic*},ref=\textbf{K.\arabic*}]
    \item For each $i \in \{x, y\}$, assume that there exist integrable random variables, $\overline{\xi}$, $\underline{\xi}$, such that $0 \leq \underline{\xi} \leq \xi_{ii}(\om) \leq \overline{\xi}$ a.s. for each $\om$ and $E[\overline{\xi}]^2 < \infty$, $E[\underline{\xi}] > 1$. Further, let the attack offspring $\xi_{ij}(\om)$ be integrable for each $\om$ and for each $i \neq j \in \{x, y\}$. 
    \label{k1}
\end{enumerate}
    The above assumption immediately implies \ref{a1}. Define the expectations conditioned on $\om$ as $\bpam_{ij}(\om) := E[\xi_{ij}(\om)]$ for  $i, j \in \{x, y\}$. We further assume (see \eqref{eqn_dynamics_BPA}):
\begin{enumerate}[label=\textbf{K.\arabic*},ref=\textbf{K.\arabic*}]
\setcounter{enumi}{1}
    \item For  $i, j \in \{x, y\}$, let $\bpam_{ij}^\infty \geq 0$  with $\bpam_{xy}^\infty > 0$. Assume $\minf_{ij}(\bc)$ satisfy the following: \label{k2}
    \begin{align*}
    \begin{aligned}
    \minf_{xy}(\bc) &=  -\bpam_{xy}^\infty 1_{\{\bc < 1\}}, \ \minf_{yx}(\bc) = -\bpam_{yx}^\infty  1_{\{\bc > 0\}},\\
    \minf_{xx}(\bc) &= \bpam_{xx}^\infty  + \bpam_{xy}^\infty  1_{\{\bc < 1\}} \mbox{ and }  \minf_{yy}(\bc) = \bpam_{yy}^\infty  + \bpam_{yx}^\infty  1_{\{\bc > 0\}}.
    \end{aligned}
    \end{align*}Further, let the conditions of \ref{a2} be satisfied with $\{(m_{ij}, m_{ij}^\infty)\}_{i, j}$ replaced by \\ $\{(\bpam_{ij}, \bpam_{ij}^\infty)\}_{i, j}$.
\end{enumerate}
We are interested in the BP where attack is prominent\footnote{If both $\bpam_{xy}^\infty, \bpam_{yx}^\infty = 0$, then it will lead to two independent (non-attacking) BPs at limit; if required, one can derive the analysis for this case, as done in Theorem \ref{thrmBPA}.} even at the limit, thus, $\bpam_{xy}^\infty > 0$ without loss of generality in \ref{k2}. If $\bpam_{yx}^\infty = 0$, then it leads to single-sided attack at limit, but recall anything is possible in transience. 
Observe the cross-mean function in \ref{k1} converge to (almost) constant limit, e.g., $\bpam_{xy}(\om) \stackrel{s^c \to \infty}{\to} \bpam_{xy}^\infty1_{\{\bc < 1\}}$. The reason behind the indicator is that there is no attack at limit when $\bc = 1$; this is because $\Cy_n \to 0$ when $\limsup_{n \to \infty} \bc(\Ups_n) = 1$ as proved at the end of Appendix \ref{appendix_prelim}.

For BP with attack, the ODE \eqref{eqn_ODE} has the following form: 

\vspace{-4mm}
{\small
\begin{align}\label{ODE_BPA}
    \begin{aligned}
        \dot{\ups} &= \mathbf{h}(\bc) 1_{\{\pc > 0\}} - \ups, \mbox{ where } \mathbf{h}(\bc) := (h_\psi^c, h_\theta^c, h_\psi^a, h_\theta^a), \mbox{ is such that}\\
        h_\psi^c &=  \bc \bpam_{xx}^\infty + \big(1 -\bc \big) \bpam_{yy}^\infty - 1, \
        h_\theta^c = \bc\left(\bpam_{xx}^\infty +\bpam_{xy}^\infty1_{\{\bc < 1\}} -1\right) - \left(1-\bc\right)\bpam_{yx}^\infty 1_{\{\bc > 0\}}, \\
        h_\psi^a &=  \bc \bpam_{xx}^\infty + \big(1 -\bc \big) \bpam_{yy}^\infty, \mbox{ and } h_\theta^a = \bc\left(\bpam_{xx}^\infty+\bpam_{xy}^\infty1_{\{\bc < 1\}}   \right) - \left(1-\bc\right)\bpam_{yx}^\infty 1_{\{\bc > 0\}}.
    \end{aligned}
\end{align}}

 We begin with the analysis of the above ODE towards providing ODE approximation result for BP with attack using Theorem \ref{thrm1}.
 
\subsection{Analysis of ODE for BP with attack} 
Define the parameter vector $\mathbf{e} := \{\bpam_{ij}^\infty: i, j \in \{x, y\}\}$, and consider the following class of limit mean functions (by \ref{k2}, the vector ${\mathbf e}$ defines $M^\infty$):
\begin{align}\label{eqn_setE}
    \cM &:= \{ \mathbf{e}: \bpam_{yx}^\infty  > 0 \} \cup \{ \mathbf{e}:  \bpam_{yx}^\infty = 0  \mbox{ and } \bpam_{xx}^\infty  + \bpam_{xy}^\infty  < \bpam_{yy}^\infty \}, \mbox{ which implies}\\
    \cM^\complement &= \{ \mathbf{e}: \bpam_{yx}^\infty  = 0 \} \cap \{ \mathbf{e}:  \bpam_{yx}^\infty > 0  \mbox{ or } \bpam_{xx}^\infty  + \bpam_{xy}^\infty  \geq \bpam_{yy}^\infty \} = \{ \mathbf{e}:  \bpam_{yx}^\infty = 0, \bpam_{xx}^\infty  + \bpam_{xy}^\infty  \geq \bpam_{yy}^\infty \}. \nonumber
\end{align}
Observe that the first and second sub-classes in $\cM$ consider double and single-sided attack, respectively (at the limit); both classes consider acquisition.
An important question for a BP with attack is regarding the survival of the individual types and co-survival. Corollary \ref{corollary_BPA} of Theorem \ref{thrm1} given later provides answers to such questions. Prior to that, the next theorem derives the asymptotic analysis of \eqref{eqn_beta_ode_simple} and also shows that this analysis is sufficient for analysis of \eqref{ODE_BPA} (see proof in Appendix \ref{proof_thrmBPA}).

\begin{theorem}\label{thrmBPA}
Assume \ref{k1} and \ref{k2}. Then, \ref{a3} holds for \eqref{ODE_BPA}. Further, we have:
    \begin{enumerate}[label=(\roman*)]
        \item For ODE \eqref{eqn_beta_ode_simple}, no interior $\bc \in (0, 1)$ is  an attractor, $\bstar = 1$ is always an attractor, but  $\bstar = 0$ is  an attractor only if $\mathbf{e} \in \cM$.
        
        Further, again for \eqref{eqn_beta_ode_simple} in $[0,1]$: if $\mathbf{e} \in \cM$, then, $\bstar_r$, the unique zero of $g_\beta$, is the only repeller; while if $\mathbf{e} \notin \cM$, then $0$ is the only repeller.
        \item The attractors and repellers of ODE \eqref{eqn_beta_ode_simple} determine the attractor ($\cA$) and saddle ($\cR$) sets of ODE \eqref{ODE_BPA} respectively:
            \[
            \cA = 
            \begin{cases}
            \{ \mathbf{h}(1), \mathbf{h}(0)\}, &\mbox{ if } \mathbf{e} \in \cM,\\
            \{\mathbf{h}(1)\}, &\mbox{ if } \mathbf{e} \notin \cM,
            \end{cases}
            \mbox{ and }
            \cR = 
            \begin{cases}
            \{ \mathbf{0}, \mathbf{h}(\bstar_r)\}, &\mbox{ if } \mathbf{e} \in \cM,\\
            \{ \mathbf{0}, \mathbf{h}(0)\}, &\mbox{ if } \mathbf{e} \notin \cM, \ \ \mbox{ where}
            \end{cases}
            \]
            for example, $\mathbf{h}(1) = (\bpam^\infty_{xx}-1, \bpam_{xx}^\infty-1, \bpam_{xx}^\infty, \bpam_{xx}^\infty)$ and $\mathbf{h}(0) = (\bpam^\infty_{yy}-1, 0, \bpam^\infty_{yy}, 0)$.
        \item The combined domain of attraction of $\cA\cup \cR$, i.e., $\cD = \cD_I$ defined in \eqref{invariant_set}. \eop
    \end{enumerate}
\end{theorem}


\subsection{Analysis of random trajectory of BP with attack}\label{subsec_BPA}
By Theorem \ref{thrm1}, the following holds (proof in Appendix \ref{proof_cor_BPA}):
\begin{corollary}\label{corollary_BPA}
Consider the BP as in \eqref{eqn_dynamics_BPA}, and assume \ref{k1}-\ref{k2}. Then, we have:
\begin{enumerate}[label=(\roman*)]
    \item The assumption \ref{a3} holds for ODE \eqref{ODE_BPA}, and hence Theorem \ref{thrm1}(i) is applicable.
    \item The following is true w.p. $1$ for BP with attack:
    
    $\noindent$ $\bullet$ if $\mathbf{e} \in \cM$, either $\Ups_n$ converges to $\{\mathbf{0}, \mathbf{h}(0), \mathbf{h}(\bstar_r), \mathbf{h}(1)\}$ or  hovers around $\{\mathbf{0}, \mathbf{h}(\bstar_r)\}$, where $\bstar_r$ is as in Theorem \ref{thrm_attractors_beta} and
    
    $\noindent$ $\bullet$ if $\mathbf{e} \notin \cM$, either $\Ups_n$ converges to $\{\mathbf{0}, \mathbf{h}(0), \mathbf{h}(1)\}$ or hovers around $\{\mathbf{0}, \mathbf{h}(0)\}$. \eop
\end{enumerate}
\end{corollary}



Recall from Theorem \ref{thrmBPA}, ODE  for \eqref{ODE_BPA} has three types of saddle points: $\mathbf{h}(0)$ when $\mathbf{e} \notin \cM$, $\mathbf{h}(\bstar_r)$ when $\mathbf{e} \in \cM$ and vector $\mathbf{0}$ for all cases. 
The sample paths in which BP hovers around $\mathbf{0}$ or 
  $\mathbf{h}(0)$ or converges to/hovers around $\mathbf{h}(\bstar_r)$ indicate co-survival. Both populations survive in insignificant numbers in the first case, 
 $x$-population is comparatively small in the second case and both populations survive in large numbers in the last case. Further, only $x$ or $y$-population survives when the process converges to $\mathbf{h}(1)$ or $\mathbf{h}(0)$ respectively, see the end of Appendix \ref{appendix_prelim}.

We re-iterate that our approach does not provide the probability with which BP converges or hovers around different limit points of the ODE \eqref{ODE_BPA}.

\begin{figure}[htbp]
    \centering
    \includegraphics[trim = {0cm 1.5cm 0cm 0.4cm}, clip, scale = 0.35]{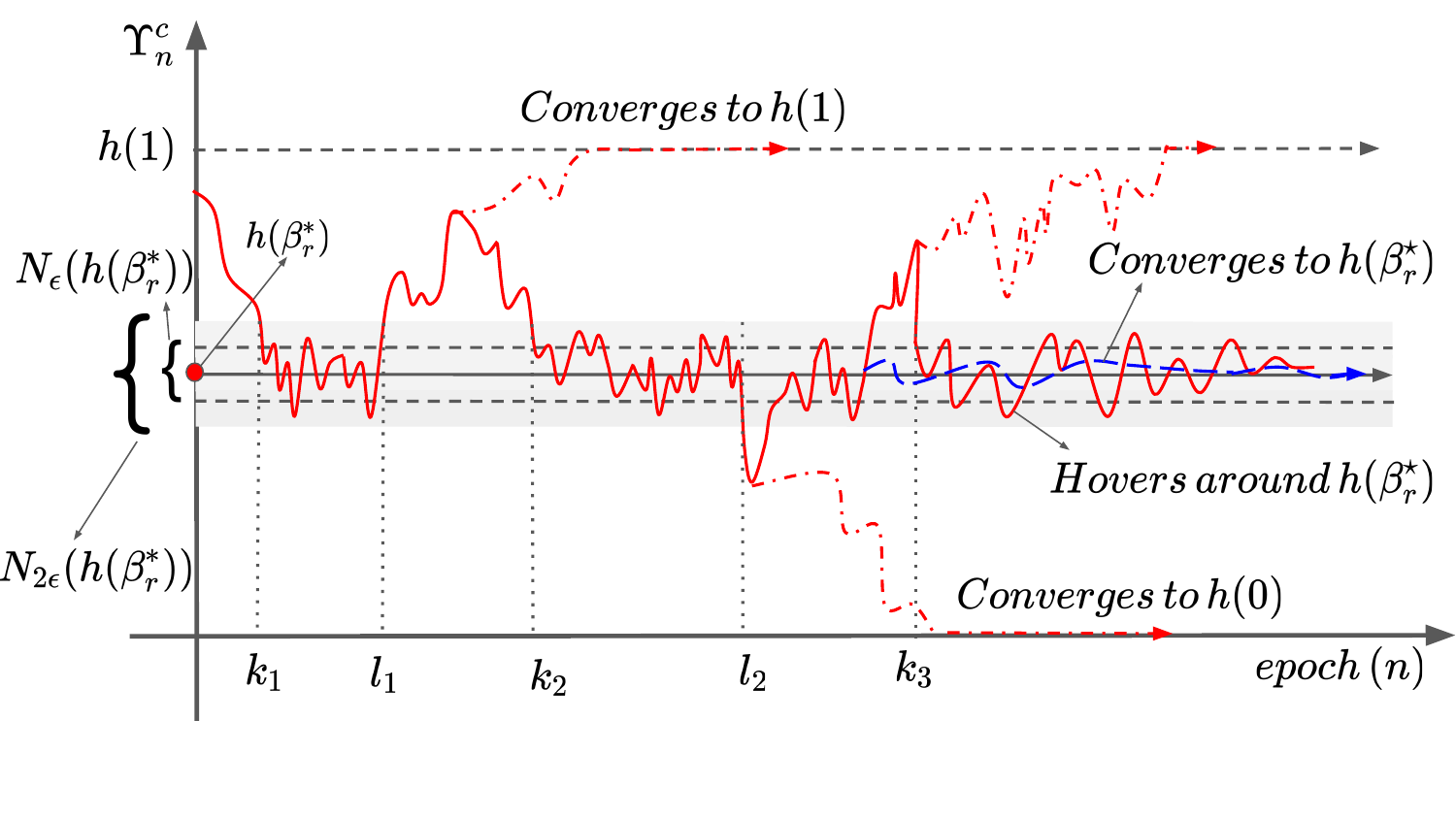}
    \caption{Behavior of BP with attack trajectory when $\mathbf{e} \in \cM$}
    \label{fig_repeller_BPA}
\end{figure}
Now, we would like to explain the behaviour of the BP with a pictorial representation in Figure \ref{fig_repeller_BPA}. Consider $\mathbf{e} \in \cM$ and survival paths. Say, the process enters $\epsilon$-neighbourhood of $\mathbf{h}(\bstar_r)$ at epochs say $k_1, k_2, \dots$ (for some $\epsilon>0$), remains in its $2\epsilon$-neighbourhood for some epochs and then exits at epochs $l_1, l_2, \dots$ At every exit, it can either get attracted to $\mathbf{h}(0)$ or $\mathbf{h}(1)$ or it can re-enter the neighbourhood. The solid red line in the figure represents the sample path when the trajectory enters and exits the $\epsilon$-neighbourhood i.o., i.e., hovers around $\mathbf{h}(\bstar_r)$ with $\delta_1 = 2\epsilon$. Some sample paths can converge to $\mathbf{h}(\bstar_r)$ - see blue dashed line. Similar behaviour is exhibited when $\mathbf{e} \notin \cM$.

\subsection{Application - Viral competing markets} \label{subsec_viral_competing}
In online social networks, content providers (CPs) share a variety of content, which is shared (again) by the recipients and thus may get viral (i.e., the number of copies of the post grows significantly with time). After reading the post, the user most likely loses interest in it forever. Thus, reading the post is analogous to death, while the number of new shares by a user is analogous to offspring. Further, unread and total (read $+$ unread) copies are analogous to the current and total population, respectively.
\begin{figure}[http]
    \centering
    \includegraphics[scale = 0.4]{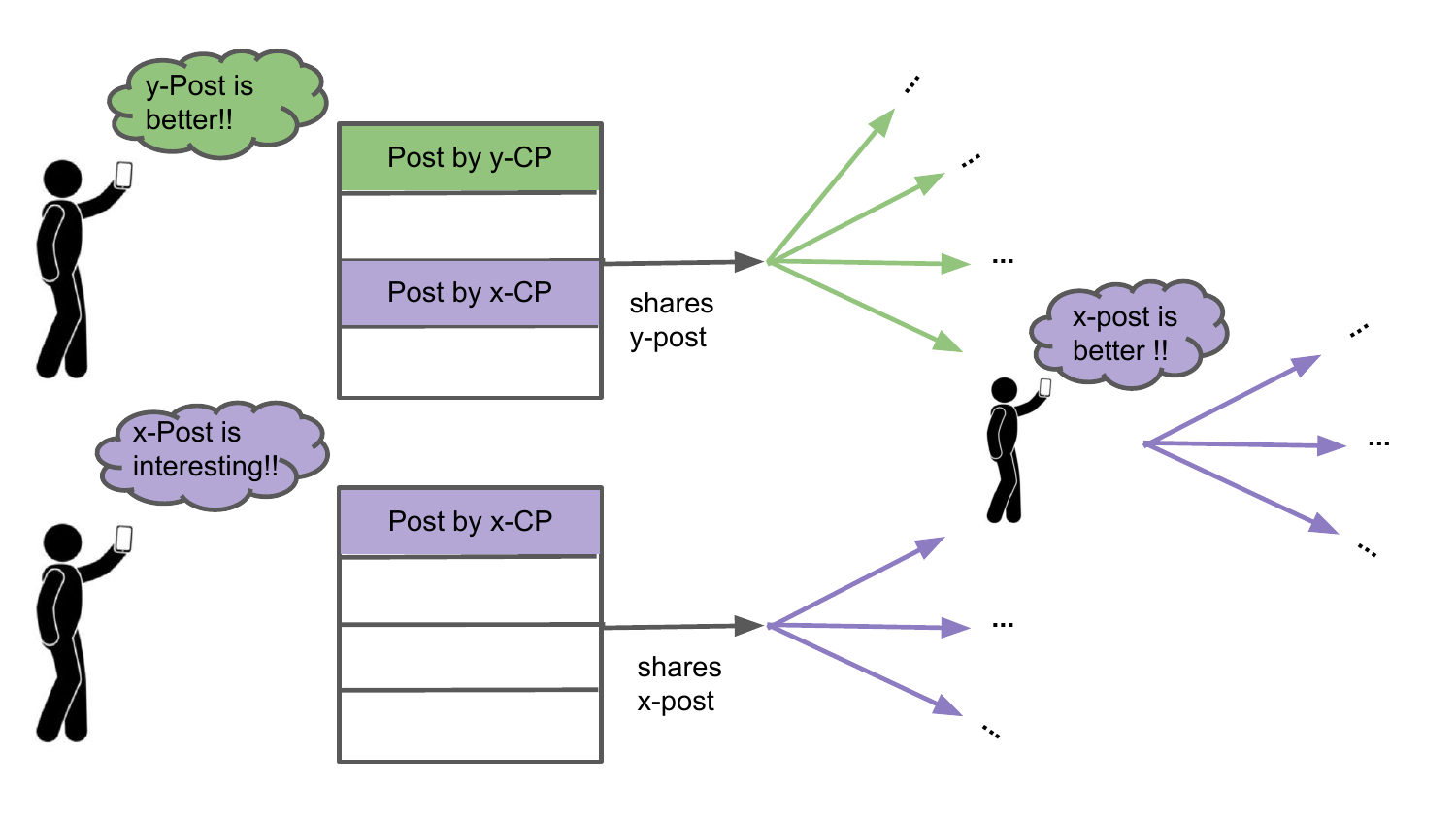}
    \caption{Viral competing markets}
    \label{fig:viral_competing}
\end{figure}

On such networks, contents often compete with each other (e.g., advertisements of similar products); when a new competing post (say $y$-type) is shared on the user's screen, the user might find $y$-post more attractive than an older $x$-post (see Figure \ref{fig:viral_competing}). This aspect leads to viral competing markets, where we say $y$-post has attacked and acquired the opportunities of $x$-post. Such attacks are dependent on the current copies. Further, the network is closed, and some users may share with previous recipients who would not be interested in the post again. Thus, the effective shares depend on the total copies. BP with attack precisely captures such dynamics (see \cite{agarwal2021co} for modelling details).

In \cite{agarwal2021co}, we analyzed such markets in a restricted setting, while Corollary \ref{corollary_BPA} can handle the generality mentioned here. Both the posts are prominent when the process converges to or hovers around $\mathbf{h}(\bstar_r)$. While, the convergence to $\mathbf{h}(0)$ or $\mathbf{h}(1)$ represents the dominance of one of the posts.

From Corollary \ref{corollary_BPA}, one can get more interesting insights. For instance, let $y$-CP be more influential, and thus $y$-post is shared more on average in the limit, so $\bpam_{xx}^\infty < \bpam_{yy}^\infty$. If the competition is ignored, the analysis is provided using independent BPs. Such analysis indicates the possibility of co-virality (both posts get viral simultaneously). However, when a typical user receives both posts, it may find $x$-post more appealing, leading to $\bpam_{xx}^\infty + \bpam_{xy}^\infty > \bpam_{yy}^\infty$ with $\bpam_{yx}^\infty = 0$. Therefore, $\mathbf{e} \notin \cM$, thus $\mathbf{h}(1)$ is a limit, which implies that $x$-post can dominate the post of more influential $y$-CP. Further, none of the limits indicate co-virality. 

On the other hand, when some users prefer the $y$-post ($\bpam_{yx}^\infty > 0$), while others prefer the $x$-post, then, co-virality is possible due to interior saddle point $\mathbf{h}(\bstar_r)$.

\section{Finite horizon approximation}\label{sec_6}
In Theorem \ref{thrm1}(i), we proved the finite time approximation of $\Ups_n$ using the autonomous ODE \eqref{eqn_ODE}; such an ODE is obtained using the limit proportion-dependent mean functions ($\minf_{ij}(\bc)$). However, directly using the population-dependent mean functions $m_{ij}(\om)$, one may anticipate better approximation in transience. 

We claim that ODE, $\dot{\ups} = \gna(\ups, t)$, constructed using the actual conditional expectation, $E[\mathbf{L}_n | \mathcal{F}_n] = \gna(\ups, t)$ given in \eqref{eqn_nonauto_g} better approximates the BP; recall, the difference term ${\cal E}_1^{n}(\cdot)$ 
 of \eqref{eqn_diff_term1} converges to $0$ as shown in the proof of Theorem \ref{thrm1}. The approximation should further improve when the new ODE is initialised with $\Ups_{n_m}$, and not with $\lim_{n_m \to \infty} \Ups_{n_m}$ as in Theorem \ref{thrm1}. From \eqref{eqn_nonauto_g}, the new ODE is non-autonomous and discontinuous. Also by \ref{a2},  the right hand side  $\gna(\Ups, t)$, converges to that of ODE \eqref{eqn_ODE}, $\ga(\Ups)$, as $t \to \infty$. Approximation by such non-autonomous ODE is proved for super-to-sub critical total population-dependent BP in \cite{agarwal2022saturated}.  


We support our claim using three numerical examples of different types. 
\begin{example}
Consider a population-dependent BP with only one (say $x$-type) population, and let $\Cx(0) = 2$. Assume that initially, the population-dependent mean offsprings reduce linearly with an increase in total population size ($\ax$), and then gets fixed to $1.2$  as below:
\[
    m_{xx}(\om) =
    \begin{cases}
    3 - 0.002 \ax, \mbox{ if } \ax \leq 400,\\
    1.2, \mbox{ if } \ax > 400,
    \end{cases}
    \mbox{ for any } \om = (\cx, \ax).
\]
Clearly, the limit mean function is $\minf_{xx} = 1.2$, when $\cx \to \infty$. From FIGURE \ref{figure1}, one can see that the the curves ($\pc_n = \cx_n/n, \pa_n = \ax_n/n$ versus $n$, for all $n \geq n_m = 5$) for random trajectory (black curve) and non-autonomous ODE trajectory (red) are close by. However, the curve for autonomous ODE trajectory (blue) matches with the other curves only as $n$ grows large. It can also be seen from the plots that the random BP trajectory converges to the attractor of the autonomous ODE, as $n$ increases.
\begin{figure}[htbp]
\centering
\begin{minipage}{.5\textwidth}
  \centering
  \includegraphics[trim = {0cm 8cm 0cm 8cm}, clip, scale = 0.3]{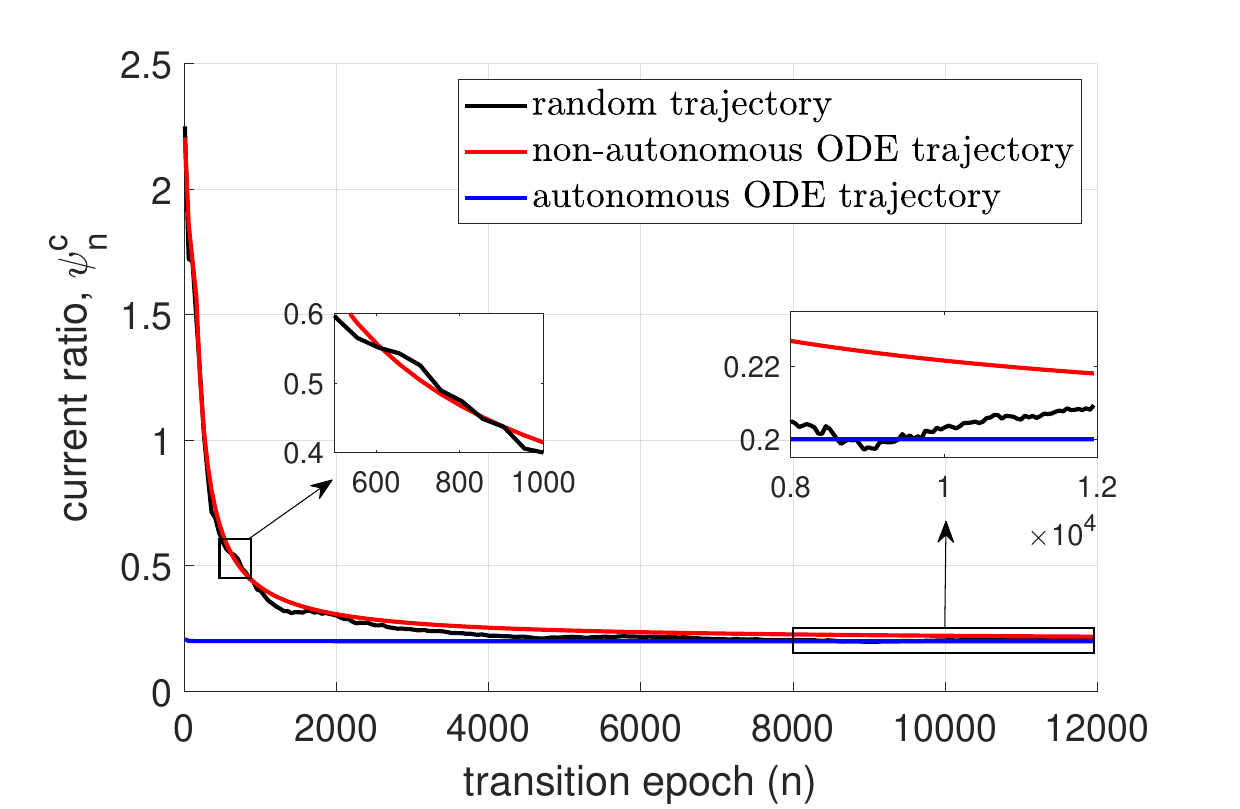} 
\end{minipage}%
\begin{minipage}{.5\textwidth}
  \centering
  \includegraphics[trim = {0cm 8cm 0cm 8cm}, clip, scale = 0.3]{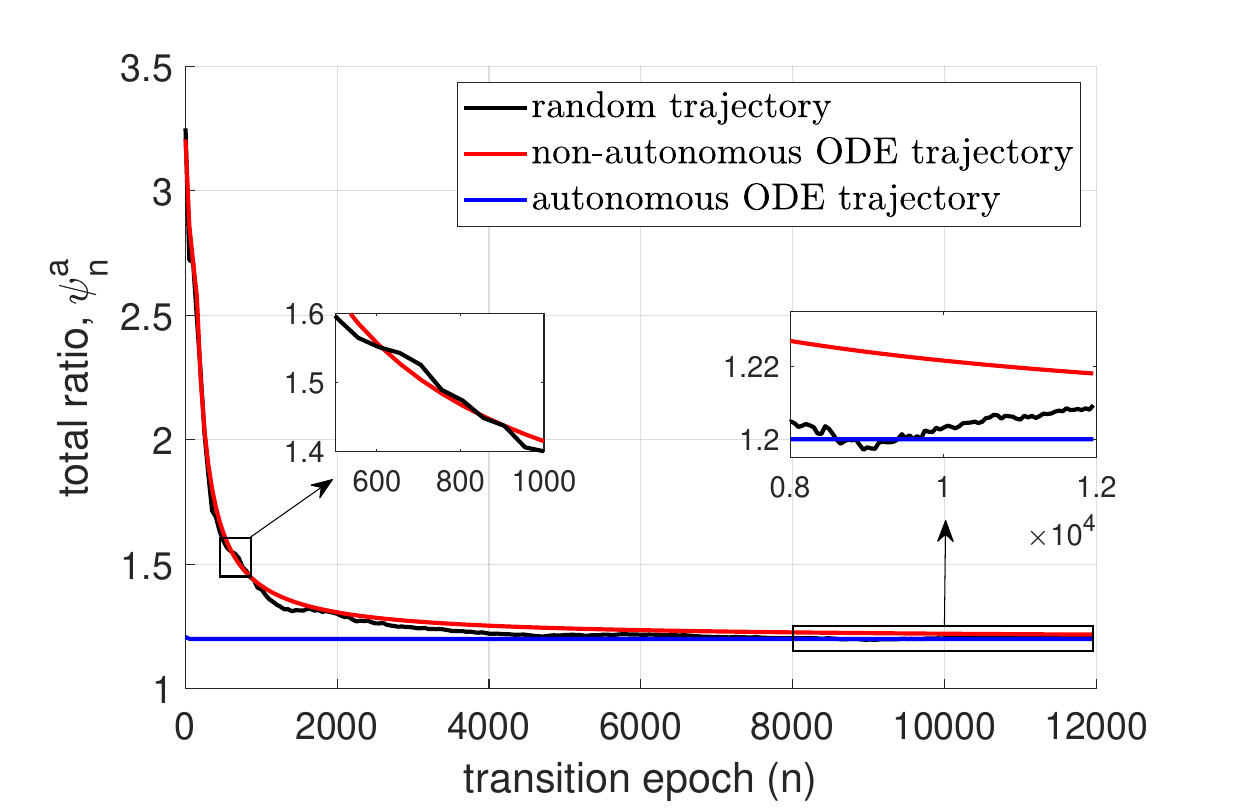}
\end{minipage}
\caption{Finite horizon approximation, single type PD-BP - one sample path}
\label{figure1}
\end{figure}
\end{example} 

\begin{example} Consider a PrD-BP with two population types ($x$ and $y$-type), $\Cx(0) = \Cy(0) = 100$ and the mean matrix:
\begin{align*}
M(\om) = M(\bc) = M^\infty(\bc) =
 \begin{bmatrix}
 6\bc & 2\bc\\
 4\bc & 5.6 \bc
 \end{bmatrix}.
 \end{align*}
\begin{figure}[htbp]
    \centering
    \includegraphics[trim = {0cm 6cm 0cm 6cm}, clip, scale = 0.4]{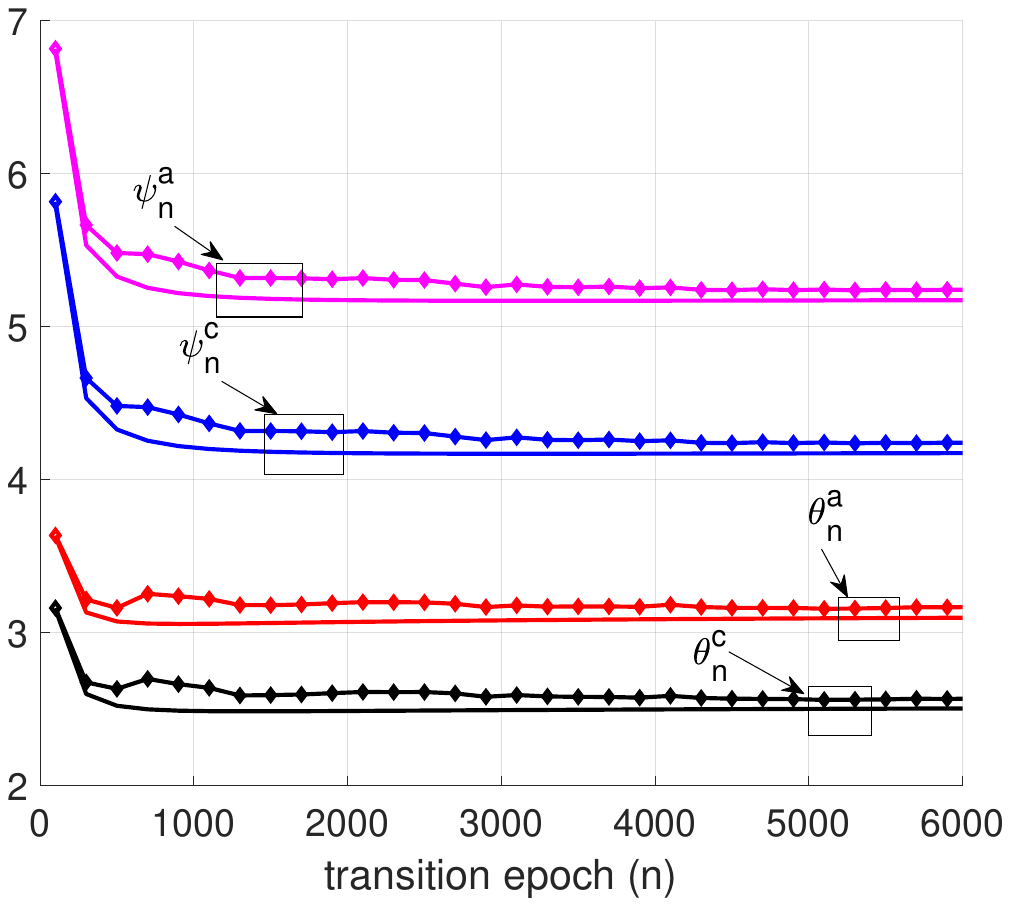}
    \caption{PrD-BP:  marked line- random trajectory (one sample path), and solid line- ODE trajectory}
\label{figure2}
\end{figure}
In the above, the mean matrix is always proportion-dependent. Further, the process is not in throughout super-critical regime, however, in a neighbourhood of the attractor of the corresponding ODE and stochastic system (in survival paths), the process is in super-critical regime. This chapter does not cover the theoretical analysis of such processes, nonetheless, we numerically illustrate in FIGURE \ref{figure2} that the curves for random trajectory and ODE trajectory (for all $n \geq n_m = 100$) match well with each other; observe that the mean matrix has the same structure from the start and hence the autonomous and non-autonomous ODE solutions are the same, except for the initial values. We leave the analysis of such processes as a part of future work.
\end{example}
\begin{example}
Let $\Cx(0) = \Cy(0) = 1200$ and let the dynamics be as in BP with attack till $S^a$ is below a certain threshold, and then let the population progress with proportion-dependent mean offspring. Specifically, $M(\om) = M^t(\om)1_{\{s^a \leq 10^4\}} +  M^\infty(\bc)1_{\{s^a > 10^4\}}$, where
  \begin{align*}
M^t(\om) = 
\begin{bmatrix}
4 & -\min(2, \cy)\\
-\min(1, \cx) & 2.2
\end{bmatrix} \mbox{ and }  M^\infty(\bc) = \begin{bmatrix}
4 \bc + 1 & 9\bc+1\\
8\bc+1 & 2.2\bc+1
\end{bmatrix}.
\end{align*}
\begin{figure}[htbp]
\centering
\vspace{-4mm}
\begin{minipage}{.5\textwidth}
  \centering
  \includegraphics[trim = {0cm 8cm 0cm 8cm}, clip, scale = 0.36]{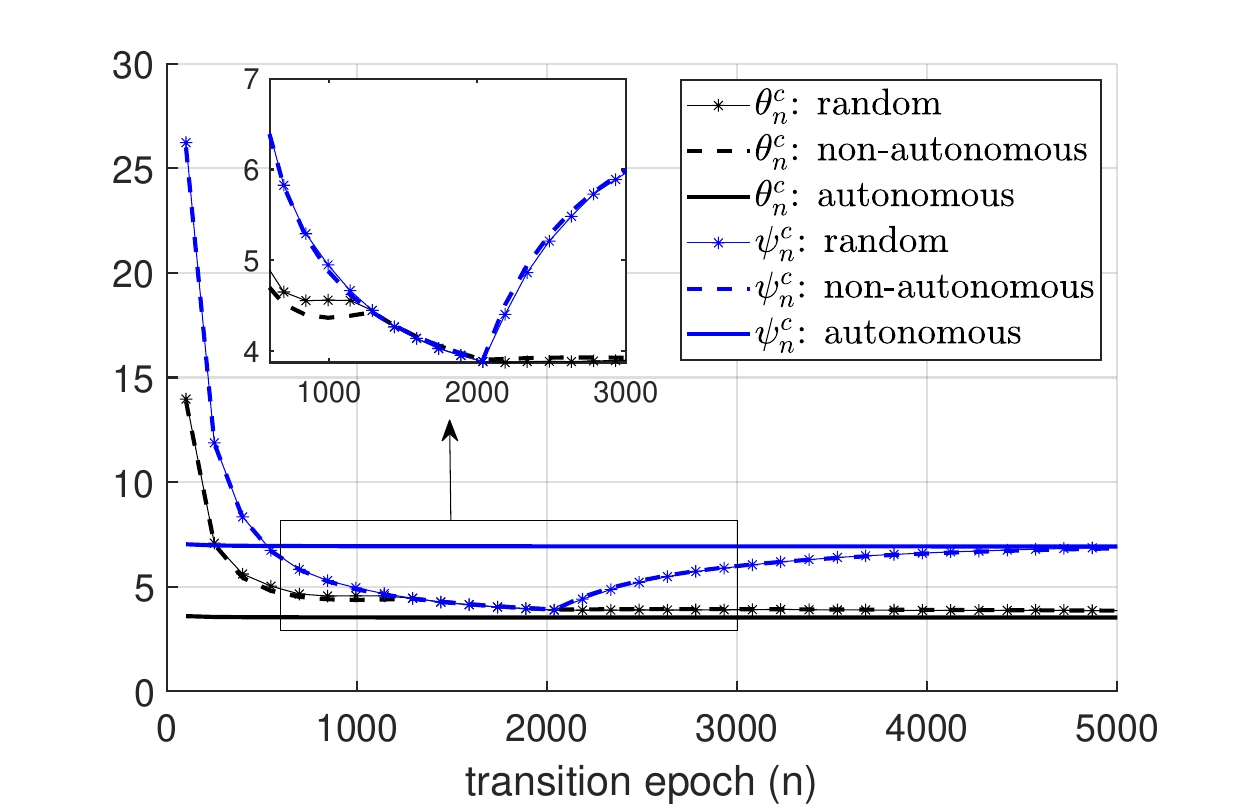} 
\end{minipage}%
\begin{minipage}{.5\textwidth}
  \centering
  \includegraphics[trim = {0cm 8cm 0cm 8cm}, clip, scale = 0.36]{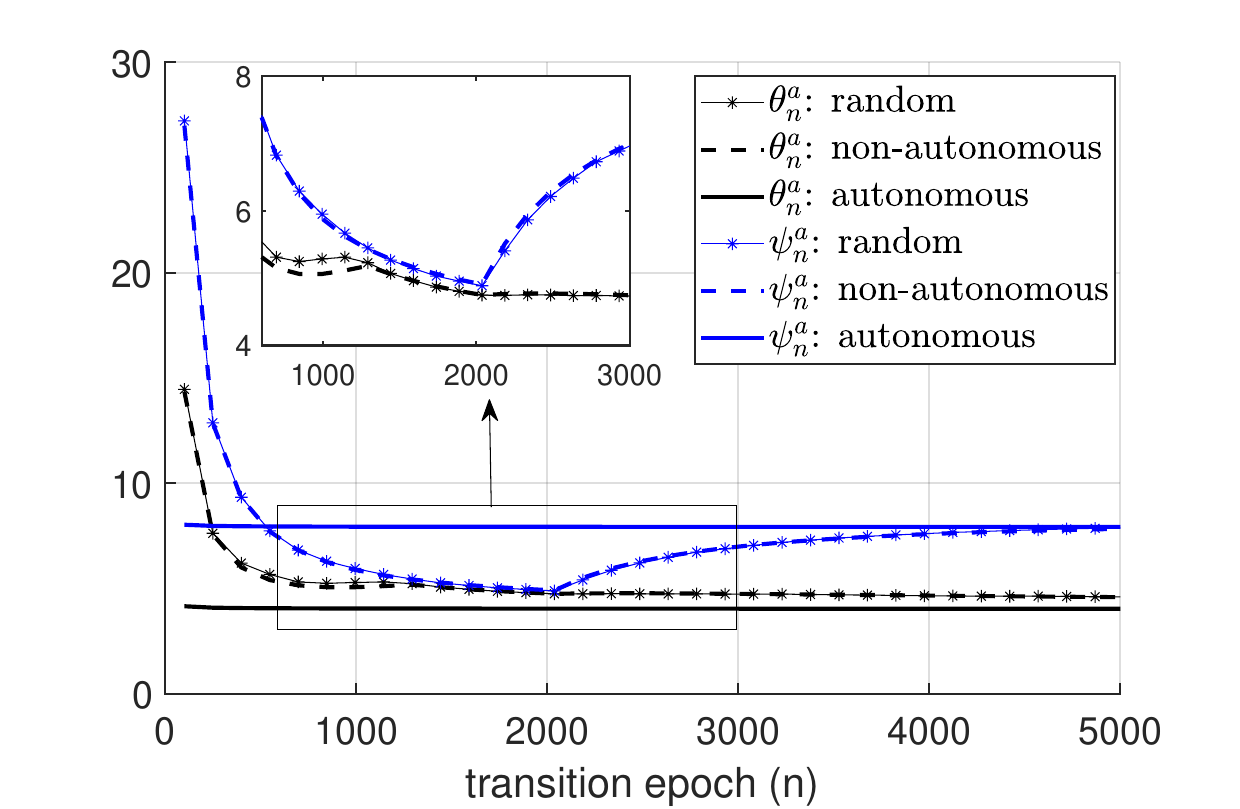}
\end{minipage}
\caption{Finite horizon approximation (current on left, and total on right side)}
\label{fig_eg3}
\end{figure}

The process is in throughout super-critical regime. We plot one sample path of BP and corresponding solutions of autonomous and non-autonomous ODEs\footnote{The ODE trajectories are estimated using the well known Piccard's iterative method (e.g., \cite{piccinini2012ordinary}).} (for all $n \geq n_m = 100$ and $T=12$). The current and total populations are in Figure \ref{fig_eg3}, while the proportion $\bc(\ups_n)$ is in Figure \ref{fig_eg3_beta}. From the plots, one can see that the non-autonomous ODE solution (dashed lines) better approximates the random BP trajectory (dotted lines), than the autonomous ODE (solid lines). As seen from the sub-figures, the non-autonomous ODE well captures the transition, unlike ODE \eqref{eqn_ODE}.
\begin{figure}[http]
\centering
  \includegraphics[trim = {0cm 8.2cm 0cm 8.8cm}, clip, scale = 0.4]{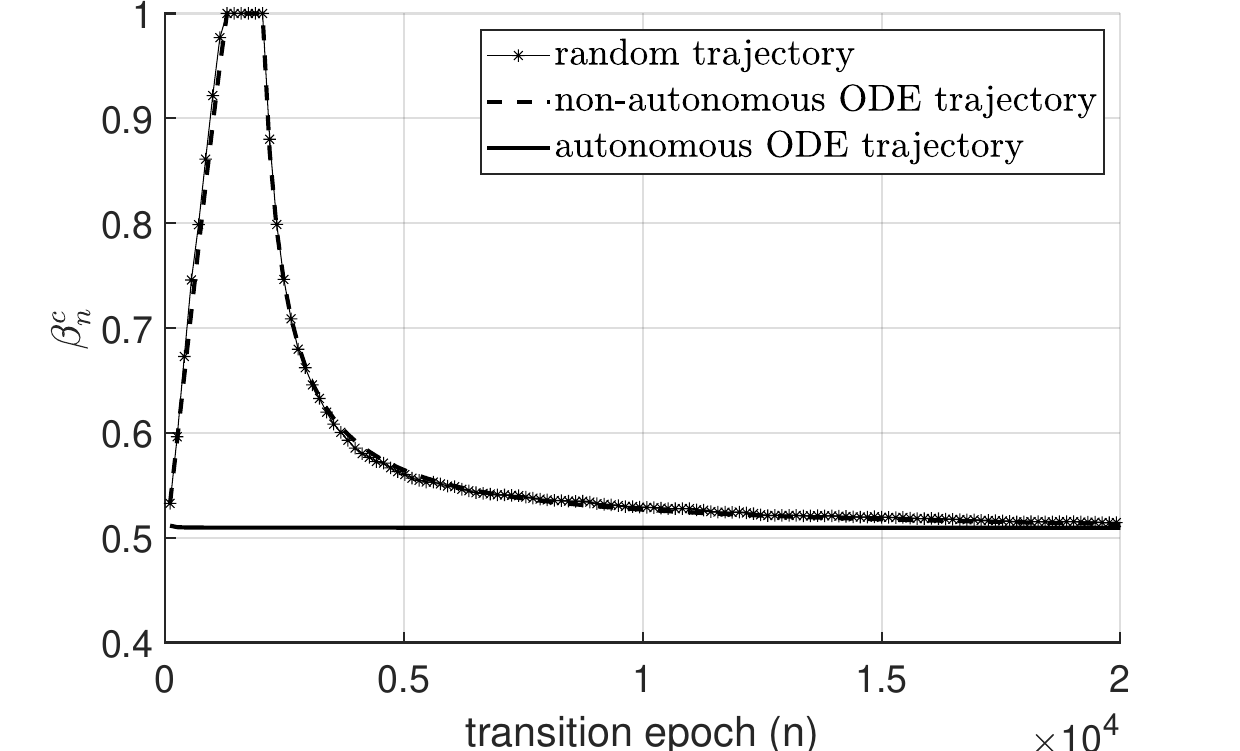}
  \caption{Proportion trajectory, $\bc_n$}
  \label{fig_eg3_beta}
\end{figure}

Initially, $x$-type individuals attack more aggressively than $y$-type and thus, the $y$-population depletes faster. In fact, by transition epoch ($1300$) proportion $\bc_n = 1$. Later, $M(\om) = M^\infty(\bc)$ does not have attack component, the $y$-population is regenerated and $\bc_n$ declines to $\approx 0.51$ indicating co-survival. This example also illustrates that the dynamics in transience (here, BP with attack) does not influence the limiting behaviour.
\end{example}

\section{Summary and conclusion}
We studied time-asymptotic proportion for a class of two-type continuous-time total-current population-dependent Markov BPs. We extended the stochastic approximation result to include the possibility of hovering around the saddle points of an appropriate ODE and to analyze BPs. The summary to derive the limiting behaviour is:

(i) if the BP satisfies the assumption \ref{a1}, then the sum current population exhibits dichotomy with probability $1$ (see Lemma \ref{lemma_sum_pop});

(ii) identify the limit mean functions $\minf_{ij}(\bc)$ satisfying \ref{a2}, if required using the discussion in Appendix \ref{appendix_prelim} for BPs with negative  offspring or attack;

(iii) identify the attractors and repellers of one-dimensional ODE \eqref{eqn_beta_ode_simple};

(iv) identify the attractor and saddle sets of ODE \eqref{eqn_ODE} using (iii) and Theorem \ref{thrm_attractors_beta}; these provide the limit proportion; 

(v) Theorem \ref{thrm_attractors_beta} also facilitates the proof of \ref{a4} to conclude about limiting behaviour of BP via Theorem \ref{thrm1}.

Interestingly, the limit proportion of any BP depends only on the limit mean matrix, irrespective of the dynamics in transience. A finite-time approximation result is also provided. We analyzed a recently introduced variant of BP with attack and acquisition under significantly more general conditions; such BP captures essential aspects of competing content propagation over online social networks.

\chapter{Robust fake post detection: BPs with unnatural deaths}\label{ch:journal2}

In this chapter, we design robust mechanisms\footnote{The work in this chapter is published as ``Agarwal, Khushboo, and Veeraruna Kavitha. ``Robust fake-post detection against real-coloring adversaries." Performance Evaluation 162 (2023): 102372". Further, an initial study of this work is published as a letter, see ``Suyog Kapsikar, Indrajit Saha, Khushboo Agarwal, Veeraruna Kavitha, and Quanyan Zhu. ``Controlling fake news by collective tagging: A branching process analysis.'' IEEE Control Systems Letters 5, no. 6 (2020): 2108-2113.''} for maximizing fake post detection on OSNs, while minimally affecting the incorrect identification of real/authentic posts. The main idea is to leverage users' responses to warn new recipients of the post about the actuality of the posts. Users may not always respond actively and even behave adversarily. The analysis is provided by a new variant of BP where unnatural deaths can occur, which we also analyze here. Additionally, this chapter provides the proofs of the remaining results of Chapter \ref{ch:basics} (the ones which are not covered by Chapter \ref{ch:journal1}).

\section{Introduction}
\label{sec_intro}
The prevalence of online social networks (OSNs), like Facebook or Twitter, is unprecedented today. A variety of content is available on the OSNs for users to consume, which can be for education, entertainment, advertisement or awareness purposes, among many more. 
Users also read news on such platforms instead of using classical mediums like newspapers.

One of the reasons for such high usage of OSNs is the ease with which users can access or share information. Further, no instant check ensures the shared post is authentic. On the one hand, this freedom allows users to express their views freely. However, at the same time, it provides users with the flexibility to post fake content, i.e., the posts that contain fabricated (mis)information that propagates through OSNs like authentic posts (see \cite{lazer2018science}). Once a post is shared on the OSN with an initial set of users, called seed users, the post can be further shared repeatedly by the recipients of the post to the extent of getting viral (the copies of the post grow significantly with time), or the post can get extinct in the initial phase (\cite{ranbir2019decomposable, agarwal2021co, agarwal2022saturated, van2010viral}). 

There are several reasons for a fake post to go viral. Authors in \cite{talwar2019people} theorize that users may share any information obtained from their reliable
source, or they can share any exciting post to seek their peers’
attention and have a sense of belonging. Also, users share posts that match their beliefs to continue using social media (due to its perceived usefulness). 
There have been many instances in the past where fake posts have proven to be fatal, and the most controversial of all is the 2016 US Presidential elections (\cite{allcott2017social}). Thus, studies on the generation, propagation, detection, and control of fake posts are the need of the hour. In this chapter, we focus on the detection aspect of fake posts.

Machine learning or deep learning is one of the commonly used approaches for fake post detection (see \cite{feng2022misreporting, sharma2019combating, ruchansky2017csi, ahmed2021detecting}). However, as argued in \cite{ahmed2021detecting}, such algorithms often face difficulty in obtaining training datasets in certain languages, and it gets difficult to determine the actuality using only the content (\cite{sharma2019combating}). 
Another approach used for fake post identification is using crowd-signals. The basic idea is to allow users to declare any post as real or fake, and then leverage user responses to identify the actuality of the post. Such an approach is being used by Facebook\footnote{\url{https://www.facebook.com/help/1753719584844061}}, where any user can report any post on the OSN. They can also provide specific reasons for reporting the post. When a post is reported, it is reviewed by third-party fact-checking organizations and is removed if it is against their policies. 
However, until the post is reviewed, the users on the OSN can view it without any warning. 

In \cite{kapsikar2020controlling}, the authors design a warning-based mechanism to control fake posts using crowd-signals. The idea is to leverage users' fake/real responses (tags) to the post and generate a warning signal for future recipients. Since the real-time warning signal/status of the post is continuously displayed to the users, this approach of using crowd-signals is different and should be more effective than that of Facebook. The objective is to ensure the maximal correct identification of the fake post while maintaining the proportion of fake tags for the real post within a given threshold. The paper assumes that each user participates in the tagging process.

In this chapter, we consider a more realistic framework. Firstly, we assume that not all users would be willing to tag. Secondly, if a user tags, it can consider the warning signal provided by the OSN; or it can tag without viewing the warning. And lastly, the users can be adversarial who always assign the real tag to any post.

We compare and show that the warning mechanism in \cite{kapsikar2020controlling} is insufficient for such a system. With just $1\%$ (with $2\%$) adversaries in the system and  everyone else tagging exactly as in \cite{kapsikar2020controlling}, we observed that the performance decreases approximately by $10\%$ (nearly $18.2\%$). This observation highlights the need for mechanisms that are robust against adversaries. We design such mechanisms in this chapter.

The new warning mechanisms are designed by cleverly eliminating the effect of adversarial users. We derive a one-dimensional ordinary differential equation (ODE) that captures the performance of any such general warning mechanism, and utilizing that ODE, we design the new warning mechanisms as well as illustrate the improved performance guarantees theoretically. 

We also present a Monte-Carlo simulation-based exhaustive numerical study to confirm our theoretical findings. The performance is expressed in two ways: (i) quality of service (QoS), which measures the proportion of fake tags for the fake post, and (ii) improved QoS (i-QoS), which represents the proportion only from non-adversarial users. The second metric, i-QoS, provides a better interpretation of the performance of warning mechanisms, as actions of adversarial users can not be controlled. Accordingly, the threshold with respect to the real post also changes to consider the responses only from non-adversarial users. 

According to the parameters in \cite{kapsikar2020controlling}, the non-adversarial users are assumed to be smart (i.e., have high intrinsic ability to identify the actuality of the posts). Thus, no warning mechanism can accentuate their ability beyond a limit -- we observe minor improvements in QoS of $2.66\%$ and $5.34\%$ with $1\%$ and $2\%$ of adversary respectively; these numbers translate to $98.64\%$ and $98.63\%$ of i-QoS under new mechanisms as compared to $95.8\%$ and $92.53\%$ with the mechanism as in \cite{kapsikar2020controlling}.

In another instance, where users are less informed and more likely to wrongly recognize the posts (as in reality), significant improvements are noticed even for a larger fraction of adversaries. Under the newly proposed mechanism, the QoS is $52.89\%$ (i-QoS is $80.86\%$) improving from only $45.31\%$ (i-QoS is only $45.31\%$) under the old mechanism, when an exorbitantly high fraction of adversarial users ($32.5\%$) are involved. This performance is achieved with minimal knowledge about users' sensitivity to the warning and their behavioural type. 
 
The warning dynamics are modelled using a new variant of branching processes (BPs).
This chapter also contributes towards total-current population-dependent two-type branching processes with population-dependent death rates and also considers a variety of unnatural deaths. In particular, we derive all possible limits and limiting behaviours of the population sizes as time progresses.

\textbf{Related Literature for Branching processes with unnatural deaths:}  
The literature on BPs has previously investigated unnatural deaths in a restricted setting. The BP analyzed in \cite{BPwithinteraction} is population-independent, while the authors in \cite{etheridge2013conditioning} consider unnatural deaths due to competition, modelled using a quadratic function of population size. The BP with pairwise interaction in \cite{ojeda2020branching} models natural births and deaths, along with additional births and deaths occurring due to cooperation and competition. Further, the birth and death rates in \cite{ojeda2020branching} are proportional to current population sizes. Our work provides a more generalized framework where the interactions are not limited to cooperation or competition. Further, the birth and death rate functions can additionally depend on the total and current population-sizes. 

\section{Problem description}\label{sec_prob_desc}
Consider an OSN with a large user base like Facebook or Twitter. Any post, $u$ on the OSN can be either fake ($u = F$) or real ($u = R$). The OSN aims to identify the actuality of the post. In \cite{kapsikar2020controlling}, the authors have proposed a warning mechanism where the recipients of the post themselves are guided in such a way that it leads to correct identification. We first study its robustness against adversarial users and then propose improved mechanisms.

We begin by describing the system and the warning mechanism of \cite{kapsikar2020controlling}. The posts are stored in a last-in-show-at-top structure named timeline for each user. The users are given a warning for each post and asked to assign a tag (fake or real) to it. Whenever a user views the post on its timeline, it guesses the actuality of the post, assigns the tag as real or fake accordingly and then forwards the same to its friends. This results in more unread copies of the post tagged as fake or real. 
The process continues when another user with the post on its timeline visits the OSN. The warning mechanism relies on the tags provided by the users and is updated with each new tag. 

We will now introduce a few notations and then discuss the propagation and tagging dynamics of the post. Let the fake and real tagged copies of the $u$-post be denoted as $x$-type and $y$-type, respectively. Further, let $\Cx(t)$ and $\Cy(t)$ be the number of users who have received the post tagged as fake and real, respectively, but have not yet read/shared it; thus, these are the number of unread copies of the post with fake or real tag. The total number of users who have received the post tagged as fake or real are represented by $\Ax(t)$ and $\Ay(t)$ respectively; these are read plus unread copies of the post. Define $\Om(t) := (\Cx(t), \Cy(t), \Ax(t), \Ay(t))$ be the tuple of number of copies at time $t$.

Each post contains two pieces of information: first, the sender's tag and second, the warning by the OSN, which is available at the click of a button (see Figure \ref{fig_post_design}). Users can exhibit different behaviours about utilizing the provided information. For example, some users may prefer to read the warning before tagging, while others may not. Therefore, motivated by \cite{agarwal2023single}, we broadly divide user behaviour into four categories. 

\begin{figure}[http]
    \centering
    \includegraphics[trim = {1.2cm 0cm 7.2cm 2cm}, clip, scale = 0.35]{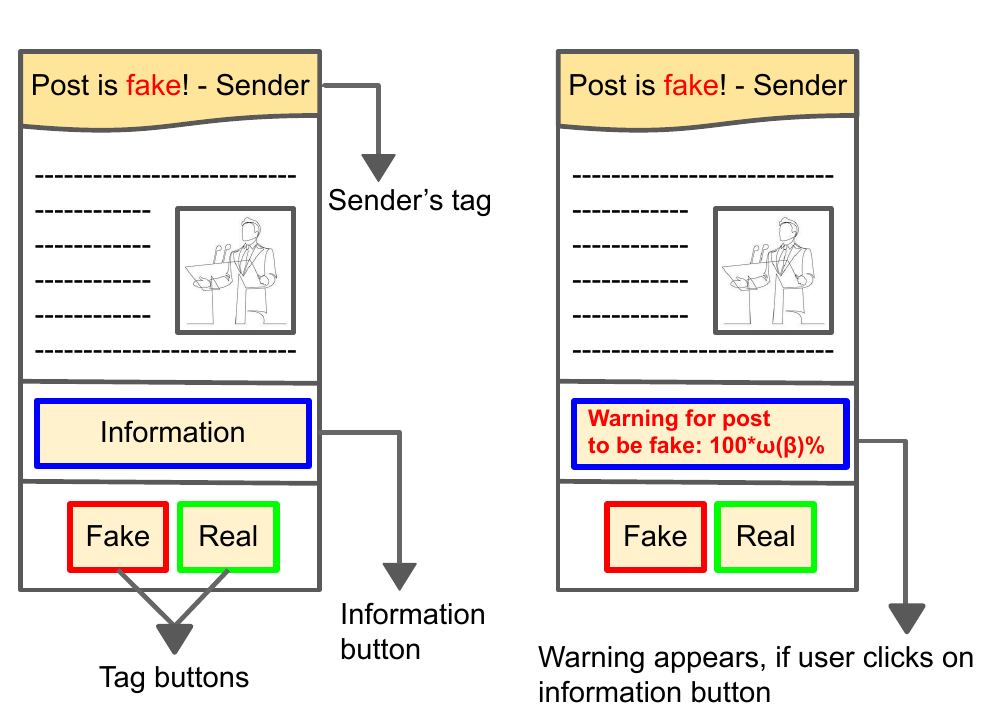}
    \caption{Design of the post}
    \label{fig_post_design}
\end{figure}
\subsection{Warning-ignoring (wi) users} These users tag the post only based on the sender's tag and their intrinsic ability to judge the post's actuality, not the warning. They prefer to invest less time in the system. Let $\tau$ be the time when a wi-user (with an unread copy of the post) reads it. At this time, the user will tag and then share the post with its friends. Let $I_{x, wi}(\Om(\tau^-))$ and $I_{y,wi}(\Om(\tau^-))$ be the indicator that the $wi$-user with fake or real tagged copy of the post tags it as fake.

If the sender has tagged the post as fake, then the recipient tags the post as fake or real with probability (w.p.) $p_x^u \in (0,1)$ and $1-p_x^u$ respectively. Similarly, let $p_y^u \in (0, 1)$ be the probability of fake tagging the post received with a real tag. Therefore:
\begin{align}\label{eqn_tag_wi}
    P(I_{x,wi}(\Om(\tau^-)) = 1|{\cal G}_\tau) &= p_x^u \mbox{ and } P(I_{y, wi}(\Om(\tau^-)) = 1|{\cal G}_\tau) = p_y^u, 
\end{align}where $\mathcal{G}_t$ is the sigma-algebra generated by $\{\Om(t'); t'\le t\}$. Naturally, the users get more suspicious about the post when received with fake tag. Thus, we assume $p_x^u > p_y^u$ for any $u \in \{R, F\}$.

As mentioned, after tagging, the user forwards the post to some/all of its friends. The number of shares depends on how attractive the post is, which we measure by $\eta^u \in (0,1)$. As argued in \cite{sui2023falsehood}, the design of fake posts is deceptive and more appealing; therefore, we assume $\eta^F > \eta^R$. 

Let ${\mathcal F}$ be the number of friends of a typical user of the OSN and assume that  ${\cal F}$ is independent and identically distributed across various users. Let $\tau^+$ and $\tau^-$ be the usual limits, e.g., $\Cx(\tau^-) := \lim_{t \uparrow \tau} \Cx(t)$. When a wi-user receives a post with a fake tag and shares it with a fake tag, it generates $\xi_{xx, wi}$ number of fake tagged copies. Similarly, when it tags the post as real, it shares to $\xi_{xy, wi}$ friends. Define $\xi_{yx, wi}$ and $\xi_{yy, wi}$ in a similar manner. We assume ($k$ is some constant):
\begin{eqnarray}
\label{eqn_shares_wi}
 \xi_{ix, wi}(\Om(\tau^-))  = \xi_{iy, wi}(\Om(\tau^-)) \sim Bin\left(\mathcal{F},\eta^u + \frac{k}{(\Sa(\tau^-))^2}\right) \mbox{ for } i \in \{x, y\},
\end{eqnarray}$Bin(\cdot, \cdot)$ denotes a binomial random variable; many times, users receive the post more than once, however, they may not be interested in it again. Thus, the new effective shares in \eqref{eqn_shares_wi} reduces with the total copies/shares of the post generated so far, i.e., $\Sa(\tau^-) := \Ax(\tau^-) + \Ay(\tau^-)$, for example as in \eqref{eqn_shares_wi}. The distribution considered in \eqref{eqn_shares_wi} is a specific example; however, our analysis can extend to any total-current shares-dependent sharing-distribution that satisfies assumption \ref{a2_prop} (see Section \ref{sec_prop_BP}).



\subsection{Warning-seeking (ws) users}
These users also click on the warning button, i.e., they incorporate the sender's tag, their innate capacity and the warning provided by the OSN to decide the tag.

Suppose a ws-user views the fake tagged post at time $\tau$. Let $\omega_\tau$ be the warning at this time. Then, as in \cite{kapsikar2020controlling}, we assume that such user tags the post as fake (real) w.p. $\min\{\alpha_x^u \omega_\tau, 1\}$ (respectively, $1-\min\{\alpha_x^u \omega_\tau, 1\}$) before sharing; here, $\alpha_x^u > 0$ is the sensitivity parameter to the warning when the post is received with fake tag. Similarly, if the post received by the ws-user has a real tag, then it tags the post as fake or real w.p. $\min\{\alpha_y^u \omega_\tau, 1\}$ and  $1-\min\{\alpha_y^u \omega_\tau, 1\}$, respectively, where $\alpha_y^u > 0$ is the sensitivity parameter when the post is received with real tag. Thus, we have:
\begin{align}\label{eqn_tag_ws}
    P(I_{x, ws}(\Om(\tau^-))=1|{\cal G}_\tau) = \min\{\alpha_x^u \omega_\tau, 1\} \mbox{ and }  P(I_{y, ws}(\Om(\tau^-))=1|{\cal G}_\tau) = \min\{\alpha_y^u \omega_\tau, 1\}.
\end{align}
The \textit{sensitivity parameters are indicative of the user's intrinsic ability to recognize the actuality of the post}. These parameterize \textit{warning-aided identification}, while  $p_F^u, p_R^u$ are the probabilities of \textit{un-aided identification}; both are characteristics of the users of the OSN. We thus assume a linear dependence between the two as in \cite{agarwal2023single}, i.e., we assume a $\rho \in (0,1)$ such that:
\begin{align}\label{eqn_prob_wi}
p_F^u = \alpha_x^u \rho \mbox{ and } p_R^u = \alpha_y^u \rho.
\end{align}

Now, similar to wi-users, a ws-user also shares the post with its friends. Using notations as in \eqref{eqn_shares_wi}, we have ($k$ is some constant):
\begin{eqnarray}
\label{eqn_shares_ws}
 \xi_{ix, ws}(\Om(\tau^-)) = \xi_{iy, ws}(\Om(\tau^-)) \sim Bin\left(\mathcal{F},\eta^u + \frac{k}{(\Sa(\tau^-))^2}\right) \mbox{ for } i \in \{x, y\}.
\end{eqnarray}

\subsection{Adversaries (a)} As is usually the case, there can be a small fraction of adversarial users on the OSN. These users aim to harm the efficacy of the system-generated warning by incorrectly tagging the post. Their agenda for doing so can be in self-interest or political. Often, such users do not have prior information about the actuality of the post, but to meet their objective they target to confuse the users about the actuality of the posts. Towards this, we consider that they always tag any post as real. In a way, such users are the ones who wish to color (tag) the posts as real, irrespective of the actuality of the posts. 

Let $I_{x, a}(\Om(\tau^-))$ and $I_{y, a}(\Om(\tau^-))$ be the indicator that an a-user with a fake or real tagged copy of the post tags the post as fake, where $\tau$ is the time when an a-user views the post. Here, we have:
\begin{align}\label{eqn_tag_rc}
    P(I_{x, a}(\Om(\tau^-))=1|{\cal G}_\tau) = P(I_{y, a}(\Om(\tau^-))=1|{\cal G}_\tau) = 0. 
\end{align}

An adversarial user shares the post with a real tag to its friends with probability $\eta_a \in (0,1)$, irrespective of the attractiveness of the post. Therefore, we have ($k$ is some constant):
\begin{align}\label{eqn_shares_adv}
    \xi_{ix, a}(\Om(\tau^-)) \equiv 0 \mbox{ and } \xi_{iy, a}(\Om(\tau^-)) \sim Bin\left(\mathcal{F},\eta_a  + \frac{k}{(\Sa(\tau^-))^2}\right) \mbox{ for } i \in \{x, y\}.
\end{align}

\subsection{Non-participants (np)} In \cite{kapsikar2020controlling}, it is assumed that all users viewing the post share and tag it. In reality, there can be users named as non-participants who neither participate in the tagging process nor share the post. In other words, when they receive a copy of the post, they do not respond, which we capture as:
\begin{align}
    P(I_{i, np}(\Om(\tau^-))=1|{\cal G}_\tau) &= P(I_{i, np}(\Om(\tau^-))=1|{\cal G}_\tau) = 0, \label{eqn_tag_np}
\end{align}
and shares to none, i.e.,
\begin{align}
    \xi_{ix, np}(\Om(\tau^-)) &= \xi_{iy, np}(\Om(\tau^-)) \equiv 0, \mbox{ for } i \in \{x, y\}. \label{eqn_share_np}
\end{align}


\noindent \textbf{Number of shares:} Let ${\cal B} := \{\mbox{wi, ws, a, np}\}$ be the set of types of users in the system. Let $\mu_0, \mu_1, \mu_2, \mu_a$ be the respective proportions of np, wi, ws, a-users on the OSN such that $\mu_1 + \mu_2 + \mu_a + \mu_0 = 1$; \textit{we assume that the OSN knows these proportions}. Since our approach is based on crowd-signals, therefore, it is meaningful to assume that $\mu_2 \in (0, 1)$. Any user of the OSN visits it after a random time which is exponentially distributed with parameter $1$ (without loss of generality); this is a commonly made assumption in the literature (see, for example, \cite{agarwal2021co, dhounchak2023viral, van2010viral}). If required, one can model different users visiting the OSN at different rates; for example, users might visit more often; our framework can easily extend to such a case. Any user of $j$-type, after viewing the post with fake tag ($i = x$) or real tag ($i = y$), generates $\Gamma_{ix, j}$ and $\Gamma_{iy, j}$ number of new fake and real tagged copies of the post respectively, where:
\begin{align}\label{eqn_final_shares}
\begin{aligned}
    \Gamma_{ix, j}(\Om(\tau^-)) &:= I_{i, j} (\Om(\tau^-)) \xi_{ix, j} (\Om(\tau^-)), \mbox{ and }\\
    \Gamma_{iy, j}(\Om(\tau^-)) &:= \bigg(1-I_{i, j}(\Om(\tau^-))\bigg) \xi_{iy, j}  (\Om(\tau^-)), \mbox{ for } i \in \{x, y\} \mbox{ and } j \in {\cal B}.
\end{aligned}
\end{align}

Next, we discuss some meaningful assumptions (inspired by \cite{kapsikar2020controlling}).

\vspace{2mm}
\noindent \textbf{Regime of parameters and assumptions:} The probability of a user fake tagging any $u$-post is higher when the sender's tag is fake; thus, $\alpha_x^u > \alpha_y^u$, for $u \in \{R, F\}$. We assume that the users are more likely to tag fake posts as fake, as compared to tagging real posts as fake, irrespective of the sender's tag, i.e., $\alpha_i^F > \alpha_i^R$, for each $i \in \{x, y\}$. \revg{Since the intent of a-users is to share the post rigorously, therefore, we assume $\eta_a > \eta^u$, for each $u$, only in the numerical experiments; the theoretical results follow even if $\eta_a \leq \eta^u$. }
Thus, in all, we assume the following:
\begin{align}\label{eqn_relation_parameters}
\begin{aligned}
    \alpha_x^u &> \alpha_y^u > 0 , \mbox{ for each }  u \in \{R, F\},   \alpha_i^F > \alpha_i^R \mbox{ for each } i \in \{x, y\},\\
    \eta_a &> \eta^F > \eta^R > 0,
    \mu_2 \in (0,1) \mbox{ and } \rho \in (0,1).
\end{aligned}
\end{align}
For the sake of clarity, we summarize all the notations which will be used consistently throughout the chapter:
\begin{table}[htbp]
\centering
\scalebox{0.75}{
\begin{tabular}{|c|c|l|}
\hline
Sr. No. & Notation                        & \multicolumn{1}{c|}{Description}                                \\ \hline
1.      & ${\cal B} = \{\mbox{wi, ws, a, np}\}$                   & types of users: warning-ignoring, warning-seeking, adversarial, non-participating \\ \hline
2.      & $\mu_0, \mu_1, \mu_2, \mu_a$           & proportion of np, wi, ws and a-users \\ \hline
3.      & $u \in \{R, F\}$ & actuality of the post as real or fake respectively                                        \\ \hline
4.      & $\eta^u, \eta_a$                        & probability of a user/adversary sharing the post to its friend                                     \\ \hline
5.      & $x, y$                    & fake or real tag by the sender \\ \hline
6.      & $\alpha_x^u, \alpha_y^u$                    & sensitivity of a user towards the warning when received with fake or real tag  \\ \hline
\end{tabular}}
\caption{Summary of the notations} \label{table_notations}
\end{table}

\subsection{Warning Mechanism (WM) - system-generated warning}
In \cite{kapsikar2020controlling}, the authors designed a warning mechanism (WM) by leveraging upon the responses of the users. They assumed all users are ws-users and did not consider the adversaries (i.e., $\mu_2 = 1$). The main idea behind the design of the mechanism is to exploit the collective wisdom of the users (via responses of all users), as depicted in Figure \ref{fig_warning_mech} (left side). 
The warning considered in \cite{kapsikar2020controlling} is:
\begin{equation} \label{eqn_warning}
 \omega_t = \bigg ( \frac{w\Cx(t)}{\Cx(t)+b\Cy(t)} + \gamma \bigg) = \bigg ( \frac{w\Beta(t)}{\Beta(t)+b(1 - \Beta(t))} + \gamma \bigg), \mbox{ where } \Beta(t) := \frac{\Cx(t)}{\Cx(t) + \Cy(t)}
\end{equation}represents the relative fraction of (unread) fake tagged copies at time $t$; $w$ and $b$ are the control parameters; $\gamma > 0$ is the parameter which captures the prior knowledge OSN has about the post via some fact-check mechanism. Here, $w \in [0, \overline{w}]$ for $\overline{w} := \frac{1}{\alpha_x^F}-\gamma$. This ensures that a ws-user tags the fake tagged copy of the post as fake with probability $\min\{\alpha_x^u \omega(\beta), 1\} =  \alpha_x^u \omega(\beta)$ for any $\beta \in [0,1]$, when the warning is as in \eqref{eqn_warning}; thus, $\min\{\alpha_y^u \omega(\beta), 1\} = \alpha_y^u \omega(\beta)$ (since $\alpha_y^u < \alpha_x^u$, see \eqref{eqn_relation_parameters}). Further, the parameter $b \in [0, \infty)$. The warning in \eqref{eqn_warning} is generated individually for each post.

\begin{figure}[http]
    \centering
    \includegraphics[trim = {0cm 3.3cm 0cm 0cm}, clip, scale = 0.5]{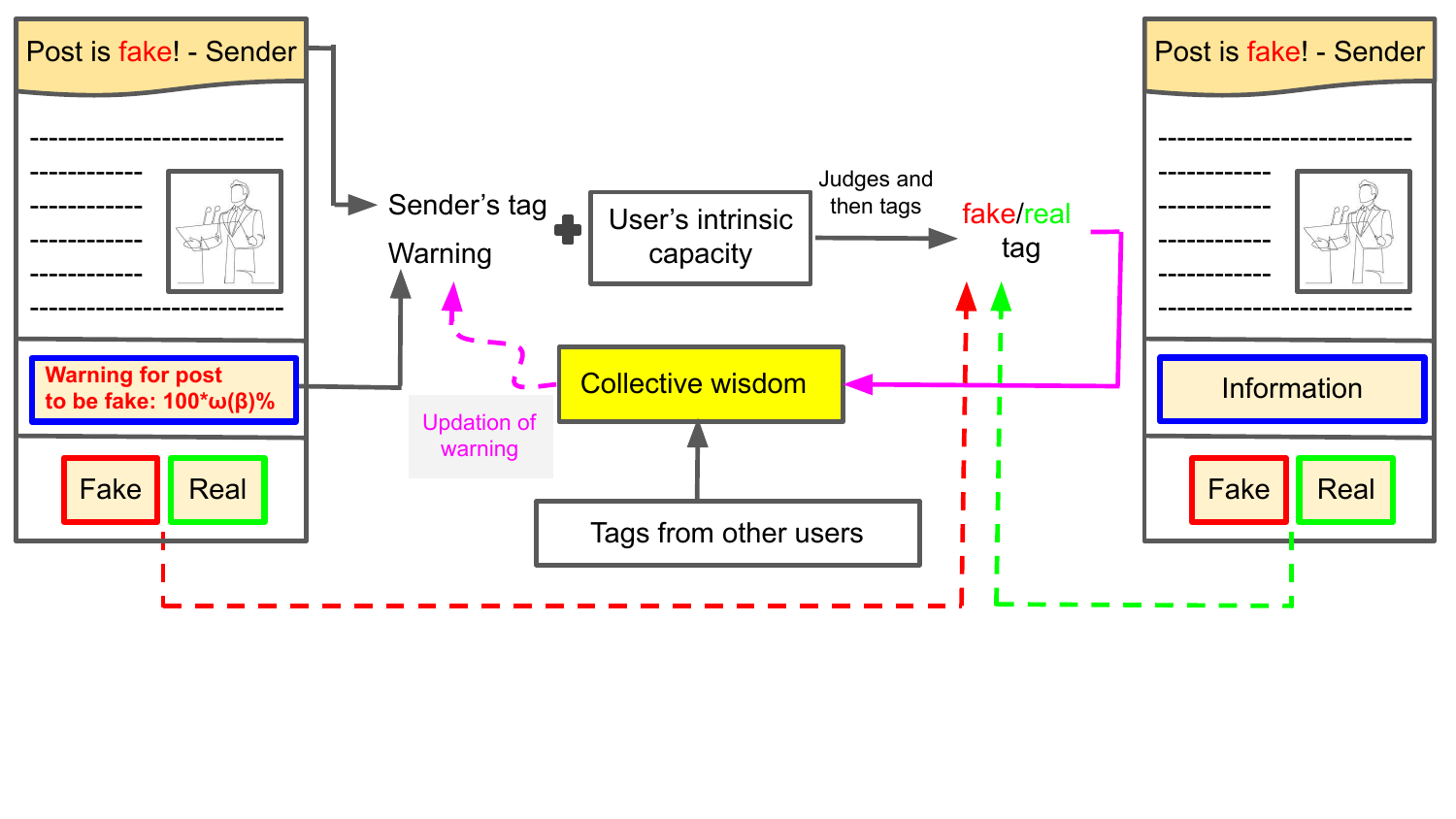}
    \caption{On the left, ws-user tags the post as fake. On the right, a-user tags the post as real, without checking the warning or sender's tag.}
    \label{fig_warning_mech}
\end{figure}

In this chapter, we are considering a variety of user behavior. Therefore, the warning is now influenced by the responses of users who ignore the warning while tagging or are purposely providing incorrect tags. 
In Figure \ref{fig_warning_mech}, we depict that the warning is updated by the response (fake) of the ws-user (left side of the figure) and also by that of a-user (right side of the figure). Similarly, one can visualize how a warning gets updated when a wi-user tags. This suggests that the warning \eqref{eqn_warning} needs to be studied for our complex and more realistic system.

It is clear from the discussion so far that the end goal of the OSN is to nudge users towards the correct identification of the posts. \revg{Let $B^u(t)$ represent the proportion of fake tags, given that the actuality of the post is $u \in \{R, F\}$.} Then, similar to \cite{kapsikar2020controlling}, we aim to optimally choose $w, b$:
\begin{itemize}
    \item to maximize the proportion of fake tags for the fake post, $\max \lim_{t\to \infty} \Beta^F(t)$,  and
    \item to ensure that the proportion of fake tags for the real post, $\lim_{t\to \infty} \Beta^R(t)$, is at most $\delta$, for some $\delta \in (0,1)$.
\end{itemize}
\revr{The above objective is well defined if the limits in the above exist and are unique almost surely. By Theorem \ref{thrm_BP_to_fake} stated in Section \ref{sec_warning_old}, we prove that the limits indeed exist (but need not be unique) for any general warning mechanism. Hence, define ${\cal L}^u := \left\{\lim_{t\to \infty} \Beta^u(t)\right\}$ as the set of all possible limits for $u$-post, across all sample paths, and consider the following optimization problem:
\begin{align}\label{eqn_gen_opt}
    \max_{w, b} \inf({\cal L}^F) \mbox{ subject to } \sup({\cal L}^R) \leq \delta.
\end{align}}
Further, we shall investigate the following two questions:
\begin{enumerate}
    \item How does the optimal WM in \eqref{eqn_warning} perform in the presence of wi-users and a-users? 
    \item If the performance degrades, can we design improved WMs which are robust against adversaries?
\end{enumerate}

\subsection{Warning dynamics and Branching process}
It is clear that when a user tags the post as fake, the fake number of copies (represented by $x$)  gets updated; otherwise, the real ($y$) number of copies gets updated. Further, the user who receives the post can be one among the type $i$, for $i \in {\cal B}$, w.p. given by the proportion of the type it belongs to; for example, the recipient can be a wi-user w.p. $\mu_1$. As discussed in \eqref{eqn_shares_wi}, \eqref{eqn_shares_ws}, \eqref{eqn_shares_adv} and \eqref{eqn_share_np}, the distribution of the number of shares depends on the type of the user who received the post.

Let $\tau$ be the time when a type-$i$ user views the post on its timeline with a fake tag. Then, the number of fake tagged and real tagged copies of the underlying post evolves at time $\tau$ as follows:
\begin{equation}\label{eqn_transition_fake_tag}
\begin{aligned}
\Cx(\tau^+) & = \Cx(\tau^-) - 1 + \Gamma_{xx, i}(\Om(\tau^-)), 
\Cy(\tau^+) = \Cy(\tau^-) + \Gamma_{xy, i}(\Om(\tau^-)),\\ 
\Ax(\tau^+) & = \Ax(\tau^-) + \Gamma_{xx, i}(\Om(\tau^-)), \mbox{ and }
\Ay(\tau^+) = \Ay(\tau^-) + \Gamma_{xy, i}(\Om(\tau^-)).
\end{aligned}
\end{equation}
We argued before that once a user reads a post, it is seldom interested in the same post again; thus, the current (unread) number of fake tagged copies decreases by $1$. Similarly, when a type-$i$ user who received the post with the real tag views the post,  the system evolves as:
\begin{equation}\label{eqn_transition_real_tag}
\begin{aligned}
\Cx(\tau^+) & = \Cx(\tau^-) + \Gamma_{yx, i}(\Om(\tau^-)), 
\Cy(\tau^+) = \Cy(\tau^-) - 1 + \Gamma_{yy, i}(\Om(\tau^-)),\\
\Ax(\tau^+) & = \Ax(\tau^-) + \Gamma_{yx, i}(\Om(\tau^-)), \mbox{ and }
\Ay(\tau^+) = \Ay(\tau^-) + \Gamma_{yy, i}(\Om(\tau^-)).
\end{aligned}
\end{equation} 
We shall briefly call the above warning-mechanism aided dynamics as \underline{warning dynamics}. 
At this point, it is important to state that the dynamics described above can be modelled as a continuous-time total-current population-dependent branching process (TC-BP) discussed in previous chapter, except for varying death rates. We will discuss how such correspondences can be made in Section \ref{sec_warning_old}; in particular, we will see that the viewing of the post can be modelled as a death in an appropriate TC-BP and hence, will have different death-types and rates owing to different types of users. However, we first analyze the TC-BPs with multiple death types in the next section using ODE based stochastic approximation technique, which will be instrumental for our study. 

\vspace{2mm}
\noindent \textbf{Informal outline for design of WMs:}
\revr{
We consider any general warning mechanism $\omega(\beta)$, which depends on the proportion of fake tags ($\beta$) provided by the previous recipients of the post. The limiting behaviour of the warning-guided post-propagation process is analyzed using the ODE derived via the analysis of the underlying BP. In particular, we will show that the analysis of a one-dimensional ODE suffices to study the limits of the underlying process; of course, the limits depend upon the warning mechanism utilized. 
The main idea is to reverse-engineer:  consider the design of the warning mechanism (to achieve the desired output) based  on the anticipated attractors of the one-dimensional ODE. We will follow this approach in Section \ref{sec_warning_old}, and after that, we bring our attention back to the control of fake post propagation over OSNs.}

\section{Total-Current population-dependent BP with multiple death types} \label{sec_prop_BP}

Consider two types of populations, namely $x$ and $y$-types, and  
let $c_{x, 0}$ and $c_{y, 0}$ be their respective initial sizes. An individual can either die naturally, or it may die  differently due to unnatural circumstances. We refer any death which is not natural as `unnatural death'\footnote{
In biological systems, unnatural deaths may occur due to exposition to a virus, competition with other species, etc. We discuss unnatural deaths for the application at hand in Section \ref{sec_warning_old}.}. Let $D_i := \{0, 1, \dots, d_i\}$ be the set of variety of deaths for $i$-type individual, where $d_i \in [0, \infty)$. Here, $d = 0$ represents the natural death and $d \in D_i-\{0\}$ represents an unnatural death; $D_x$ need not equal $D_y$ as some circumstances may affect only one population. We shall briefly refer to the death of variety $d$ as $d$-death.

Now, given that the interest of this chapter is in controlling the fake post propagation over OSNs, our focus is on the time-asymptotic proportion of the population (fake tags). Therefore, it is sufficient to study the embedded process (discrete-time chain defined at death instances) of the continuous-time Markov process. In Chapter \ref{ch:journal1}, we analyzed the TC-BP using stochastic approximation based approach, where only natural deaths occur. In this section, we will follow same approach to incorporate different varieties of deaths. We begin by introducing few notations which are exactly as in previous chapter, however are re-written here for the ease of reading.

Let $\tau_n$ be the time at which $n$-th individual dies. Consider any $n \geq 1$. Let $\Om_n := (C_{x, n}, C_{y, n}, A_{x, n}, A_{y, n})$, where $C_{x, n}, C_{y, n}$ represent the \textit{current population} and $A_{x, n}, A_{y, n}$ are the \textit{total population} sizes immediately after ${\tau}_n$, e.g., $C_{x, n} = \Cx({\tau}_n^+)$. Let $S_n :=  C_{x, n} + C_{y, n}$ be the sum current population, again immediately after $\tau_n$.  Let $\om = (\cx, \cy, \ax, \ay)$ be a realisation of the random vector $\Om$. \revg{Any individual can die naturally or unnaturally. We assume that the time till $d$-death of an $i$-type individual is exponentially distributed with parameter $\lambda_{i, d} \in (0, \infty)$. An individual in the population will die according to the first death (variety) event that occurs. By memoryless property, after any given instance of time (e.g., $\tau_n$), the death-time of any $i$-type individual in the population is again exponentially distributed with parameter $\sum_{d} \lambda_{i, d}$, and hence the first death in the two populations is exponentially distributed with parameter $\left(\sum_{d} \lambda_{x, d} + \sum_{d} \lambda_{y, d}\right)$. We further assume that the parameter $\lambda_{i, d}$ depends on the population-size, i.e., $\lambda_{i, d}(\om_n)$. conditioned on $\om_n$, for each $i \in \{x, y\}$. Observe that we have population-dependency even for the natural deaths, in contrast to the classical models studying only natural deaths (see, for example, \cite{klebaner1993population, athreya2012classical, jagers1969proportions}).}


The current population can get extinct, and thus let $\nu_e := \inf \{n : S_n = 0\}$ be the extinction epoch, with the usual convention that $\nu_e = \infty$, when $S_n > 0$ for all $n$. \textit{For the sake of completion, define $\Om_n := \Om_{\nu_e}$ and $\tau_{n} :=\tau_{\nu_e}$, for all $n \geq \nu_e$, when $\nu_e < \infty$.} \revg{We refer the sample paths in which $\nu_e = \infty$ as the non-extinction paths, and the complementary ones as the extinction paths.} \textit{Define $\Beta_n :=  C_{x, n}/S_n$ as the proportion of $x$-type population among current population}. Let $\beta = \cx/(\cx+\cy)$ be a realisation of $\Beta$. 

\hide{
Define $\Om(t) := (\Cx(t), \Cy(t), \Ax(t), \Ay(t))$, where $\Cx(t), \Cy(t)$ represent the \textit{current population} and $\Ax(t), \Ay(t)$ are the \textit{total population} sizes at time $t$. Observe 
$(\Ax(0), \Ay(0)) = (\cx_0, \cy_0)$.  Let $\om = (\cx, \cy, \ax, \ay)$ be a realisation of the random vector $\Om$. Now, conditioned on $\om$, say an $i$-type individual lives for an exponentially distributed with parameter $\lambda_{i, d}(\om) \in (0, \infty)$ before it $d$-dies, where $i \in \{x, y\}$. In classical BPs, such parameters are population-independent, while here, we consider them to depend not only on the population sizes, but also on the variety of death. 
Here, $\lambda_{i, 0}(\om)$ represents the lifetime parameter if populations were living (and reproducing) independently.

Define $\beta := \cx/(\cx + \cy)$ as the proportion of $x$-type individuals, conditioned on $\om$. }

\subsection{Evolution of embedded process}
In classical BPs, each individual lives for a random time which is exponentially distributed with a common parameter  (say) $\lambda > 0$. Thus, an individual to die at $n$-th epoch is of $x$-type w.p.\footnote{This happens due to the memory-less property of exponential distribution and as minimum of $k$ independent and identically distributed exponentially distributed random variables with parameter $\lambda$ is exponentially distributed with parameter $k\lambda$. } $\beta_n$, conditioned on $\Om_n = \om_n$. In similar lines, with the possibility of unnatural deaths, the probability that an $i$-type individual $d$-dies is given by:
\begin{align}\label{eqn_prob_d_death}
\begin{aligned}
\mbox{P($x$-type individual $d$-dies$|\om$)} &= \frac{\lambda_{x, d}(\om) \beta}{d(\om)} \mbox{ and} \\
\mbox{P($y$-type individual $d$-dies$|\om$)} &= \frac{\lambda_{y, d}(\om) (1-\beta)}{ d(\om)}, \mbox{ where }
\end{aligned}
\end{align}
\begin{align*}
d(\om) &:= \sum_{d \in D_x} \lambda_{x, d}(\om) \beta + \sum_{d \in D_y} \lambda_{y, d}(\om) (1-\beta).
\end{align*}
In all, the overall probability that an $i$-type individual is the first to die after previous death instance, $\tau$, is given by:
\begin{align}\label{eqn_prob_death}
\begin{aligned}
\mbox{P($x$-type individual dies$|\om$)} &= \frac{\beta \sum_{d \in D_x} \lambda_{x, d}(\om)}{d(\om)} =: f_{\beta}(\om) \mbox{ and}\\ \mbox{P($y$-type individual dies$|\om$)} &= 1 - f_{\beta}(\om).
\end{aligned}
\end{align}

Suppose an individual of $i$-type dies at $n$-th epoch. Then, the current size (not the total size) of $i$-type reduces by $1$ due to death. Further, if it $d$-dies for $d \in D_i$, it produces $\offs_{ii, d}(\Om_{n-1})$ and $\offs_{ij, d}(\Om_{n-1})$ offspring of $i$-type and $j$-type ($j\neq i$) respectively, conditioned on the sigma algebra $\sigma\{\Om_{n-1}\}$, where $\offs_{ij, d} (\Om_{n-1})$ is an  integer-valued random variable. Basically, when $\Om_{n-1} = \om_{n-1}$, the random offspring are represented by $\offs_{ij, d}(\om_{n-1})$ for each $i, j$ and $d$. Thus, the embedded process immediately after an $i$-type individual $d$-dies at $n$-th epoch is given by:
\begin{equation}\label{evolve_x_up_time_prop}
\begin{aligned}
C^i_n &= C^i_{n-1}  + \offs_{ii, d}(\Om_{n-1}) - 1, \ \ \  A^i_n = A^i_{n-1}  + \offs_{ii, d}(\Om_{n-1}), 
\\
C^j_n &= C^j_{n-1} + \offs_{ij, d}(\Om_{n-1}), \ \ \ \ \ \ \ \ \ A^j_n = A^j_{n-1} + \offs_{ij, d}(\Om_{n-1}), \mbox{ for } i \neq j.
\end{aligned}
\end{equation}
Conditioned on $\om$, we assume the $\om$-dependent random offspring satisfy the following, which also ensures throughout super-criticality, a notion defined in previous chapter:
\begin{enumerate}[label=\textbf{C.\arabic*}, ref=\textbf{C.\arabic*}]
    \item \label{a1_prop} There exist two integrable random variables $\overline{\offs}$ and $\underline{\offs}$ which bound the random offspring as: $0 \leq \underline{\offs} \leq \offs_{ix, d}(\om) + \offs_{iy, d}(\om) \leq \overline{\offs}$ a.s., for each $\om$, for each $d$. Also,  $E[\overline{\offs}^2] < \infty$ and $E[\underline{\offs}] > 1$. Further, $\offs_{ii, d}(\om) \geq 0$ a.s., for each $i, \om, d$. Furthermore, assume that $\inf_{\om} \lambda_{x, d}(\om) > 0$ for each $d \in D_x$ and $\inf_{\om} \lambda_{y, d}(\om) > 0$ for each $d \in D_y$.
\end{enumerate}

\subsection{Mean matrix}

Let \textit{$m_{ij, d} (\om) := E[ \offs_{ij, d} (\om) ]$ denote the expectation of the number of $j$-type offspring, when an $i$-type parent $d$-dies, conditioned on $\om$, for $i, j \in \{x, y\}$ and $d \in D_i$}. Further, define the mean matrix $M(\om) := [m_{ij}(\om)]_{i, j\in\{x, y\}}$ as given below:
\begin{eqnarray}\label{eqn_mean_matrix}
M(\om) :=
\left [  
\begin{array}{ll}
    \frac{\sum_{d \in D_x} \lambda_{x, d}(\om) m_{xx, d}(\om)}{\sum_{d \in D_x} \lambda_{x, d}(\om)}    & \frac{\sum_{d \in D_x} \lambda_{x, d}(\om) m_{xy, d}(\om)}{\sum_{d \in D_x} \lambda_{x, d}(\om)}   \\ \\
      \frac{\sum_{d \in D_y} \lambda_{y, d}(\om) m_{yx, d}(\om)}{\sum_{d \in D_y} \lambda_{y, d}(\om)}    & \frac{ \sum_{d \in D_y} \lambda_{y, d}(\om) m_{yy, d}(\om)}{\sum_{d \in D_y} \lambda_{y, d}(\om)}
\end{array}
\right ].  
\end{eqnarray}
Then,  for $j \in \{x, y\}$, we have (see \eqref{eqn_prob_d_death}, \eqref{eqn_prob_death} and \eqref{eqn_mean_matrix}):

\vspace{-4mm}
{\small
\begin{align}
\begin{aligned}
    E[j\mbox{-type offspring produced by an } x\mbox{-type parent}|\om] &= \sum_{d \in D_x} \frac{\lambda_{x, d}(\om) \beta}{d(\om)} m_{xj, d}(\om) \\
    &= f_{\beta}(\om) m_{xj}(\om),\\
    E[j\mbox{-type offspring produced by a } y\mbox{-type parent}|\om] &= \sum_{d \in D_y} \frac{\lambda_{y, d}(\om) (1-\beta)}{d(\om)} m_{yj, d}(\om)\\
    &= (1-f_{\beta}(\om)) m_{yj}(\om).
\end{aligned}
\end{align}}

As in Lemma \ref{lemma_sum_pop}, one can prove the dichotomy for the sum current population of TC-BP with multiple death types, as in the following:
\begin{lemma}\label{lemma_dichotomy}
Assume \ref{a1_prop} and define $\underline{m} =: E[\underline{\offs}]$. Then, we have:
$$
P\left(\left\{\liminf_{n} S_n e^{-\underline{\lambda}(\underline{m}-1)\tau_n} > 0\right\} \cup \left\{\lim_{n \to \infty} S_n = 0\right\}\right) = 1, 
$$where $\underline{\lambda} := \inf_\om\{\lambda_{x, 0}(\om), \dots, \lambda^x_{d_x}(\om), \lambda_{y, 0}(\om), \dots, \lambda^y_{d_y}(\om)\} > 0$.
\end{lemma}
\revg{Thus, the sum current population either gets extinct or in non-extinction paths, it explodes (i.e., it grows exponentially larger at rate $\underline{\lambda}(\underline{m}-1)$).
}

\subsection{Main Result}
We will now provide the first main result of the chapter which determines the limit proportion, $\lim_{t \to \infty} \Bc(t)$ in non-extinction paths and additionally, provides the deterministic approximate trajectories for the underlying BP. The result follows in similar lines to Theorem \ref{thrm1}, while accommodating some important changes for multiple deaths. As established in Lemma \ref{lemma_dichotomy}, the underlying BP can explode. In such a case, it is a common practice to scale the process appropriately that enables convergence to a finite limit (see, for example, \eqref{eqn_ratios} and \cite{kapsikar2020controlling}).



\hide{Let $\tau_n$ be the time at which $n$-th individual dies.
Consider any $n \geq 1$. Let $\Om_n := (C_{x, n}, C_{y, n}, A_{x, n}, A_{y, n})$ be the individual (current and total) populations and $S_n :=  C_{x, n} + C_{y, n}$ be the sum current population, both immediately after $\tau_n$, e.g., $C_{x, n} = \Cx(\tau_n^+)$. The current population can get extinct, and thus let $\nu_e := \inf \{n : S_n = 0\}$ be the extinction epoch,  
with the usual convention that $\nu_e = \infty$, when $S_n > 0$ for all $n$. \textit{For the sake of completion, define $\Om_n := \Om_{\nu_e}$ and $\tau_{n} :=\tau_{\nu_e}$, for all $n \geq \nu_e$, when $\nu_e < \infty$.}}

To this end, define the scaled ratios $\Pc_n := S_{n}/n$ and $\Tc_n := C_{x, n}/n$. Let $Z_n := A_{x, n} + A_{y, n}$ be the total population size immediately after $\tau_n$, and then analogously, define  $\Pa_n$ and $\Ta_n$ for the total population.  Let $\Ups_n := (\Pc_n, \Tc_n, \Pa_n, \Ta_n)$, and let $\Ups_0 := (s_0^c, c_{x, 0}, s_0^c, c_{x,0})$ denote the initial population, where $s_0^a = s_0^c := c_{x, 0} + c_{y, 0}$. \hide{\textit{Define $\Bc_n := \Tc_n/\Pc_n  = C_{x, n}/S_n$ as the proportion of $x$-type population among current population}. } Let $\ups := (\pc, \tc, \pa, \ta)$ be a realisation of $\Ups$. \hide{ and $\beta = \cx/(\cx+\cy) = \tc/\pc$ be that of $\Bc$. }

In \ref{a2}, we assumed that the total-current population-dependent  mean functions converge to proportion-dependent mean functions, which can further be discontinuous. Similar to that, we now assume that the resultant mean functions ($m_{ij}(\om)$, and not $m_{ij, d}(\om)$) converge to proportion-dependent mean functions at a certain rate. However, to accommodate for the variety of deaths, we assume that the the lifetime parameters of the populations also become proportion-dependent asymptotically (at the same rate of convergence as that of the mean functions).
\begin{enumerate}[label=\textbf{C.\arabic*}, ref=\textbf{C.\arabic*}]
\setcounter{enumi}{1}
    \item Define $\beta(\ups) := \tc/\pc = \cx/s$. As sum current population,  $s \to \infty$:  \label{a2_prop}
    \begin{align*}
     |m_{ij}(\om) - \minf_{ij}(\beta(\ups))| &\leq \frac{1}{(s)^\alpha}, \mbox{ for each }i, j \in \{x, y\} \mbox{ and}\\
    |\lambda_{i, d}(\om) - \lambda_{i, d}^\infty(\beta(\ups))| &\leq \frac{1}{(s)^\alpha}, \mbox{ for each } d \in D_i \mbox{ for each } i \in \{x, y\},    \mbox{ for some } \reva{\alpha \ge 1}.
    \end{align*}
\end{enumerate}
Further, under \ref{a2_prop}, the function $f_\beta(\om)$ converges to $f_{\beta}^\infty(\beta)$ as given below (see \eqref{eqn_prob_death}):
\begin{align}\label{eqn_f_infty}
\begin{aligned}
|f_{\beta}(\om) - f_{\beta}^\infty(\beta) | &\leq \frac{1}{(s)^\alpha}, \mbox{ where } f_{\beta}^\infty(\beta) := \frac{\beta \sum_{d \in D_x} \lambda_{x, d}^\infty(\beta)}{d^\infty(\beta)} \mbox{ with}\\
d^\infty(\beta) &:= \beta \sum_{d \in D_x} \lambda_{x, d}^\infty(\beta) + (1-\beta) \sum_{d \in D_y} \lambda_{y, d}^\infty(\beta).
\end{aligned}
\end{align}
In all, under \ref{a1_prop}-\ref{a2_prop}, we analyze the ratios $\Ups_n$ using the solutions of the following ODE:
\begin{align}\label{eqn_ODE_prop}
\dot{\ups} &= \ga(\ups) = \mathbf{h}(\beta)1_{\{\pc > 0\}} - \ups, \mbox{ where } \mathbf{h}(\beta) := (h_{\psi}^{c},  h_{\theta}^{c},  h_{\psi}^{a},  h_{\theta}^{a}), \mbox{ with} 
\end{align}
\begin{align*}
        h_{\psi}^{c}(\beta) &= f_{\beta}^\infty(\beta) \bigg(\minf_{xx}(\beta) +     \minf_{xy}(\beta)\bigg) + \left(1-f_{\beta}^\infty(\beta)\right)\bigg(\minf_{yy}(\beta) + \minf_{yx}(\beta)\bigg) - 1,  \nonumber\\
        h_{\theta}^{c}(\beta)  &= f_{\beta}^\infty(\beta)\bigg(\minf_{xx}(\beta) - 1\bigg) + \left(1-f_{\beta}^\infty(\beta)\right) \minf_{yx}(\beta), \\
        h_{\psi}^{a}( \beta) &=  f_{\beta}^\infty(\beta) \bigg(\minf_{xx}(\beta) + \minf_{xy}(\beta) \bigg) + \left(1-f_{\beta}^\infty(\beta)\right)\bigg(\minf_{yy}(\beta) + \minf_{yx}(\beta) \bigg)   \mbox{ \normalsize  and}  \nonumber  \\
        h_{\theta}^{a}( \beta)  &= f_{\beta}^\infty(\beta)\minf_{xx}(\beta) + \left(1-f_{\beta}^\infty(\beta)\right) \minf_{yx}(\beta).  \nonumber
\end{align*}
Now, exactly as in \ref{a3}, we assume the following (see  Definition \ref{defn_solution} for the definition of extended solution):
\begin{enumerate}[label=\textbf{C.\arabic*}, ref=\textbf{C.\arabic*}]
\setcounter{enumi}{2}
    \item There exists a unique solution $\ups(\cdot)$ for ODE \eqref{eqn_ODE_prop} in the extended sense over any bounded interval.\label{a3_prop}
\end{enumerate}
As per Definition \ref{defn_attractors}, let $\cA$ be the attractor set and $\cR$ be the saddle set with respect to the ODE \eqref{eqn_ODE_prop}. For systems modelling the BPs,  the following subset of the combined domain of attraction of $\cA$ and $\cR$ is relevant (recall the definition of ratios $\ups$):
\begin{align}\label{eqn_domain of attraction_prop}
    \cD &:= \{\ups \in (\mathbb{R}^+)^4: \tc \leq \pc \leq \pa, \ta \leq \pa \mbox{ and } \ups(t) \to {\cA}\cup \cR \mbox{ as } t \to \infty, \mbox{ if } \ups(0) = \ups\}.
\end{align}
Therefore, we will be interested in initial conditions $\ups(0) \in \cD_I$ for the ODE \eqref{eqn_ODE_prop}.

In Definition \ref{defn_hovers}, we introduced a new notion of limiting behavior of the stochastic process, named `hovering around the saddle set' - here, the stochastic trajectory visits every neighborhood of $\cR$ infinitely often (i.o.), but also leaves some neighborhood of $\cR$ i.o. The main result in Theorem \ref{thrm1} states that with certain positive probability, the random trajectory either converges to the attractor set or it converges to/hovers around a special kind of saddle set.  In particular, if any non-zero saddle point, $\ups^* \neq \mathbf{0}$, is attracted exponentially to $\cR$ along a particular affine sub-space, $\revg{{\mathbb S}(\ups^*) := \{\ups: \beta(\ups) = \beta(\ups^*)\}}$ and to $\cA$ in the remaining space, then such $\ups^*$ are named as (quasi) q-attractor in Definition \ref{defn_q_as}.  We have a similar result for the case with multiple deaths; we would like to mention again that the coming result does not assert the positive probability of hovering around.

Similar to \ref{a4}, under above definition, we finally assume the following:
\begin{enumerate}[label=\textbf{C.\arabic*}, ref=\textbf{C.\arabic*}]
\setcounter{enumi}{3}
    \item (a) Let $\cA\cap\cD_I$ be the attractor set and each $\ups \in \cR\cap\cD_I$ be the q-attractor. Consider $\cD$ as in \eqref{eqn_domain of attraction_prop} and let  $\cS := \cD \cap \{\pa \leq b\}$, for some $b > 0$, be a compact subset of combined domain of attraction.
    
    (b) Assume $p_{b} := P(\mathcal{V}) > 0$, where $\mathcal{V} :=  \{\omega : \Ups_n(\omega) \in \cS \mbox{ i.o.}\}$. \label{a4_prop}
\end{enumerate}
We have the following result:
\begin{theorem}\label{thm1}
  Under \ref{a1_prop}-\ref{a3_prop}, we have:
  \begin{enumerate}[label=(\roman*)]
        \item For every $T>0$, a.s.  there exists a sub-sequence $(n_l)$ such that:
            $$
            \sup_{k: t_k \in [t_{n_l}, t_{n_l} + T]} d(\Ups_k, \ups(t_k - t_{n_l})) \to 0  \mbox{ as } l \to \infty, \mbox{ where } t_n := \sum_{k=1}^n \frac{1}{k} \mbox{ and}
            $$
        $\ups(\cdot)$ is the extended solution of ODE \eqref{eqn_ODE_prop} which starts at $\ups(0) =
        \lim_{n_l \to \infty} \Ups_{n_l}$.
        \item Further, assume \ref{a4_prop}. Then, $P({\cal C}_1 \cup {\cal C}_2) \geq p_b$, where
        \begin{align*}
            \hspace{5mm} {\cal C}_1 &: =\{\Ups_n \to (\cA \cup \cR)\cap \cD_I \mbox{ as } n \to \infty\}, \mbox{ and }
                {\cal C}_2 := \{ \Ups_n \mbox{ hovers around } \cR \}. \hspace{4mm}  \mbox{ \eop} 
        \end{align*}
  \end{enumerate}
\end{theorem}
\textit{All the proofs of this Chapter are provided in Appendix \ref{appendix_journal2}.}

\subsection{Derivation of attractor and saddle sets}
It is evident from Theorem \ref{thm1} that the limit proportion, $\lim_{n \to \infty} \Beta_n$ can be deduced if one derives the attractor and saddle (specifically, q-attractor) sets. 
In Section \ref{procedureAR}, we proposed a procedure to derive these sets for the ODE \eqref{eqn_ODE_prop}, when only natural deaths occur. The main idea was to  exploit the dependence of limit mean functions on $\beta$ as in \ref{a2_prop} and finally, it is showed that the analysis of $\beta(\ups)$-ODE suffices. We extend the same approach for the new process with both natural and unnatural deaths. Towards this, one can derive the following limit $\beta$-ODE, using \eqref{eqn_ODE_prop}:
\begin{align}\label{eqn_beta_ODE_prop}
\begin{aligned}
\dot{\beta} &= \frac{1}{\pc} g_\beta(\beta)1_{\{\pc > 0\}},\mbox{ where}\\ 
g_\beta(\beta) &:= -f^\infty_{\beta}(\beta) m_{xy}^\infty(\beta) + (1-f^\infty_{\beta}(\beta))m_{yx}^\infty(\beta)   + \beta - f^\infty_{\beta}(\beta) \\
&\hspace{1cm}+ (1-\beta)f^\infty_{\beta} (\beta) (m_{xx}^\infty(\beta) + m_{xy}^\infty(\beta)) - \beta (1-f^\infty_{\beta}(\beta))  (m_{yy}^\infty(\beta) + m_{yx}^\infty(\beta))  .
    \end{aligned}
\end{align}
Similar to \eqref{eqn_beta_ODE}, we will also show that \textit{the asymptotic analysis of $\beta$ is independent of other components of $\ups$}. In particular, the result stated below shows that the analysis of the following one-dimensional ODE suffices:
\begin{align}\label{eqn_beta_ode_simple_prop}
    \dot{\beta} &= g_\beta(\beta).
\end{align}

\begin{theorem} \label{thrm_beta_ODE_prop}
\reva{Consider the interval $[0,1]$ such that $g_\beta(0) \geq 0$ and $g_\beta(1) \leq 0$.} Define ${\cal I} := \{x^*: g_\beta (x^*) = 0\}$ and say ${\cal I}= \{x_i^*:     1 \leq i \leq n\}$, for some $1 \leq n < \infty$. For each $i$, let there exist an open/closed/half-open non-empty interval around $x_i^* \in {\cal I}$, say ${\cal N}_i^*$, such that $\cup_{1\leq i\leq n} {\cal N}_i^* = [0,1]$ and ${\cal N}_i^* \cap {\cal N}_j^* = \emptyset$ for $i\neq j$. Define ${\cal N}_i^- := {\cal N}_i^*\cap[0, x_i^*)$ and 
${\cal N}_i^+ := {\cal N}_i^*\cap (x_i^*, 1]$. Let $g_\beta(x)$ be Lipschitz continuous on ${\cal N}_i^-$ and $ {\cal N}_i^+$ for each $i$:
\begin{enumerate}[label=(\roman*)]
    \item if $g_\beta(x) > 0$ for all $x \in {\cal N}_i^-$, $g_\beta(x) < 0$ for all $x \in {\cal N}_i^+$, then, $x_i^*$ is an attractor for ODE \eqref{eqn_beta_ode_simple_prop}; 
    \item   if $g_\beta(x) < 0$ for all $x \in {\cal N}_i^-$ and $g_\beta(x) > 0$ for all $x \in {\cal N}_i^+$, then, $x_i^*$ is a repeller for ODE \eqref{eqn_beta_ode_simple_prop}; 
    \item else if $g_\beta(x) > 0$ (or $g_\beta(x) < 0$) for all $x \in {\cal N}_i^-$ and $g_\beta(x) > 0$ (or $g_\beta(x) < 0$ respectively) for all $x \in {\cal N}_i^+$, then, $x_i^*$ is a saddle point for ODE \eqref{eqn_beta_ode_simple_prop}. 
\end{enumerate}
Further, ODE \eqref{eqn_ODE_prop} satisfies \ref{a3_prop}. Furthermore, the attractor and saddle sets in $\cD_I$ are respectively given by:
\begin{align*}
{\cal A} &:= \{\mathbf{h}(x^*): x^* \in {\cal I} \mbox{ is an attractor for the ODE \eqref{eqn_beta_ode_simple_prop}}\},\\
{\cal S} &:= \{\mathbf{h}(x^*): x^* \in {\cal I} \mbox{ is a repeller or saddle point for the ODE \eqref{eqn_beta_ode_simple_prop}}\} \cup \{\mathbf{0}\}, \mbox{ and}
\end{align*}entire $\cD_I$ is the combined domain of attraction for  \eqref{eqn_ODE_prop}. \eop
\end{theorem}
The above result provides the limiting behaviour of a one-dimensional ODE with possibly discontinuous right hand sides, that typically arises while studying our type of application. \revg{The condition $g_\beta(0) \geq 0$ and $g_\beta(1) \leq 0$ ensures that the interval $[0,1]$ is positive invariant for the ODE \eqref{eqn_beta_ode_simple_prop}.}  
It is important to note that in Theorem \ref{thrm_attractors_beta}, we considered the function $g_\beta$ such that its zeroes could be either attractors or repellers for the ODE \eqref{eqn_beta_ode_simple}. The above result is an extension of the former as here the zeroes of the function $g_\beta$ can be either attractors or repellers or saddle points for the ODE \eqref{eqn_beta_ode_simple_prop}. Such an extended result is required for the application at hand, as we will see in the coming sections. 

\hide{\section{Proportion-dependent BP and controlling fake news}\label{sec_prop_BP}

Consider a two-type (say $x$ and $y$ types) population-dependent BP. Here, each individual of any type lives for an exponentially distributed time, $\tau$. Before dying, it produces a random number of non-negative population-dependent offsprings of both $x$ and $y$ types. Now, the classical models assume the existence of the limit population-independent mean matrix, $M^\infty$ for the population-dependent mean matrix observed during transience. We consider the variant of BP, where in the limit, the mean matrix is proportion-dependent, and call such a BP as \textit{proportion-dependent BP} (PrD-BP). The dynamics of such a process can be captured by \eqref{evolve_x_up_time_prop}.
    
    Observe that since $\lim_{t \to \infty} \beta(t)$ may not be known a priori, it is not easy to compute $M^\infty(\cdot)$ for PrD-BP. However, by virtue of Theorem \ref{thm1} and upcoming Theorem \ref{thrmProp}, the possible limits of $M^\infty$ can be derived using the attractors of the ODE \eqref{eqn_beta_ode_simple_prop} under 
 an additional assumption given below:
    
    
\begin{enumerate}[label=\textbf{A.\arabic*}, ref=\textbf{A.\arabic*}]
\setcounter{enumi}{4}
    \item  Assume $\minf_{ij}(\cdot)$, for $i, j \in \{x, y\}$ to be non-negative Lipschitz continuous on $[0, 1]$. Let $\cR_\beta$ be the finite set of repellers for the ODE \eqref{eqn_beta_ode_simple_prop}. Further, let $\cA_\beta := \{\bstar_1, \dots, \bstar_k  \in [0, 1]\} $ be the attractor for the ODE \eqref{eqn_beta_ode_simple_prop} for some $k < \infty$, with respective domain of attraction of each $\bstar_i$ as $\cD(\bstar_i)$ such that $\cup_{i=1}^k \cD(\bstar_i) = [0,1] - \cR_\beta$. \label{a5}
\end{enumerate}

\subsection{Analysis of ODE for PrD-BP}
We will now see that attractors/repellers of the ODE \eqref{eqn_beta_ode_simple_prop} once again translates to appropriate attractors/saddle points for the ODE \eqref{eqn_ODE_prop} as summarized below (see proof in Appendix \ref{appendix_B}):

\begin{theorem}\label{thrmProp}
Assume \ref{a5}. Then, the ODE \eqref{eqn_ODE_prop} satisfies \ref{a3_prop}. Further, the following statements are true for the ODE \eqref{eqn_ODE_prop}:
\begin{enumerate}[label=(\roman*)]
    \item the attractor set, $\cA =  \{\mathbf{h}(\bstar) : \bstar \in \cA_\beta\}$, 
    \item the repeller set, $\cR = \{\mathbf{0}\}\cup \{\mathbf{h}(\bstar_r) : \bstar_r \in \cR_\beta\}$, and
    \item the combined DoA of $\cA$ and $\cR$, i.e., $\cD = \cD_I$. \eop
\end{enumerate} 
\end{theorem}

\subsection{Analysis of random trajectory of PrD-BP}
By Theorem \ref{thm1}, we have: 

\begin{corollary}\label{corollary_prop}
Consider the dynamics as in \eqref{evolve_x_up_time_prop}, and assume \ref{a1_prop}-\ref{a2_prop}, \ref{a5}. Then, following statements are true:
\begin{enumerate}[label=(\roman*)]
    \item the assumption \ref{a3_prop} holds for the ODE \eqref{eqn_ODE_prop}, and hence Theorem \ref{thm1}(i) is true,
    \item $\Ups_n$ either converges to $\cA \cup \cR$, as $n \to \infty$ or hovers around $\cR$  w.p. $1$, where $\cA$ and $\cR$ are given in Theorem \ref{thrmProp}.  \eop
\end{enumerate}
\end{corollary}
Clearly, Corollary \ref{corollary_prop}(i) follows by standard results of ODE (e.g., \cite{piccini1984ordinary}) under \ref{a5}, and rest can be proved exactly as in \cite[Corollary 1]{agarwal2021new}.
We now discuss the implications of this Corollary.

$\bullet$ Let $M^\infty(\beta) \equiv M^\infty$, where all entries of constant matrix, $M^\infty$ are strictly positive, i.e., we have an irreducible mean matrix at limit which does not depend on population sizes/proportions. Then, by simple algebra (concavity/convexity of $\ga_\beta(\cdot)$), there exists a unique zero, $\bstar \in (0, 1)$, which is the only attractor for the ODE \eqref{eqn_beta_ode_simple_prop} with DoA as $[0,1]$. 
Thus, $\Ups_n \to \{\mathbf{h}(\bstar), \mathbf{0}\}$, and therefore, the two populations co-survive in proportion equals to $\bstar$, if at all. This particular feature is analyzed already in the literature, except that we allow anything to occur during transience, along with total-population dependency (see Example $3$ in Section \ref{} where the process progresses as in BPA in the initial transition epochs, and then as PrD-BP).

$\bullet$ Now, consider $M^\infty(\beta) \equiv M^\infty$, where all entries of $M^\infty$ are strictly positive except say $\minf_{yx} = 0$, i.e., consider a decomposable proportion-independent process in the limit. Here, if $\minf_{xx} > \minf_{yy}$, the attractor for the ODE \eqref{eqn_beta_ode_simple} is uniquely given by:
\begin{align}\label{eqn_beta_decomposable}
    \bstar = \frac{\minf_{xx} - \minf_{yy}}{\minf_{xx} - \minf_{yy} + \minf_{xy} } \in (0,1).
\end{align}
 with DoA as $(0,1]$; thus, $1$ is not an attractor; further, $\bstar_r = 0$ is a repeller for the ODE \eqref{eqn_ODE}. In such a case, we have four kinds of sample paths: (i) both populations get extinct, (ii) only $x$-type gets extinct ($\Bc(t) \to 0$, in spite of $\bstar_r = 0$ being a repeller for the ODE), (iii) $x$-type population survives in small numbers (due to hovering around $\mathbf{h}(0)$), and (iv) the populations co-survive in proportion equals to $\bstar$. This result matches with \cite{jagers1969proportions} when one considers the Markovian variant of their age-dependent population-independent model. 
 
 Otherwise if $\minf_{xx} \leq \minf_{yy}$, $\bstar = 0$ is the only attractor for the ODE with domain of attraction as $[0,1]$, which implies a.s. extinction of $x$-type population. This result matches with \cite[Theorem 1(ii)]{ranbir2019decomposable}, \cite{kesten1967limit} which considers population-independent decomposable BP. 

$\bullet$ Lastly, if $M^\infty(\beta)$ is proportion dependent, then the set of limit proportions is $\cA_\beta \cup \cR_\beta$ (see \ref{a5}). This result is true even if the population dependent mean offspring functions, $m_{ij}(\om)$ depend only on $\beta$ (in the transience also). We discuss one example of the latter kind next directly while discussing the application.}

\section{Modelling of warning dynamics using TC-BP with multiple deaths}\label{sec_warning_old}
We begin this section by demonstrating how the warning dynamics can be modelled using TC-BPs with multiple deaths discussed in the previous section. Towards this, we model the copies with fake and real tags as the $x$ and $y$-type populations respectively. The time instance when a user views, tags and shares the post corresponds to the time of death of an individual in the BP. \revg{As seen in Section \ref{sec_prob_desc}, in \eqref{eqn_tag_wi}-\eqref{eqn_final_shares}, the distribution of shares, types of shares, etc., depends on the type-$d$ of the user that reads the post with $d \in {\cal B}$. Thus, one can correspond each $d$-type user to a $d$-death because of the following details.  When a $d$-type user reads and shares the post, the said post becomes a read copy, resulting in a $d$-death. Further, clearly $D_x = D_y = {\cal B}$. At any given time, the proportions of the users of any type are given by $\mu_0, \mu_1, \mu_2$ and $\mu_a$, which also correspondingly represent the proportions of unread copies with np, wi, ws and a-users. Thus, one can easily infer that a type-$d$ user reads the post first among the existing unread copies, or in other words, $d$-type death occurs first with probability $\mu_d/(\mu_0+ \mu_1+\mu_2+\mu_a) = \mu_d$. 
Therefore, one can set the parameter of $d$-death as:}
\begin{align}\label{eqn_lambda}
    \lambda_{z,d} (\phi) := \mu_d  \mbox{ for all } \om, \mbox{ for each }  d \in {\cal B} \mbox{ and } z \in \{x, y\}.
\end{align}
Now, after viewing the post, if a ws-user with fake tagged copy shares the post with fake tag, then we say that the number of shares, $\Gamma_{xx, ws}$, corresponds to the number of $x$-type offspring produced by an $x$-type parent, when ws-death occurs. In general, the number of shares with fake and real tag correspond to offspring of $x$ and $y$-type respectively, see \eqref{eqn_final_shares}; the number of shares (offspring) also depend upon $\Om(\cdot)$.

\revg{The underlying TC-BP with multiple death-types that models the warning dynamics \eqref{eqn_final_shares}  is exactly like the well-known irreducible BP, except for the inclusion of multiple death-types (see 
\cite{athreya2004branching}). In irreducible BPs, the extinction occurs only when both the population-types die; individual extinction of a population-type is not possible. The same is the case with our model. For example, suppose there are no unread copies with fake tag, i.e., $\Cx(t) = 0$ at some time $t > 0$, while the system still has real tagged unread copies ($\Cy(t) > 0$). Then, if at some time $t' > t$, a wi-user or ws-user reads and shares the post, then, with non-zero probability, it can tag the post as fake (see \eqref{eqn_tag_wi}, \eqref{eqn_tag_ws}). If so happens, then it will lead to new unread copies with fake tag, i.e., $\Cx(t') > 0$. Thus, \textit{the number of fake tagged copies can be regenerated even after they are not present on the OSN, as long as there are some unread copies of the post on the OSN}.}

Next, we provide the general framework for analyzing the warning dynamics with respect to any warning mechanism ($\omega$). \revg{Observe that when any real/fake post gets extinct, then it's effect is harmless.} Therefore, our focus shall only be on the non-extinction paths.

\subsection{Analysis of warning dynamics for general WM}
\revg{Consider a general warning mechanism defined using a continuous-function $\omega : [0,1] \to \mathbb{R}^+$ which depends only on the proportion of fake tags $\beta$.} Further, consider any post with actuality, $u \in \{R, F\}$. Then, for each $u$, it is clear from the previous section that the analysis of the TC-BP with multiple-death types, and hence the warning dynamics, is driven by the limit mean matrix (see \eqref{eqn_mean_matrix} and \ref{a2_prop}). Thus, we first construct the limit mean matrix, $M^{\infty, u}(\beta) := [m_{ij}^{\infty, u}(\beta)]_{\{i, j \in\{x,y\}\}}$, as follows
(see \eqref{eqn_tag_wi}-\eqref{eqn_final_shares}):

{\footnotesize
\begin{align}\label{eqn_mean_matrix_red_coloring}
    M^{\infty, u}(\beta) = 
    \begin{bmatrix}
    \bigg(\mu_1\rho \alpha_x^u + \mu_2 \min\{\omega(\beta) \alpha_x^u, 1\}\bigg) m_f\eta^u  & \bigg(\mu_1(1-\alpha_x^u \rho) + \mu_2(1-\min\{\omega(\beta) \alpha_x^u, 1\}) \bigg) m_f\eta^u + \mu_a m_f \eta_a\\
    & \\ 
    \bigg(\mu_1\rho \alpha_y^u + \mu_2 \min\{\omega(\beta) \alpha_y^u, 1\}\bigg) m_f\eta^u & \bigg(\mu_1(1-\alpha_y^u \rho) + \mu_2(1-\min\{\omega(\beta) \alpha_y^u, 1\})\bigg) m_f\eta^u + \mu_a m_f \eta_a
    \end{bmatrix}.
\end{align}}
Next, we will identify the attractor, repeller and saddle sets for the ODE \eqref{eqn_beta_ode_simple_prop} which will lead to the limits for the stochastic trajectory corresponding to the warning dynamics by using Theorem \ref{thrm_beta_ODE_prop} and Theorem \ref{thm1}. Towards this, observe that $d(\phi) = d^\infty(\beta) = 1$, as $\sum_{d } \lambda_{z,d} (\phi) = 1$ for any $\om$ and $z \in \{x, y\}$. This implies, $f^\infty_\beta (\beta) = \beta$ (see \eqref{eqn_f_infty}). Thus, by \eqref{eqn_beta_ODE_prop}, the function $g^u_\beta$ and the corresponding ODE \eqref{eqn_beta_ode_simple_prop} for the warning dynamics for both types of posts, $u \in \{R, F\}$, is given by:
\begin{align}\label{eqn_general_g_beta}
\dot{\beta^u} &= g_\beta^u(\beta) \mbox{ where }\\
    g^u_\beta(\beta) 
    :=  \bigg(-\beta \mu_2 - \beta \mu_1 (1-\alpha_x^u \rho) &+ (1-\beta) \mu_1 \rho  \alpha_y^u \nonumber \\
    + \mu_2 \bigg(\beta \min\{\omega(\beta) &\alpha_x^u, 1\} + (1-\beta)\min\{\omega(\beta) \alpha_y^u, 1\}\bigg) \bigg) m_f\eta^u - \beta \mu_a m_f \eta_a. \nonumber
\end{align}
Define $\cA^u_\beta$ as the set of attractors in $[0,1]$ and $\cR^u_\beta$ as the combined set of repellers and saddle points in $[0,1]$ for the above ODE. Then, we have the following result: 
\begin{theorem}\label{thrm_BP_to_fake}
Consider the warning dynamics as in \eqref{eqn_transition_fake_tag} and \eqref{eqn_transition_real_tag}. Let the distribution of number of friends, ${\cal F} \geq 0$ be such that $E[{\cal F}] \eta_R > 1$ and $E[{\cal F}^2] < \infty$. Then, the following statements are true for each~$u$, the actuality of post:
\begin{enumerate}[label=(\roman*)]
    \item the assumptions \ref{a2_prop} and \ref{a3_prop} hold for the ODE \eqref{eqn_ODE_prop}; Theorem \ref{thm1}(i) is true,
    \item the set $\cA^u_\beta \neq \emptyset$, and then $\Ups_n$ converges to $\cA^u \cup \cR^u$, as $n \to \infty$ or hovers around $\cR^u$ w.p. $1$, where $\cA^u = \{\mathbf{h}(\beta): \beta \in \cA^u_\beta\}$ and $\cR^u = \{\mathbf{0}\} \cup \{\mathbf{h}(\beta): \beta \in \cR^u_\beta\}$. 
    \item Further, any potential limit proportion corresponding to the warning dynamics, i.e., $\beta \in \cA^u_\beta \cup \cR^u_\beta$, can be bounded as below:
    \begin{align}\label{eqn_beta_bar}
        0 < \frac{\mu_1 \rho \alpha_y^u \eta^u }{q^u} =: \underline{\beta}^u &< \beta \leq \overline{\beta}^u := \frac{(\mu_2 + \mu_1 \rho \alpha_y^u) \eta^u}{q^u} \leq 1, \mbox{ where} 
    \end{align}
    the constant $q^u := \bigg( \mu_2 + \mu_1 (1 - (\alpha_x^u - \alpha_y^u) \rho)\bigg) \eta^u + \mu_a \eta_a$. \eop
\end{enumerate}
\end{theorem}
At first, observe that any warning mechanism $\omega$ only affects the likelihood of tagging the post as real or fake by a ws-user (see \eqref{eqn_tag_ws}). \revg{It does not affect the probability of a post getting viral or extinct as extinction depends on the sum current number of unread copies (i.e., sum current population in the BP)}. Now, given that our interest lies in non-extinction paths, the above Theorem gives a generalised result which holds for any warning dynamics. It is important to note that \revg{viral paths are possible only when the probability of non-extinction is non-zero; this is possible if $E[{\cal F}] \eta_R > 1$, as then the TC-BP with multiple deaths can be in throughout super-critical regime  (see Lemma \ref{lemma_dichotomy}).}

Theorem \ref{thrm_BP_to_fake}(i) implies that the warning dynamics can be approximated by the solution of the ODE \eqref{eqn_general_g_beta} over any finite-time window, where the limit mean functions are given by \eqref{eqn_mean_matrix_red_coloring}. \revg{The more important result for our context is the second part of the Theorem which} states that the stochastic trajectory $\Ups_n$ either converges to $\cA^u \cup \cR^u$ or hovers around $\cR^u$. The set \revg{$\cR^u$} contains $\mathbf{0}$ which represents the limiting behavior of the stochastic trajectory in the extinction paths. \textit{Thus, all the results henceforth will focus on deriving the limits \revg{which are not equal to $\mathbf{0}$, which in turn provide the limit proportion of fake tags for the warning dynamics in non-extinction paths.}}

\revr{Further, Theorem \ref{thrm_BP_to_fake} provides the above limits using the zeroes $\{\beta^{\infty, u}\}$ of $g^u_\beta$ (see \eqref{eqn_general_g_beta}). Now, observe that the function $g^u_\beta$ and therefore the zeroes $\{\beta^{\infty, u}\}$ depends on $\chi$, where $\chi := \{\mu_1, \mu_2, \mu_a, b, w\}$ is the set of parameters.
For some warning mechanisms, the function $g^u_\beta$ can have multiple zeroes, and the warning dynamics can converge to one of them. Thus, one would want to ensure the  maximum limit proportion of fake tags for the real post is within a given limit and optimise the minimum proportion of fake tags for the fake post. This aspect is considered in the optimization problem \eqref{eqn_opt_prob} given in the next section. 
In this context, we define the following Quality of Service (QoS) for any WM:
\begin{align}\label{eqn_qos}
Q := \inf\{\beta: {\beta}\in \cA^{F}_\beta \cup \cR^{F}_\beta \}.
\end{align}
Observe here that $Q = \inf({\cal L}^F)$, the objective function of \eqref{eqn_gen_opt} and is the almost sure lower bound on the limit proportion of fake tags when the underlying post is fake. It measures the minimal extent to which a fake post is identified by the users. From \eqref{eqn_beta_bar} of Theorem \ref{thrm_BP_to_fake}, $Q \in (\underline{\beta}^F, \overline{\beta}^F]$. We would see in the coming sections how (optimal) $Q$ varies with different warning mechanisms.}

Next, in Theorem \ref{thrm_unique_att}, we will derive some properties of $\{\beta^{\infty, u}\}$ with respect to each parameter in $\chi$\revg{, when $g_\beta^u$ has a unique zero}. This result will be instrumental in deriving important results in the coming sections. To keep it simple, we shall write $\beta^{\infty, u}(\kappa)$ and $g^u_\beta(\beta; \kappa)$ to show the dependency on the required parameter $\kappa$, an element of the tuple $\chi$. Towards this, we require the following difference term (note that $g^u_\beta(\beta^\infty(\kappa); \kappa) = 0$):
\begin{align}\label{eqn_nabla}
    \nabla^u(\kappa, \partial \kappa) := g^u_\beta(\beta^\infty(\kappa); \kappa+\partial \kappa)  -  g^u_\beta(\beta^\infty(\kappa); \kappa) = g^u_\beta(\beta^\infty(\kappa); \kappa+\partial \kappa).
\end{align}
\begin{theorem} \label{thrm_unique_att} Consider any warning mechanism, $\omega(\beta)$.
Let $\kappa$ be any parameter. Let $g^u_\beta(\beta; \kappa)$ be either a convex or concave function of $\beta$ with a unique zero, $\beta^{\infty, u}(\kappa) \in (0,1)$, for each $u \in \{R, F\}$. Keeping all parameters in $\chi$ fixed, other than $\kappa$, if the difference term $\nabla^u(\kappa, \partial \kappa)   > 0$ for some $\partial \kappa > 0$, then $\beta^{\infty, u}(\kappa + \partial \kappa) > \beta^{\infty, u}(\kappa)$. Else if $\nabla^u(\kappa, \delta\kappa) < 0$, then $\beta^{\infty, u}(\kappa + \partial \kappa) < \beta^{\infty, u}(\kappa)$. Else, $\beta^{\infty, u}(\kappa + \partial \kappa) = \beta^{\infty, u}(\kappa)$.
\eop
\end{theorem} 

\vspace{-2.2cm}
Hereon, we will analyse the warning dynamics for some specific mechanisms. 

%
\section{Analysis of Extended Original WM (eo-WM)}
This section will analyze the warning dynamics when the OSN provides the warning as in \eqref{eqn_warning}, initially proposed in \cite{kapsikar2020controlling}. Recall that in \cite{kapsikar2020controlling}, the system has only ws-users who interact with the warning mechanism. Since we study the original mechanism \eqref{eqn_warning} under the influence of a variety of user behaviour, we refer to $\omega$ as \underline{extended original warning mec} \underline{hanism} (eo-WM) in our context.

Consider any post with actuality $u \in \{R, F\}$. Recall that we have $w \leq \overline{w} := \frac{1}{\alpha_x^F} - \gamma$, thus leading to $\alpha_j^u \omega(\beta) < 1$ for each $j \in \{x, y\}$ and for any $\beta \in [0,1]$. We begin the analysis by analyzing the ODE \eqref{eqn_general_g_beta} for the eo-WM. The $g_\beta^u$ defined in \eqref{eqn_general_g_beta} for the eo-WM, henceforth denoted as $g^{o, u}_\beta$, is as given below:

\vspace{-0.5cm}
\begin{align}\label{eqn_beta_ODE_etac1}
\begin{aligned}
    g^{o, u}_\beta(\beta) &= -\beta \mu_2 m_f\eta^u - \beta \mu_1 (1-\alpha_x^u \rho) m_f\eta^u + (1-\beta) \mu_1 \rho  \alpha_y^u m_f\eta^u \\
    &\hspace{1cm}+ \mu_2 \omega(\beta) \bigg(\beta \alpha_x^u + (1-\beta)\alpha_y^u\bigg) m_f\eta^u - \beta \mu_a m_f \eta_a.
\end{aligned}
\end{align}
Let $\cA^{o, u}_\beta \subset [0,1]$ be the corresponding attractor set and $\cR^{o, u}_\beta \subset [0,1]$ be the union of the corresponding repeller and saddle sets, i.e., with respect to ODE $\dot{\beta}^u = g^{o, u}_\beta(\beta)$. We study these sets in the following.
\begin{corollary}\label{corollary_ex_wm}
  There exists a unique zero, $\beta^{o, \infty, u}$, of $g^{o, u}_\beta$ in $[0,1]$. Further, $\beta^{o, \infty, u} \in (0,1)$, $\cA^{o, u}_\beta = \left\{\beta^{o, \infty, u}\right\}$ and $\cR^{o,u}_\beta = \emptyset$.  \eop
\end{corollary}
\revg{Thus, there is a unique attractor, $\beta^{o, \infty, u}$, of ODE \eqref{eqn_general_g_beta}. By Theorem \ref{thrm_BP_to_fake}, the stochastic trajectory $\Ups_n$ under eo-WM either converges to $\mathbf{h}(\beta^{o, \infty, u})$ or $\mathbf{0}$, or hovers around $\mathbf{0}$ almost surely. We re-iterate that our focus is on the non-extinction paths, and thus, the relevant proportion of fake tags is unique and equals $\beta^{o, \infty, u}$. Further, by Theorem \ref{thrm_BP_to_fake}, for the given choice of $w, b$ and given $\mu_1, \mu_2, \mu_a \in \chi$, $\beta^{o, \infty, u} \in (\underline{\beta}^u, \overline{\beta}^u]$.}

We now consider the following robust optimization problem for the OSN discussed before:
\begin{align}\label{eqn_opt_prob}
\begin{aligned}
    \sup_{w \in \left[0,  \overline{w}\right], b \in [0, \infty)} \beta^{o, \infty, F}(w, b)
    \mbox{ subject to } \beta^{o, \infty, R}(w, b) \leq \delta, \mbox{ for some } \delta \in (0,1).
\end{aligned}
\end{align}
\revr{By uniqueness of the attractors in the non-extinction paths, the above constrained optimization problem optimizes the QoS defined in \eqref{eqn_qos}, $Q = \beta^{o, \infty, F}$ under eo-WM by choosing $w, b$, while ensuring that the unique zero for the real post, $\beta^{o, \infty, R} \leq \delta$.}
The problem in \eqref{eqn_opt_prob} is exactly the same as in \cite{kapsikar2020controlling}, but for the inclusion of different user behaviour in our model. Thus, we need to extend the solution of \cite{kapsikar2020controlling} to the case that includes wi, ws, a, and np-users. Observe that $\delta$ is a design parameter for the OSN.

Before we solve the above problem, we observe the following qualitative behaviour which is true by the virtue of Theorem \ref{thrm_unique_att} -- this behavior is important for further analysis:
\begin{corollary}\label{cor_limits_warning}
For each $u \in \{R, F\}$, the limit $\beta^{o, \infty, u}(w, b)$ strictly increases with $w$ and strictly decreases with $b$. 
\eop
\end{corollary}
The above Corollary intuitively indicates to choose the largest $w$, i.e., $\overline{w}$ and the smallest $b$, i.e., $0$. However, since the optimal $w, b$ needs to satisfy the constraint for the real post as in \eqref{eqn_opt_prob}, therefore, formal analysis is required.

\begin{theorem}{\bf [Optimal eo-warning design]} \label{thrm_opt} The following statements hold for the optimization problem \eqref{eqn_opt_prob}:
\begin{enumerate}[label=(\roman*)]
    \item If $\beta^{o, \infty, R}(\overline{w}, 0) > \delta$, then 
    the optimizer $(w^*, b^*)$ of \eqref{eqn_opt_prob} is as below and satisfies $\beta^{o, \infty, R}(w^*, b^*) = \delta$:

    \vspace{-4mm}
    {\footnotesize
    \begin{align}\label{eqn_optimal_parameter_old}
    \begin{aligned}
        w^* = \overline{w} \mbox{ and } b^* 
        &= \left(\frac{\delta}{1-\delta}\right)\left(\frac{w^*\eta^R \mu_2(\delta \alpha_x^R + (1-\delta)\alpha_y^R)}{\delta ((\mu_1+\mu_2)\eta^R + \mu_a  \eta_a) - \eta^R (\mu_1 \rho + \mu_2 \gamma) (\delta \alpha_x^R + (1-\delta) \alpha_y^R)} - 1 \right).
    \end{aligned}
    \end{align}}
    \item Else, if $\beta^{o, \infty, R}(\overline{w}, 0) \leq \delta$, then $(w^*, b^*) = (\overline{w}, 0)$ and satisfies $\beta^{o, \infty, R}(w^*, b^*) \leq \delta$.  \eop
    \end{enumerate}
\end{theorem}
Thus, as anticipated, $w^* = \overline{w}$. Interestingly, contrary to the expectation, $b^*$ is not always $0$. If $\beta^{o, \infty, R}(\overline{w}, 0) > \delta$, then the optimal choice for $b$ is given by $b^* > 0$. Such $b^*$ is achieved by solving for $\beta^{o, \infty, R}(w^*, b) = \delta$, i.e., relaxing the constraint for the real post to the maximum $\delta$-level in a bid to achieve the maximum $\beta^{o, \infty, F}$ for fake post at optimality. In view of Corollary \ref{cor_limits_warning}, it is then easy to see that, $\beta^{o, \infty, F}(w^*, b^*) < \beta^{o, \infty, F}(w^*, 0)$, when $b^* > 0$. For simpler notations, henceforth we will refer to $\beta^{o, \infty, F}(w^*, b^*)$ as $\beta^o$ and $\beta^{o, \infty, R}(w^*, b^*)$ as $\beta^{o, R}$. 

In \cite{kapsikar2020controlling}, the optimization problem \eqref{eqn_opt_prob} is solved partially. Firstly, only the case with the hypothesis of Theorem \ref{thrm_opt}(i) is analyzed. It is shown that the optimal value is achieved for $b$, which satisfies $\beta^{o, \infty, R} = \delta$. However, the optimal choice of $w$ is not derived; rather a projected gradient descent algorithm is suggested to evaluate $w^*$. Furthermore, \cite{kapsikar2020controlling} considers $w \in [0,1]$, while one can allow $w$ to be as large as $\overline{w}$, which can be larger than $1$. As we have proved that $w^* = \overline{w}$, therefore, \textit{our optimal eo-WM should perform better than the optimal WM designed in \cite{kapsikar2020controlling}}. We  show this numerically in the next sub-section. 

\subsection{QoS under eo-WM}
\revg{It is clear from Corollary \ref{corollary_ex_wm} and Theorem \ref{thrm_opt} that the QoS under optimal eo-WM, say $Q^o$ equals $\beta^o$. Now, fix any configuration, }
$$
{\cal C} := \bigg\{\{\alpha_i^u\}_{\{i \in \{x, y\}, u \in \{F, R\}\}}, \eta_a, \{\eta^u\}_{\{u \in \{F, R\}\}}, \rho, \gamma, m_f, w, \mu_1, \mu_2 \bigg\},
$$\revg{and let $\mu_a$ vary. Then, we want to investigate how $Q^o$ changes with $\mu_a$. Towards this, define:
\begin{align}
\betana := \beta^o(\mu_a = 0) = Q^o(\mu_a = 0),
\end{align}as the proportion of fake tags for the fake post at optimality when there is no adversary. Recall that a-users deliberately tag any post as real. Therefore, 
one can anticipate that the OSN achieves the maximum QoS when there is no adversary, i.e., $\beta^o(\mu_a) = Q^o(\mu_a) < \betana$, when $\mu_a > 0$. We  prove this precisely in the next result for an appropriate range of $\delta$.}
\begin{corollary}\label{cor_beta_o_na}
    For given configuration ${\cal C}$, there exists a $\overline{\delta} > 0$ such that $\beta^o(\mu_a) < \betana$ for all $\delta \leq \overline{\delta}$, for each $\mu_a \in (0, 1-\mu_1-\mu_2]$. \eop
\end{corollary}
\revg{Thus, the above corollary confirms our anticipation that the performance degrades with introduction of the a-users in the system, however for a smaller range of $\delta$; observe that the OSN is interested in keeping $\delta$ as small as possible, therefore, such choices of $\delta$ are indeed meaningful. \textit{Henceforth, we consider such $\delta$, i.e., $\delta \leq \overline{\delta}$}. }
 In the next subsection, we will validate this result numerically and reinforce the requirement to design better WMs in the presence of adversaries.

\subsection{Numerical analysis for eo-WM}
\revg{At first, we would like to compare eo-WM with the mechanism in \cite{kapsikar2020controlling} with just a-users added -- in the first example, any user on the OSN can either be a ws-user or an a-user ($\mu_2 + \mu_a = 1$). Thus, there is no wi-user and everyone participates.} Further, let the parameters be as in \cite{kapsikar2020controlling}:
\begin{align}\label{eqn_param_plot1}
\begin{aligned}
    m_f &= 28, \eta^F = 0.08, \eta^R = 0.05, \gamma = 0.1, \eta_a = 0.55, \delta = 0.02,\\
    \alpha_x^F &= 0.85, \alpha_y^F = 0.6375, \alpha_x^R = 0.3 \mbox{ and } \alpha_y^R = 0.09.
\end{aligned}
\end{align}
For such parameters, we perform Monte-Carlo  (MC) simulation, and also evaluate the zeroes of $g_\beta^{o, u}$ for each $u \in \{R, F\}$. In Figure \ref{fig:exWM}, we plot the outputs of MC simulations and the theoretical limits against $\mu_a$, which can be seen to be close to each other. The constraint for the real post is satisfied. In fact, the proportion of tags (for both fake and real posts) decreases with $\mu_a$, which is intuitive as a-users deliberately real tag the posts.

\begin{figure}[http]
\centering
\includegraphics[trim = {1.6cm 6.5cm 0cm 6.5cm}, clip, scale = 0.3]{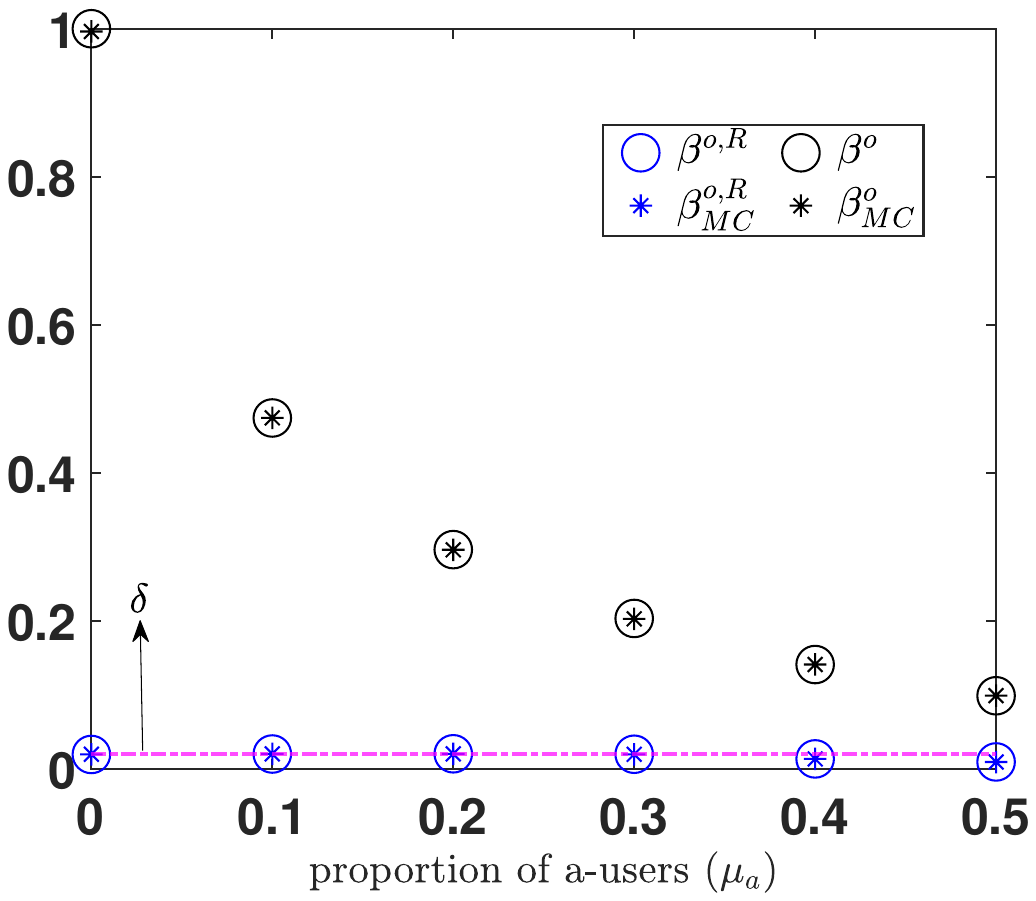}
  \caption{Limits of warning dynamics under eo-WM}
  \label{fig:exWM}
\end{figure}
Under the optimal eo-WM, $99.981\%$ of users can identify the fake post as fake; this optimal value is higher than the reported $90\%$ in \cite{kapsikar2020controlling}, as we use $w^* = 1.076$, while algorithm  in \cite{kapsikar2020controlling}  uses $w^* = 1$. 
Now, it is interesting to note that with just $1\%$ and $2\%$ of a-users on the OSN, the performance of the eo-WM decreases to $89.798\%$ and $81.74\%$ respectively (\revr{in fact, there is degradation with respect to the new QoS defined in \eqref{eqn_iqos} which focuses only on non-adversarial users; $99.981\%$ decreases to $95.8\%$ and $92.53\%$ respectively with $1\%$ and $2\%$ of a-users}). This depicts that the original WM is not sufficient to control the fake post propagation in the presence of adversaries.  

\revr{Next, we consider a second example with parameters almost as in \eqref{eqn_param_plot1}, but with 
the proportion of ws-users ($\mu_2$) fixed and with $\mu_a$, the proportion of a-users varying. We set $\mu_2 = 0.5$, 
 $\mu_1 = 0$ and let the fraction of non-participants equal $0.5 - \mu_a$.  
For ease of reference, the users of this example are referred to as `\underline{smart users}', as here $\alpha_x^F - \alpha_x^R = 0.55$ and $\alpha_y^F - \alpha_y^R = 0.5475$ indicating that the users are capable of distinguishing the fake posts from real posts to a reasonable extent, even without external aid and irrespective of sender tag. 
\begin{figure}[http]
\begin{minipage}{\textwidth}
\centering
\begin{minipage}{.5\textwidth}
  \centering
  \includegraphics[trim = {1cm 6cm 0cm 6cm}, clip, scale = 0.3]{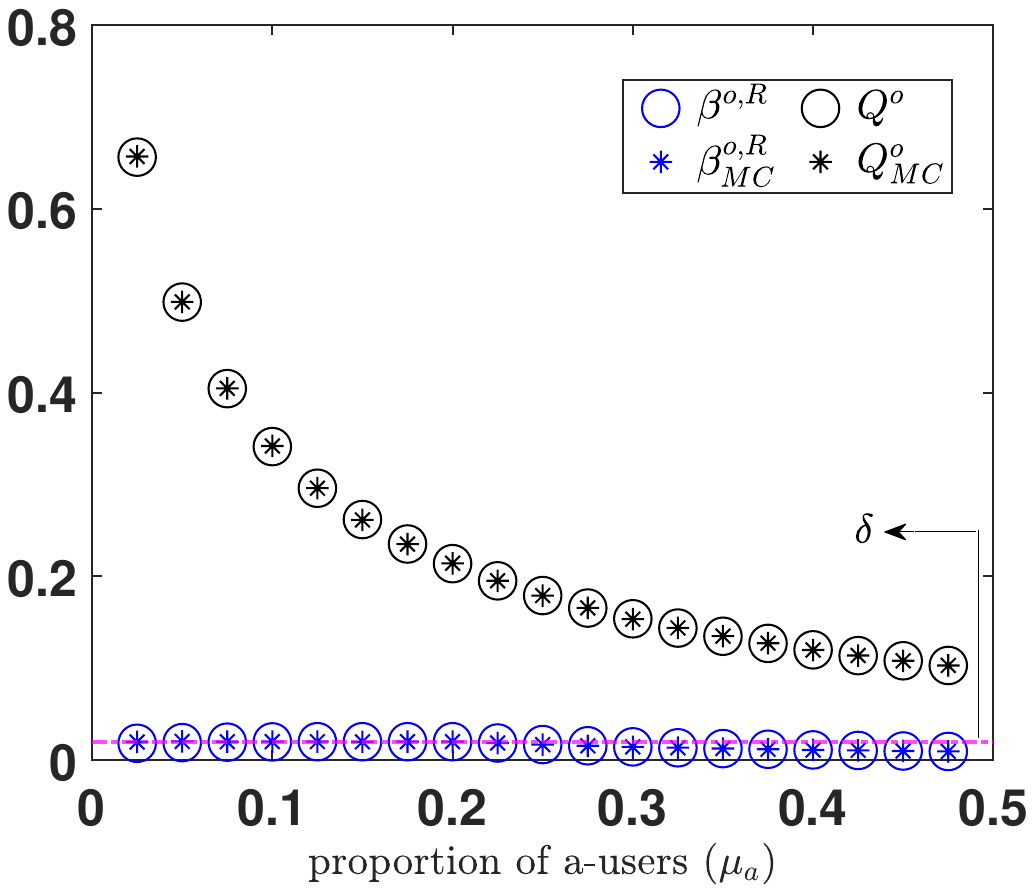}
\end{minipage}%
\begin{minipage}{.5\textwidth}
  \centering
  \includegraphics[trim = {1cm 6cm 0cm 6cm}, clip, scale = 0.3]{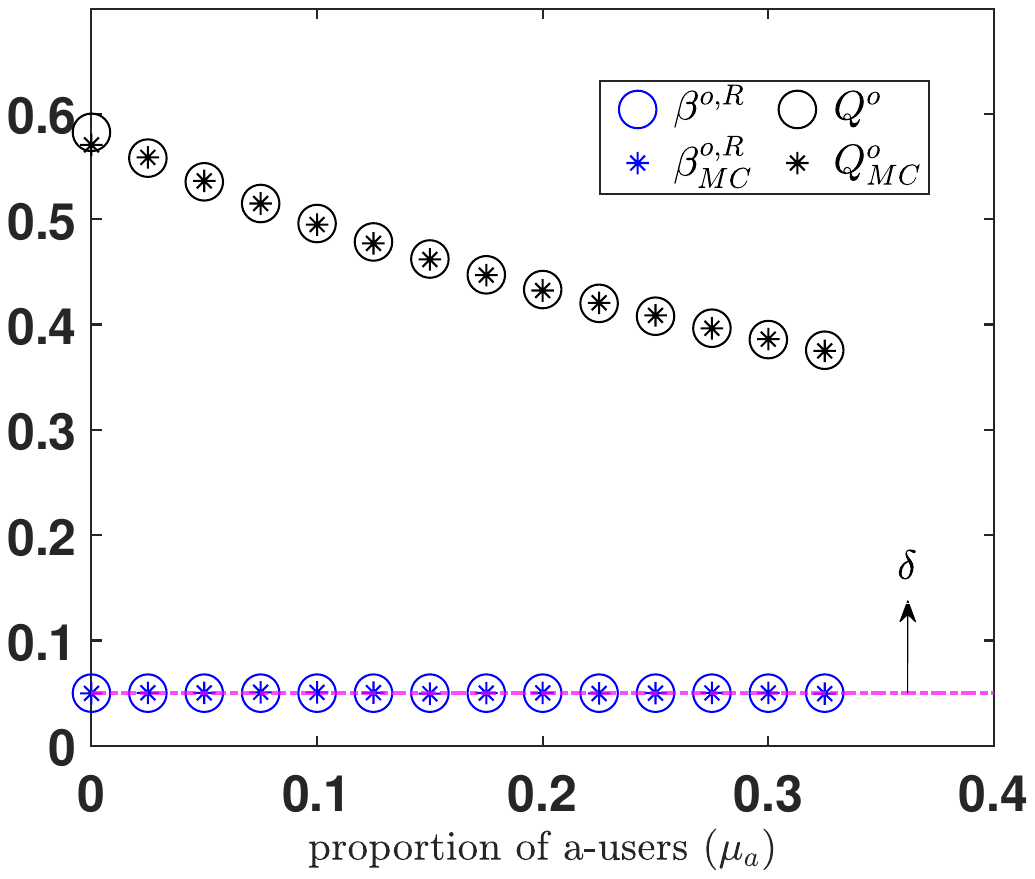}
\end{minipage}
\caption{Limits of warning dynamics under eo-WM with smart (left) and naive (right) users respectively}
\label{fig:eoWM_delta}
\end{minipage}
\begin{minipage}{\textwidth}
\centering
\begin{minipage}{.5\textwidth}
  \centering
  \includegraphics[trim = {1cm 6cm 0cm 6cm}, clip, scale = 0.3]{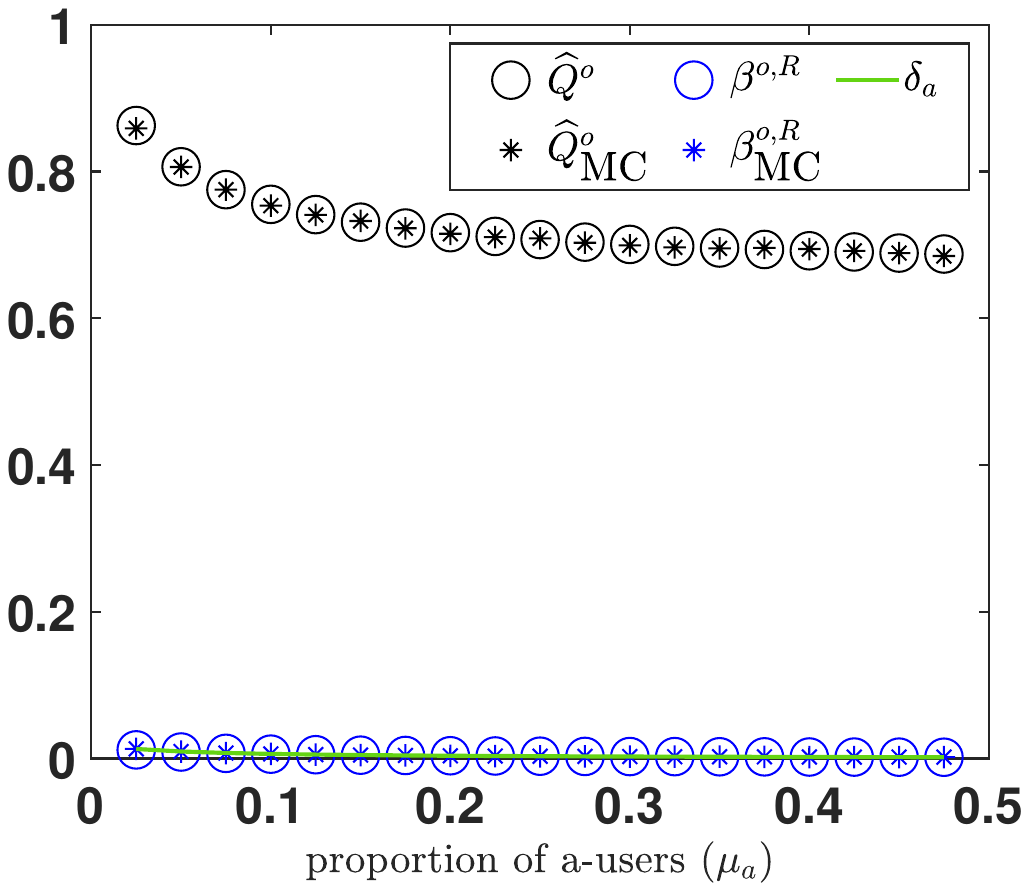}
\end{minipage}%
\begin{minipage}{.5\textwidth}
  \centering
  \includegraphics[trim = {1cm 6cm 0cm 6cm}, clip, scale = 0.3]{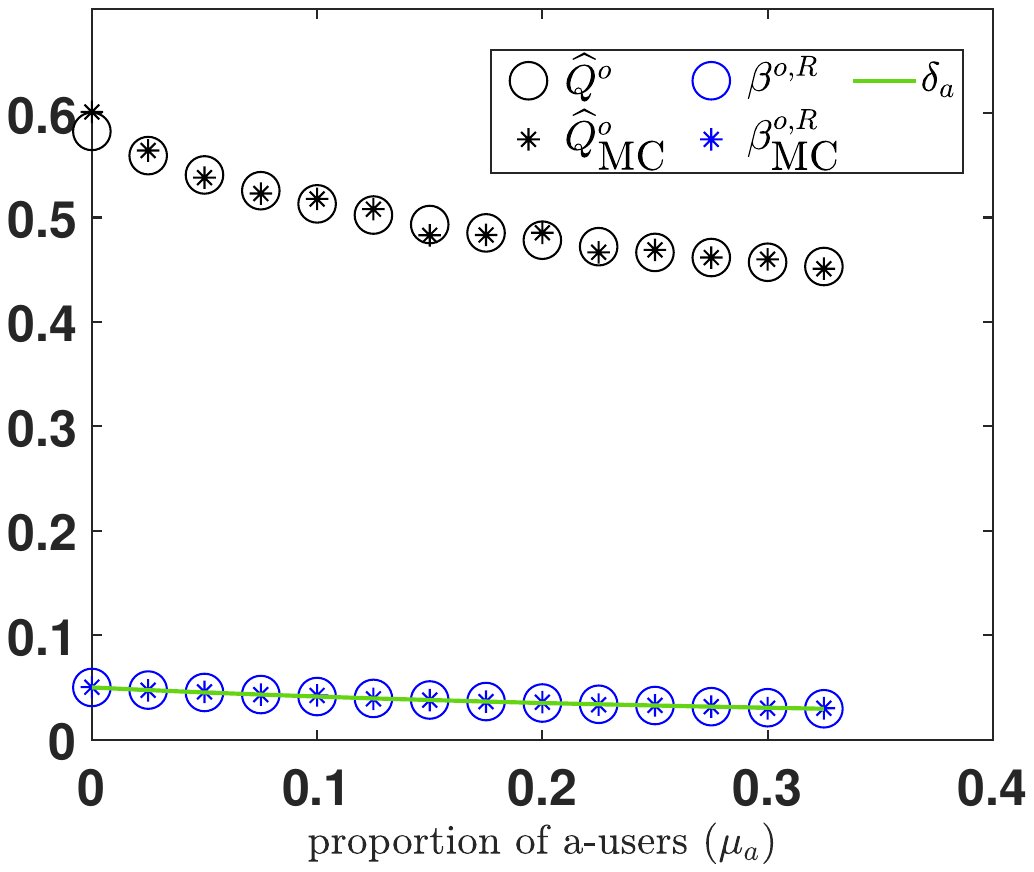}
\end{minipage}
\caption{i-QoS under eo-WM with smart (left) and naive (right) users respectively}
\label{fig:QOS_eo}
\end{minipage}
\end{figure}

We compare smart users with users in another example scenario where $\alpha_x^F - \alpha_x^R = 0.18$ and $\alpha_y^F - \alpha_y^R = 0.135$. As the differences between the distinguishing parameters are small, these users are referred to as `\underline{naive users}'. For this example, the remaining parameters are fixed as below (for diversity, we also consider more attractive posts):
\begin{align}\label{eqn_param_plot3}
\begin{aligned}
    \rho &= 0.9, m_f = 30, \eta^F = 0.52, \eta^R = 0.4, \gamma = 0.1, \eta_a = 0.55, \delta = 0.05,\\
    \alpha_x^F &= 0.3, \alpha_y^F = 0.225, \alpha_x^R = 0.12, \alpha_y^R = 0.09, \mu_1 = 0.15 \mbox{ and } \mu_2 = 0.5.
\end{aligned}
\end{align}Typically, the users may be naive -- may not possess sufficient intrinsic ability to distinguish between the posts to the level that smart users can. Interestingly, as seen below, the proposed mechanism effectively guides even naive users.

In Figure \ref{fig:eoWM_delta}, we illustrate the QoS ($Q$ defined in \ref{eqn_qos}) and the proportion of fake tags for the real post for examples with smart and naive users in left and right sub-figures, respectively. Many observations are similar to the first example: QoS decreases with an increase in $\mu_a$, and the proportion of fake tags for real post is at most $\delta$. The QoS in the left sub-figure with smart users is also less than that for first example provided in Figure \ref{fig:exWM}, which also considers smart users -- however, for the example in Figure \ref{fig:eoWM_delta}(left), the proportion of ws-users ($\mu_2$) is lesser than that in Figure \ref{fig:exWM} and the number of np-users is non-zero. Furthermore, as one may anticipate, the QoS with naive users is even smaller.}

\revr{\subsection{Improved QoS -- QoS among non-adversaries}
It is important to note that the OSN can only control/guide the fake tags from non-adversarial users. The aim is also confined to correctly identifying the actuality of the posts by such users. Hence, it is more appropriate to consider a metric/QoS focused on the proportion of fake tags only from ws and wi-users. We  aim to capture precisely the same in this sub-section and define the appropriate optimization problem. Towards this, let $X_1^u, X_2^u, X_a^u$ be the respective proportion of wi, ws and a-users at limit who fake tag the $u$-post; observe $X_a^u = 0$ and recall, np-users do not participate. Similarly, define $Y_1^u, Y_2^u, Y_a^u$ as the corresponding proportion of users who real tag; observe $Y_a^u = 1$. The limit approaches when the number of users that read the post, $t \uparrow \infty$, and consider a large enough $t$. Then, the number of fake tags by ws-users after $t$-th user reads the post can be approximated by $tX_2^u m_f \eta^u$ (one can anticipate this by the law of large numbers and because of \ref{a2_prop}). The other numbers can be approximated similarly, and as a result, one can re-write the overall proportion of fake tags as:
\begin{align*}
    \beta^u \approx \frac{(X_1^u + X_2^u)m_f \eta^u}{(X_1^u + X_2^u + Y_1^u + Y_2^u)m_f \eta^u + Y_a^u m_f \eta_a}.
\end{align*}
In a similar manner, the proportion of fake tags from non a-users represented by $\beta_a^u$, the quantity of actual interest, can be approximated as below:
$$
\beta_a^u \approx \frac{(X_1^u + X_2^u)m_f \eta^u}{(X_1^u + X_2^u + Y_1^u + Y_2^u)m_f \eta^u} = \frac{X_1^u + X_2^u}{X_1^u + X_2^u + Y_1^u + Y_2^u}.
$$
Thus, one can relate the two QoS metrics as follows:
\begin{align}\label{eqn_beta_u_only_wiws}
    \beta_a^u = \left(\frac{(\mu_1 + \mu_2) \eta_u + \mu_a \eta_a}{(\mu_1 + \mu_2) \eta^u}\right)\beta^u.
\end{align}
The above discussion motivates us to define an `improved quality of service (i-QoS)' with respect to any warning-mechanism:
\begin{align}\label{eqn_iqos}
    \widehat{Q} := \inf \left \{ \left(\frac{(\mu_1 + \mu_2) \eta^u + \mu_a \eta_a}{(\mu_1 + \mu_2) \eta^u}\right)\beta : \beta \in \cA_\beta^F \cup \cR_\beta^F\right \} =  \left(\frac{(\mu_1 + \mu_2) \eta^u + \mu_a \eta_a}{(\mu_1 + \mu_2) \eta^u}\right)Q.
\end{align}
One can interpret $\widehat{Q}$ as the almost sure lower bound on the limit proportion of fake tags for fake post from non a-users. Henceforth, we also consider the comparison of various warning mechanisms using this more relevant metric, i-QoS. Further, we illustrate a lot more improvement when optimization problem \eqref{eqn_opt_prob} is instead designed using i-QoS.
Observe that i-QoS is simply a constant multiple of QoS, and hence by Corollary \ref{corollary_ex_wm} and by \eqref{eqn_iqos}, the i-QoS for eo-WM (represented by $\widehat{Q}^o$) is unique. Thus, the original optimization problem \eqref{eqn_opt_prob} changes to the following, for some $\delta \in (0,1)$:
\begin{align}\label{eqn_new_opt}
    \sup_{w \in \left[0,  \overline{w}\right], b \in [0, \infty)} \widehat{Q}^o(w, b)
    \mbox{ subject to } {\beta}^{o, \infty, R}(w, b) \leq \delta_a := \frac{\delta ((\mu_1 + \mu_2)\eta^R) }{((\mu_1 + \mu_2)\eta^R + \mu_a\eta_a)}.
\end{align}
Observe that the above optimization problem has the same structure as in \eqref{eqn_opt_prob}, except that $\delta$ is replaced by $\delta_a$; hence, $w^*, b^*$ can be derived by Theorem \ref{thrm_opt} directly. The optimal value of the above problem represents the fraction of non a-users (wi and ws-users) who correctly identify the fake post as fake. 
When $\mu_a > 0$ is sufficiently large, then QoS is sufficiently small (lesser than $1-\mu_a$), as it includes the effects of a-users real tagging. However, this does not imply that the WM failed; in fact, on the contrary, at the extreme end, WM is completely successful in eliminating the effect of adversaries if optimal $\widehat{Q}^o = 1$  (indicating that all the non a-users correctly identify the fake post).

In Figure \ref{fig:QOS_eo}, we continue with the two examples of Figure \ref{fig:eoWM_delta}, where we plot i-QoS and its MC estimates, and the corresponding quantities for the real post; the left sub-figure has smart users and right sub-figure has naive users. It is clear that the proportion of fake tags for the real post ($\beta^{o, R}$, see blue curves) are within $\delta_a$-threshold for both cases. More interestingly, the results of the said figure for the fake post indicate that the results of Figure \ref{fig:eoWM_delta} are mis-leading; the latter figure shows extremely high level of degradation in QoS with $\mu_a$, while the same is not the case in the former; this is obviously because  the latter also counts the (intentional) real tags from a-users. For example, when $\mu_a = 0.3$, the QoS is $15.38\%$ in Figure \ref{fig:eoWM_delta}(left), while the actual fraction of fake tags among the smart non a-users is around $70.06\%$. Thus, the degradation with $\mu_a$ may not be as large as depicted in Figure \ref{fig:eoWM_delta}, nonetheless there is sufficient degradation as $\mu_a$ increases (for example, from $99.981\%$ at $\mu_a = 0$ to $70.06\%$ for $\mu_a = 0.3$). 

The above illustrations motivate us to design better warning mechanisms which achieve higher performance. The underlying theme of the entire chapter is to optimize/increase the proportion of fake tags for the fake post while still ensuring that the constraint in \eqref{eqn_new_opt} for the real post is satisfied. In this section, we optimized the performance of the eo-WM for the fake post and achieved exactly $\delta$-threshold for the real post. In the coming sections, we will attempt to design WMs that increase performance without compromising the real post. As mentioned before, this goal is achieved by designing appropriate WMs such that the resultant $g_\beta$ of \eqref{eqn_general_g_beta} has zeroes with desirable properties, which in turn dictate the limiting behaviour of WM as confirmed by Theorem \ref{thrm_BP_to_fake}. }  To this end, the first idea is to eliminate the effect of adversaries, which we consider next.

\section{Eliminating Adversarial Effect WM (ea-WM)}
The OSN may not know the exact set of adversarial users, but it knows the proportion of adversarial users ($\mu_a$). We aim to use this knowledge to design a new, improved WM which performs better even when $\mu_a$ is large. \textit{The idea is to construct a WM specific to any given $\mu_a > 0$, namely $\omega^a(\beta)$, such that $g^F_\beta$ under the new WM exactly matches that corresponding to $g^{o, F}_\beta$ with $\mu_a = 0$, at optimality (see \eqref{eqn_beta_ODE_etac1}).} In other words, using the knowledge of $\mu_a$, we are creating a hypothetical situation with no adversaries, and hence we name $\omega^a$ as \underline{eliminating adversarial effect WM} (ea-WM). If possible, one can anticipate that the performance will improve for the fake post under ea-WM; we will identify such conditions below. Further, one still needs to ensure that the performance of real post is not degraded beyond $\delta$ as in \eqref{eqn_opt_prob} \revr{(beyond $\delta_a$ as in \eqref{eqn_new_opt} when i-QoS is considered); this is ensured by the WM proposed in this section (and by coming WMs as well).} Towards this, we define $\omega^a$ as:
\begin{align}\label{eqn_warning_ea}
    {\omega}^a(\beta) = \omega(\beta) + \frac{\beta \mu_a m_f \eta_a}{ \mu_2 m_f\eta^F \bigg(\beta \alpha_x^F + (1-\beta)\alpha_y^F\bigg) }.
\end{align}
Consider $w, b$ and $\beta$ such that $\min\{\alpha_x^u \omega^a(\beta), 1\} = \alpha_x^u \omega^a(\beta)$. Then $g^F_\beta$ under ea-WM, henceforth denoted as $g^{a, F}_\beta$, matches with $g^{o, F}_\beta(\beta; \mu_a = 0)$, because (see \eqref{eqn_general_g_beta}):
\begin{align}\label{eqn_g_F_eaWM}
\begin{aligned}
    g^{a, F}_\beta(\beta; \mu_a > 0) &= -\beta \mu_2 m_f\eta^F - \beta \mu_1 (1-\alpha_x^F \rho) m_f\eta^F + (1-\beta) \mu_1 \rho  \alpha_y^F m_f\eta^u \\
    &\hspace{1cm}+ \mu_2 \omega^a(\beta) \bigg(\beta \alpha_x^F + (1-\beta)\alpha_y^F\bigg) m_f\eta^F\\
    &= g^F_\beta(\beta; \mu_a = 0).
\end{aligned}
\end{align}
Thus, if $\min\{\alpha_x^u \omega^a(\beta), 1\} = \alpha_x^u \omega^a(\beta)$ is satisfied in a neighborhood of $\betana$, then one can design the required ea-WM, if further the performance of real post is within $\delta$-threshold \revr{(or $\delta_a$-threshold)}. In view of Theorem \ref{thrm_opt}, we set $w, b$ as follows for the new ea-WM \revr{(similarly, with $\delta_a$)}:
\[
w = \overline{w} \mbox{ and } b = 
\begin{cases}
    b^*|_{\mu_a = 0} = \left(\frac{\delta}{1-\delta}\right)\left(\frac{\overline{w} \mu_2(\delta \alpha_x^R + (1-\delta)\alpha_y^R)}{\delta (\mu_1+\mu_2) - (\mu_1 \rho + \mu_2 \gamma) (\delta \alpha_x^R  + (1-\delta) \alpha_y^R)} - 1 \right), &\mbox{if } \beta^{a, \infty, R}(\overline{w}, 0) > \delta,\\
    0, & \mbox{otherwise}.
\end{cases}
\]
Now, similar to eo-WM, for each $u \in \{R, F\}$, we will first identify the set of attractors ($\cA^{a, u}_\beta$) and the combined set of repellers and saddle points ($\cR^{a, u}_\beta$) for the \revg{ODE \eqref{eqn_general_g_beta} under ea-WM, i.e., $\dot{\beta}^u = g_\beta^{a, u}(\beta)$.}
\begin{theorem}\label{corollary_ea_wm}
    Define 
    \begin{align}\label{eqn_l_threshold_mua_eaWM}
        \Delta_a := \mu_2 \eta^F \left(\frac{1}{\alpha_x^F} - \omega(\betana)\right) \left(\frac{\betana\alpha_x^F + (1-\betana)\alpha_y^F}{\betana \eta_a}\right). 
        \end{align}
    Then, the following statements are true for the fake post:
    \begin{enumerate}[label=(\roman*)]
        \item If $0 < \mu_a \leq \min\{1-\mu_1-\mu_2, \Delta_a\}$, then ${\beta}^{a} \geq \betana$ for all ${\beta}^{a} \in \cA^{a, F}_\beta \cup \cR^{a, F}_\beta$.
    
        \item Else, i.e., if $\Delta_a < \mu_a < 1-\mu_1-\mu_2$, then $ {\beta}^{a} \in (\beta^o, \betana)$ for all ${\beta}^{a} \in \cA^{a, F}_\beta \cup \cR^{a, F}_\beta$.
    \end{enumerate}
    For the real post, ${\beta}^{a, R} < \delta$ for all ${\beta}^{a, R} \in \cA^{a, R}_\beta \cup \cR^{a, R}_\beta$. \eop
\end{theorem}
In view of the above and Theorem \ref{thrm_BP_to_fake}, we get that the stochastic iterates $\Ups_n$ under ea-WM for the $u$-post either converge to $\{\mathbf{h}(\beta) : \beta \in \cA^{a, u}_\beta \cup \cR^{a, u}_\beta \} \cup \{\mathbf{0}\}$, or hover around $\{\mathbf{h}(\beta) : \beta \in \cR^{a, u}_\beta \} \cup \{\mathbf{0}\}$. \revr{Unlike eo-WM, the above Theorem does not guarantee a unique limit for the warning dynamics under ea-WM in the non-extinction paths, but Theorem \ref{thrm_BP_to_fake}(ii) ensures that there exists at least one attractor of the ODE \eqref{eqn_general_g_beta}, as} $\cA^{a, u}_\beta \neq \emptyset$.

\revr{Now, note that ea-WM provides a higher warning than the eo-WM, even for the real post. Even with such a WM, it is proved above that the proportion of the real post is maintained\footnote{Some equilibrium points can be saddle points and according to Theorem \ref{thrm_BP_to_fake}, the warning dynamics can hover around such points. However, then the warning dynamics move arbitrarily close to such points, and we assume the equilibrium points to be representative of the limiting behaviour. This leads to a small level of inaccuracy in the sense that the warning dynamics can go above or below the point in case of hovering around.} within $\delta$-threshold. Further, due to higher warning, we expect a higher QoS under ea-WM; next we discuss the same.
Let the QoS \eqref{eqn_qos} under ea-WM be represented by $Q^a$. In view of Theorem \ref{corollary_ea_wm}, we claim that $Q^a > Q^o$ for the following reasons:

(i) when $\mu_a$ is small, i.e., when  $\mu_a \leq \Delta_a$, we have $Q^a \geq \betana > Q^o$ (by Theorem \ref{corollary_ea_wm}(i) and Corollary \ref{cor_beta_o_na}). Thus, ea-WM with adversaries achieves higher QoS than the original eo-WM without adversaries. Then, one can say that the former eliminated the effect of adversaries completely.

(ii) when $\mu_a$ is larger, i.e., when $\mu_a > \Delta_a$,  ea-WM still improves over eo-WM as $Q^a > Q^o$ by Theorem \ref{corollary_ea_wm}(ii). However, in this case, the QoS under ea-WM is lesser than the QoS under eo-WM with no adversary as $Q^a < \betana $. Thus, in this case, the effect of adversaries is not completely eliminated by ea-WM.

Similar design and observations follow when one attempts to design ea-WM with i-QoS, i.e., by replacing $\delta$ with $\delta_a$. Recall again that with i-QoS, we consider a more relevant problem that focuses only on the responses from non a-users.}

\subsection{Numerical analysis for ea-WM}\label{num_exp_eaWM}
In this sub-section, we will numerically quantify the improvement achieved by ea-WM in comparison to eo-WM; we consider only i-QoS based problems and results.
\revr{In Figure \ref{fig:eaWM}, we continue with the two examples considered in Figure \ref{fig:exWM} (i.e., with smart and naive users) for ea-WM. We plot the i-QoS with respect to ea-WM (denoted as $\widehat{Q}^a$) evaluated via the exact zeroes of $g^{a, u}_\beta$ and the corresponding MC estimates for the ea-WM. Once again, we observe a close match between the theoretical expressions and the corresponding MC estimates. 
\begin{figure}[http]
\centering
\begin{minipage}{.5\textwidth}
  \centering
  \includegraphics[trim = {1cm 6cm 0cm 6cm}, clip, scale = 0.3]{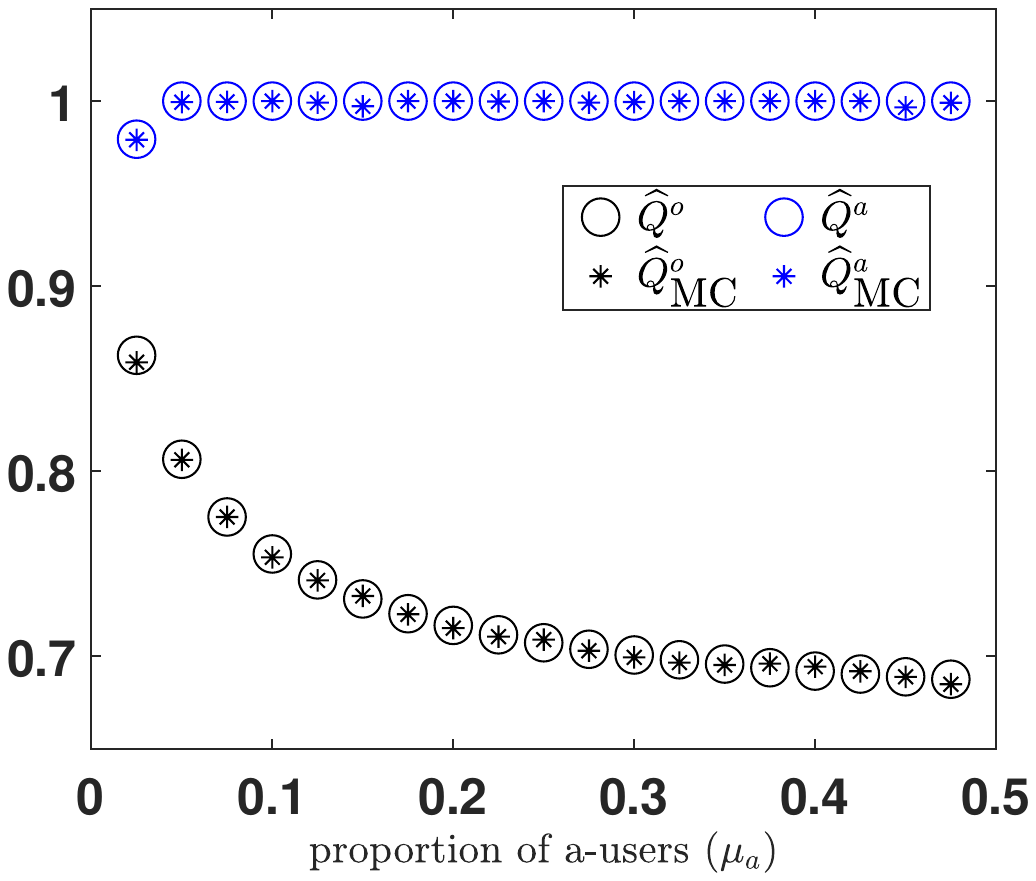}
\end{minipage}%
\begin{minipage}{.5\textwidth}
  \centering
  \includegraphics[trim = {1cm 6cm 0cm 6cm}, clip, scale = 0.3]{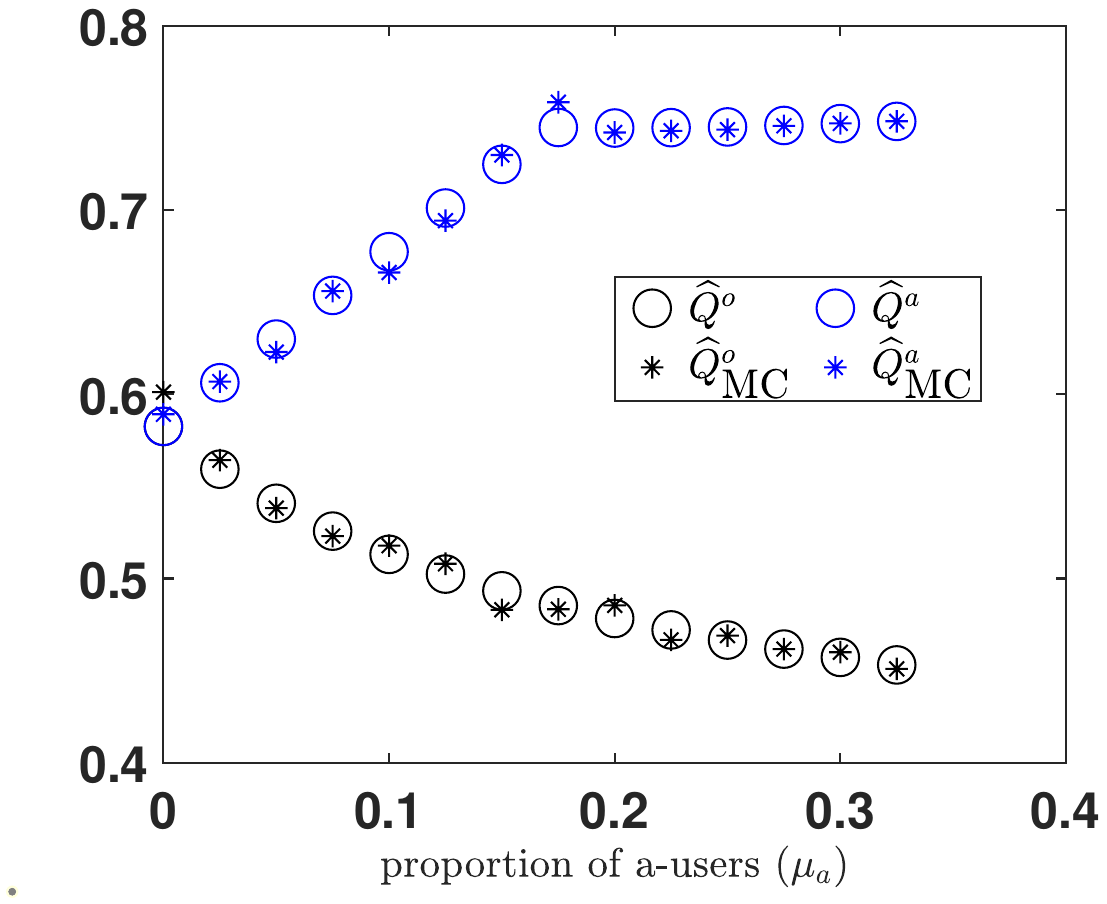}
\end{minipage}
\caption{Comparison of i-QoS under eo-WM and ea-WM, with smart (left) and naive (right) users respectively}
\label{fig:eaWM}
\end{figure}

Further, as seen from the figure \eqref{fig:eaWM}, in all the case studies, the i-QoS improves; nonetheless, this way of improvement does not degrade the performance of the real post, as confirmed by Theorem \ref{corollary_ea_wm}) and also as observed in Figure \ref{fig:ehWM_real} which plots the performance for the real post.
More interestingly,  the i-QoS and the improvement (with respect to eo-WM) both increase sharply with $\mu_a$.  
 Thus,  even in the presence of a larger fraction of a-users confusing the WM, ea-WM is able to nudge the non a-users to identify the fake post as fake correctly. In view of Theorem \ref{corollary_ea_wm}, this may  be true as  ea-WM provides increasingly high warning levels with increase in $\mu_a$ (see \eqref{eqn_warning_ea}).  
 One probably can design a better WM  that  provides higher warning levels even with a smaller value of $\mu_a$ (which again ensures the required real performance), and the quest further is precisely for the same. 
 }

\revr{From Figure \ref{fig:eaWM}(left), for the case study  with smart users, observe that $\widehat{Q}^a = 1$, the maximum possible i-QoS, for $\mu_a \geq 0.05$. However, ea-WM fails to achieve such high i-QoS with naive users --- i-QoS is less than $0.8$  in right sub-figure of Figure \ref{fig:eaWM}. The quest again is for a better WM  which works well even for naive users, and this is considered in the immediate next. 


}

\section{Enhanced WM (eh-WM)} \label{sec_enWM}
In this section, we design an improved version of ea-WM. The idea is to design a warning $\omega^h$ such that $\omega^a(\beta) < \omega^h(\beta)$ for all $\beta \in [0,1]$. \revg{In lines of Theorem \ref{thrm_unique_att},} such monotonicity of the WM will ensure that the zeroes of the function $g_\beta^F$ (see \eqref{eqn_general_g_beta}) corresponding to the new WM are larger than that of $g_\beta^{a, F}$. However, the design should be such that the performance of the new WM for the real post is not compromised. Towards this, we design an \underline{enhanced warning mechanism (eh-WM)} as follows:
\begin{align}\label{eqn_warning_eh}
    \omega^h(\beta)  = \zeta \omega^a(\beta), \mbox{ for an appropriate choice of } \zeta > 1, \mbox{ with } w, b \mbox{ as in ea-WM.}
\end{align}
For given $\zeta$, denote the $g_\beta^u$ of \eqref{eqn_general_g_beta} corresponding to the eh-WM as $g_{\beta, \zeta}^{h, u}$. Further, define $\beta^{h}_\zeta$ as a zero of $g_{\beta, \zeta}^{h, F}$ in $[0,1]$ and $\beta^{h, R}_\zeta$ as a zero of $g_{\beta, \zeta}^{h, R}$ in $[0,1]$. Observe that:

\vspace{-4mm}
{\small
\begin{align*}
    g_{\beta, \zeta}^{h, F}(\beta) &= g_\beta^{a, F}(\beta) + \mu_2 m_f \eta^F \bigg\{ \beta \bigg( \min\{1, \zeta \omega^a(\beta) \alpha_x^F\} - \min\{1, \omega^a(\beta) \alpha_x^F\} \bigg) \\
    &\hspace{1cm}+ (1-\beta) \bigg( \min\{1, \zeta \omega^a(\beta) \alpha_y^F\} - \min\{1, \omega^a(\beta) \alpha_y^F\} \bigg) \bigg\} \geq g_\beta^{a, F}(\beta),
\end{align*}}with equality only if $\alpha_j^F \omega^a(\beta) > 1$ for each $j \in \{x, y\}$. This implies that any zero of $g_{\beta, \zeta}^{h, F}$ is larger or equal to the smallest zero of $g_\beta^{a, F}$. Thus, it clear that $\beta_\zeta^{h} \geq Q^a$ for any $\beta^h_\zeta \in \cA^{h, F}_{\beta, \zeta} \cup \cR^{h, F}_{\beta, \zeta}$. Therefore, we have:
$$
\inf\{\beta : \beta  \in \cA^{h, F}_{\beta, \zeta} \cup \cR^{h, F}_{\beta, \zeta}\} =: Q^h_\zeta \geq Q^a.
$$
That is, the QoS under eh-WM (for any $\zeta$) is higher or at par with the QoS under ea-WM.

Now, one can anticipate that higher the warning level is, the more cautiously users tag the posts. Thus, as $\zeta$ increases, the proportion of fake tags must increase. However, one can not choose an arbitrarily large $\zeta$ as then the performance for the real post is degraded. Thus, we consider the following problem to optimally choose $\zeta = \zeta^*$ such that $Q^h_\zeta$ is maximized, while satisfying constraint in \eqref{eqn_opt_prob}:
        \begin{align}
        \label{eqn_opt_phi}
            \max_\zeta & \hspace{2mm} \revg{Q^h_\zeta}
             \mbox{ subject to } \beta \leq \delta \mbox{ for each } \beta \in \cA^{h, R}_{\beta, \zeta} \cup \cR^{h, R}_{\beta, \zeta}.
        \end{align}
We have the following optimal design for the   eh-WM (proof is in  appendix):
\begin{theorem}\label{corollary_eh_WM}
Define the constant 
\begin{align}\label{eqn_phi_bar}
\overline{\zeta} := \frac{\delta \bigg( \mu_2 \eta^R + \mu_1 (1-\alpha_x^R \rho) \eta_R + \mu_a \eta_a \bigg) - (1-\delta) \mu_1 \rho \alpha_y^R \eta_R}{\mu_2 \omega^a(\delta) \bigg( \delta \alpha_x^R + (1-\delta) \alpha_y^R \bigg) \eta^R}.
\end{align}
The $\zeta^*$ defined below is greater than $1$ and is the optimizer of the problem \eqref{eqn_opt_phi}:
\begin{eqnarray}\label{eqn_phi_star}
    \zeta^* :=
    \left \{ \begin{array}{ll}
       \overline{\zeta}, & \mbox{ if \hspace{2mm}}  \overline{\zeta} < \frac{1}{\alpha_y^R \omega^a(\delta)}, \mbox{ or if  \hspace{2mm}} \overline{\zeta} \geq \frac{1}{\alpha_y^R \omega^a(\delta)}, \hspace{2mm} \underline{\beta}^F = 0 \mbox{ and } b = 0,\\
    \frac{1}{\omega^a (\underline{\beta}^F) \alpha^F_y}, & \mbox{ else.} 
     \end{array} \right.  \mbox{\eop}
\end{eqnarray}   
\end{theorem}
\revg{Thus, the choice of $\zeta$, which gives the maximum proportion of fake tags for the fake post, is given by $\zeta^*$. Such a $\zeta^*$ also ensures that the performance of eh-WM for the real post is not degraded beyond $\delta$-level. The problem \eqref{eqn_opt_phi} can also be designed and solved in terms of the better metric i-QoS and by replacing $\delta$ by $\delta_a$ analogously. 
Henceforth, when we refer to eh-WM, it corresponds to the case with $\zeta = \zeta^*$ and when $\delta = \delta_a$.
We present the numerical results with respect to eh-WM directly in terms of i-QoS and the correspondingly modified $\delta_a$-threshold.}

\subsection{Numerical analysis for eh-WM}
We now (MC) simulate the warning dynamics under eh-WM \revr{for the two examples with smart and naive users, and the  MC-estimates again well match the theoretical values, as seen from  Figure \ref{fig:ehWM_real} (for real post) and Figure \ref{fig:ehWM} (for fake post).  
Next, we discuss the qualitative analysis. To begin with, the Figure \ref{fig:ehWM_real} re-affirms the results of Theorem \ref{corollary_eh_WM} about the real post --- the proportion of fake tags for the real post is at most $\delta_a$.

\begin{figure}[http]
\centering
\begin{minipage}{.5\textwidth}
  \centering
  \includegraphics[trim = {1cm 6cm 0cm 6cm}, clip, scale = 0.3]{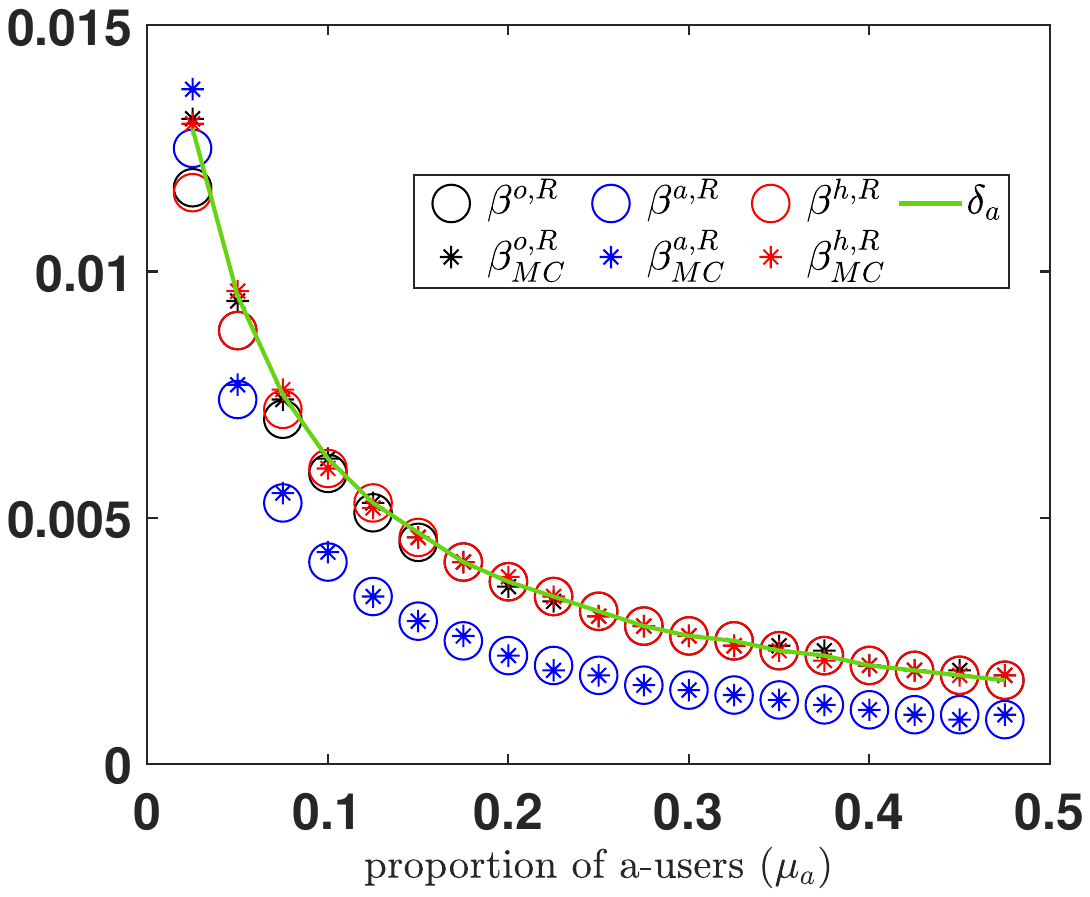}
\end{minipage}%
\begin{minipage}{.5\textwidth}
  \centering
  \includegraphics[trim = {1cm 6cm 0cm 6cm}, clip, scale = 0.3]{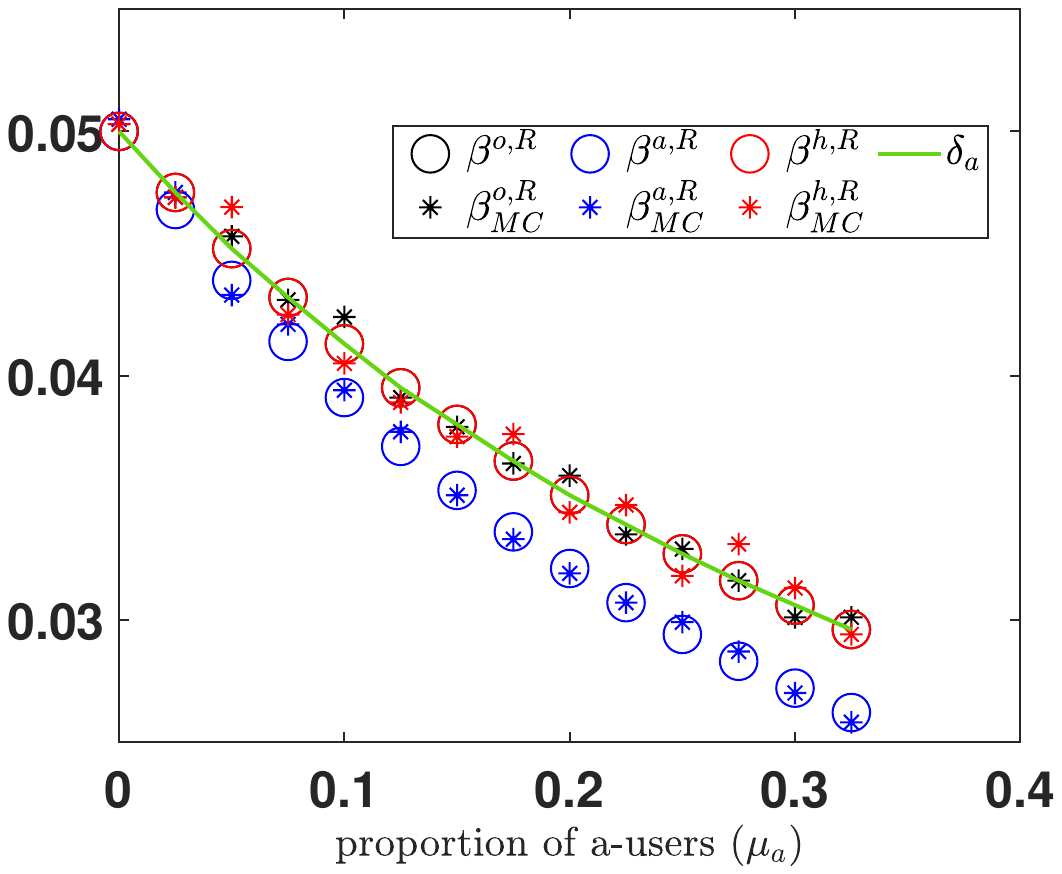}
\end{minipage}
\caption{Limits of warning dynamics for real post under three WMs with smart (left) and naive (right) users  respectively}
\label{fig:ehWM_real}
\end{figure}
\begin{figure}[http]
\centering
\begin{minipage}{.5\textwidth}
  \centering
  \includegraphics[trim = {1cm 6cm 0cm 6cm}, clip, scale = 0.3]{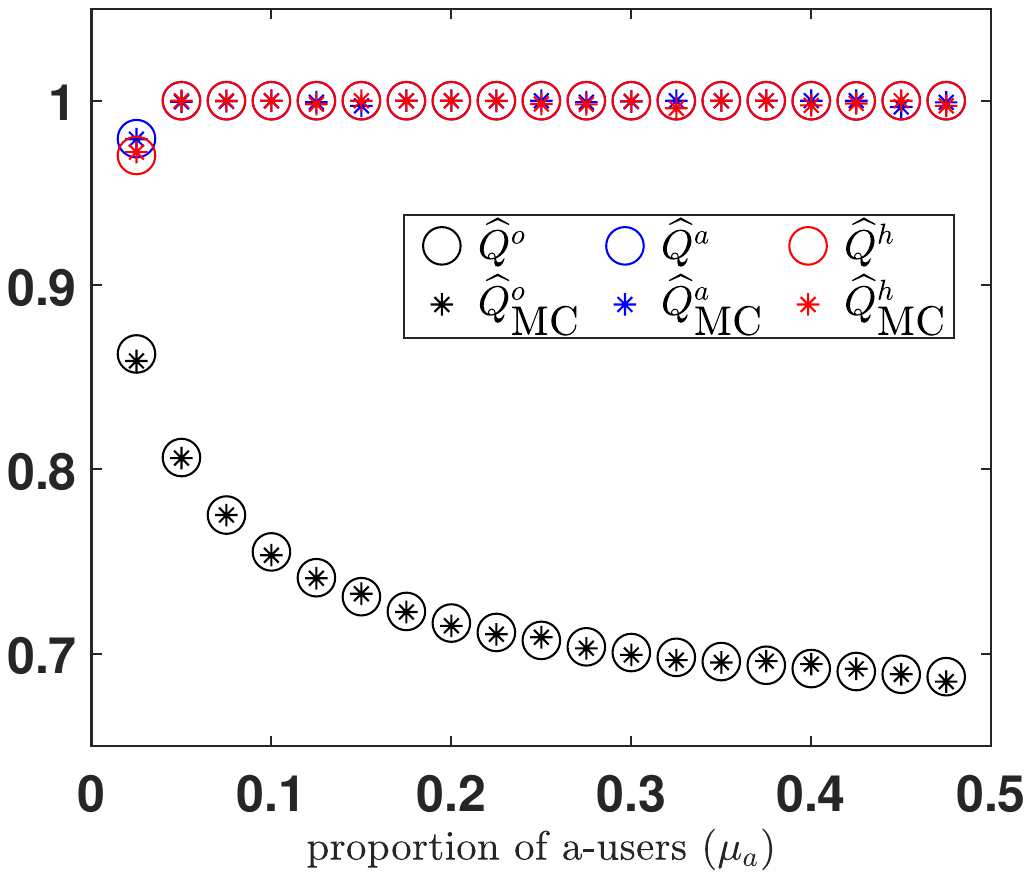}
\end{minipage}%
\begin{minipage}{.5\textwidth}
  \centering
  \includegraphics[trim = {1cm 6cm 0cm 6cm}, clip, scale = 0.3]{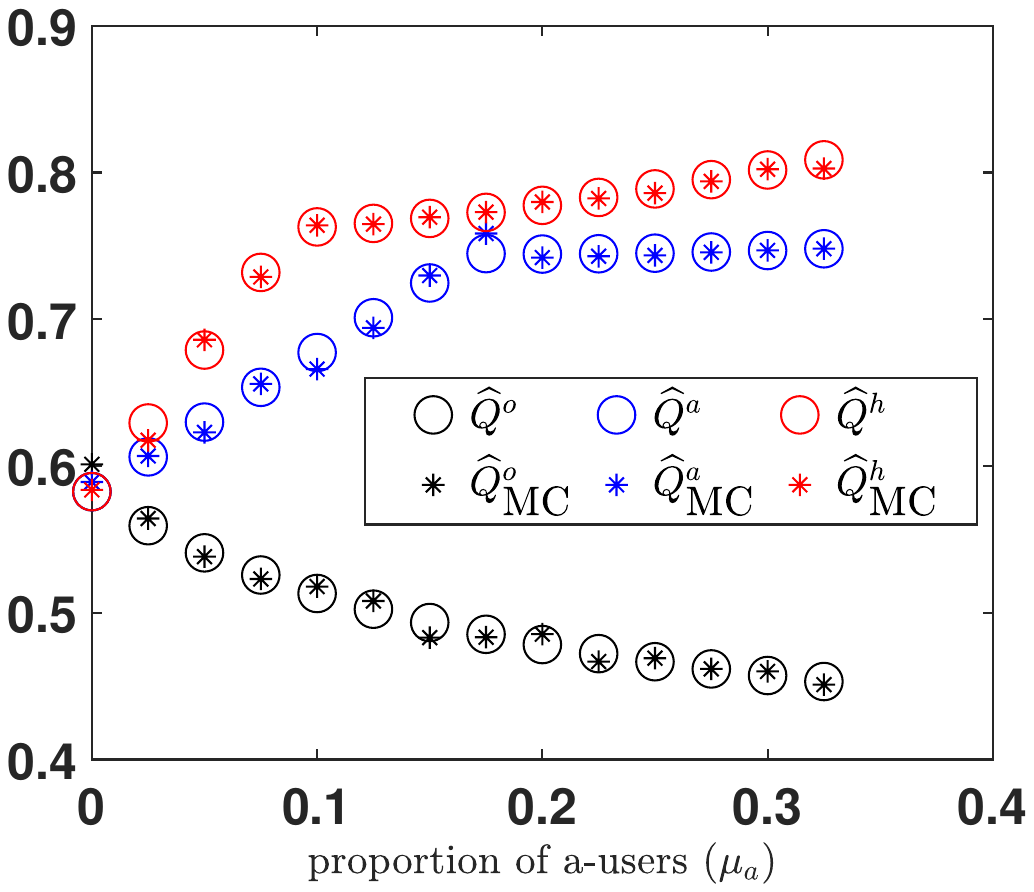}
\end{minipage}
\caption{Comparison of i-QoS under three WMs with smart (left) and naive (right) users  respectively}
\label{fig:ehWM}
\end{figure}
In Figure \ref{fig:ehWM}, we plot the i-QoS under eh-WM $\widehat{Q}^h$ (i.e., with $\zeta^*$), along with that corresponding to the previous two WMs. For the example with smart users, eh-WM performs at par with ea-WM; recall, ea-WM almost achieved $\widehat{Q}^a = 1$. However, for the case with naive users, $\widehat{Q}^h \gg \widehat{Q}^a$; thus, eh-WM is more robust against adversaries than ea-WM. Therefore, eh-WM is able to guide the naive non a-users about the actuality of fake posts better than ea-WM. }

\revr{As an example, when $10\%$ of a-users are trying to harm the system, the eh-WM ensures that $76.29\%$ of naive non a-users correctly identify the fake post, while this fraction is only $51.31\%$ under ea-WM (observe, $\widehat{Q}^h - \widehat{Q}^a$ is as large as $0.2498$, for $\mu_a = 0.1$). }

\revr{As seen from the example with naive users, eh-WM (red curve) performs significantly better than ea-WM (blue curve). Even then, the i-QoS under eh-WM is much better with higher values of $\mu_a$. This limitation probably calls for a very different design of WM, which can generate high warning levels even for smaller values of $\mu_a$. This is attempted in the immediate next.



}

\section{Enhanced-2 WM (eh2-WM) and learning}
It is intuitive that as warning increases, the users are alarmed rigorously about the actuality of the posts; this should lead to more users correctly identifying the posts, and thus higher QoS; in fact, Theorem \ref{thrm_unique_att}  captures precisely this intuition. If one can control the warning so that it does not harm the performance of the real post beyond $\delta$-threshold, providing a higher warning should be effective. We designed ea-WM and eh-WM along these intuitions with a higher warning than eo-WM (recall, there is an additive term in \eqref{eqn_warning_ea} and multiplicative term in \eqref{eqn_warning_eh}), and still managed to ensure the performance of the real post is within the desired level (see Theorem \ref{corollary_ea_wm} and Theorem \ref{corollary_eh_WM}). Further motivated by this, in this section, we aim to design another improved version of eo-WM, named \underline{enhanced-2 WM (eh2-WM)} and denoted by $\omega^{h2}$, which provides higher warning signals to the users (in fact, even for the cases with smaller $\mu_a$); this mechanism also facilitates learning the required parameters $b$ and $w$.

To achieve the same, we again utilize the eo-WM but now with a larger $w$, and ensure that there is a unique limit proportion for the real post which satisfies the $\delta$-threshold. From \eqref{eqn_warning}, a bigger $w$ results in higher warning levels; hence, we set $w = w^{h2} := \nicefrac{1}{\alpha_x^R} - \gamma$ and choose a corresponding $b$ as in Theorem \ref{thrm_opt}. This value of $w$ ensures that $\alpha_j^R \omega^{h2} (\beta) \le 1$ for all $j \in \{x,y\}$ and all $\beta \in [0,1]$ for real post (i.e., when $u = R$) and hence using the same logic as in Corollary \ref{corollary_ex_wm}, we have a unique zero/attractor for the real post; further the choice of $b$ as in Theorem \ref{thrm_opt} ensures the said unique attractor $\beta^{R,h2}$ corresponding to the real post is within the required threshold $\delta$. 
However, unlike eo-WM, with larger $w$ we may not have   a unique limit proportion for the fake post under eh2-WM. Nonetheless, the resultant QoS (and hence i-QoS) is bigger than that with eo-WM by Theorem  \ref{thrm_unique_att}, as with bigger $w$, $\omega^{h2}(\beta) > \omega(\beta)$ for all $\beta$.  

It is important to observe here that the new enhanced WM (eh2-WM) generates high levels of warning signals, and its design does not depend on parameters like $\mu_a$. Thus, one can anticipate that it will enhance the performance even for the smaller values of $\mu_a$. To illustrate the same, we tabulate the i-QoS, $\widehat{Q}^{h2}(w^{h2}, b(w^{h2}))$, achieved under eh2-WM for the case with naive users:
\begin{table}[http]
    \centering
    \begin{tabular}{|c|c|c|c|c|}
    \hline
         & $\mu_a = 0$ &   $\mu_a = 0.1$ &  $\mu_a = 0.2$ &  $\mu_a = 0.3$\\ \hline
         $\widehat{Q}^{h2}(w^{h2}, b(w^{h2}))$ & $0.8289$ & $0.8270$ & $0.8257$ & $0.8246$\\ \hline
    \end{tabular}
    \caption{i-QoS under perfect knowledge of user sensitive parameters}\label{table_WM_perfect}
\end{table}
Clearly, the i-QoS under eh2-WM is consistently higher than that with eh-WM (see Figure \ref{fig:ehWM}, where the red curve is below $0.8$ for all $\mu_a$). More importantly, the i-QoS under eh2-WM is almost the same for all values of $\mu_a$.  


%

\noindent \textbf{Learning the parameters:} At this point, it is important to note that all the discussions so far assumed that the user-sensitive parameters ($\rho$ and $(\alpha_i^u)$ for each  $i \in \{x, y\}$ and $u \in \{R, F\}$) and proportions of users of different types ($\mu_1, \mu_2$ and $\mu_a$) are known to the OSN. However, such information is not easily accessible to the OSN, and the purpose now is to design a WM without such knowledge. Towards this, \textit{we propose an algorithm which directly learns the parameters of the WM, $b$ and $w$}. We only require that there is a non-zero proportion of ws-users\footnote{it can be checked by noticing the users who click on the information button (see Figure \ref{fig_post_design}) }, i.e., $\mu_2 > 0$ and the knowledge of ratio $\alpha_x^R/\alpha_y^R$ (details are given below). The design would only utilize various random quantities observed during the post propagation process.

The main idea is to consider a real post known to the OSN and train the parameters $w$ and $b$ using the responses of the users. 

Basically, we add a SA-based step which tunes $b$ such that the corresponding $\beta^{o, R}$ eventually approaches $\delta$ - recall, the constraint in optimization problem \eqref{eqn_opt_prob} requires that $\beta^{o, R} \leq \delta$. Further, $w$ is tuned such that $\alpha_y^R \omega^{h2}(1)$ approaches $ 1-\kappa$, where constant $\kappa \geq 1 -~\alpha_y^R/\alpha_x^R$. From \eqref{eqn_warning}, $\omega^{h2}(1; w, b) = w+ \gamma$, and hence such a tuning ensures that $w$ approaches $\nicefrac{(1-\kappa)}{\alpha_y^R} - \gamma$ (and by choice of $\kappa$, eventually $w \leq \nicefrac{(1-\kappa)}{\alpha_y^R} - \gamma$) --- thus, eventually $\alpha_j^R \omega^{h2}(1) \leq 1$ for each $j \in \{x, y\}$, as planned for the real post. Here, we would like to stress that the tuning of $w$ is done with respect to $\alpha_y^R$, instead of $\alpha_x^R$, as there may not be sufficient estimates corresponding to fake tags for the real posts (recall, $\delta$ is typically a small value). Thus, the algorithm requires some idea on the ratio $\alpha_x^R/\alpha_y^R$. In all, if such a tuning (of both $w$ and $b$) is possible, then it would ensure a unique attractor below $\delta$-threshold for the real post. 

The above tuning for $w$ requires warning levels $\omega^{h2}(1)$, corresponding to $\beta = 1$; however, in the eo-WM, the warning levels were generated according to the then estimates of $\beta$, the proportion of fake tags. To minimally disrupt the normal functioning of the WM, we propose some special epochs at which such special warning is provided -- at time epoch $k$, if a ws-user who received the post with real tag clicks on the information button, the OSN generates such a warning with probability $\eta_k$, where $\eta_k \downarrow 0$, as $k \to \infty$. Only such special epochs are used to learn $w$. To summarize, the updates for $w$ at epoch $k$ are as follows: if a ws-user that received the post with real tag reads the post, then we have:
\begin{align}\label{update_w}
    w_k \gets 
    \max \left\{1, w_{k-1} - \epsilon_k\left( I_{k} - (1-\kappa) \right) \right\}, &\mbox{ with probability }\eta_k, 
\end{align}where $I_{k}$ is the indicator that the user tags the post as fake and $\epsilon_k := c_1(\frac{1}{k+1})^{c_2}$ with some appropriate $c_1 > 0$ and $c_2 \in ( 0.5, 1]$. In all other cases, we set $w_k = w_{k-1}$.



Next, we discuss the updates for $b$.  For each $k \geq 1$, update $b_k$ as below:
\begin{align}\label{update_b}
\begin{aligned}
    b_k &\gets \max \left\{0, b_{k-1} + \epsilon_k (B_k^{h2, R} - \delta) \right\},  \mbox{ where as before } B_k^{h2, R} := \frac{\Cx(\tau_k^-)}{\Cx(\tau_k^-) + \Cy(\tau_k^-)},
\end{aligned}
\end{align}and the post-propagation process updates as in \eqref{eqn_transition_fake_tag} and \eqref{eqn_transition_real_tag} --- the warning shown to the $k$-th user reading the post would have been generated using $(w_k, b_k)$ as below: 
\begin{align}\label{eqn_warning_learn}
\omega^{h2}(B^{h2, R}_k) := \omega(B^{h2, R}_k) = \frac{w_k B^{h2, R}_k}{B^{h2, R}_k + b_k(1-B^{h2, R}_k)} +~\gamma, 
\end{align}
at the normal epochs (when $w_k$ is not updated); for the special epochs, the warning $\omega^{h2}(B^{h2, R}_k) := \omega(1) = w_{k} + \gamma$ is generated.


The brief idea behind such a design is that as is usually the case with SA algorithms, the SA iterates $b_k$ and  $w_k$ converge to ensure the expected values of the respective  update-terms $B_k^{h2, R} - \delta$ in \eqref{update_b} and  $I_{k} - (1-\kappa)$ in \eqref{update_w} converge to $0$ as $k \to \infty$. That is, $\beta^{h2,R}_k = E[B_k^{h2, R} ]$ approaches $\delta$ and $w_k$ approaches\footnote{Observe that the conditional expected value conditioned that the user is a ws-user who received the post with real tag, $E[I_{ k}] = \alpha_y^R \omega(1) = \alpha_y^R  (w_k + \gamma) $.} $(1-\kappa)/\alpha_y^R - \gamma$. As already mentioned, such a limit of $w$ ensures that the unique limit for the real post $\beta^{h2,R}$ is near $\delta$; thus, the constraint in \eqref{eqn_opt_prob} is satisfied, and the discussion at the beginning of this section also ensures that the QoS is strictly improved in comparison to the eo-WM. 

The learning algorithm is summarized in Algorithm \ref{alg_WM}. The analysis of the above learning algorithm  would require rigorous two-time scale (projected) SA-based tools - observe $w_k$ is updated minimally and further probability $\eta_k \downarrow 0$. We skip the analysis here but validate and illustrate the improved performance of the \underline{learning WM} (referred to as l-eh2-WM)  via numerical examples in the following sub-section.
\newcommand{\SampS}{{\tiny \mathbb{S}}}
\subsection{Numerical analysis for l-eh2-WM}\label{subsec_numerical_estimate}
In Table \ref{table:WM}, we continue with the example with naive users to test the learning algorithm. Towards this, we fix   $\kappa = 1 - \nicefrac{\alpha_y^R}{\alpha_x^R} + 10^{-3}$, $\eta_k = 1.5(\nicefrac{1}{k})^{0.8}$, $\eta_0 = 0.008$, $w_0 = 6$ and $b_0 = 10^{-4}$. The choice of $\epsilon_k$ for learning $b$ and $w$ is $2.2(\nicefrac{1}{k})^{0.7}$. We initialize the system such that the content provider shares a real post with the real tag to $20$ users.

For a given sample size (number of samples available for learning and represented by $\SampS$), we consider $150$ sample paths for the post-propagation of the real post under l-eh2-WM; the idea is to measure  the  efficacy of l-eh2-WM algorithm  via the fraction of times it achieves an i-QoS   within $\pm0.05$ of that corresponding to the case with perfect information (i.e., $\widehat{Q}^{h2}(w^{h2}, b(w^{h2}))$). 
We consider different sample sizes $\SampS$ in the range $10^4$ to $10^5$.

In Table \ref{table:WM}, for different values of $\SampS$, we tabulate $f_\SampS$, the fraction of sample paths for which $|\widehat{Q}^{h2}(w^{h2}, b(w^{h2}))- \widehat{Q}^{o}_\SampS(b_\SampS, w_\SampS)| \leq 0.05$.

\begin{table}[http]
\centering
\begin{tabular}{|c|ccccc|}
\hline
                              & \multicolumn{5}{c|}{$\SampS$ }                                                                                                                                                                           \\ \hline
                                 & \multicolumn{1}{c|}{$10^4$} & \multicolumn{1}{c|}{$2.5* 10^4$} & \multicolumn{1}{c|}{$5*10^4$} & \multicolumn{1}{c|}{$7.5*10^4$} & \multicolumn{1}{c|}{$10^5$} \\ \hline
\multirow{1}{*}{$\mu_a = 0$}  
                               & \multicolumn{1}{c|}{0.73}    & \multicolumn{1}{c|}{0.89}      & \multicolumn{1}{c|}{0.91}    & \multicolumn{1}{c|}{0.95}    &  \multicolumn{1}{c|}{0.93}
                               \\ \hline
\multirow{1}{*}{$\mu_a = 0.1$} 
                               & \multicolumn{1}{c|}{0.41}    & \multicolumn{1}{c|}{0.57}      & \multicolumn{1}{c|}{0.76}    & \multicolumn{1}{c|}{0.84}      & \multicolumn{1}{c|}{0.91}    \\ \hline
\multirow{1}{*}{$\mu_a = 0.2$} 
                               &  \multicolumn{1}{c|}{0.19}    & \multicolumn{1}{c|}{0.44}      & \multicolumn{1}{c|}{0.64}    & \multicolumn{1}{c|}{0.74}      & \multicolumn{1}{c|}{0.79}           \\ \hline
\end{tabular}
\caption{Fraction of sample paths that learnt the parameters ($b, w$) sufficiently well and achieved the desired level of i-QoS under l-eh2-WM}
\label{table:WM}
\end{table}

\RestyleAlgo{ruled}
\begin{algorithm}
\caption{Design of learning WM}\label{alg_WM}
(i) Consider a real post.

(ii) Initialize $\Cx(\tau_0)$ and $\Cy(\tau_0)$; calculate $B_0^{h2, R}$. Fix a large enough $\SampS < \infty$.

(iii) Initialize $b_0$ and $\eta_0$ sufficiently small, and choose a  $w_0 > 1$.

(iv) At $k$-th epoch, $\tau_k$, when $k$-th user reads the post, for $k \in \{1, 2, \dots, \SampS\}$:
\begin{itemize}
\item  set the $w$-update flag, $J_{ws} = 0$
    \item if the reader is a ws-user, then provide warning, $\omega^{h2}$, which is set as below:
        \begin{itemize}
            \item  toss a biased coin such that $P(\mbox{head appears}) = \eta_{k-1} > 0$, let $\eta_{k-1} \to 0$
            \item if head appeared and if the said user received with post with real tag, 
            \begin{itemize}
                \item set warning corresponding to $\beta = 1$, i.e., set $\omega^{h2}(B^{h2, R}_{k-1}) := w_{k-1} + \gamma$
                \item set the indicator $J_{ws} = 1$
            \end{itemize}
             
            \item else, set warning as per WM, i.e., set $\omega^{h2}(B^{h2, R}_{k-1})$  as in \eqref{eqn_warning_learn}
        \end{itemize}
    \item observe the tag $I_k$ and the number of shares by the said user and accordingly, update proportion of fake tags, $B_k^{h2, R} = \frac{\Cx(\tau_{k-1}^+)}{\Cx(\tau_{k-1}^+) + \Cy(\tau_{k-1}^+)}$
    
    \item update the parameters, using the new estimate $B_k^{h2, R}$ and $I_k$
    \begin{itemize}
        \item if $w$-update flag, $J_{ws} = 1$, then update  $w_k$ as in  \eqref{update_w}
        \item update $b_k$ as in \eqref{update_b}
    \end{itemize}
\end{itemize}

\end{algorithm}



It can be seen from the table that the fraction of sample paths with the desired property ($f_\SampS$) increases with $\SampS$, thus depicting that the l-eh2-WM is progressively able to achieve the performance close the case with perfect knowledge. One may anticipate that more iterations/shares should be required to achieve i-QoS of eh2-WM (i.e., with perfect knowledge) as $\mu_a$ increases; the same is evident from the table; for example, when $\SampS = 10^5$, $f_\SampS$ is as large as $0.91$ for $\mu_a = 0.1$, but $\mu_a = 0.3$, it is much smaller and equals $0.79$. 
 Thus, this example illustrates that l-eh2-WM has learned and tuned the WM sufficiently well when it has more than $7.5*10^4$ samples for the proportion of a-users up to $0.2$. 

The performance of the learning algorithm is sensitive to the initial conditions and the parameters of the two-timescale algorithm (like, $\epsilon_k$), as is the usual case with SA-based algorithms. Using the trial-and-error method, we picked a good enough set of values, while an extensive study on a better choice of these parameters is outside the scope of this work. 


Next, in Figure \ref{fig:estimate}, we continue with
the two examples considered in Figure \ref{fig:eoWM_delta}. In the left and right sub-figures, we consider the instances with smart and naive users, respectively and present the results directly in terms of i-QoS. The learning algorithm is again initialized and tuned appropriately, and now, with a large sample size, $\SampS = 10^6$.

\begin{figure}[http]
\centering
\begin{minipage}{.5\textwidth}
  \centering
  \includegraphics[trim = {1cm 6cm 0cm 6cm}, clip, scale = 0.3]{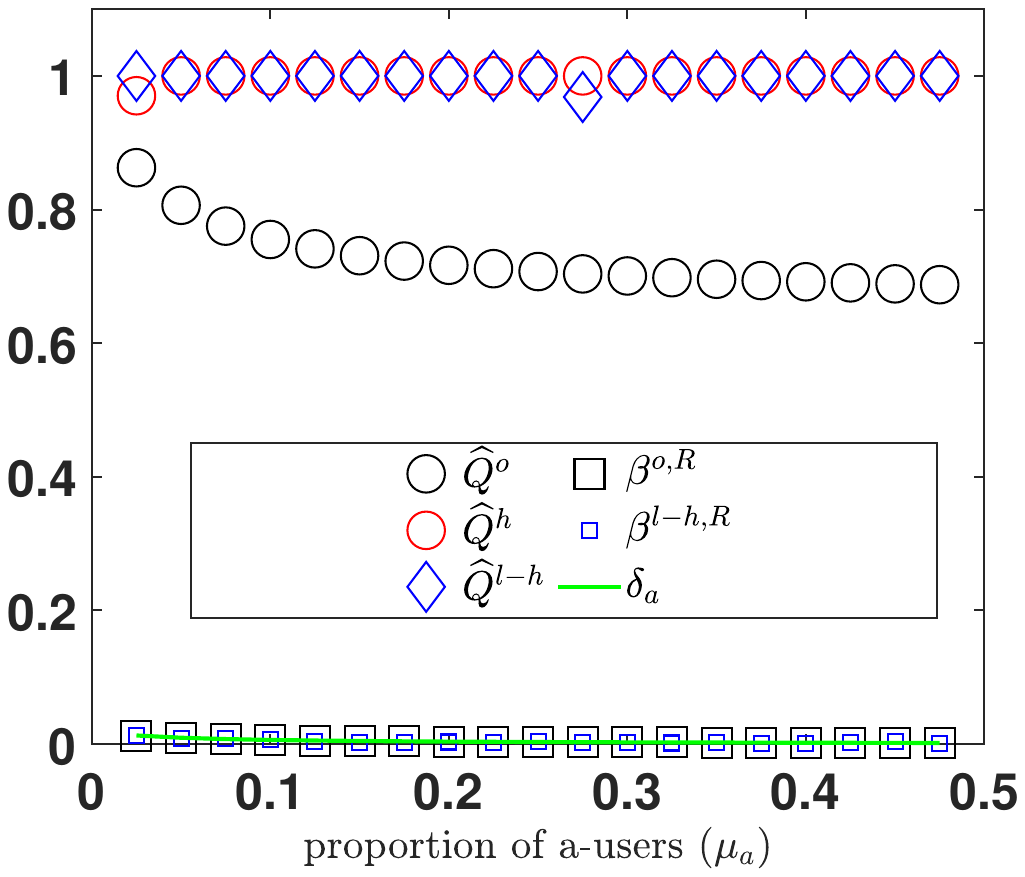}
\end{minipage}%
\begin{minipage}{.5\textwidth}
  \centering
  \includegraphics[trim = {1cm 6cm 0cm 6cm}, clip, scale = 0.3]{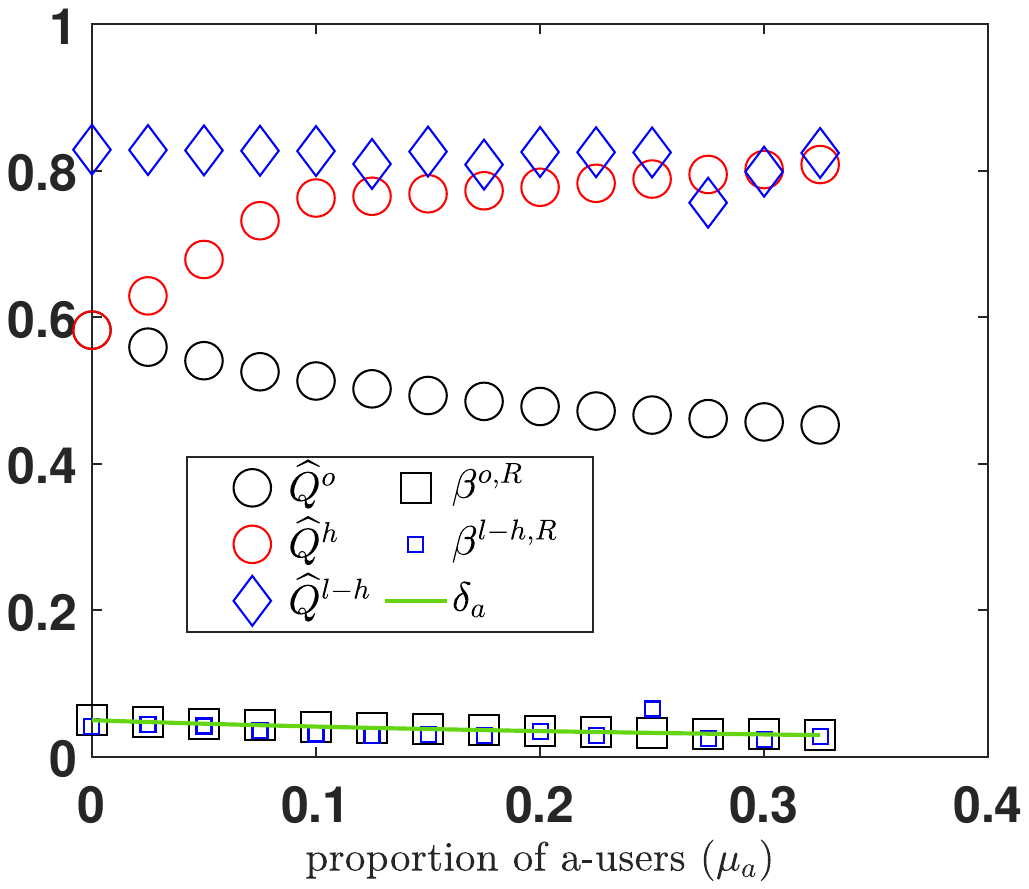}
\end{minipage}
\caption{Comparison of limits of warning dynamics under eo-WM, eh-WM and l-eh2-WM with smart (left) and naive (right) users  respectively}
\label{fig:estimate}
\end{figure}
The figure shows that for all values of $\mu_a$, the i-QoS under l-eh2-WM  (marked in diamond) is higher than the eo-WM; in fact, it performs superior to all the previous WMs. Of course, the i-QoS can not be further improved for smart users --- even l-eh2-WM achieves i-QoS close to $1$, as eh-WM. The superior performance of eh2-WM (actually that of l-eh2-WM with large $\SampS$) is clearly depicted in the case with naive users. 
From Table \ref{table_WM_perfect} and Figure \ref{fig:ehWM}, it is clear that the eh2-WM outperforms eh-WM and performs similarly for all values of $\mua$. The l-eh2-WM  with large $\SampS$ has exactly similar performance traits, as seen from Figure \ref{fig:estimate}.
Furthermore, the proportion of fake tags for the real post is also within the $\delta_a$-threshold, thus satisfying the constraint in \eqref{eqn_new_opt}.


\section{Conclusions and Future Work}
There is a huge requirement to identify fake posts on ever-active OSNs. Further, any algorithm attempting to identify fake posts faces challenges from adversarial users and users unwilling to participate. Our first aim in this chapter is to derive the performance of a promising recently proposed algorithm in the presence of adversaries who always real tag any post. A severe degradation in performance is observed with just 1\% of adversaries.

The algorithm collects binary signals (fake/real tags) from all previous users, generates a warning based on the fraction of fake tags and compels further users to judge and consume the post cautiously based on the warning level provided. Using new results in branching processes (also derived in this chapter), we obtain a one-dimensional ordinary differential equation (ODE) that analyses any generic iterative warning mechanism depending upon the fraction of fake tags. 
This ODE is instrumental in deriving robust adaptations of the previous mechanism -- in particular, we use concepts like eliminating the effects of adversaries,  the inherent monotone characteristics of relevant performance on certain parameters, etc. 
The new mechanisms illustrate significant performance improvement both in the presence and the absence of adversaries compared to the previous method. An algorithm that improves the existing method without relying on the usually inaccessible users-specific information is also proposed.

This chapter also contributes towards total-current population-dependent two-type branching processes with population dependent death rates and also considers a variety of unnatural deaths. In particular, we derive all possible limits and limiting behaviours of the population sizes as time progresses.

In future, one can think of several new directions. The one-dimensional ODE can also be utilized to study other types of adversaries, like always fake tagging adversaries or more informed adversaries that mis-tag both posts (fake tag authentic post and real tag the fake post). One can again derive improved algorithms, as we illustrated with real tagging adversaries. One can also study the influence of users who share but refuse to tag or other important behavioural characteristics. Further, we designed two types of enhanced warning mechanisms, which improved over the existing mechanism. However, the two new mechanisms are not comparable, as one can perform better than the other in some instances. In future, one can attempt to design a combination of the two which outperforms all of them and also design the corresponding learning mechanism.

\chapter{Saturated viral markets: Saturated total population-dependent BP}\label{ch:STPBP}

In this chapter, we study the effect of re-forwarding the post to the same set of users on the OSN. This effect leads to saturated viral markets\footnote{The work in this chapter is published, see ``Agarwal, Khushboo, and Veeraruna Kavitha. ``Saturated total-population dependent branching process and viral markets.'' 2022 IEEE 61st Conference on Decision and Control (CDC). IEEE, 2022.''}, which we analyze using a new variant of single-type total-population dependent BP that transitions from super-to-sub critical regime. Theoretical trajectories for the unread and read copies of the post and several important metrics are derived.

\section{Introduction}
Social media offers a global platform for communication;  people share their content, forward interesting posts among the ones shared with them, etc. It plays a significant role in marketing, e.g., a post advertising a specific product/service can get viral (reach a large number of users). It is also used for spreading propaganda, sharing knowledge, influencing a large population's beliefs, etc.

The content propagation (CP) over an online social network (OSN) can be outlined as follows: a) a post is shared with an initial set of users, called seed users; b) the seed users forward the post to their friends/followers if they like it;  c) the recipients follow suit, and this continues. The content either gets extinct in the initial phase or gets viral, and eventually, copies get \textit{saturated}, and the propagation ceases.

There are several approaches for studying CP, and we use branching processes (BPs) to analyze the same. Further, since a post can witness a huge surge in the shares in a short duration, we consider continuous-time BPs (e.g., \cite{dhounchak2023viral, iribarren2011branching, van2010viral}). We discuss and compare other techniques at the end. 

\noindent \textbf{Branching processes:} A variety of BPs have been studied before (e.g., \cite{athreya2001branching,agarwal2021co,klebaner1984population} are a few strands of  them). We discuss a few relevant varieties and provide some details to model the CP process using BPs. 
 The course of any BP largely depends upon the expected number of offspring ($\Gamma$).
 We have super-critical BPs   when the  
expected offspring $E[\Gamma] > 1$ (see, e.g., \cite{athreya2001branching}); such BPs have a positive probability of exploding (population  grows exponentially) and  can mimic the viral CP.    In the critical/sub-critical regime ($E[\Gamma]\leq1$), the BP gets extinct with probability one, i.e., the population eventually declines to zero. The current population in a BP represents the live population, while the total population also includes the dead ones. In the context of  CP, the number of unread copies (the corresponding recipients are yet to view/read the post) represents the current population. While the total number of recipients, including the ones that already read the post, represents the total population.

\noindent \textbf{Saturation:} When a post is viral and is already forwarded to a noticeable fraction of the network, a significant fraction of further forwards (by future users) can overlap a  part of the network that already received the post. The expected number of effective forwards (after deleting re-forwards) represents the expected offspring when one attempts to model CP using a BP. The number of re-forwards depends on the total copies (number of users that already received the post) and not just on the current population/unread copies.
We are only aware of current-population  dependent BPs (e.g., \cite{klebaner1984population}).

Thus the existing BP models are insufficient to mimic a saturated CP due to two imperative factors: a) the overlaps in forwards/offspring depend upon the total copies/ population, and b) the usual content propagation process traverses from super-critical to sub-critical regime before getting extinct. 
We thus consider a  new variant of continuous-time total-population dependent Markovian BP, named as \textit{saturated total-population dependent BP (STP-BP)}. The total copies either increase with time or saturate; if re-forwards are proportional to the total copies, then the resultant value of expected offspring decreases with time. This also ensures the desired transition from the super-critical to the sub-critical regime. The saturated BP well mimics the CP.

In previous chapters, we analyzed many new variants of BPs using a new approach based on stochastic approximation techniques, e.g., attack and acquisition  BP (competing viral markets, see Chapter \ref{ch:journal1}), and proportion-dependent BP (fake news on OSNs, see \cite{kapsikar2020controlling} and Chapter \ref{ch:journal2}). Further modifying the said approach to address the required finite horizon analysis allows us to analyze saturation resulting from total population dependency.

\noindent \textbf{Key contributions:} At first, we formally analyze the STP-BP. \textit{We derive an appropriate ordinary differential equation (ODE) and its solution, which (time asymptotically) almost surely approximates the embedded chain of the CP process over any finite time window. } These deterministic solutions depict exponential growth and linear fall for unread copies. Secondly,  we model and fit an appropriate total-copies dependent piece-wise linearly decreasing function for the expected offspring, having two different slopes. 
We further derive important metrics like the peak number of unread copies, time asymptotic value of the total copies, and others. The growth of the total copies is exponential and depends on the reduction rates in the expected forwards.

We corroborate our theoretical results by performing Monte-Carlo simulation on SNAP Twitter-dataset \cite{mcauley2012learning}. 
The description of the theoretical trajectories depends only on four parameters of the OSN (e.g.,  two rates of reduction of the expected forwards) and the attractiveness of the post.
  
\noindent \textbf{Related work:} The study of CP on OSNs has been a topic of interest for a long time, and several approaches have been used for its analysis. 
Random graph models are widely used to analyze CP on OSNs (e.g., see \cite{deijfen2016winner, zhou2019cost, lu2019beyond}); in particular, \cite{lu2019beyond} considers re-exposure of users with a post on OSN. However, such models can not capture aspects like virality, as discussed in \cite{van2010viral,dhounchak2023viral, iribarren2011branching}, while  BPs facilitate virality analysis like growth patterns, extinction/virality probability etc.

The epidemiology-based models (in particular SIR) are also used to study CP, which succeed in capturing saturation (e.g., \cite{rodrigues2016can, freeman2014viral, jiang2019quantitative}). The current and total copies are respectively modelled as infected and (infected$+$recovered) populations. We argue that this approach is not suitable for analyzing viral markets on OSNs/email platforms: (a) this set of literature directly starts with an appropriate ODE (e.g.,  \cite{freeman2014viral} considers exponentially diminishing infection rate/interest in the post, while \cite{rodrigues2016can} considers standard SIR ODE), (b) they do not delve into the details of the random dynamics (e.g.,  the chance encounters between various individuals), and more importantly, (c) on OSNs,  majority of users share the post to a subset of their friends  `only once' after viewing and lose interest immediately after; in contrast, in SIR based models, infected individuals keep infecting/spreading the information for a random/prefixed time before recovering/losing-interest. In other words, \textit{SIR-based models well capture the typical behaviour of word-of-mouth dynamics (individuals remain interested in gossip,  keep sharing it, and then lose interest) and are insufficient for viral markets over OSNs}.

There is a brief indirect mention of the saturation effect in \cite{van2010viral}, where using BPs, the authors predict the future progress of CP using the available history of an ongoing campaign; they use the well-known Kolmogorov's backward equations  (for PGFs in BPs) to achieve this. In contrast, we provide a theoretical study of a new relevant variant of BP, facilitating an exhaustive study of the saturated CP on OSNs.

\section{Problem description and background}\label{sec_model}
Consider an OSN where the content of interest is forwarded by its users. At the start of the propagation, the post is shared by its content provider to  an initial set of users,  called seed users.
These users view the post on  the OSN at random time instances. Then, they forward it to some or all of their friends, depending on how much they like the post.\footnote{The post may seem appealing to the users for the offers mentioned, the creativity or the informational quotient of the content (see \cite{van2010viral}).} This subset of users further forward the post when they visit the OSN\footnote{When a user views, reads and forwards some of the posts on its timeline.}, and the post propagation continues likewise (more details in  \cite{dhounchak2023viral, agarwal2021co, van2010viral}). 
The random instances at which users visit the OSN are called wake-up times. The time gap between the wake-up and the post-reception times of any recipient  is  exponentially distributed with parameter $\lambda$.  
Let $\mathcal{F}$ be the (random) number of friends of a typical user, with finite mean. 
The attractiveness of the post  is captured by  factor $\rho$, which specifies a subset of friends $(\offs)$ to whom the post is forwarded. Naturally, $\rho$ determines the  growth of the post.

We will see that the content propagation (CP) majorly depends on the expected number of such forwards, which in turn depends on the then total recipients of the post (say $a$). That is, $E[\offs] := m(a)$, where $m$ is some function which will be discussed in detail in the coming sections. The study of such a dependency (using the SNAP dataset) and its influence on the CP is the key differentiating feature of our work. As one may anticipate again, such dependency will eventually lead to a reduction in the effective forwards of the post, in turn leading to the saturation of  shares. This effect has been  observed majorly through numerical studies in the past, e.g.,  see \cite[Fig. 6]{dhounchak2023viral} and \cite[Fig. 1]{hemsley2016studying} for total and current shares, respectively. \textit{We aim to provide an analytical explanation of saturation in terms of an approximating deterministic trajectory and some relevant performance measures.} We next describe the new variant of `saturated' BP, STP-BP.


\subsection{Saturated Content propagation and Branching process}

Let $A(t)$ denote the   total number of copies of the post on the network, i.e., all the copies which has been received (viewed or not, forwarded or not) by the users on the OSN till time $t$; we briefly refer this number as \textit{total shares}. Further, let $C(t)$ be the unread/live/current number of copies of the post till time $t$. These represent the copies that have been received but not yet viewed  by the users till time $t$, and are solely responsible for further propagation of the post; we refer to this number as \textit{current shares}.

Let $A(0) = C(0) = a_0$ be the number of seed users. The previous discussion clearly shows that the number of forwards directly depends on the total shares, i.e., $\offs = \offs(A(t))$. Observe that $\offs$  does not depend on the current shares, as is usually considered in the BP literature.

Let $\tau^+$ and $\tau^-$ denote the time instances immediately after and before wake-up time $\tau$ (of any user),  e.g., $A(\tau^+) := \lim_{t \downarrow \tau} A(t)$. 
Then, the evolution of the system at transition epoch $\tau$ is:
\begin{equation}\label{evolve_cont}
\begin{aligned}
C(\tau^+) &= C(\tau^-)  + \offs(A(\tau^-)) - 1, \mbox{ and } A(\tau^+) = A(\tau^-)  + \offs(A(\tau^-)).
\end{aligned}
\end{equation}

The above dynamics can easily be placed  in a BP framework when the current shares are modelled as a population. Accordingly, the number of forwards $(\offs(A(\tau^-)))$ can be viewed as the offspring of the population. Further, when a user shares the post, the current shares reduce by $1$ (see \eqref{evolve_cont}); this is  exactly like a death in a BP (see \cite{dhounchak2023viral, kapsikar2020controlling, agarwal2021co} for similar details). 

To analyse total population dependent BPs, one needs to study two-dimensional tuple $\Om(t) := (C(t), A(t))$ simultaneously, a realisation of which is denoted by $\om = (c, a)$. In contrast, the existing BP models can analyse  $C(t)$ alone if the need is only to analyse the current population. We now proceed towards analytically deriving the trajectories of the tuple $\Om(t)$ and other salient features.

\section{Dynamics and ODE approximation }\label{sec_dynamics}

To facilitate the study of STP-BP, we analyse the embedded chain corresponding to the underlying continuous-time jump process (CTJP), as in \cite{kapsikar2020controlling, agarwal2021co}. It is a standard technique to use embedded chains when transience, recurrence, extinction and similar properties of CTJP are studied. We use it for a similar purpose in this chapter. 
In particular, we observe the  dynamics in \eqref{evolve_cont} at the time instances when a user with an unread copy of the post wakes-up. 
Let $\tau_n$ be $n^{th}$ such transition epoch.\footnote{If the post gets extinct at $n^{th}$ epoch, we set $\tau_{k} :=\tau_n$ for all $k \geq n$ and the same is true for rest of the quantities.}
Let $C_n := C(\tau_n^+)$ be the current shares of the post immediately after $\tau_n$. Similarly define $A_n$. Note that the  time taken by the first user to wake-up after $n^{th}$ transition epoch, ($\tau_{n+1}-\tau_n$), is exponentially distributed with parameter  $\lambda C_n$. Thus, if a user wakes up at $\tau_n$, then:
\begin{align}\label{evolve_SA}
C_{n} &= C_{n-1}  + \offs_{n}(A_{n-1}) - 1, \ A_{n} = A_{n-1}  + \offs_{n}(A_{n-1}).
\end{align}

As in \cite{kapsikar2020controlling, agarwal2021co}, we use stochastic approximation (SA) approach to study the embedded chain. Towards this, define the following fractions of current and total shares respectively, $\Pc_n := \nicefrac{C_n}{n}$ and $\Pa_n := \nicefrac{A_n}{n}$ for $n \geq 1$, with $\Pc_0 = \Pa_0 := a_0$. Let $\Ups_n := (\Pc_n,  \Pa_n)$. Further, define:
\begin{eqnarray}\label{eqn_eps_n}
    \epsilon_n = \frac{1}{n+1}, \ t_n := \sum_{k=1}^n \epsilon_{k-1}, \mbox{ and } \eta(t) := \max\left  \{ n: t_n \le t \right \}. \hspace{-2mm}
\end{eqnarray}Then, the evolution of $\ups_n$ can be captured by 2-dimensional SA based updates given below (see \eqref{evolve_SA}):
\begin{equation}\label{eqn_SA}
\begin{aligned}
\Pc_n &= \Pc_{n-1} + \epsilon_{n-1} \left[\offs_{n}(A_{n-1}) - 1 - \Pc_{n-1} \right ]1_{\Pc_{n-1} > 0}, \\
\Pa_n &= \Pa_{n-1} + \epsilon_{n-1} \left [\offs_{ n}(A_{n-1})  - \Pa_{n-1} \right ]1_{\Pc_{n-1} > 0}.
\end{aligned}
\end{equation}

We analyse these fractions using SA techniques (e.g., \cite{kushner2003stochastic}), which helps in approximating the same using the solutions of the ODE (see \eqref{eqn_eps_n}):
\begin{align}\label{eqn_ODE_stpbp}
&\dot{\pc} = \left(m(a) - 1 - \pc\right) I, \  
\dot{\pa} = \left(m(a)  - \pa\right)I, \mbox{ with}\\
&I := 1_{\pc > 0}, \ a(t) :=     \pa(t) \eta(t), \mbox{ and } m(a) := E[\offs(a)]. \nonumber
\end{align}

By Lemma \ref{lemma_existence}, the solution for this non-autonomous and non-smooth ODE exists over any finite time interval in the extended sense (satisfies the ODE for almost all $t$). In Theorem \ref{thrm1_stpbp} given below,  we will prove that the above ODE indeed approximates \eqref{evolve_SA}.

\noindent \textbf{Approximation result:} The study of continuous-time population size-dependent BPs has been limited in the literature. 
This chapter uses the ODE approximation result to study the saturated BP. Now, for mathematical tractability, we require the following assumption on offspring distributions: 
\begin{enumerate}[label=\textbf{D.\arabic*}, ref=\textbf{D.\arabic*}]
    \item There exists an integrable random variable, $\hat{\offs}$, such that $\offs(a) \leq \hat{\offs}$ almost surely for every $a$ and $E[\hat{\offs}]^2 < \infty$.  \label{d1}
    \item The mean function, $m(\cdot)$ is Lipschitz continuous. \label{d2}
\end{enumerate}The assumption \ref{d1} is readily satisfied (details in the next section), while \ref{d2} is an extra assumption required for additional affirmation (see Theorem \ref{thrm1_stpbp}(ii)).
We will now see that the piece-wise constant interpolation $\Ups^n(\cdot) := (\Psi^{n, c}(\cdot), \Psi^{n, a}(\cdot))$ of $\Ups_n$ trajectory defined as:
\begin{align}
\Ups^n(t) = \Ups_n \mbox{ if } t \in [t_n, t_{n+1}),
\end{align}
satisfies an almost integral representation  as below, with $I_n := 1_{\Psi^{n, c}(s) > 0}$ (see \eqref{eqn_linear2} for derivation):
\begin{align}\label{eqn_piecewise_main}
\Psi^{n, a}(t) &= \Pa_n + \int_0^t \hspace{-1mm} \bigg(m(\Psi^{n, a}(s) n) - \Psi^{n, a}(s)\bigg)I_n ds + \varepsilon^{n, a}(t),\\
\Psi^{n, c}(t) &= \Pc_n + \int_0^t \bigg(m(\Psi^{n, a}(s) n) - 1 -  \Psi^{n, c}(s)\bigg)I_n ds  + \varepsilon^{n, c}(t). \nonumber
\end{align}Let $\hat{\ups}^n(\cdot)$ be the  solution of ODE \eqref{eqn_ODE_stpbp}, with $\hat{\ups}^{n}(0) = \ups_{n}$. Observe that $\Ups^n(\cdot)$ in \eqref{eqn_piecewise_main} is similar to $\hat{\ups}^n(\cdot)$, except for the difference term $\varepsilon^n(t) := (\varepsilon^{n, c}(t), \varepsilon^{n, a}(t))$. 
Further, if at all $\varepsilon^n(t) = 0$ for all $t \leq T $, then, by uniqueness of the  solution (see Lemma \ref{lemma_existence}), the BP trajectory \eqref{eqn_SA} would have coincided with it, i.e., $\ups_k = \hat{\ups}(t_k)$ for all $k$ such that $t_k \leq T$. However, it is not true in general; nevertheless, we will show that $||\varepsilon^n|| \to 0$ as $n \to \infty$ (see norm $||\cdot||$ in \eqref{eqn_norm}). 

Thus, we have two operators which are converging towards each other; the first operator including $\varepsilon^n$ in \eqref{eqn_piecewise_main} provides the BP trajectory, while the second operator without $\varepsilon^n$ in \eqref{eqn_piecewise_main} provides the ODE solution.
Further, using the Maximum theorem, we show that  the difference between  the two solutions of the operators \eqref{eqn_piecewise_main} with  and without $\varepsilon^n$ is small when $||\varepsilon^n||$ is small. Formally, we state the result as follows:

\begin{theorem}\label{thrm1_stpbp}
For any $\ups(\cdot) = (\pc(\cdot), \pa(\cdot))$, define the norm with any finite $T > 0$:
\begin{align}
\begin{aligned}\label{eqn_norm}
||\ups|| &:= \max\{||\pc||, ||\pa||\}, \mbox{ where } \\
||\psi^i|| &:= \sup\{t \in [0,T]: |\psi^i(t)|\} \mbox{ for any } i \in \{a, c\}.
\end{aligned}
\end{align}
Under assumption \ref{d1}, we have the following almost surely:
\begin{enumerate}
    \item[(i)] $||\varepsilon^n|| \to 0$ as $n \to \infty$, and 
    \item[(ii)] if \ref{d2} also holds, then, the difference  $\sup_{k : k \geq n, t_k \leq T} || \Ups^n(t_k) - \hat{\ups}^n(t_k) ||$ depends upon the magnitude of $||\varepsilon^n||$.
\end{enumerate}
\end{theorem}
\noindent \textbf{Proof} is provided in Appendix \ref{Appendix_STPBP}. \eop

\noindent \textbf{Remarks:} (i) For each $n$, consider the ODE initialised with the value of the embedded chain, $\ups_{n}$. Then, the embedded chain values at transition epochs, $k \in [n , \eta(t_{n} + T)]$, are  close to the ODE solution, $\hat{\ups}^{n}(t_k-t_{n})$,  at time points $t_k \in [t_{n}, t_{n}+T] $. This approximation improves as $n$ increases, and \textit{the result is true almost surely} (a.s.) and for all $T < \infty$.

\noindent (ii) \textbf{Dichotomy:} We have a `modified 
dichotomy': either the population gets extinct ($\ups_{n} = 0$ in initial epochs), or the population explodes exponentially  as confirmed by  ODE-solution \eqref{eqn_total_ode}   in Section \ref{ode_based_analysis}.  In contrast to classical dichotomy (e.g., \cite{athreya2001branching}), \textit{in both the sets of sample paths, the saturated BP eventually gets extinct}.

\noindent (iii) As  the ODE solution approximates the embedded chain (BP) trajectory, one can analyse the latter using the former.  \textit{In contrast to many existing studies, we would consider the analysis of the ODE trajectories and not the attractors, which is more relevant here. } Prior to that, we derive an appropriate function that can represent $m(\cdot)$, the total-share-dependent expected forwards (TeF) in a typical OSN.

\section{Population dependent expected forwards}\label{sec_mean}

It is clear from the ODE \eqref{eqn_ODE_stpbp} and Theorem \ref{thrm1_stpbp} that the expected number of forwards, $m(\cdot)$, influence the course of any post; we construct a piece-wise linear function to capture the same. We estimate the parameters and validate the model using the SNAP dataset \cite{mcauley2012learning} in Section \ref{sec_numerical}. The following aspects are considered for the model:

$\bullet$ Any user on the OSN forwards the post to a subset of friends based on its attractiveness.
A fraction of users in this subset will have previously received the post. 
\textit{Most likely, such users will not be interested in the same post again.} Suppose $\kappa$ is the fraction of common friends between any two typical users, then $\kappa a$ denotes the number of such re-forwards if `$a$' number of users already had the post. \textit{Thus a linearly decreasing function is a suitable choice for $m(\cdot)$.   }

$\bullet$ It is likely that a user with a high number of friends is also a friend to a large number of users; the dataset supports this observation. 
Thus, a user with more connections is more likely to receive the post earlier. 
To capture such aspects, {\it we model the friends of a user, ${\cal F}(a)$, to be independent across users, but with decreasing expected values, i.e., $a \mapsto E[{\cal F}(a)]$ is itself a decreasing function}. With this, {\it the assumption \ref{d1} is readily satisfied by considering for example, $\hat{\offs} = {\cal F}(a_0)$}, with $E[{\cal F}(a_0)]^2 < \infty$.

$\bullet$ Further, as time
progresses, the post would naturally proceed towards saturation. Therefore, it is likely that the slope of $m(\cdot)$ is drastically different towards the end. To summarize,
\textit{we model the TeF function $m(\cdot)$ as a piece-wise linearly decreasing function, with two different slopes, where the initial slope is bigger than that towards the end}.

$\bullet$ It is reasonable to assume that the expected number of forwards is proportional to the attractiveness factor $\rho > 0$. Hence, we model the TeF function $m(\cdot)$ as $m(a) = \rho m_N(a) ~ \forall ~ a$, where $m_N(\cdot)$ is the TeF with $\rho = 1$. Such a factorization of $m(\cdot)$  is supported by the experiments  on the SNAP dataset 
for a wide range\footnote{\label{footnote}Such a common  fit is good mostly for $\rho \geq 0.4$, for others we directly derive a good linear fit of $m(\cdot)$ by trial-and-error, see Section \ref{sec_numerical}.} of $\rho$.
The latter function $m_N(\cdot)$ is determined solely by the characteristics of the network. Basically, $m_N(\cdot)$ corresponds to a hypothetical situation where every recipient shares with all of its friends, but new shares are only the effective forwards. Conclusively:
\begin{align}\label{eqn_network_mean_func}
m(a) &= \rho m_N(a), \mbox{ where with } \widetilde{m} := \overline{m} - \bax (\kappa_1-\kappa_2), \nonumber \\
m_N(a) &:= 
 (\overline{m} - \kappa_1 a)1_{ a \leq \bax} + (\widetilde{m} - \kappa_2 a)1_{ a > \bax}.
\end{align}
Thus, the TeF is determined by four parameters, $\overline{m}, \kappa_1, \kappa_2$ and $\bax$. Here, $\rho\overline{m} = \rho E[\mathcal{F}(a_0)]$ is the expected number of forwards in the beginning, and $ \rho\kappa_1, \rho\kappa_2$ reflect the slopes of the TeF.  Observe that the slope reduces to $\rho\kappa_2 < \rho\kappa_1$ when total shares reach a particular value, $\bax$. We assume\footnote{Then, there is a possibility of exponential growth leading to virality.} $\rho \overline{m} > 1$.

$\bullet$ We corroborate that the above model well captures the content propagation using an instance of the SNAP dataset: (a) we first obtain the TeF curve by estimating $m_N(a)$ (i.e., with $\rho = 1$) for some values of $a$;  (b) then fit a two slope curve using a naive approach; and (c) use the estimated (four) parameters to derive the solution of the ODE \eqref{eqn_ODE_stpbp} and other theoretical conclusions, for posts with different $\rho$ values.

Estimating the parameters mentioned above in step (b) is crucial in any such study. For now, we estimate the same using a simple trial-and-error method, as explained in footnote \ref{footnote} and Section \ref{sec_numerical}. However, more sophisticated study is essential to estimate these parameters more accurately, and we leave it for the future. As said earlier, this study aims to provide a theoretical understanding of the CP process.

\section{ODE  analysis}\label{ode_based_analysis}
Using the TeF  \eqref{eqn_network_mean_func}, we solve the ODE \eqref{eqn_ODE_stpbp}. By Theorem \ref{thrm1_stpbp}, this solution (a.s.) approximates the CP process. We also derive the approximate trajectories for $a(t)$ and $c(t)$, which show the desirable saturation effect. Let $\tau_s := \inf\{t: a(t) = \bax \}$ and $\tau_e := \inf\{t: c(t) = 0 \}$, the extinction time. 

\textit{In extinction sample paths (we have, $\ups_{n} = 0$),  from ODE \eqref{eqn_ODE_stpbp}, $\hat{\ups}^{n}(t) = 0$ $\forall$ $t$. We now consider viral sample paths.}

\vspace{0.5mm}
\noindent \textbf{Total and current fractions:} The fractions $(\pa, \pc)$ are given by the solution of the ODE  \eqref{eqn_ODE_stpbp}, which can be directly obtained using the integrating factor (IF) approach. For $t \leq \tau_e$, the closed-form (extended) solution for $\pa$ is given by: 
\begin{align*}
    \pa(t) = e^{-\int_{u_1}^t (u_4 \eta(s') + 1) ds'}  \left(u_2 + u_3\int_{u_1}^t  e^{\int_{u_1}^s (u_4 \eta(s') + 1) ds'} ds\right), 
\end{align*}where  constants {\small$u : = (u_1, u_2, u_3, u_4)$} in the two phases equal:
\begin{equation*}
u = 
\begin{cases}
\left(0, c_0, \overline{m} \rho, \kappa_1 \rho \right), & 0 \leq t \leq \tau_s,\\
\left( \tau_s,  \pa(\tau_s), \widetilde{m} \rho , \kappa_2 \rho  \right), & \tau_s < t \leq \tau_e.
\end{cases}
\end{equation*}

After extinction, i.e, for $t > \tau_e$, from the ODE \eqref{eqn_ODE_stpbp},  $\pa(t) = \pa(\tau_e)$. Similarly, the current fraction of shares till time $t$ evaluates to the following, again using IF approach:
\begin{align*}
    \pc(t) &=  e^{-(t-v_1)}\left(v_2 + \int_{v_1}^t  e^{(s - {v_1})}(v_3 - v_4 a(s) - 1)  ds \right), \mbox{with} \nonumber \\ 
    v  &= (v_1, v_2, v_3, v_4)  =   
    \begin{cases}
    \left(0, c_0, \overline{m}\rho, \kappa_1\rho \right), & 0 \leq t \leq \tau_s,\\
    \left( \tau_s,  \pc(\tau_s), \widetilde{m}\rho, \kappa_2\rho \right), & \tau_s < t \leq \tau_e.
    \end{cases}
\end{align*}Further,  $\pc(t) = 0$ for  all $t \geq \tau_e$.

\noindent \textbf{Trajectories of shares:} It is more relevant to analyse the trajectories corresponding to the current and total shares; and we provide approximate expressions for the same.  Recall from \eqref{eqn_ODE_stpbp}, $a(t) = \pa(t) \eta(t)$ and $c(t) = \pc(t) \eta(t)$. Thus, we begin with an approximation for  $\eta(t)$, defined in \eqref{eqn_eps_n}:
\begin{align}\label{eqn_eta}
\begin{aligned}
    \eta(t) &\approx \max\{n : \gamma + \ln(n) \leq t \}= \floor{e^{t-\gamma}} \approx e^{t-\gamma},
\end{aligned}
\end{align}where $\gamma$ is Euler-Mascheroni constant. \textit{Henceforth, we use this approximation of $\eta(t)$ in all the computations.} With this approximation and using \eqref{eqn_ODE_stpbp}, the ODE for $a(\cdot)$ is given by:
\begin{align}\label{approx_total_ode}
    \dot{a} &=  
     \dot{\pa}e^{t-\gamma} + \pa e^{t-\gamma}1_{c > 0} = m(a)e^{t - \gamma}1_{c > 0}.
\end{align}
The solution of the above ODE with $m(\cdot)$ as in \eqref{eqn_network_mean_func}, can be obtained using the standard techniques in ODE theory:
\begin{equation}\label{eqn_total_ode}
a(t) = 
\begin{cases}
 w_1 - w_2 e^{-w_3 e^{t}}, &\mbox{ when } 0 \leq t \leq \tau_e,\\
a(\tau_e), & \mbox{ when } t > \tau_e,
\end{cases}
\end{equation}
where the constants are given by the following (with $w = (w_1, w_2, w_3)$, observe by continuity $a(\tau_s) = \bax$):
\begin{equation}\label{total_pop_constants}
w = 
\begin{cases}
\left( \frac{\overline{m}}{\kappa_1}, (w_1 - a_0)e^{w_3}, \kappa_1 \rho e^{-\gamma} \right), & 0 \leq t \leq \tau_s,\\
\left( \frac{\widetilde{m}}{\kappa_2}, (w_1- \bax)e^{w_3  e^{\tau_s}},  \kappa_2 \rho e^{-\gamma} \right), & \tau_s < t \leq \tau_e.
\end{cases}
\end{equation}Observe that $a(\cdot)$ saturates as the current shares get extinct. We believe, only STP-BP can capture these effects. Proceeding as in \eqref{approx_total_ode}, the ODE for $c(t)$ is given by:
\begin{align}\label{approx_current_ode}
    \dot{c} = (m(a) - 1)e^{t-\gamma}1_{c > 0}.
\end{align}
By solving, the trajectory of the current shares is given by:
\begin{align}\label{eqn_current_ode}
    c(t) &= \left(c(\varphi) - a(\varphi) + a(t) + e^{-\gamma}(e^\varphi - e^t)\right)1_{t < \tau_e}, 
\end{align}
where $\varphi := \tau_s 1_{t > \tau_s}$. After extinction (for $t \geq \tau_e$), $c(t) = 0$. 

\vspace{0.5mm}
\noindent \textbf{Shares at transition epochs:} From Theorem \ref{thrm1_stpbp}, the value of the embedded chain at the transition/wake-up epoch $n$ is approximated by the ODE solution at $t_n$, defined in \eqref{eqn_eps_n}. Thus, it is more important to evaluate $a(t)$ and $c(t)$ at time points $ t = t_n$. Towards this, define $n_s := \eta(\tau_s)$ and $n_e := \eta(\tau_e)$ as the respective counterparts of $\tau_s$ and $\tau_e$. Using the same approximation ($t_n \approx \gamma + ln(n)$) as in \eqref{eqn_eta}, now for the mapping $n \mapsto t_n$, the shares, by \eqref{evolve_SA}, \eqref{eqn_total_ode} and \eqref{total_pop_constants}, equal:
\begin{align}\label{eqn_total_epoch}
a(t_n) &=
     w_1 - w_2 e^{-n w_3 e^\gamma},  \mbox{ when } 0 \leq n \leq n_e,\\
c(t_n) &=  w_1 - w_2 e^{-n w_3 e^\gamma} - n, \mbox{ when } 0 \leq n \leq n_e.\label{eqn_current_epoch}
\end{align}
\noindent In the above, \eqref{eqn_current_epoch} follows because  $a(t_n)-c(t_n) = n$ from \eqref{evolve_SA}. After extinction, $(a(t_n), c(t_n)) $ $=(a(\tau_e), 0)$  $\forall$ $n \geq n_e$. 

We now investigate other important metrics related to CP.

\noindent \textbf{Growth rates:} In standard BPs, the population exhibits dichotomy, i.e., either the population grows exponentially large (at a constant rate),  or declines to zero. In the former case, both $C(t_n)$ and $A(t_n)$ grow like $e^{(\rho \overline{m} -1)\tilde{\tau}_n}$, where $\tilde{\tau}_n := \lambda \tau_n$ is independent\footnote{Using properties of exponential distribution, this can easily be proved. Such an effect is seen since we are discussing the embedded chain.} of $\lambda$, as time progresses. While in STP-BP, from \eqref{eqn_total_epoch}, which again approximates $A(t_n)$ a.s., and time asymptotically, it is clear that the total shares have exponential growth, however, the rates are different in the two phases. The growth rate in the initial phase is $w_3 e^\gamma = \kappa_1 \rho$, while it decreases to $\kappa_2 \rho$ in the later phase (since $\kappa_1 > \kappa_2$). Further, \textit{the current shares also experience an initial exponential growth, which is further modulated by the growing factor of $n$, in \eqref{eqn_current_epoch}, leading to an eventual linear fall}. This illustrates the modified dichotomy discussed after Theorem \ref{thrm1_stpbp}. More attractive posts have higher growth rate.

At this point, we would like to admit that discussing the growth rates at transition epochs is a non-standard practice in the BP literature. However, for the saturated BP, that does not grow forever, we believe such a discussion is relevant and important. In future, we plan to include the influence of $\{\tau_n\}$ on these growth rates, by extending the ODE analysis to fractions $\left\{\nicefrac{\tau_n}{n}\right\}$; this would help us derive the standard growth patterns discussed in the BP literature for STP-BP. 
 
\noindent  \textbf{Peak of current shares:} Define $c^* = \sup_t c(t)$, i.e., the peak (maximum) current shares. It can be obtained from \eqref{eqn_current_ode} $\left(c''(t) < 0\right)$ and equals (recall $n_s = a(\varphi) - c(\varphi)$):
\begin{eqnarray}\label{peak}
    c^* = w_1 - \frac{\left(1 + \ln(w_2 w_3 e^\gamma) \right)}{w_3 e^\gamma}, \mbox{ as } e^{-\gamma  + \varphi} = n_s,
\end{eqnarray}
where $\varphi$ takes different values in two phases as in \eqref{eqn_current_ode}.

\noindent \textbf{Life span and Max reach:} Recall $n_e$ is the epoch at which CP of the post terminates, or in other words, it represents the life span of the post. Now, substituting $c(t_n)$ from \eqref{eqn_current_epoch}, $n_e$ can be written as the solution of the equation:
\begin{eqnarray}\label{eqn_max_share}
    w_1 - w_2 e^{-n_e w_3 e^\gamma} = n_e,
\end{eqnarray}
where  $w$ is given by the second line of \eqref{total_pop_constants}. 
Further, by \eqref{eqn_current_epoch}, the \textit{max reach} or saturated total shares $a(t_{n_e})$ equal $n_e$.

\noindent \textbf{Probability of virality:} When one transitions from super-critical to sub-critical regime, it is clear that the process gets extinct a.s. (see modified dichotomy). One of the important questions related to CP is the chances of virality, where total/current shares grow significantly large. This necessitates the definition of a different probability related to STP-BP, the probability of virality (PoV). In view of the initial exponential growth of current shares \eqref{eqn_current_epoch} and the dichotomy remark after Theorem \ref{thrm1_stpbp}, we define PoV, $p_\Delta := P(C(t) > \Delta \mbox{ for some } t)$ for some threshold $\Delta > 0$. We conjecture that for small $\Delta$, $p_\Delta \approx 1 - p_e$, where $p_e$, the probability of extinction of standard population independent BP solves the equation: $f(s) = s$, where $f(\cdot)$ is the PGF of $\Gamma(a_0)$ (see \cite{athreya2001branching}). This is an important aspect for future study.

\section{Numerical Experiments}\label{sec_numerical}
We now perform exhaustive Monte-Carlo (MC) simulations on the SNAP Twitter dataset to validate our theory (see \cite{mcauley2012learning}). The dataset consists of inter-connections among $81,306$ users, with $29.77$ average number of friends. \textit{In all the case studies, we consider only viral sample paths.}

\hide{
\noindent \textbf{Validation using synthetic data:} In Theorem \ref{thrm1_stpbp}, we established that the solutions of the ODE \eqref{eqn_ODE_stpbp} can almost surely approximate the dynamics of the CP process over any finite time window. We corroborate this result using MC simulations for  $\bax = 1400$, $\overline{m} \rho = 5$,  $\rho \kappa_1 = 0.003$, and $\rho \kappa_2 = 28 \times 10^{-5}$ in Fig. \ref{fig_synthetic_and_mean_curve}(i). In particular, we plot two sample paths of the embedded chain \eqref{evolve_SA}, against the Piccard's iterates of the ODE \eqref{eqn_ODE_stpbp}, approximate trajectories derived in \eqref{eqn_total_ode}, \eqref{eqn_current_ode} and the trajectories obtained at transition epochs in \eqref{eqn_total_epoch}, \eqref{eqn_current_epoch} for the total and current shares. In the said figure, we apply Piccard's iterates with $n_m = 5$. Then, we evaluate the shares (represented by red triangles) by multiplying the fractions obtained by $\eta(t_k)$ for finite $k > n_m$. Similarly, the approximate trajectories are plotted at transition epochs, and are displayed by red circles. 
From the figure, it can be seen that Theorem \ref{thrm1_stpbp} indeed holds true for synthetic data, and further the derived trajectories also well approximate the embedded chain. 
}

\noindent \textbf{MC simulations over the dataset:} Suppose the content provider initially shares the post with $2$ random seed users chosen from the dataset, identified by their user IDs. These users are added to the total and current-shares lists. At any time, one random user from the current-shares list wakes-up and forwards the post to a random subset of its friends, each chosen independently with probability $\rho$. Out of this subset, we ignore the friends who had the post. Then, we delete this user from the current-shares list, and update the two lists with the new effective forwards. The propagation continues in this manner and terminates when the current-shares list is empty, thus completing one sample run of the CP process.

\noindent \textbf{Estimation of TeF:} 
To estimate the TeF discussed in Section \ref{sec_mean}, we create bins of equal length ($1000$ in our case). Then, we simulate the CP process $861$ times over the dataset  for $\rho = 1$. In each run and bin $n$, we count: (i) the number of transitions that have occurred and (ii) the sum total of effective forwards,  while the number of total shares belongs to that bin. These two entries are accumulated bin-wise over $861$ viral sample paths. For each bin, we divide the effective forwards by the number of transitions, which  represents the estimated TeF curve (see solid black curve in Fig. \ref{fig_synthetic_and_mean_curve}(i)).

We repeat the above routine for $\rho = 0.4, 0.6$ as well, and the corresponding estimates of $m(\cdot)$ are plotted in Fig. \ref{fig_synthetic_and_mean_curve}(i) after dividing by the respective $\rho$ values. The resultant picture  gives the confidence that one can derive the individual $m(\cdot)$ curves by using $m_N(\cdot)$ as suggested by \eqref{eqn_network_mean_func}. We also plot an approximate piece-wise linear curve (see the dashed curve in Fig. \ref{fig_synthetic_and_mean_curve}(i)), with the parameters $\overline{m}  = 21.321042, \kappa_1 = 532\times 10^{-6}, \kappa_2 = 83\times 10^{-6}$ and $\bax = 35000$ obtained using trial and error method. Henceforth, we refer this $m_N(\cdot)$ curve as the \textit{common-fit (C-fit) curve}. For some sets of simulations, again using the trial and error method, we individually choose \textit{best fit $m(\cdot)$ curve} for the given case study. 

\begin{figure}
    \centering
    \includegraphics[trim = {2.5cm 0cm 1cm 0cm}, clip, scale = 0.3]{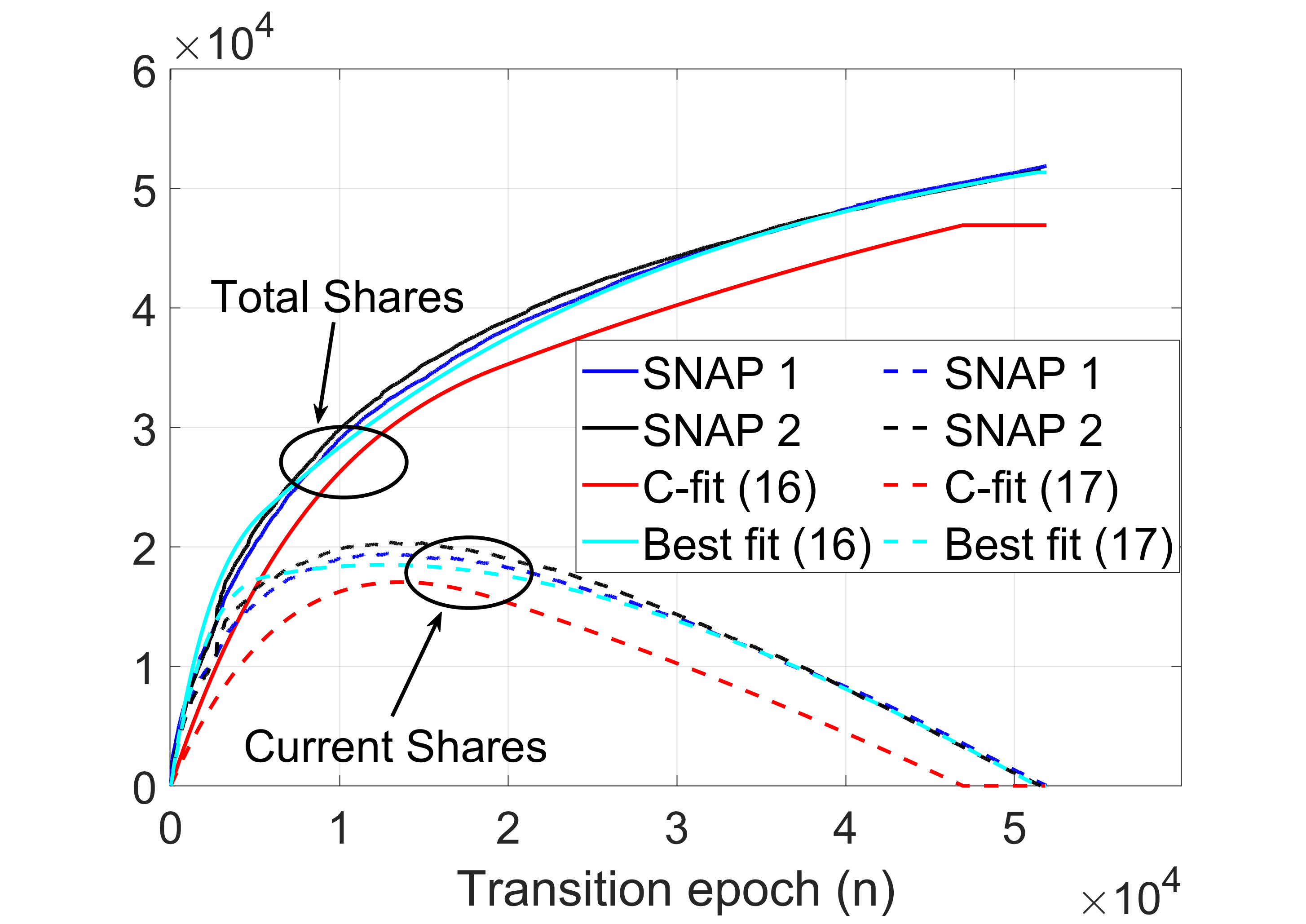}
    \caption{Validation of theoretical trajectories against two instances of CP on SNAP dataset, for $\rho = 0.2$}
    \label{fig:SNAP_2}
\end{figure}
\begin{figure*}
\begin{minipage}{0.5\textwidth}
\centering
    \includegraphics[trim = {1cm 0cm 1cm 0cm}, clip, scale = 0.3]{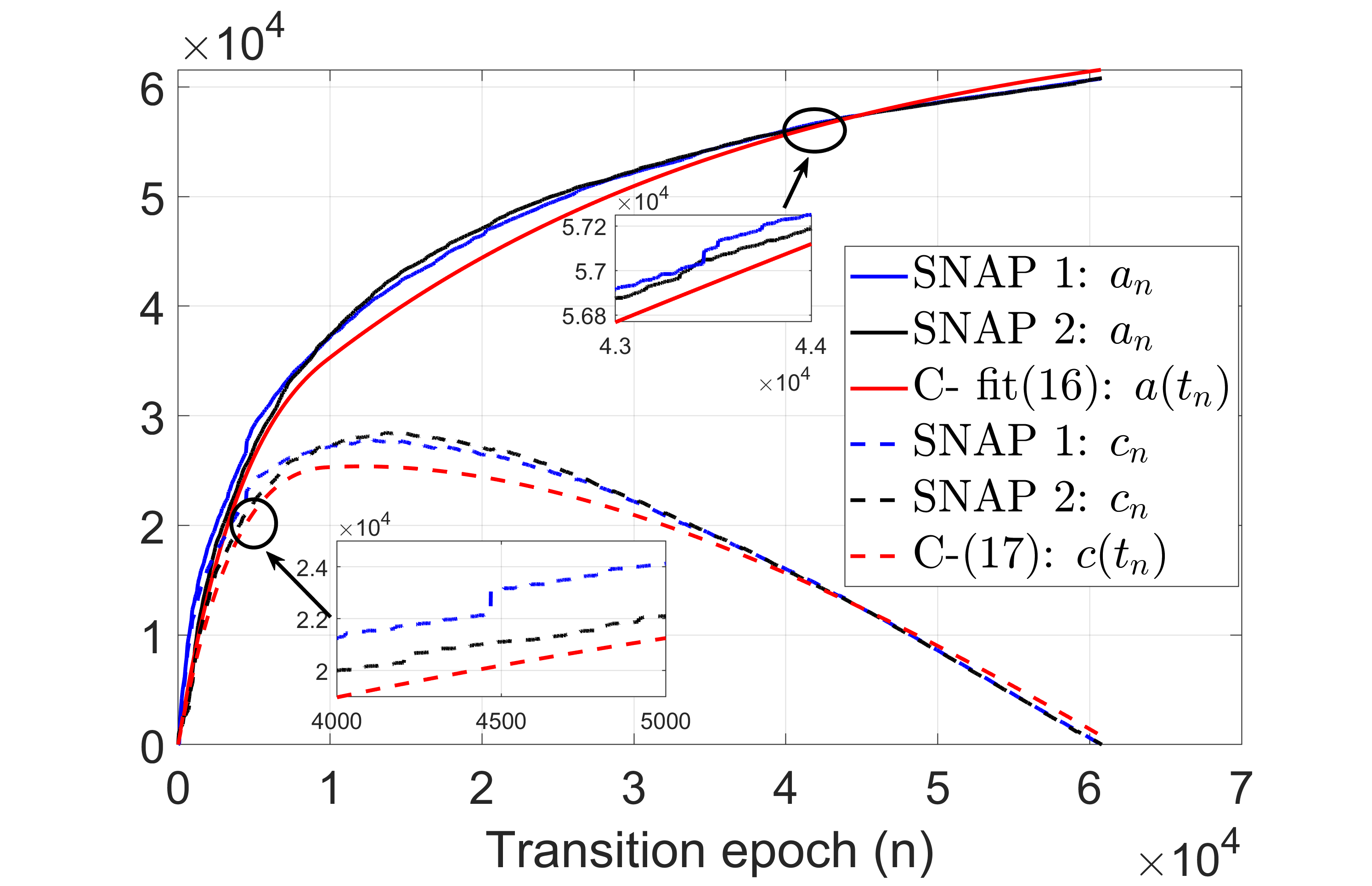}
\end{minipage}%
\begin{minipage}{0.5\textwidth}
\centering
    \includegraphics[trim = {1cm 0cm 1.5cm 0cm}, clip, scale = 0.35]{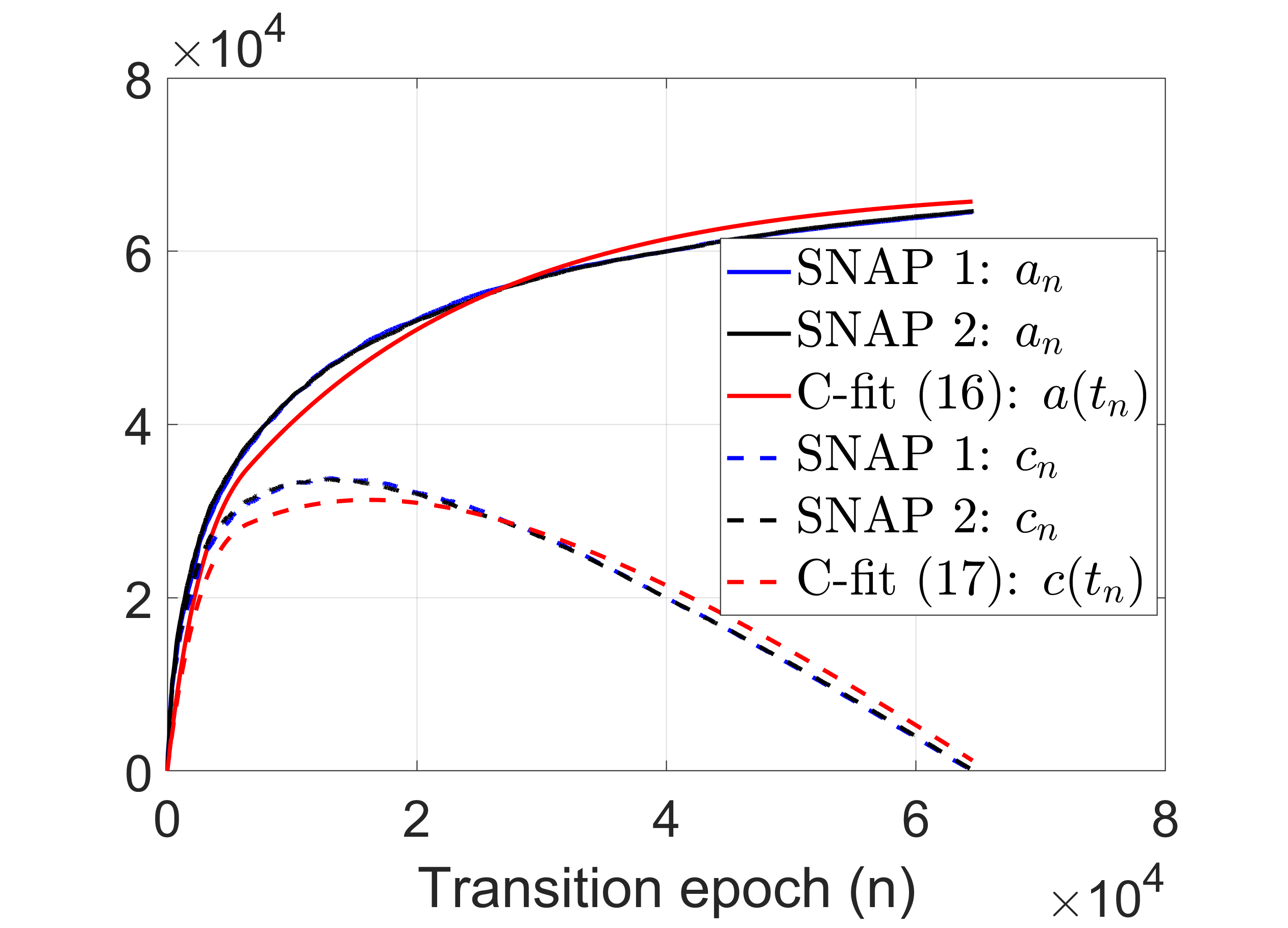}
\end{minipage}
\caption{\label{fig_SNAP}Validation of theoretical trajectories against two instances of CP on SNAP dataset, for $\rho = 0.4, 0.6$ respectively}
\end{figure*}
In Fig. \ref{fig_SNAP}, we obtain the theoretical curves for total and current shares using C-fit TeF (discussed above), which are then compared with two instances of the CP process on the SNAP dataset. The theoretical curves well approximate the dataset curves for $\rho = 0.4, 0.6$. In fact, we observed this is true for $\rho \geq 0.4$, in general. However, for lower values of $\rho$ (e.g., $\rho = 0.2$ in Fig. \ref{fig_SNAP}), the best fit curve better approximates. This suggests that one can obtain C-fit curves for higher and lower ranges of $\rho$ separately. At last in Fig. \ref{fig_synthetic_and_mean_curve}(ii), for different $\rho$ values, we show that the peak shares and max. reach from the C-fit curve approximate the respective values for an instance of CP on the dataset with a maximum error of $11.4645\%$ and $2.4837\%$ respectively (for $\rho \geq 0.4$). 

\begin{figure}[htbp]
\vspace{-3mm}
\begin{minipage}{0.5\linewidth}
\centering
    \includegraphics[trim = {1.2cm 0cm 1.2cm 0cm}, clip, scale = 0.4]{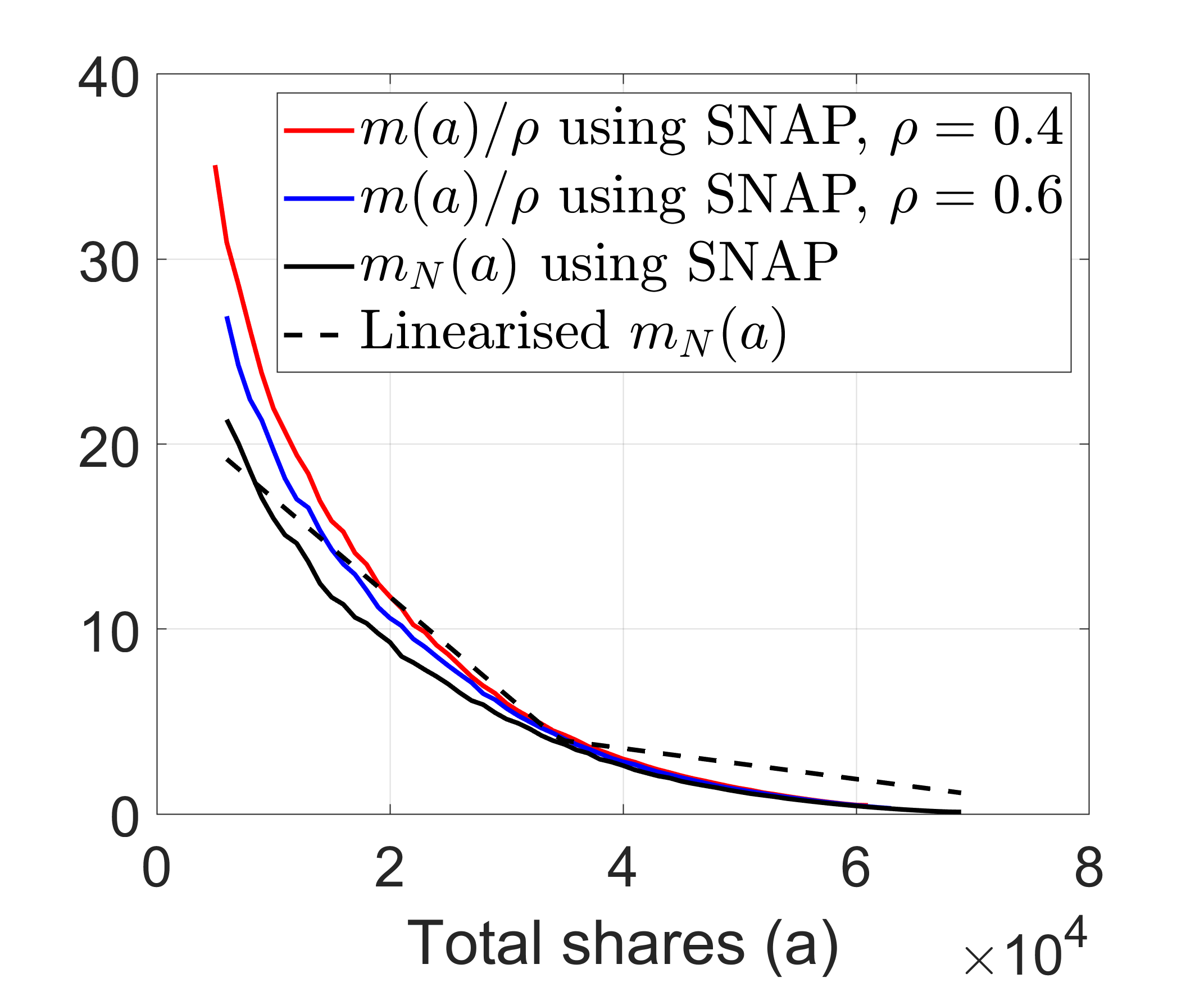}
\end{minipage}%
\begin{minipage}{0.5\linewidth}
\centering
\includegraphics[trim = {3cm 0cm 2cm 0cm}, clip, scale = 0.35]{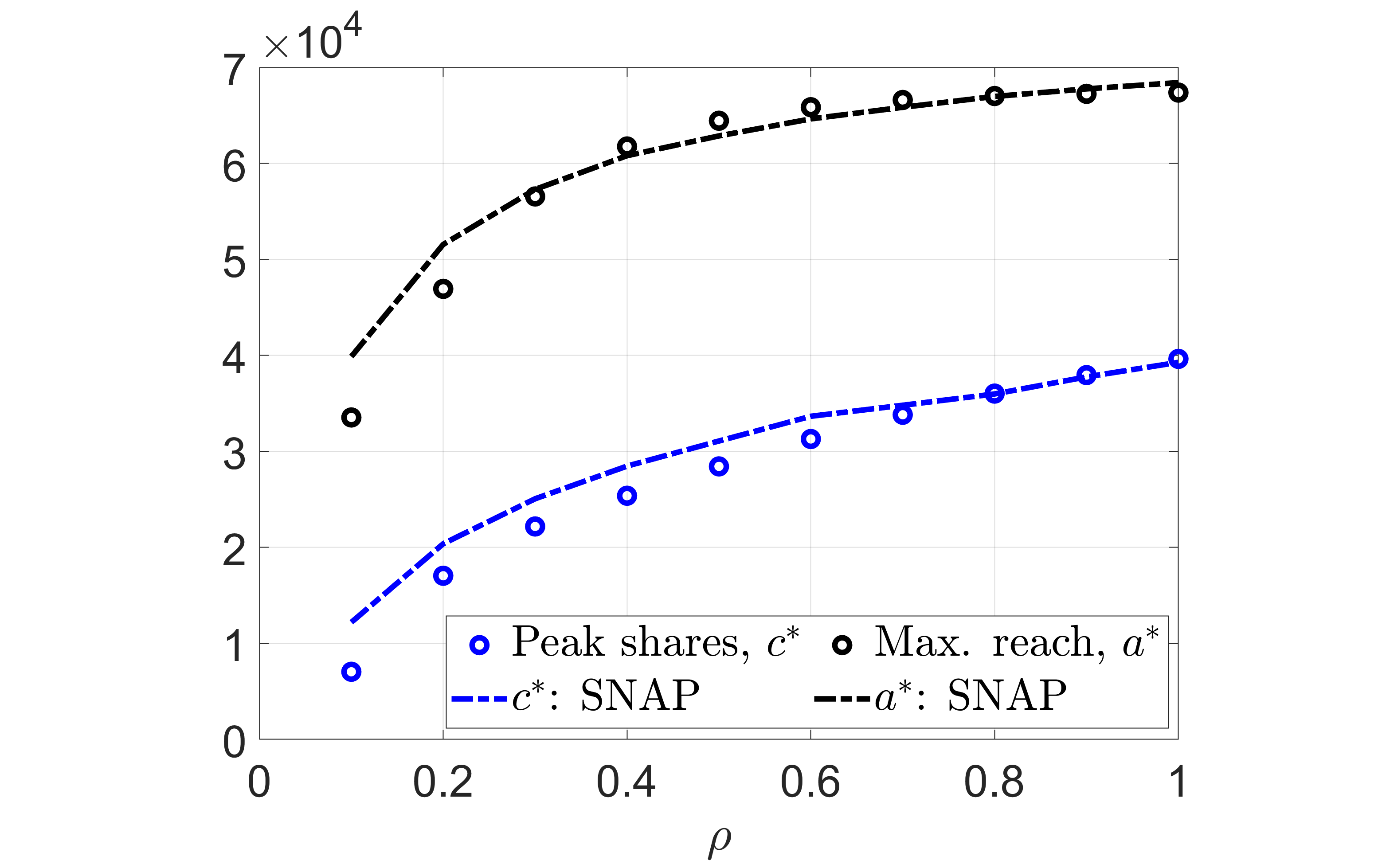}
\end{minipage}
\caption{(i) Left: piece-wise linear TeF; (ii) Right: Peak shares \eqref{peak} and Max reach  \eqref{eqn_max_share} versus SNAP estimates \label{fig_synthetic_and_mean_curve}}
\vspace{-1mm}
\end{figure}

\section{Conclusions}
In this chapter, we studied the saturation effect experienced by the total shares/copies of the post due to continual re-forwarding on OSNs. We captured the dynamics via a newly introduced, saturated total-population dependent branching process. The analysis uses the stochastic approximation technique, which provides an ODE dependent on the post's total expected (effective) forwards, TeF. We modelled TeF as a piece-wise linearly decreasing function with two slopes. The derived trajectories (dependent on four network-specific parameters) asymptotically and almost surely approximate the total and unread copies over any finite time interval. 

Unlike classical dichotomy (either explosion or extinction), the unread copies observe an explosion followed by extinction or direct extinction under saturated BPs. Interestingly, maximum reach (number of users that received the post) and the life span (number of users that read the post) are equal. We showed that derived expressions provide a good fit to the simulated instances of propagation over the SNAP dataset. Here, we fitted the parameters of TeF using a naive method; however, advanced methods of estimation/learning can improve the approximation.

\chapter{Single out fake posts: participation game and its design}  \label{ch:MFG}  

In this chapter, we continue with the problem of fake post detection on OSNs. In Chapter \ref{ch:journal2}, we assumed that a non-zero fraction of users consider the proposed warning mechanism; however, this need not always be true. Therefore, we now design a participation mean-field game\footnote{The work in this chapter is accepted and presented at the American Control Conference (ACC) 2023, see ``Agarwal, Khushboo, and Veeraruna Kavitha. ``Single out fake posts: participation game and its design.'' arXiv preprint arXiv:2303.08484 (2023).''} among the users of the OSN, which ensures the desired proportions of users and achieves the desired levels of actuality identification of posts on the OSN at the Nash Equilibrium.  

\section{Introduction}
It is well-established that fake news has disastrous implications for society. The severity has increased multi-fold with the increasing usage of OSNs\footnote{$2.93$B people are active monthly on Facebook, out of $5.3$B people using the internet in 2022 (\url{https://www.itu.int/en/ITU-D/Statistics/Pages/stat/default.aspx}, \url{https://s21.q4cdn.com/399680738/files/doc_news/Meta-Reports-Second-Quarter-2022-Results-2022.pdf}).}, e.g., several fake news were either intentionally/unintentionally shared by the users  during the COVID-19 pandemic (\cite{cinelli2020covid}).   

As said in Chapter \ref{ch:journal2}, machine learning or deep learning is one of the commonly used approaches for fake post detection (see \cite{feng2022misreporting, sharma2019combating, ruchansky2017csi, ahmed2021detecting}). However, as argued in \cite{ahmed2021detecting}, such algorithms often face difficulty in obtaining training datasets in certain languages, and it gets difficult to determine the actuality using only the content (\cite{sharma2019combating}). Hence, algorithms whose applicability is not restricted due to limited datasets and language barriers are required.

Another approach used for fake post identification is crowd-signals, which we discussed in Chapter \ref{ch:journal2}. To the best of our knowledge, only \cite{kapsikar2020controlling} attempts to guide the users in this manner, which strengthens the collective wisdom; we further build upon this idea here. 

Authors in \cite{freire2021fake} mention that the limited users' willingness to give their opinions publicly is a limitation for the application of crowdsourcing models. 
This calls for the design of an appropriate participation game which sufficiently motivates the users to provide their responses.


Recent algorithms also learn the users' credibility based on their shared posts while utilizing their signals (\cite{freire2021fake, tschiatschek2018fake}). However, it is computationally expensive to learn each user's credibility on enormously large platforms like Facebook; hence, improved algorithms with less knowledge are required. Our mechanism requires just the knowledge of the fraction  of the adversarial users (who purposely mis-tag fake posts as real).

Along with other authors in \cite{kapsikar2020controlling}, we conducted an initial study towards singling out fake posts using crowd-signals. A mechanism is designed where each user tags the post as real/fake based on its intrinsic ability to identify the actuality, the sender's tag and the system-generated warning. The warnings are generated by compiling the tags of the previous recipients. The mechanism proposed in this chapter differs from that in \cite{kapsikar2020controlling} as (i) the previous work assumed that the network has some prior knowledge about the actuality, while the OSN is oblivious in our case, (ii) we do not assume that all users participate in the tagging process, and (iii) no users foul played in \cite{kapsikar2020controlling}, while our analysis is robust against adversarial users (similar to Chapter \ref{ch:journal2}).  

We extended the work of \cite{kapsikar2020controlling} in Chapter \ref{ch:journal2}, where we considered users exhibiting different behaviours, including adversaries. In this chapter, we further extend the work by considering a mean-field  participation game among the users of the OSN. For each post on the user's timeline, the user can tag the post as real/fake. The user can utilize the warning level of that post to make a more informed decision. The main objective is to design an actuality identification (AI) $(\theta, \delta)$ game where at some Nash Equilibrium (called AI-NE), at least $\theta$ fraction of  non-adversarial users tag the fake post as fake, and not more than $\delta$ fraction of non-adversarial users mis-tag the real post as fake. Towards this, the users are rewarded if $\td$ is achieved at NE and earn more if they consider the warning level in their judgement.

We propose an easily implementable warning mechanism for the polynomial response function of the users. The designed AI game has at most two NEs - one NE always exists and is always AI, and the other NE, if it exists, achieves the desired $\delta$-detection of the real posts. We also identify the conditions required to design an AI game.  

In Chapter \ref{ch:journal2}, we considered a linear response function for the warning-seeking users, but here we model the polynomial response function. Unlike the work in \cite{kapsikar2020controlling} and Chapter \ref{ch:journal2}, here, we do not assume that OSN has any prior knowledge about the actuality of the posts. Further, in Chapter \ref{ch:journal2}, we assumed that there is a given proportion of users of different types, while here, we design a game which results in desired proportions at Nash equilibrium. Furthermore, in Chapter \ref{ch:journal2}, the objective function was to maximize the proportion of fake tags for the fake post while ensuring the $\delta$-threshold for the real post; as said above, in this chapter, we again keep the objective same for the real post, but now we can design a game to achieve any $\theta$-threshold for the fake post (which can be higher than the optimal value achieved in Chapter \ref{ch:journal2}) under certain conditions. Thus, this chapter considers a more general framework and can achieve desired results under certain conditions.


\section{System Description}\label{sec_system_desc}

Consider an online social network (OSN) where   content providers create and share  fake ($F$) or real ($R$) content in the form of posts. More often than not, neither the users of the OSN nor the network are aware that the given post is fake/real. Additionally, there is an \textit{adversary} in the system who designs fake posts and creates fake accounts/employs bots to confuse the other users about the actuality of its post by \textit{declaring fake posts as real}; \textit{they do not mis-tag the real post}. We refer to all such fake accounts as adversarial ($a$) users. In the presence of such $a$-users, the OSN is interested in designing a mechanism to detect the actuality of any given post. In particular, for any $\theta \gg \delta > 0$, the OSN aims to guide $\na$ users (referred to as just `users' or $na$-users) such that at least $\theta$-fraction of them correctly detect the fake post (denoted as $F$-post) as fake, and at maximum $\delta$-fraction of them consider the real post ($R$-post) as fake. 

\begin{figure}[htbp]
    \centering
    \includegraphics[trim = {1cm 1.7cm 1cm 1.7cm}, clip, scale=0.35]{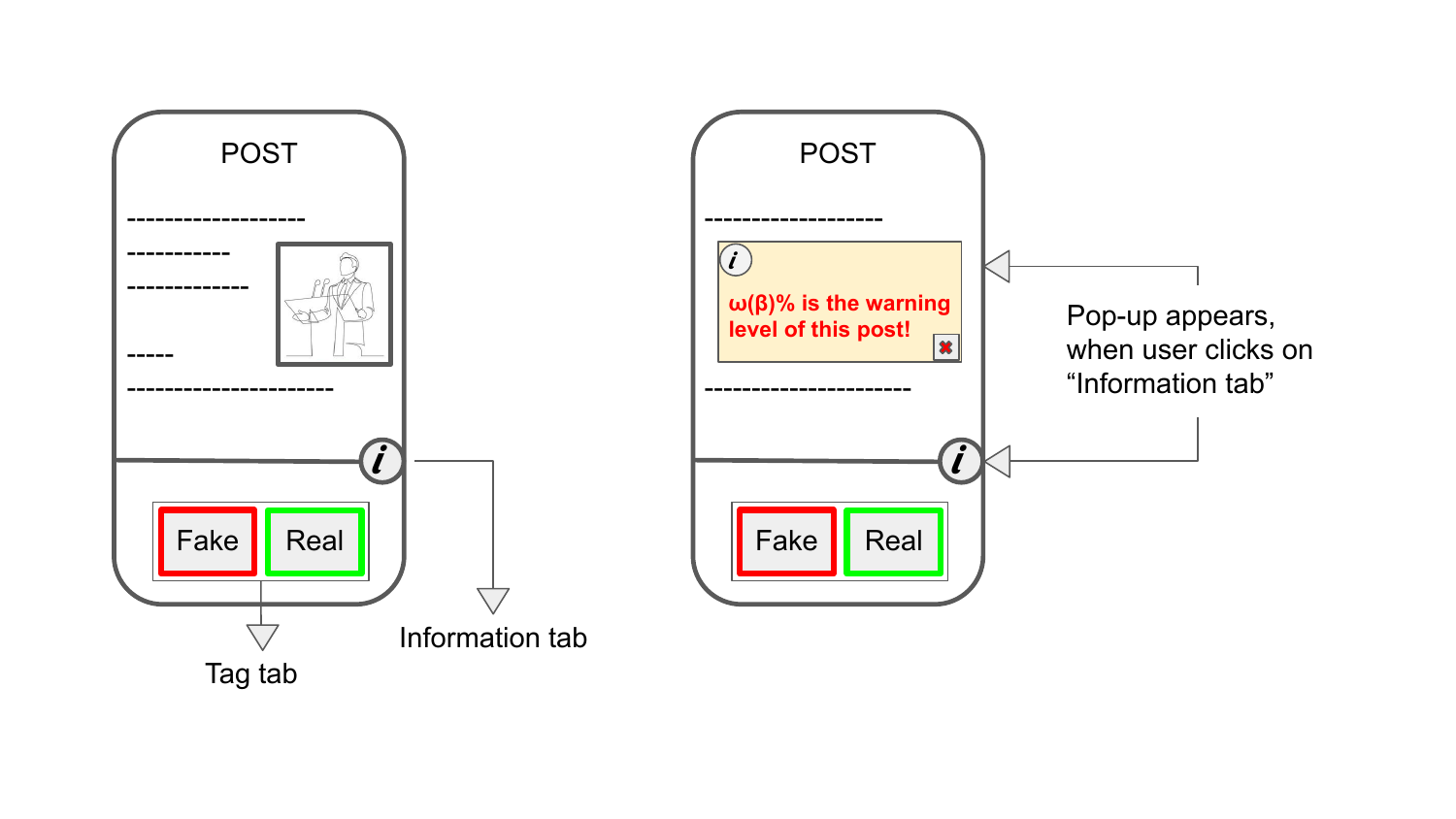}
    \caption{Warning mechanism for each post on OSN}
    \label{fig:mechanism}
\end{figure}
Towards this, motivated by the work in \cite{kapsikar2020controlling}, we propose the following new warning mechanism. For each post shared on the OSN, the OSN additionally designs two tabs - the tag tab and the information tab\TR{ (see \cite[Fig. 1]{arxiv})}{ (see \autoref{fig:mechanism})}. When a user clicks the former tab, it directly tags the post (as $R$ or $F$) based on its innate capacity to judge the actuality, while on clicking the latter tab, it tags additionally using the warning level. By the innate capacity, any user judges the $F$-post as fake with probability (w.p.) $\alpha_F > 0$ and mis-judges  the $R$-post as fake w.p. $\alpha_R > 0$; assume $\Delta_R > 1$ for $\Delta_u := \nicefrac{\alpha_F}{\alpha_u}$. Recall that $a$-users mis-tag any $F$-post w.p. $1$.

\textit{Observe that the difference between $\alpha_F$ and $\alpha_R$ might be minimal, so users might not accurately distinguish between the $F$/$R$-posts without any additional information.} The OSN aims to leverage upon this difference and accentuate the capacity of the users by providing them with additional information based on tags of other users (i.e., via collective wisdom) to single out the fake posts; this information is based on the responses of all the users, as the OSN can not differentiate between $a$-users and others. Now, the users can access this additional information by clicking on the information tab, after which a pop-up window appears with a warning level, $\omega$, whose design is discussed later in detail. The warning level influences the decision of the users to tag the post as real/fake. 

In reality, some users might not participate/tag, some might participate only based on their innate capacities, and the rest participate by considering both the warning provided by the information tab and their innate capacity. We represent the actuality of the underlying post as $u \in \{F, R\}$ and assume that the post is fake w.p. $p$, i.e., $P(u = F) = p$ with $p \in (0,1)$. Thus, the second type of users tag the $u$-post as fake w.p. $\alpha_u$. While, the third type of users tag the post as fake w.p. $r(\alpha_u, \omega(\beta))$ where $r: [0,1] \times \mathbb{R}^+ \mapsto [0, 1] $, the \underline{response function}, is defined as:
\begin{align}\label{eqn_response_func}
r(\alpha, \omega) = \min\{h(\alpha, \omega), 1\}.
\end{align}In the above, the function $h(\cdot, \cdot)$ is such that $r(\cdot, \cdot)$ is Lipschitz continuous in $\omega$ and $\omega : [0,1] \mapsto \mathbb{R}^+ $ is the warning level designed by the OSN - the design of the warning depends on the fraction of fake tags so far.
The warning continuously updates as more users tag the post; we discuss these dynamics in Section \ref{sec_participation_game}. We anticipate that users perceive the post as fake with a higher probability if the warning level is high. Further, the response ($r$) towards fake tagging the post should naturally increase with innate capacity, $\alpha_u$.

\subsection{Objectives of the OSN}\label{sec_objectives}
The OSN aims to design an appropriate warning mechanism. Its primary objective is to achieve \underline{$\td$}, for some $\theta > \alpha_F$ and $\delta > \alpha_R$, defined as follows:
\begin{itemize}
    \item if the underlying post is fake, then at least $\theta$ fraction of non-adversarial users tag it as fake, or
    \item if the underlying post is real, then at most $\delta$ fraction of non-adversarial users tag it as fake.
\end{itemize}Thus, the OSN is willing to compromise a slightly larger fraction ($\delta$) of mis-tags for the real posts to obtain the desired level ($\theta$) of identification of fake posts. However, one may have to ensure a sufficient number of users participate in the tagging process. Towards this, the OSN induces a participation game among users, which we discuss next.

\section{Participation game by OSN}\label{sec_participation_game}
Let $n$ be the number of users on the OSN. Let $N_a$ be the $a$-users out of $n$ users. The remaining users on the OSN have three actions/strategies, $s$: (i) not to participate (say $s = 0$), (ii) participate and tag only based on intrinsic ability (say $s = 1$), and (iii) participate and tag based on both intrinsic ability and the warning level (say $s = 2$); let us call the   users as type $0, 1, 2$ in the order of described actions. Let $N_i$ be the number of type $i$   users for $i \in \{0,1, 2\}$. Define $\mu_i := \nicefrac{N_i}{n}$ for $i \in \{0,1,2,a\}$; let $\mua \in [0,1)$ be fixed. \textit{We assume that OSN knows the fraction $\mua$, but not the identity of the adversarial users.} 

To motivate the   users to participate, the OSN provides publicly visible attributes to each user. For example, the OSN might reflect each user's average participation status on its profile, which gets updated with each post. We believe that such public recognition leads to pro-social behaviour (e.g., \cite{reiffers2019reputation}), i.e., participating in the tagging process. Thus, for the perceived stardom among its peers, each type $1, 2$ user receives a positive utility, say $\cp > 0$. 

The OSN further announces that the participants would get a total reward of $nR$, if the mechanism achieves $(\theta, \delta)$-success; furthermore, each type $2$ user gets $\gamma$ times\footnote{If a user clicks on information tab, then the OSN perceives such a user to be type $2$ user.} the reward provided to the type $1$ user (where $\gamma > 1$); observe, $R$ and $ \gamma$ are  the design parameters for the OSN. \textit{We assume that $a$-users directly tag the post like type $1$ users; and as already mentioned, OSN can not differentiate between $a$ and type $1$ users.} Thus, the OSN provides the same reward to $a$-users, as it does to type $1$ users. Hence, upon success, a type $1$ user gets a reward of  $\nicefrac{nR}{(N_1 + N_a + \gamma N_2)} =  \nicefrac{R}{(\mu_1 + \mua +\gamma\mu_2)}$, and a type $2$ user gets $\nicefrac{\gamma R}{(\mu_1+ \mua + \gamma\mu_2)}$ reward. It usually takes a long time for the OSN to ascertain the actuality of the posts, and thus, the reward is delivered to the users after the confirmation of the actuality.

The type $0$ users earn  a positive utility for the perceived comfort they experience by not participating and a negative utility for the public disapproval; this amounts to a consolidated utility $\cnp \in \mathbb{R}$ for each such user. We assume $\cp \geq \cnp$; if $\cp < \cnp$, then the OSN needs to provide an additional participation reward ($\geq \cnp - \cp > 0$) to all participants, irrespective of the outcome of the game. 

A type $2$ user also incurs a cost, $\ce > 0$, for the extra time invested in the process.

Thus, given the mechanism announced by the OSN, the   users make choices and participate in the tagging process. The tagging process may take some finite positive time as the users tag asynchronously. However, since the users make a \underline{participation choice} oblivious to the choices of the others, one can model it as a simultaneous move strategic form game after capturing the tagging process by the fixed-point (FP) equation in \eqref{eqn_fp} described next. Observe here that the tagging process, and hence utility (of any user) depends on the strategy profile ${\bf s}=(s_1, \cdots, s_n)$ chosen by all the users.

Let $X_u$ be the number of fake tags provided by the users for the $u$-post\footnote{Henceforth, we reserve sub-script $u$ for the actuality of the post, i.e., at places, we may not explicitly write that $u \in \{R, F\}$.}, where $u \in \{R, F\}$. Define $\beta_u := \nicefrac{X_u}{(N_1+N_2 + N_a)}$ as the proportion of fake tags (observe, $N_1 + N_2 + N_a$ is the number of participants). The proportion satisfies the following FP equation: 
\begin{align}\label{eqn_fp}
    \beta_F(\bmu) &= \alpha_F \eta + (1-\eta-\eta_a) r(\alpha_F, \omega(\beta_F(\bmu))), \\
    \beta_R(\bmu) & = \alpha_R \eta + (1-\eta-\eta_a) r(\alpha_R, \omega(\beta_R(\bmu))), \mbox{ where} \nonumber \\
    \bmu &= (\mu_0, \mu_1, \mu_2), \ \eta = \eta(\bmu) := \frac{\muone}{\muone+\mutwo+\mua}, \mbox{ and } 
    \eta_a = \eta_a(\bmu) := \frac{\mua}{\muone+\mutwo+\mua},\nonumber
\end{align}
with the terms explained as below:

$\bullet$ $\eta$ and $\eta_a$ represent the fraction of type $1$ and $a$-users among participants, and empirical measure $\bmu$ is defined by ${\bf s}$ $=(s_1, \cdots, s_n)$ as:
$$
    \mu_i = \frac{\sum_{j = 1}^n 1_{\{s_j = i\}}}{n}, \mbox{ for } i \in \{0,1,2\};
$$
$\bullet$ Let $Z_{i}^{(j)}$ be the indicator that $i$-th user, among type $j$, has tagged the post as fake. The fraction of type $1$ users that tagged the post as fake equals $\sum_i \nicefrac{Z_i^{(1)}}{N_1} \approx \alpha_u$, when $n$ is sufficiently large  and $\mu_1 > 0$ (w.p.~$1$, by law of large numbers). Thus, the first term $\alpha_u \eta$ results from $\sum_i \nicefrac{Z_i^{(1)}}{(N_1+N_2+N_a)}$, for $u\in\{F, R\}$;

$\bullet$ While the tagging process is ongoing, the type $2$ users are provided refined warnings based on the tags of the previous users. Thus, when $n$ is large, we anticipate the warning to stabilise. The stabilised value is reflected by the tags of the future users as well. Then, $r(\alpha, \omega)$ is the response of the type $2$ users at stabilised warning level $\omega$, resulting from the stabilised fraction of fake tags $\beta_u$. Thus, the second term (similar to the first term) approximately equals $\sum_i \nicefrac{Z_i^{(2)}}{(N_1+N_2+N_a)} \approx (1-\eta-\eta_a) r(\alpha, \omega(\alpha, \beta_u))$.  The overall fraction of fake tags ($\beta_u$) equals the sum of the corresponding terms; hence, the FP equation;


$\bullet$ \textit{When the number of OSN users increases (i.e., as $n \to \infty$), the limit fraction of fake tags (for any given strategy profile of the users) is indeed given by the FP equation \eqref{eqn_fp}}; Lemma \ref{lemma_beta} of next subsection proves this.

The OSN is   interested in the fake tags from $na$-users only, $\nicefrac{X_F}{(N_1+N_2)}$, but it can not distinguish between the tags from $na$ and $a$-users; it can observe only the overall fraction of fake tags, $\beta_u$. Hence, \textit{$\td$ is redefined as follows} in terms of $\beta_u$ (see \eqref{eqn_fp}):

\vspace{-4mm}
{\small
\begin{align}\label{eqn_td}
\begin{aligned}
    \beta_F(\bmu) &\geq \theta \left(1 - \eta_a(\bmu)\right) := \theta_a(\bmu), \mbox{ and} \\
    \beta_R(\bmu) &\leq \delta \left(1 - \eta_a(\bmu)\right) := \delta_a(\bmu).
\end{aligned}
\end{align}}
Thus, $\td$ depends on $\textbf{s}$ only via $\bmu$, when $n$ is large (also, see \TR{\footref{footnote_uni_dev}}{footnote \ref{footnote_uni_dev}}). Let $P_{\bmu}(\cdot) := P(\cdot|\bmu)$ be the corresponding conditional probability. From \eqref{eqn_td}, the probability of mechanism being $\td$ful, given the empirical ($\textbf{s}$-dependent) distribution $\bmu$ and parameterised by $(\theta, \delta)$, is given by (recall $p$ is the probability of underlying post being fake):
\begin{eqnarray}\label{eqn_prob_success}
P_{\bmu}(S; \theta, \delta) = p  P_{\bmu}(\beta_{F, \overline{k}} \geq \theta_a) +  (1-p)  P_{\bmu}(\beta_{R, \overline{k}} \leq \delta_a), 
\end{eqnarray}where $\beta_{u, k}$ for any  $1\leq k\leq n$ represents the proportion of the fake tags for the $u$-post, immediately after the  $k$-th participant tags (see details in sub-section \ref{subsec_ode}), and ${\overline k} = n(1-\mu_0)$ is the index of the last participant to tag. Recall that the game is played continually, however, OSN will design a mechanism  anticipating the eventual responses of the users.
Finally, the utility of $i$-th user is\footnote{\label{footnote_uni_dev}Here, the influence of a single user's action is negligible, as is usually the case in MFGs; observe that if, for example, $\muone, \mutwo$ fractions correspond to $(s_i = 1, s_{-i})$, then the fractions corresponding to $(s_i = 2, s_{-i})$ equal $\muone - \frac{1}{n}, \mutwo + \frac{1}{n}$ which respectively converge to $\muone, \mutwo$ as $n \to \infty$.} ($\textbf{s} = (s_i, \textbf{s}_{-i})$, a standard game-theoretic notation):
\begin{align}\label{eqn_util}
U(s_i, \textbf{s}_{-i}) = 
\begin{cases}
\cnp, &\mbox{if } s_i = 0,\\
\cp + \frac{R P_{\bmu}(S; \theta, \delta)}{\muone + \mua + \gamma \mutwo}, &\mbox{if } s_i = 1,\\
\cp - \ce + \frac{\gamma R P_{\bmu}(S; \theta, \delta)}{\muone + \mua + \gamma \mutwo}, &\mbox{if } s_i = 2.
\end{cases}
\end{align}

This completes the description of the participation game represented by $\big<\{1, \dots,  n-N_a\},$ $ \{0,1,2\}, (U_i) \big>$ for any given $N_a$. We derive the solution of this game, Nash Equilibrium\TR{}{\footnote{
\begin{definition}\cite{narahari2014game}
Given a strategic form game $\left<\{1, \dots, n\}, (S_i), (U_i) \right>$, the strategy profile $\textbf{s}^* = (s_i^*)_{i = 1}^n$ is called a \underline{pure
strategy Nash equilibrium} if $U_i(s^*) = \mbox{argmax}_{s\in S_i} U_i(s, \textbf{s}_{-i}^*)$ for each $i \in \{1, \dots, n\}$.
\end{definition}}} (NE).
From \eqref{eqn_util}, $U(s_i, \textbf{s}_{-i}) = U(s_i, \bmu)$ for $n$ large enough (see \footref{footnote_uni_dev}). As a result, the utility of any user depends on its strategy and the relative proportion of users, $\bmu$. It is thus appropriate to analyse such large population game using mean field game (MFG) theory. For MFGs with countable action set (as described below), the solution concept is again NE, and is equivalently given by the following (see \cite{carmona2018probabilistic}):
\begin{definition}\label{defn_NE_MFG}
Consider a mean field game, with a countable strategy set, $S$. Let $\mu_s$ represents the fraction of players choosing action $s$, for some $s\in S$. Further, let the utility of any player be $U(s, \bmu)$, for $\bmu = (\mu_s)_{s \in S}$. 
Then, $\bmu^* = (\mu_s^*)_{s \in S}$ is called a \underline{Nash equilibrium of the MFG} if $\S(\bmu^*) \subseteq \mbox{argmax}_s U(s, \bmu^*)$, where $\S(\bmu) := \{s : \mu_s > 0\}$ represents the support. \label{defn_NE}
\end{definition}
We now consider the MFG variant of the participation game by letting $n \to \infty$. The utilities remain as in \eqref{eqn_util}, but the probability of success\footnote{The $\lim_{k \to \infty} P_{\bmu}(\cdot)$ may not exist for all distributions $\bmu$, and hence it is appropriate to define probability of success with $\liminf$ in \eqref{eqn_prob_success}.} changes to (when $\mu_0 < 1-\mua$):
\begin{align}\label{eqn_prob_success_n_large}
\begin{aligned}
P_{\bmu}(S; \theta, \delta) &= p \liminf_{k \to \infty} P_{\bmu}(\beta_{F, k} \geq \theta_a) +  (1-p) \liminf_{k \to \infty} P_{\bmu}(\beta_{R, k} \leq \delta_a),
\end{aligned}
\end{align}
and $P_{(1-\mua, 0, 0)}(S; \theta, \delta) := 0$.
Let $\bmu^* = (\muzero^*, \muone^*, \mutwo^*)$ be the NE of MFG (when its exists), where $\mu_i^* \geq 0$ for $i \in \{0,1,2\}$ and $\sum_{i=0}^2 \mu_i^* = 1-\mua$. The OSN aims to appropriately design ($R, \gamma, \omega$) such that the equilibrium outcome of the resultant game achieves $\td$; \textit{we represent the game compactly by $\G(R, \gamma, \omega)$}. We call the MFG as Actuality identification (AI) game as per the definition below:
\begin{definition}\label{defn_AI}
A game $\G(R, \gamma, \omega)$ is called an \underline{AI game} if there exists parameters $R > 0, \gamma \geq 1$ and a warning mechanism $\omega$ such that for some NE $\bmu^*$ of the participation game  the following is true:
\begin{align*}
\liminf_{k \to \infty} P_{\bmu^*}(\beta_{F, k} \geq \theta_a) &= 1 \mbox{ and }
\liminf_{k \to \infty} P_{\bmu^*}(\beta_{R, k} \leq \delta_a) = 1.
\end{align*}
Such an NE is called an \underline{AI-NE}.
\end{definition}
In simpler words, a participation game is called an AI game if at some NE, $\td$ is achieved. 
Observe that the Definition \ref{defn_AI} requires that $P_{\bmu^*}(S;\theta, \delta) = 1$ (see \eqref{eqn_prob_success_n_large}), which implicitly demands that the random tagging-dynamics driven by the warning mechanism leads to $\td$ w.p. $1$. In other words, the two limit infimums in \eqref{eqn_prob_success} equal $1$.

We derive sufficient conditions for designing an AI game after providing the analysis of the MFG in the next section. 
For the sake of clarity, we re-state that the design parameters of the system are $R, \gamma$, along with the warning mechanism $\omega$, the parameters $\alpha_R, \alpha_F, \cp, \cnp, \ce$ are user-specific, while $p$ is a post specific parameter. We henceforth assume the following: 

\noindent \textbf{(P)} Assume $R > 0, \gamma > 1$, $\alpha_F \in (\alpha_R, 1)$, $\cnp \leq \cp$, $\ce > 0$, $p \in (0,1)$, $\mua \in [0,1)$. 


\section{MFG: Analysis and Design}\label{sec_MFG}
We first derive the limit of $\beta_{u, k}(\bmu)$ that defines \eqref{eqn_prob_success_n_large} for any $\bmu$. Then, we appropriately choose $R, \gamma$ and $\omega$ such that the resultant is an AI game.

\subsection{Tagging dynamics}\label{subsec_ode}

For any given $\bmu$, recall that users asynchronously visit the OSN and provide the tag; some users also utilise the warning level.\footnote{We skip explicit mention of the dependence of various entities (e.g., $\eta, \eta_a$) on  $\bmu$ at few places for simplifying notations and improving on clarity.} This leads to continuous-time evolution of the proportion of fake tags, ($\beta_{u, k}$), and the corresponding warning levels, $(\omega(\beta_{u, k}))$. However, it is sufficient to observe the tagging process whenever a user decides to tag, i.e., the embedded process; let $k \in \mathbb{Z}^+$ be the index of such decision epochs. The time duration between two decision epochs must follow some distribution, however, it's specific details are immaterial for the study. At any decision epoch, the participant can be an $a$-user, type $1$ or $2$ user w.p. $\eta_a$, $\eta$ or $1-\eta-\eta_a$ respectively.

Let $X_{u,k}$ be the number of fake tags for $u$-post at $k$-th epoch, where $u \in \{F, R\}$ is fixed. The fraction of fake tags at $(k+1)$-th epoch, $\beta_{u, k+1}(\bmu)$, can then be written as:
\begin{align}\label{eqn_dynamics_beta}
\begin{aligned}
    \beta_{u, k+1} &:= \frac{X_{u,k+1}}{k+1} = \frac{ X_{u,k} + 1_{\{\mbox{tag for $u$-post} = F\}}}{k+1} \\
    &= \beta_{u,k} + \frac{1}{k+1} L_{u,k}, \mbox{ where }  \\
    L_{u,k} &:= 1_{\{\mbox{tag for $u$-post} = F\}} - \beta_{u,k}.
\end{aligned}
\end{align}
This iterative process can be analysed using stochastic approximation tools (see \cite{kushner2003stochastic}). Ordinary Differential Equation (ODE) based analysis is a common approach to study such processes - define the conditional expectation of $L_{k, u}$ with respect to sigma algebra, ${\mathcal F}_{k, u} := \sigma\{X_{u,j}: j < k\}$:
\begin{align}\label{eqn_cond_exp}
\begin{aligned}
    E[L_{u,k}|{\mathcal F}_k] &= g_u(\beta_{u, k}), \mbox{ where}\\
g_u(\beta) := 
     \alpha_u \eta + (1-&\eta-\eta_a) r(\alpha_u, \omega(\beta)) - \beta. 
\end{aligned}
\end{align}
Then, the dynamics in \eqref{eqn_dynamics_beta} can be captured via the following autonomous ODE (proved in Lemma \ref{lemma_beta} given below):
\begin{align}\label{eqn_ode}
    \dot{\beta_u} = g_u(\beta_u).
\end{align}
The right hand side of the above ODE is Lipschitz continuous, and thus has unique global solution (see \cite[Theorem 1, sub-section 1.4]{piccinini2012ordinary}).
Define the domain of attraction (DoA)
$$
\cD_u := \{\beta \in [0,1]: \beta_u(t) \stackrel{t \to \infty}{\longrightarrow} \cA_u, \mbox{ if } \beta_u(0) = \beta\},
$$for asymptotically stable (AS) set $\cA_u$ of the ODE \eqref{eqn_ode} in the interval $[0, 1]$ (see \cite{piccinini2012ordinary}).
Assume the following: 

\noindent \textbf{(A)} $P(\beta_{u, k} \in \cD_u \mbox{ infinitely often}) = 1.$
\begin{lemma}\label{lemma_beta}
Under \textbf{(A)}, the sequence $(\beta_{u,k})$ converges to $\cA_u$ w.p. $1$, as $k \to \infty$.  \eop
\end{lemma}
The proof of above Lemma is in \TR{\cite[Appendix]{arxiv}}{Appendix \ref{appendix_MFG}}.
It is clear that the attractor set ($\cA_u$) depends on the choice of warning mechanism $\omega(\cdot)$, $\bmu$ and the response function $r(\cdot, \cdot)$. \textit{Henceforth, we assume $\cA_u$ to be a singleton (for each $\bmu$) and provide the analysis; we prove this assumption and \textbf{(A)} for a special response function $r$ and the correspondingly chosen $\omega$ in Section \ref{sec_response1}}. One needs to characterise $\cA_u$ and prove \textbf{(A)} for other response functions considered in future. Thus, by Lemma \ref{lemma_beta}, $\beta_u \in \cA_u$ uniquely signifies the eventual fraction of fake tags for the $u$-post. Hence, if $\cA_u = \{\beta_u^*\}$ for both $u$, then \eqref{eqn_prob_success_n_large} simplifies to:
\begin{align}\label{eqn_simplified_prob_success}
P_{\bmu}(S; \theta, \delta) = p 1_{\{\beta_{F}^* \geq \theta_a(\bmu)\}} + (1-p) 1_{\{\beta_R^* \leq \delta_a(\bmu)\}}.
\end{align}

\subsection{Design of AI game}\label{sec_analysis}
Any $\bmu$ that satisfies Definition \ref{defn_NE_MFG} qualifies to be a NE of the MFG. However, the OSN is interested in designing a game (i.e., choosing $R, \gamma, \omega$) for which at least one AI-NE exists (see Definition \ref{defn_AI}). We achieve the same in Theorem \ref{thrm_AI}. We will further derive the conditions under which the game has only AI-NE for a specific response function in Section \ref{sec_response1}.

Before we state the result, define \underline{$\bmu_x := (0,x, 1-x-\mua)$}  and \underline{$\beta_u^x := \beta_u^*(\bmu_x)$}, i.e., the attractor obtained via Theorem \ref{thrm_AI} for $\bmu = \bmu_x$, for any $x \in [0, 1-\mua]$, for $u\in\{R,F\}$.

Consider a response function $r(\alpha, \omega)$. Suppose there exists a warning mechanism $\omega$   that satisfies the following conditions related to the ODE \eqref{eqn_ode}  for some $\eta \in (0,1-\mua)$:

\noindent (B.i) $\exists$ a  $\beta_R^\eta  \in [0,\delta_a(\bmu_\eta)]$ such that $g_R(\beta_R^\eta) = 0$, and 

\noindent (B.ii) $\exists$ a $\beta_F^\eta \in [\theta_a(\bmu_\eta),1]$ such that $g_F(\beta_F^\eta) = 0$,

\noindent (B.iii) there are no other equilibrium points of the ODE \eqref{eqn_ode} in $[0,1]$ for each $u$, with respect to $\bmu_\eta$, and

\noindent (B.iv) $\frac{\partial g_u(\beta_u^\eta)}{\partial \beta_u^\eta} < 0$ for each $u$. 

For such a pair of $(r, \omega)$, one can anticipate that the tagging dynamics converge to a unique limit point for each $u$, leading to an AI game. This indeed is true as claimed below (proof in \TR{\cite[Appendix]{arxiv}}{Appendix \ref{appendix_MFG}}).
\begin{theorem}\label{thrm_AI}
Consider a pair of response and warning mechanism that satisfy assumptions (B.i)-(B.iv). Then, $\beta_u^\eta$ is AS with DoA as $[0,1]$ for each $u$. Further, choose $R, \gamma$:
\begin{align}\label{eqn_R_gamma}
\begin{aligned}
\gamma &> \underline{\gamma}(\eta) := \frac{1}{1-p}\left(\frac{1-(\eta+\mua)(1-p)}{1-\eta-\mua} \right) \mbox{ and }
R = \ce \left(1 -\eta - \mua + \frac{1}{\gamma-1}  \right).
\end{aligned}
\end{align}Then, $\G(R, \gamma, \omega)$ is an AI game with the following properties:

\noindent (a) $\bmu_\eta$ is an AI-NE,

\noindent (b) any $\bmu_x$, with $x \in [0, \eta)\cup\{1-\mua\}$ is not an NE, and 

\noindent (c) any $\bmu$ with $\mu_0 > 0$ can not be a NE.  \eop
\end{theorem}

Thus, if the OSN chooses $(R, \gamma, \omega)$ as per Theorem \ref{thrm_AI}, then the resultant is an AI game. Observe that at AI-NE ($\bmu_\eta$), there is a non-zero proportion of type $1$ and $2$ users. The OSN is able to motivate all the $na$-users to participate in the tagging process at $\bmu_\eta$. 

It can be seen from \eqref{eqn_R_gamma} that the OSN can monetarily benefit (reduced $R$) by either choosing a larger $\gamma$ (a bigger disparity between rewards provided to type $1$, $2$ users) or larger $\eta$ (that leads to a larger fraction of type $1$ users). Interestingly, $R$ can be reduced to an arbitrarily small value (observe the infimum of the achievable $R$ equals $0$). Further, even if the perceived cost of processing a warning ($\ce$) is high, one can design the desired AI game by appropriately scaling $\gamma, \eta$ with the same reward, $R$. 

Any NE requires a mixed behaviour; if all the users consider the warning ($\mu_2 = 1-\mua$) or tag only based on their intrinsic abilities ($\mu_1 = 1-\mua$), then there is no NE. 

Ideally, the OSN would want to design a game where any NE, $\bmu_x^*$, is an AI. Theorem \ref{thrm_AI} provides such guarantees only for $x \in [0, \eta]\cup\{1-\mua\}$ (any such $x \neq \eta$ is not a NE). We next delve into the remaining configurations for a specific response function. 





\vspace{-5mm}
\section{A specific response function}\label{sec_response1}
In this section, we specifically consider a class of polynomial response functions, $r(\alpha, \omega) = \min\{h(\alpha, \omega), 1\}$ with $h(\cdot, \cdot)$ defined as (extension of the linear response in \cite{kapsikar2020controlling}):
\newpage
\begin{align}\label{eqn_response1}
    h(\alpha, \omega) = c \alpha^a \omega^b, \mbox{ where } a, b, c \in \mathbb{R}^+.
\end{align}Recall that any type $1$ user fake tags a $u$-post w.p. $\alpha_u$. However, if a user incorporates a warning level as well, it responds differently - it fake tags the $u$-post w.p. $r(\alpha_u, \omega)$. We model the said effect in \eqref{eqn_response1} via $a, b$, which indicates the positive correlation of the user's response to the innate capacity and the warning level, respectively. Next, we introduce a few notations:
\begin{align}\label{eqn_notations}
    \eta^*_l &:=  \frac{(1-l)(1-\mua)}{1-\alpha_F} \mbox{ for any } l \in \mathbb{R}, \nonumber\\
    K_\delta &= \kappa^2-4\delta\alpha_R\alpha_F(\Delta_R)^a \geq 0 \mbox{ for } 
    \kappa := \delta\left((\Delta_R)^a(1-\alpha_F) - 1 \right) - \alpha_R (\Delta_R)^a, \\
    \delta_a &:= \delta(\bmu_x) = \delta(1-\mua) \mbox{ for any } x \in (0, 1-\mua), \mbox{ and}  \nonumber\\
    \overline{\eta} &:= \frac{\delta_a((1-\mua)cw\alpha_R - 1)}{cw\alpha_R \delta_a - \alpha_R}.  \nonumber
\end{align}

We choose the following warning mechanism that would modulate users' response given in \eqref{eqn_response1} for $\td$ful identification of the posts (see Theorem \ref{thrm_response1}):
\begin{align}\label{eqn_warning_MFG}
    \omega( \beta) = w^{1/b} \alpha_R^{(1-a)/b} \beta^{1/b},
\end{align}where $w$ will be appropriately chosen as per Algorithm \ref{alg_AI}. Note that while designing the warning mechanism ($\omega$), we assume that the OSN knows $\alpha_R$. 

\begin{theorem}\label{thrm_response1}
Consider the response function as in \eqref{eqn_response1}. Let $\theta \in \left(\max\left\{\alpha_F, \frac{\delta}{(\Delta_R)^a}\right\}, 1\right]$ and $\delta \in (\alpha_R, \theta)$. If the parameters satisfy the conditions in Algorithm \ref{alg_AI},
then, choosing $R, \gamma$ as in \eqref{eqn_R_gamma} and the warning mechanism as in \eqref{eqn_warning_MFG} for $w$ given in Algorithm \ref{alg_AI} leads to a game $\G(R, \gamma, \omega)$ such that:

\noindent (i) $\bmu_\eta$ is an AI-NE, for $\eta$ in Algorithm \ref{alg_AI}, and

\noindent (ii) $\bmu_{x_\eta}$ is the only other NE, with $x_\eta := \frac{p}{\gamma-1} + p(1-\mua-\eta) + \eta$, if $x_\eta > \eta_{\widetilde \theta}^*$.  
\eop
\end{theorem}

\noindent The proof of above Theorem is in  \TR{\cite[Appendix]{arxiv}}{Appendix \ref{appendix_MFG}}. Thus, the OSN can design an AI game for any $(\theta, \delta)$ in the following:
\begin{align}\label{eqn_feasible}
\RAI := \{(\theta, \delta) : \theta > f(\theta, \delta) \mbox{ or } \theta \leq f(\theta, \delta) \mbox{ with } K_\delta \geq 0\}.
\end{align}Note that the above algorithm assumes the knowledge of user-specific parameters, the estimation of which is independent of linguistic barriers, as discussed in the introduction.

\RestyleAlgo{ruled}
\begin{algorithm}[ht]
\caption{Design of AI game, assume $K_\delta \geq 0$}\label{alg_AI}
\eIf{
$
\theta > f(\theta, \delta) := \frac{\delta_a-\eta^*_\theta \delta}{(\Delta_R)^a(\delta_a-\alpha_R \eta^*_\theta)}
$
}{
$\widetilde{\theta} \gets \theta$
}{
\If{$K_\delta \geq 0$}{
$\widetilde{\theta} \gets \min\left\{\max\left\{\frac{-\kappa + \sqrt{K_\delta}}{2(\Delta_R)^a\alpha_R}, 1 - \frac{\delta(1-\alpha_F)}{\alpha_R} \right\} + \epsilon, 1\right\}$, \mbox{ for } $\epsilon > \max\left\{0, \theta - \frac{-\kappa + \sqrt{K_\delta}}{2(\Delta_R)^a\alpha_R}\right\}$
}}

\textbf{Choose:}

(i) $w \gets \frac{1}{c \alpha_R}
\left( 
\frac{1}{1-\mua} \max\left\{1, \frac{1}{(\Delta_R)^a \widetilde{\theta}}\right\} + \epsilon_1\right)$, where

\vspace{-2mm}
{\footnotesize
$$
\hspace{-6mm} 0 < \epsilon_1 < \min\left\{ \frac{1}{\delta_a}, \fone \right\} - \frac{1}{1-\mua}\max\left\{1, \frac{1}{(\Delta_R)^a\widetilde{\theta}} \right\},
$$}


(ii) $\eta \gets \overline{\eta} + \epsilon_2$, where $\epsilon_2 \in (0, \eta^*_{\widetilde{\theta}}-\overline{\eta}]$.
\end{algorithm}

We show in \TR{\cite[Lemma 5]{arxiv}}{Lemma \ref{lemma_feasibility_w}} that a feasible $w$ exists as per Algorithm \ref{alg_AI}.  It is important to note that the warning mechanism in \eqref{eqn_warning_MFG} is designed such that $\beta_F^\eta$ and $\beta_R^\eta$ correspond to $r = 1$ and $r < 1$ respectively, for some $\eta$. Such a design helps to ensure that $\beta_F^\eta \geq \theta_a$ and $\beta_R^\eta \leq \delta_a$. However, if the desired $\theta$ is small, AI is not achievable due to insufficient difference between $\theta$ and $\delta$. Then, we choose some $\widetilde{\theta} > \theta$ as in Algorithm \ref{alg_AI}. We formalise these ideas in the proof of Theorem \ref{thrm_response1}. In all, the OSN actually achieves 
$(\widetilde{\theta}, \delta)$-success at $\bmu_\eta$, for  $\widetilde{\theta} \geq \theta$\TR{ (see \cite[Lemma 4]{arxiv}).}{(see Lemma \ref{lemma_theta_tilde}).}


The designed game $\G(R, \gamma, \omega)$ has a unique NE, which is AI if $x_\eta \leq \eta_{\widetilde \theta}^*$; else, there is another NE, $\bmu_{x_\eta}$. In the latter case, $x_\eta > \eta_{\widetilde \theta}^* \geq \eta$, and therefore, the performance of $\omega(\cdot)$ might degrade due to a larger proportion of type $1$ users at this NE. Next, we characterise  $\bmu_{x_\eta}$ (see proof in \TR{\cite{arxiv}}{Appendix \ref{appendix_MFG}}).
\begin{theorem}\label{thrm_perf_x_eta}
Define $x_F := \xdoubleF$. Under the hypothesis of Theorem \ref{thrm_response1}, if $x_\eta > \eta^*_{\widetilde{\theta}}$, then:
\[
\hspace{3cm}\beta_R^{x_\eta} \leq \delta_a, \ \beta_F^{x_\eta} \geq 
\begin{cases}
 \frac{1}{cw\alpha_R (\Delta_R)^a}, &\mbox{if } x_\eta \in (\eta^*_{\widetilde{\theta}}, x_F]. \\
 \alpha_F(1-\mua), &\mbox{if } x_\eta \in (x_F, 1-\mua).  \hspace{3.2cm}\mbox{\eop}
\end{cases}
\]
\end{theorem}
It is easily verifiable that $\beta_F^{x} = \widetilde{\theta}(1-\mua)$ for $x = \eta^*_{\widetilde{\theta}}$. Now, as one might expect, we prove in \TR{\cite[Equations (16), (17)]{arxiv}}{\eqref{eqn_betaF}, \eqref{eqn_betaR}} that the proportion of fake tags should decrease with an increase in type $1$ users for any $u$-post. In view of this, Theorem \ref{thrm_perf_x_eta} states that if the proportion of type $1$ users, \hide{$x_\eta \leq \eta^*_{\widetilde{\theta}}$, then, the second NE ($\bmu_{x_\eta}$) achieves $(\widetilde{\theta}, \delta)$-success, and hence $\td$. Then. we have a game with both the NEs as  AI. 
Even when  }$x_\eta > \eta^*_{\widetilde{\theta}}$, the proposed warning mechanism does not degrade the quality of the tagging process for $R$-post, as from Theorem \ref{thrm_perf_x_eta}, we have $\beta_R^{x_\eta} \leq \delta_a$. However, the users can not identify the $F$-post up to $\theta_a(\bmu_{x_\eta})$-level. Theorem \ref{thrm_perf_x_eta} provides the worst performance of the warning for tagging of $F$-posts.

We next numerically comment upon a more detailed performance of $\bmu_{x_\eta}$ for $F$-post using the normalized degradation metric, ${\cal P} := \nicefrac{(\theta_a(\bmu_{x_\eta}) - \beta_F^{x_\eta})100}{\theta_a}$. Towards this, we consider a large number of samples/configurations of system parameters chosen randomly and independently from some appropriate uniform distributions to obtain the fraction of configurations that achieve AI; we also obtain  the fraction of configurations that have   ${\cal P} < 10\%$. Pick  $\alpha_R \sim  U(0.25, 0.3),\ \mua \sim U(0, 0.2),\ a \sim U(2,3), \ p \sim U(0, 0.5) \mbox{ and } \delta = \alpha_R + 0.01$. 
Define $d$ as the normalised difference  between the innate capacity of users to identify $F, R$-posts, i.e., $d = \nicefrac{(\alpha_F - \alpha_R)}{\alpha_F}$.
We generate $10^4$ samples for different values of $d$. 
%
%

Now, say that the OSN demands $\theta = 0.75$. Then, it can always design an AI game (for all random configurations) if $d \geq 0.01$. 
Further, $21.16\%$ of samples have ${\cal P} < 10\%$ for $d = 0.08$, which gradually increases to $58.57\%$ as users get smarter, $d = 0.28$. Thus, if the OSN aims to achieve higher performance with respect to $\bmu_{x_\eta}$ as well, then it requires users to be slightly more intelligent (higher $d$).



\old{Before we state the result, we introduce few notations:
\begin{align}
        \underline{\gamma}_2^*(\eta) :=  1-\frac{p}{q(\eta)}, \mbox{ where}
\end{align}the function $q(\eta)$ is defined as:
\begin{align}\label{eqn_func_q}
    q(\eta) &:=  (1-\mua)\left(p - \frac{1-\theta}{1-\alpha_F} \right) + \eta (1-p).
\end{align}
Next, define the following terms which will provide the possible proportion of type $1$ users under to-be designed AI game (see Theorem \ref{thrm_response1}):
\begin{align}
    \eta^* &:= \frac{(1-\mua)(1-\theta)}{1-\alpha_F}, \ \eta_*:= \frac{\eta^* - p(1-\mua)}{1-p}, \nonumber \\
    \underline{\eta} &:= \max\left \{ (1 - \mu_a)(1 - \theta),                     \frac{(1-\delta)(1-\mua)\mua}{\delta(1-\mua) + \mua-\alpha_R} \right\}.     \nonumber
\end{align}
Denote the regime of parameters for which OSN can design an AI game (see Theorem \ref{thrm_response1}) as:

\vspace{-2mm}
{\small
\begin{align*}
\mathcal{R}_{AI} &:= \{(p, \eta, \gamma) : 0 < p \in \frac{\eta^*-\underline{\eta}}{1-\mua-\underline{\eta}}, \eta \in (\underline{\eta}, \min\left\{\eta^*, \eta_*\right\})\\
&\hspace{30mm}\mbox{ and } \gamma > \max\{ \underline{\gamma}_1^*(\eta), \underline{\gamma}_2^*(\eta)\}.
\end{align*}}
It is natural to assume a threshold on the proportion of $a$-users, and thus:

\noindent \textbf{(B)} Assume $\mua < \Gamma := \frac{(\delta-\alpha_R)(1-\theta)}{(1-\delta)(\theta-\alpha_F)}.$
\begin{theorem}\label{thrm_response1}
Assume \textbf{(B)} and let $\cp \geq \cnp$. Consider the response function as in \eqref{eqn_response1}. Let $\Delta := \frac{\alpha_F}{\alpha_R} > 1$ with $\Delta^a \theta(1-\mua) > 1$. Consider the warning mechanism $\omega(\cdot)$ defined as,

\vspace{-2mm}
{\small
\begin{align}\label{eqn_warning_response1}
 \omega( \beta)= w^{1/b} \alpha_R^{(1-a)/b} \beta^{1/b}, \mbox{ with }
 w := \frac{1}{c \alpha_R \Delta^a \theta (1-\mua)}.
\end{align}}
Let $R$ be as in Lemma \ref{lemma_R_gamma}. Then, the following are true:

\noindent (i) If $(p , \eta, \gamma) \in \mathcal{R}_{AI}$, then $\G(R, \gamma, \omega)$ is an AI game, with $\bmu^* = (0, \eta, 1-\eta - \eta_a)$ as the unique equilibria of the game.
    
\noindent (ii) If $p \in (0,1)$, $\eta \in (\max\{\underline{\eta}, \eta_*\}, \eta^*)$ and $\gamma > \underline{\gamma}_1^*(\eta)$, then $\G(R, \gamma, \omega)$ has only two NE -
    
    (a) $\bmu_1^* := (0, \eta, 1-\eta - \eta_a)$ with $P_{\bmu_1^*}(S) = 1$, and
    
    (b) $\bmu_2^* := (0, \eta_0, 1-\eta_0 - \eta_a)$, where 
    \begin{align}\label{eqn_eta_0}
        \eta_0 = \frac{p + (\gamma-1)[p(1-\mua-\eta)+\eta]}{\gamma-1},
    \end{align}with $P_{\bmu_2^*}(S) = 1-p$.
\eop
\end{theorem}}
\section{Conclusion}
OSNs are flooded with fake posts, and several techniques have been proposed to detect the same. A significant fraction of them depend upon crowd signals; however, none focuses on the limited willingness of the crowd to participate. We filled this gap by formulating an appropriate (mean-field) participation game where the users are encouraged to provide their responses (fake/real) for each post via a simple reward-based scheme. Further,  our algorithm ensures minimal wrong judgement of real/authentic posts and maximal actuality identification of the fake ones.

We proposed a simple warning mechanism for the polynomial response function of the users. Our mechanism is robust against adversarial users, independent of language barriers, and continually guides the users in making more informed decisions by utilizing the warning signals shared by the OSN. 
Under our design, the resultant game always has a Nash Equilibrium (NE), which meets the desired objective. We also identify the condition in which another NE exists; it achieves the desired identification level for real posts, but fails to achieve the desired level for fake posts. 
\chapter{Summary and Conclusions} \label{ch:summary}
The thesis mainly focused on branching processes (BPs) and online social networks (OSNs). Towards the first domain, we introduced new generalized variants of multi-type population-dependent BPs and analyzed the same in continuous-time and Markovian framework. The key features of our two-type BPs in an appropriately defined super-critical regime can be summarized as follows: 
\begin{enumerate}
    \item In a departure from the classical literature, which considers offspring distribution dependent only on the current (living) population, we consider that the offspring distribution can depend on the current and/or total (living and dead) population.
    \item Further, the literature considers that the population dependency diminishes as time progresses so that the limit of population-dependent mean offspring functions are constants. We assume that the limit mean functions can depend on the proportions of the populations; thus, there is population dependency even at the limit. This change implies that different limit means are possible sample path-wise. 
    \item We also introduced a BP where any individual of a population type can produce negative offspring of (i.e., can attack) the other population and also produces offspring of its type. The attacking population then acquires the attacked individuals. 
    \item We also studied BPs where death can occur unnaturally (due to external factors like competition or climate change) or naturally. 
\end{enumerate}
We also considered a single-type BP, where the offspring distribution is only total population dependent, and the BP permanently transitions from the super-to-sub critical regime as total population size grows with time. Such a variant models systems where the reproductive capacity of individuals diminishes (due to, say, resource constraints) to the extent that it leads to the eventual extinction of the current population and saturation of the total population size. 

In general, we analyzed the BPs using the stochastic approximation (SA) technique approach. We focused on deriving the limits of such BPs, particularly the time-asymptotic (limit) proportion of the populations when two populations interact. Interestingly, we showed that the limit proportion either converges to the attractors or saddle points or hovers around the saddle points of an appropriate ODE with a certain non-zero probability. In fact, we showed that the said probability is one for BP with attack and BP with unnatural deaths (under some assumptions).
The approximating ODE is non-trivial, proportion-dependent, autonomous and measurable.

The convergence to saddle points is new to the SA-based literature, where existing results focus only on attractors of the ODE. Further, the behaviour of hovering around, induced due to the consideration of saddle points, is also not seen before. Thus, our time-asymptotic result is novel for BPs and SA literature. We also proved an almost sure finite time approximation result for the BP trajectory, again using the ODE trajectory. 

All the BPs discussed above are theoretically relevant and instrumental in analyzing various aspects of content propagation over OSNs. On OSNs, users share the post (which they like) with friends. Some of the friends of the user may  already have a read/unread copy of the same post; such users are most likely not interested in the post again. Thus, one needs to consider both unread copies (current population) and the unread plus read copies (total population) while modelling such dynamics. Next, we summarize the results derived in this thesis related to  OSNs.

Content providers (CPs) often use OSNs to share their product information to make it viral, as then, their product may attract huge attention and sales. However, the viral markets on OSN are competitive, as one can have simultaneous propagation of  posts related to  similar products on the same OSN. We precisely studied such viral competing markets using BP with attack. We provided  the explicit conditions in which both the posts can get viral simultaneously and derived the limit proportion of the  copies of the two posts. Interestingly,  the design/content of the post is the critical factor for virality in such a competitive environment; the influence of the CP is secondary.

Further, as discussed above, re-forwarding the post is an important aspect. We captured this effect using saturated total-population dependent BP because the (effective) mean number of shares decreases as the total number of shares  increases. Using the approximation result over finite-time, we derived the deterministic approximate trajectories for the current and total shares, which depend only on the network characteristics. These trajectories led to the expressions for important metrics like the peak number of unread copies, the lifespan of posts and others. 

We also designed warning mechanisms based on users' responses to identify the fake posts propagating over OSNs. Towards this, we proposed a model where the OSN allows the users to assign a fake or real tag to each post, and then, based on each response of the user, it generates a warning for future recipients of the post. In reality, users may respond differently to the warning mechanism: some users may not tag at all, some users may tag only based on their understanding, some users may consider warning as well for tagging, and others may adversarily tag any post as real. The dynamics of such a complicated process are captured via BP with unnatural deaths. We show theoretically and numerically that the designed mechanisms are robust against adversaries and lead to maximal correct (and minimal wrong) identification of fake (and real) posts. Further, we designed an algorithm that estimates the parameters required for the warning mechanism without assuming the knowledge of the proportions of users exhibiting different behaviour and with minimal knowledge of other user-specific parameters.

In the above problem, we assumed that there is a non-zero fraction of users who tag based on warning. However, in reality, users are reluctant to tag and  may not necessarily consider the warning. To overcome these issues, we designed a mean-field game where users are given rewards to participate in the tagging process. The rewards are designed to achieve desired levels of actuality identification for fake and real posts at Nash equilibrium.

Next, we provide various possible interesting future directions for our work.

\textbf{Future directions for BPs:} First, one needs to identify the conditions under which the  BP-trajectory hovers around saddle set with non-zero probability. 
One may also find it worthwhile to extend the analysis to the case where more than two types of populations are involved. The extension should follow analogously and be straightforward if the structure of  the limit mean functions is preserved in terms of proportions. 

\textbf{Future directions for applications:} In this thesis, we applied our BP-based results to understand post-propagation dynamics over OSNs. In the future, one may even consider using our new `BP with unnatural deaths' to extend the existing numerical understanding of the complicated interactions in ecological systems and provide a rigorous theoretical analysis.

Further, regarding OSNs, one can answer numerous interesting questions by exploiting the structure of the derived deterministic trajectories of content propagating over OSNs. Some potential questions are as follows:
\begin{enumerate}[label = (\roman*)]
    \item What is the optimal number of initial (seed) users that CPs should buy to make their post viral in a competitive environment over OSNs? How should a CP strategically divide its money into buying seed users and designing an attractive post?
    \item In reality, any CP  shares its post repeatedly (and not just once) with new seed users each time. This leads to more rigorous re-forwarding and, in fact, clustering effects. In such a case, how do the trajectories change?
\end{enumerate}
Lastly, recall that in our final problem about singling out fake posts, we designed a participation mean-field game (MFG) where the game starts with a given fixed proportion of users who react differently to the warning mechanism; thus, the underlying game is in a static setting. In practice, the users exhibit different behavioural traits as the game proceeds. Thus, the dynamic MFG is needed to capture the dynamics more appropriately.


\begin{appendices}
    \chapter{For Chapter \ref{ch:journal1}}\label{appendix_journal1}

\section{Some preliminary results}\label{appendix_prelim}
In this Appendix, we state some important auxiliary results, which are also helpful in further understanding of the subject at hand. The Lemma \ref{lemma_sum_pop} and the discussion thereafter provide insights into the derivation of the limit mean matrices of \ref{a2}.


\begin{lemma}{\bf [Dichotomy]}
\label{lemma_sum_pop}
Let assumption \ref{a1} hold and define $\underline{m} =: E[\underline{\offs}]$. Then, we have:
$$
P\left(\left\{\liminf_{n} S_n^c e^{-\lambda(\underline{m}-1)\tau_n} > 0\right\} \cup \left\{\lim_{n \to \infty} S_n^c = 0\right\}\right) = 1. 
$$
\end{lemma}
\begin{proof}
Let $\Cx(0) = \cx_0$ and $\Cy(0) = \cy_0$. Consider a fictitious population-independent BP with single-type population, say $z$-type. Let $Z(0) = \cx_0 + \cy_0$. Each time an individual dies in the new process, assume that random number of offspring (distributed as $\underline{\offs}$ in \ref{a1}) are produced.  Further, assume that if $Z(t) = 0$ for some $t < \infty$, then, exactly $1$ individual is immigrated into the new system; this leads to the classical continuous time branching process with state-dependent immigration as in \cite{yamazato1975some}. Observe $\sum_{j=2}^\infty j P(\underline{\offs} = j) log(j) < \infty$ due to finite second moment assumption on $\overline{\offs}$ in \ref{a1}. Thus, by \cite[Theorems 6 and 8]{yamazato1975some}, $P(Z(t) \to \infty) = 1$, under \ref{a1}.

For completing the proof, we couple the embedded chains of the two BPs, for all $n \leq \nu_e$, where $\nu_e$ is the extinction epoch of the given system (see Section \ref{sec_probdesc_mainresult}); the offspring in the $Z(\cdot)$ branching process are given by $\underline{\offs}$ of \ref{a1}. If $\nu_e < \infty$, then $S_n^c = 0$ for all $n \geq \nu_e$. Otherwise, by coupling arguments, $S^c_n \geq Z_n$ for all $n$, and thus $S^c_n \to \infty$ as $n \to \infty$. Further, in the latter case, by \cite[Theorem 1, Chapter 1]{athreya2012classical}, the growth rate of $S_n^c$ is at least as large as that of $Z_n$, i.e., $\lambda(\underline{m}-1)$.
\end{proof}


\noindent \underline{\textbf{Limit mean matrices for BPs with negative offspring:}} 

In BPs with negative offspring, in the the survival sample-paths, by Lemma \ref{lemma_sum_pop}, $S^c_n \to \infty$. In such cases, one needs to identify the limit mean matrix of \ref{a2}. Say $0 < \liminf_{n \to \infty} \bc(\Ups_n) \leq \limsup_{n \to \infty} \bc(\Ups_n) <1$. Then, for such sample-paths, both populations would have exploded, i.e., $(\Cx_n, \Cy_n) \to (\infty, \infty)$. Hence, there are sufficient number of individuals to be attacked of both types, which results in the saturation of the number of attacks\footnote{To be realistic, the number of attacks by a single individual should saturate, i.e., for example, $\lim_{\cy \to \infty}m_{xy}(\cy) = \minf_{xy} < \infty$. The case with unsaturated attacks in easier to analyze, and one can easily prove for BP with attack that only one of the two population types survives with probability $1$.
\vspace{2.4mm}
}; thus, it is appropriate to consider $\minf_{xy}(\bc)$ as some constant for all $\bc \in (0,1)$, and so is the case with $\minf_{yx}(\bc)$.

On the other hand, say $\limsup_{n\to \infty} \bc(\Ups_n) = 1$, then, $\bc(\Ups_n) = 1$ i.o. This implies $\bc(\Ups_n) = 1$ for all $n$ large enough, as $\bc(\Ups_n) = 1$ is an absorbing state for processes with attack, like BP with attack and prey-predator BP. Thus, clearly $\minf_{xy}(\bc) = 0$ for $\bc = 1$. Similarly, $\minf_{yx}(\bc) = 0$ for $\bc = 0$.

\section{For Results in Chapter \ref{ch:journal1}} \label{appendix_B}

Throughout the Appendix, we will consider the solution of the integral operator as the extended solution of ODE \eqref{eqn_ODE}. The fact that these two solutions are equivalent, is proved towards the end of the proof of Theorem \ref{thrm1}(i).

\noindent \textbf{Proof of Lemma \ref{lemma_equi_cont_thrm1} (contd.).} \label{proof_lemma1}
By \eqref{eqn_bounded_iterates}, $(\Ups^n(0))_n$ is bounded. We will now prove \eqref{eqn_footnote} for $(\Theta^{n, c}(t))$; it can be proved analogously for other components of $\Ups^n(\cdot)$. 
\hide{Towards this, re-write the scheme for $(\Theta_n^c)$ as follows (see
\eqref{eqn_bias},  \eqref{eqn_nonauto_ODE}):
\begin{align*}
    \Tc_{n+1} &= \Tc_n + \epsilon_n L_n^{\theta, c} =  \Tc_n + \epsilon_n (\delta M_n^{\theta, c} + g_\theta^{c}(\Ups_n) + D_n^{\theta, c}), \mbox{ where}
\end{align*}
$\delta M_n^{\theta, c} :=  L_n^{\theta, c} - g_\theta^{c}(\Ups_n) - D_n^{\theta, c} = L_n^{\theta, c} - \rho_\theta^{c}(\Ups_n, t_n)$. }
Observe from \eqref{eqn_diff_term1} and \eqref{eqn_diff_term2} that the interpolated trajectory can be re-written as:
\begin{align}\label{eqn_interpolated_traj}
\begin{aligned}
    \Theta^{n, c}(t) &:= \Tc_n + \int_0^t g_\theta^c(\Ups^n(s)) ds +  \sum_{i=n}^{\eta(t_n+t)-1} \epsilon_i L_i^{\theta, c} - \int_0^t g_\theta^c(\Ups^n(s)) ds\\
    &=  \Tc_n + \int_0^t g_\theta^{c}(\Ups^n) ds + M^{n, \theta, c}(t) + \rho^{n, \theta, c}(t) + D^{n, \theta, c}(t), \mbox{ where}\\
    M^{n, \theta, c}(t) &:=  \sum_{i=n}^{\eta(t_n + t)-1}  \epsilon_i \left(L_i^{\theta, c} - \rho_\theta^{c}(\Ups_i, t_i)\right), \\ \rho^{n, \theta, c}(t) &:=  \sum_{i=n}^{\eta(t_n + t)-1}\epsilon_i g_\theta^{c}(\Ups_i) - \int_0^t g_\theta^{c}(\Ups^n) ds, \mbox{ and}  \\
    D^{n, \theta, c}(t) &:= \sum_{i=n}^{\eta(t_n + t)-1} \epsilon_i \left(\rho_\theta^{c}(\Ups_i, t_i) - g_\theta^{c}(\Ups_i)\right).
\end{aligned}
\end{align}
\hide{Then, \eqref{eqn_interpolated_traj_1} can be re-written as: 

\vspace{-6mm}
{\small 
\begin{align}\label{eqn_interpolated_traj}
    \Theta^{n, c}(t) &:= \Tc_n + \sum_{i=n}^{\eta(t_n + t)-1}\epsilon_i (\delta M_i^{\theta, c} + g_\theta^{c}(\Ups_i) + D_i^{\theta, c}) \nonumber \\
    &= \Tc_n + \int_0^t g_\theta^{c}(\Ups^n) ds + M^{n, \theta, c}(t) + \rho^{n, \theta, c}(t) + D^{n, \theta, c}(t), \mbox{ where} \\
M^{n, \theta, c}(t) &:=  \sum_{i=n}^{\eta(t_n + t)-1}  \epsilon_i \delta M_i^{\theta, c}, \ \ \   D^{n, \theta, c}(t) := \sum_{i=n}^{\eta(t_n + t)-1} \epsilon_i D_i^{\theta, c}, \mbox{ and} \nonumber \\
\rho^{n, \theta, c}(t) &:=  \sum_{i=n}^{\eta(t_n + t)-1}\epsilon_i g_\theta^{c}(\Ups_i) - \int_0^t g_\theta^{c}(\Ups^n) ds. \nonumber 
\end{align}}}
Now, fix $T > 0$ and define the set $S^\delta_T := \{(s, t) : 0\leq t-s\leq \delta, 0\leq t \leq T \}$. Then:
\begin{align}\label{eqn_sup_bound} 
    \sup_{S_T^\delta} |\Theta^{n, c}(t) - \Theta^{n, c}(s)| 
     &\leq \sup_{S_T^\delta} \left|\int_s^t g_\theta^{c}(\Ups^n) dr\right| +  \sup_{S_T^\delta} \left|M^{n, \theta, c}(t) - M^{n, \theta, c}(s)\right| \nonumber \\
     &+ \sup_{S_T^\delta} \left|\rho^{n, \theta, c}(t) - \rho^{n, \theta, c}(s) \right| + \sup_{S_T^\delta} \left|D^{n, \theta, c}(t) - D^{n, \theta, c}(s) \right|.
\end{align}
To prove our claim, we begin with the first term  of   \eqref{eqn_sup_bound}. From  \eqref{eqn_ODE} and \eqref{eqn_bounded_iterates}, $|g_\theta^{c}(\Ups)| \leq \hat{m}$ for an appropriate $\hat{m} > 1$, for any $\Ups$, and, thus:
\begin{align*}
     \left|\int_s^t g_\theta^{c}(\Ups^n) dr\right| 
     \hide{&\leq \int_s^t \hspace{-3mm} \left|g(\theta^{n, c}(r)) \right|dr \\
     &=  \int_s^t  \left|\bc\big(m_{xx}(\om) - 1\big) + (1-\bc) m_{yx}(\om) - \theta^{n, c}(r) \right|dr\\
     &\leq \int_s^t  \left|\bc (\hat{m} - 1) + (1-\bc) \hat{m}  \right|dr }
     \leq \hat{m}(t-s), \mbox{ so, }  \sup_{S_T^\delta} \int_s^t  \left|g_\theta^{c}(\Ups^n) \right|dr \leq  \delta \hat{m}.
\end{align*}For the second term of   \eqref{eqn_sup_bound}, define $M_n^{\theta, c} := \sum_{i=0}^{n-1}\epsilon_i \left(L_i^{\theta, c} - \rho_\theta^{c}(\Ups_i, t_i)\right)$. Then, it is easy to prove that $(M_n^{\theta, c})$ is a Martingale   with respect to $(\mathcal{F}_n)$. Thus, using Martingale inequality, 
for each $\mu > 0$ (where, $E_n(\cdot)$ denotes the expectation conditioned on $(\mathcal{F}_n)$): 
$$
P\left\{\sup_{m\leq j \leq n} |M_j^{\theta, c} - M_m^{\theta, c}| \geq \mu \right\} \leq \frac{E_n\left|\sum_{i=m}^{n-1} \epsilon_i \left(L_i^{\theta, c} - \rho_\theta^{c}(\Ups_i, t_i)\right) \right|^2}{\mu^2}.
$$
Observe, $E\left[\left(L_i^{\theta, c} - \rho_\theta^{c}(\Ups_i, t_i)\right)\left(L_j^{\theta, c} - \rho_\theta^{c}(\Ups_j, t_j)\right)\right] 
= 0$ for $i < j$. Using this: 
\begin{align*}
    P\left\{\sup_{m\leq j \leq n} |M_j^{\theta, c} - M_m^{\theta, c}| \geq \mu \right\} &\leq 
    \frac{\sum_{i=m}^{n-1} \epsilon_i^2 E_n\left| L_i^{\theta, c} - \rho_\theta^{c}(\Ups_i, t_i) \right|^2}{\mu^2}.
\end{align*}
Note that under \ref{a1} and \eqref{eqn_bounded_iterates}, for some $K > 0$:
\begin{align*}
     \sup_n E_n|L_n^{\theta, c}- \rho_\theta^c(\Ups_i, t_i)|^2 &\leq
     \DetailK{\sup_n E\left|H_{n}\left(\offs_{xx, n}(\Om_{n-1}) - 1\right) + \overline{H}_{n} \offs_{yx, n}(\Om_{n-1}) - \Tc_{n-1}\right|^2\\
     &\leq} \sup_n E_n\left(\overline{\offs}_n - 1 \right)^2 + \sup_n E_n|\rho_\theta^c(\Ups_i, t_i)|^2 < K.
\end{align*}Thus, for every $n \geq m$:
\begin{align*}
    P\left\{\sup_{m\leq j \leq n} |M_j^{\theta, c} - M_m^{\theta, c}| \geq \mu \right\} &\leq  \frac{K}{\mu^2} \sum_{i=m}^{\infty} \epsilon_i^2.
\end{align*}
By first letting $n \to \infty$ (and using continuity of probability), then, letting $m \to \infty$, 
\begin{align}
    \lim_{m \to \infty} P\left\{\sup_{m\leq j } |M_j^{\theta, c} - M_m^{\theta, c}| \geq \mu \right\} &= 0 \mbox{ for each }\mu > 0.\label{eqn_equi_cont_M}
\end{align}
Now, by \eqref{eqn_equi_cont_M} and continuity of probability, for each $\mu > 0$:
\begin{align}\label{eqn_lim_sup_second_term}
 P\left\{\lim_{m \to \infty} \sup_{m\leq j } |M_j^{\theta, c} - M_m^{\theta, c}| \geq \mu \right\} = 0. 
\end{align}
Let $A_k := \lim_{m \to \infty} \sup_{m\leq j } |M_j^{\theta, c} - M_m^{\theta, c}| < 1/k$, then, $P(A_k) = 1$ for each $k > 0$. We further restrict our attention to  sample paths  $\omega \notin \underline{N :=  (\cap_k A_k)^c \cup \{\overline{\Pi} \nto \overline{m} \}}$. Now, the second term in \eqref{eqn_sup_bound} is upper bounded by $2\sup_{t\geq 0}|M^{n, \theta, c}(t)|$. For any $\omega \notin N$:
\hide{
\vspace{-4mm}
{\small
\begin{align*}
|M^{n, \theta, c}(t)| &= \left|\sum_{i=n}^{\eta(t_n + t)-1}\epsilon_i \delta M_i^{\theta, c} \right| 
= \bigg|M^{\theta, c}_{\eta(t_n + t)} - M^{\theta, c}_n\bigg|.
\end{align*}}
This gives us:}
\begin{align*}
\sup_{t\geq 0}|M^{n, \theta, c}(t)| 
&= \sup_{t \geq 0} |M^{\theta, c}_{\eta(t_n + t)} - M^{\theta, c}_n|  = \sup_{j \geq n} |M^{\theta, c}_{j} - M^{\theta, c}_n|\\
\implies  \lim_{n \to \infty} \sup_{S_T^\delta}|M^{n, \theta, c}(t)|  &\leq \lim_{n \to \infty} \sup_{\eta(t_n + t) \geq n} |M^{\theta, c}_{\eta(t_n + t)} - M^{\theta, c}_n| < 1/k,
\end{align*}where the last inequality holds because we have considered sample paths which are not in $N$.
Letting $k \to \infty$, we get, $M^{n, \theta, c}(\cdot) \to 0$ uniformly on each bounded interval. 

For the third term in \eqref{eqn_sup_bound}, observe that when  $t = t_k - t_n$ $(k > n)$, $\rho^{n, \theta, c}(t) = 0$. Thus, 
\hide{Towards this, for $k = n+1$, we have $t_k - t_n = \epsilon_n$:
\begin{align*}
    \rho^{n, \theta, c}(t_{n+1} - t_n) &=  \sum_{i=n}^{\eta(t_{n+1})-1}  \epsilon_i g(\tc_i) - \int_0^{\epsilon_n} g(\theta^{n, c}(s)) ds\\
    &= \epsilon_n g(\tc_n) - \epsilon_n g(\tc_n) = 0,
\end{align*}
where the second equality holds true because for $0 \leq s \leq \epsilon_n$, $\theta^{n, c}(s) = \tc_n$.
Let the claim be true for $k = n+l$, $l > 0$, then for $ k = n+l+1$, we have:
\begin{align*}
    \rho^{n, \theta, c}(t_{n+l+1} - t_n) &= \rho^{n, \theta, c}(t_{n+l} - t_n)  + \epsilon_{n+l}g(\theta_{n+m}^c) - \int_{t_{n+m}-t_n}^{t_{n+m+1}-t_n} g(\theta^{n, c}(s))ds = 0.
\end{align*}
Thus, by induction the claim is true. Next, we need to show that
$\rho^{n, \theta, c}(t) \to 0$ uniformly in $t$ as $n\to \infty$.} \hide{, i.e., for every $\epsilon > 0$, there exists $n_\epsilon$ such that for all $n \geq n_\epsilon$ and for all $t > 0$, $|\rho^{n, \theta, c}(t) - 0| < \epsilon$.}  for any $|t| \leq T$ (following similar steps as in first term, and noting $\epsilon_{\eta(t_n + t)} \le \epsilon_n$):
\begin{align*}
    |\rho^{n, \theta, c}(t)| &=
    \left| \int_{t_{\eta(t_n + t)} - t_n}^t g_\theta^c( \Ups^n) ds \right|
    \hide{\\
    &\leq  \int_{ t_{\eta(t_n + t)} - t_n}^t \left|g_\theta^{c, \infty}(\ups^n) ds \right| \leq (t - \eta(t_n + t)) \overline{m} }
    < \epsilon_n \hat{m}.
\end{align*}Thus, $\rho^{n, \theta, c}(\cdot)$ uniformly converges to $0$ as $n \to \infty$ on each bounded interval.  

For the last term in \eqref{eqn_sup_bound}, we claim that $D^{n,\theta, c}(t)$ also converges to $0$ uniformly on each bounded interval in $(0, \infty)$ as $n \to \infty$, for each $\omega \notin N$. Towards this, first consider $\omega \in N^c \cap \{S_n^c \to 0\}$, i.e, extinction paths. Then, $\rho_\theta^c(\Ups_i, t_i) = 0$ and $g_\theta^{c}(\Ups_i) = 0$ for all $i > \nu_e$. Thus, trivially $\lim_{n \to \infty}D^{n,\theta, c}(t) = 0$ for all $t \in (0, \infty)$.

Next, consider $\omega \in N^c \cap \{S_n^c \nto 0\}$; for such sample paths, we first derive a uniform positive lower bound for $\Pc_n$, required to prove the claim. To this end, analogous to $\overline{\Pi}_n$ defined in \eqref{eqn_overline_S_n}, one can define $\underline{\Pi}_n$ using $ \underline{\offs}$ given in \ref{a1}. Then, following similar steps as before, i.e., using strong law of large numbers and computing as in \eqref{eqn_bounded_iterates}, we get $\Pa_n \geq \Pc_n \geq \Delta$ for an appropriate $\Delta > 0$, for all $n\geq 1$. Thus, we have for each $i \geq 1$ (see $\tc$ component of \eqref{eqn_ODE}, \eqref{eqn_nonauto_g} and assumption \ref{a2}):
\begin{align*}
|D_i^{\theta, c}| = |\Bc_i(m_{xx}(\Om_i) -m_{xx}^\infty(\Bc_i)) + (1-\Bc_i)(m_{yx}(\Om_i) - m_{yx}^\infty(\Bc_i))| \leq \frac{2}{S_i^c} = \frac{2}{\Pc_i \eta(t_i)} \leq \frac{2}{\Delta i}. 
\end{align*}This implies that, (recall $\epsilon_i = 1/(i+1)$)
\begin{align*} 
|D^{n, \theta, c}(t)| = \left|\sum_{i= n}^{\eta(t_n + t) - 1}\epsilon_i D_i^{\theta, c}\right| \leq \sum_{i= n}^{\eta(t_n + t) - 1} \frac{2}{\Delta i (i+1)} \leq \sum_{i= n}^{\infty} \frac{2}{\Delta i (i+1)}, \mbox{ for any }t.
\end{align*}Thus, $D^{n, \theta, c}(t)$ uniformly converges to $0$ as $n \to \infty$. In all, by \eqref{eqn_sup_bound} and above analysis, it is clear that for each $T>0$ and for any $\epsilon > 0$, there exists $n_\epsilon$ such that $\sup_{S_T^\delta} |\Theta^{n, c}(t) - \Theta^{n, c}(s)| < \epsilon $ for all $n \geq n_\epsilon$; hence $(\Theta^{n, c}(\cdot))$ is equicontinuous in extended sense. \eop

\noindent \textbf{Proof of Theorem \ref{thrm1} (ii).} \label{proof_thrm1}
The proof is constructed for sample paths $\omega \notin N$, however, for simplicity, we drop $\omega$ (see Lemma \ref{lemma_equi_cont_thrm1} for definition of set $N$). \hide{Consider the set $S(\omega)$ defined as:
$$
    S(\omega) := \{\ups: \Ups_n(\omega) \mbox{ exits } N_{\delta_{\mbox{\scriptsize{$\ups$}}}}(\ups) \mbox{ i.o. for some } \delta_\ups > 0\}.
$$
If $S(\omega)^c \cap \cR \neq \emptyset$
, then $\Ups_n(\omega) \to \cR$, and more precisely, $\Ups_n(\omega) \to \cR-S(\omega)$. Otherwise,  i.e., if $\cR \subseteq S(\omega)$, 
then $\Ups_n(\omega) \in \chi := \cD_b \cap \left(\cap_{\ups \in \cR} N_{\delta_\ups}^c(\ups)\right)$ i.o., observe $\chi$ is compact. For simpler notations, we drop $\omega$ henceforth. Therefore,}
By \ref{a4}, $\Ups_n\in \cD_b$ i.o. Since $\cD_b$ is compact, $({\Ups}_n)$ has a limit point ${\ups}_0 \in \cD_b$; then, there exists a sub-sequence $(n_k)$ such that ${\Ups}_{n_k} \to {\ups}_0$. Further, by (extended) equicontinuity of $(\Ups^n(\cdot))$, there exists further sub-sequence (denote it again by $(n_k)$, for simpler notations) $({\Ups}^{n_k}(\cdot))$ which converges to the extended solution $\ups(\cdot)$ of the ODE \eqref{eqn_ODE} uniformly on each bounded interval. Also observe, ${\Ups}^{n_k}(0) = {\Ups}_{n_k} \to {\ups}_0$, and recall $\ups(0) = \ups_0$  is the initial condition for ODE \eqref{eqn_ODE}. Under characterization of attractor or q-attractor in \ref{a4}, the ODE solution $\ups(t)$ converges to some $\ups^* \in (\cA \cup \cR)\cap \cD_I$ as $t \to \infty$. 

We will now show that for any $\delta_1 > 0$, $\Ups_n$ visits $N_{\delta_1}(\ups^*)$ i.o. We will also discuss other convergence aspects to complete the proof. Towards this, fix $\delta_1 > 0$.

{\bf Step A:} To begin with, assume $\ups^* \in \cA \cap \cD_I$.  Then, by \ref{a4} (local stability) it is possible to choose $0 < \delta_2 < \delta_1$ such that any ODE solution, ${\widetilde \ups}(\cdot)$, satisfies the following:
\begin{equation}\label{eqn_local_stability}
    {\widetilde \ups} (t) \in N_{\delta_1}(\ups^*)
    \mbox{ for all } t \geq 0, \mbox{ when initial condition } {\widetilde \ups}(0) \in cl( N_{\delta_2}(\ups^*) ).
\end{equation}
Further, by convergence of solution, $\ups(t) \to \ups^*$, thus there exists $T_{\delta_2} < \infty$ such that:
\begin{align}\label{eqn_dist_ODE_A}
    d({\ups}(t) , \ups^*) < \delta_2/2 \mbox{ for all } t \geq T_{\delta_2}.
\end{align}Now, following similar steps as in part (i) (see \eqref{eqn_dist_scheme_ODE_}), there exists $\overline{n} < \infty$ such that:
\hide{
Further, there exists $\overline{N}$ such that (by uniform convergence):
$$
\sup_{T_{\delta_2} \leq t \leq 2 T_{\delta_2}} d({\ups}^{n_k}(t) , {\ups}(t)) < \delta_2/2 \mbox{ for all } n_k \geq \overline{N}.
$$
Consider $t = t_l - t_{n_k}$ ($l > n_k$) such that $T_{\delta_2} \leq t \leq 2T_{\delta_2}$. Let $L:= \{l : T_{\delta_2} + t_{n_k} \leq t_l \leq 2T_{\delta_2} + t_{n_k}\}$. Then, we have (note that ${\ups}^{n_k}(t_l - t_{n_k}) = {\Theta}_l$):}
\begin{align}\label{eqn_dist_scheme_ODE}
\sup_{l \in L_k} d({\Ups}_l , {\ups}(t_l)) < \delta_2/2 \mbox{ for all } n_k \geq \overline{n},
\end{align}
for $L_k:= \{l : T_{\delta_2} + t_{n_k} \leq t_l \leq 2T_{\delta_2} + t_{n_k}\}$.
Using \eqref{eqn_dist_ODE_A} and \eqref{eqn_dist_scheme_ODE},  for all $n_k \geq \overline{n}$:
\begin{align}\label{eqn_dist_scheme_A}
\sup_{l \in L_k} d({\Ups}_l , \ups^*) &\leq \sup_{l \in L_k}d({\Ups}_l , {\ups}(t_l)) + \sup_{l \in L_k}d(\ups(t_l) , \ups^*) < \delta_2.
\end{align}Thus, ${\Ups}_n$ visits $N_{\delta_2}(\ups^*)$ i.o., and hence $N_{\delta_1}(\ups^*)$ i.o.

Henceforth, the proof is majorly as in proof of \cite[Theorem 2.3.1, pp. 39]{kushner2012stochastic}, except for \textit{few changes to consider convergence to q-attractors, not just attractors}. Contrary to the claim, assume that ${\Ups}_n$ exits $N_{\delta_1}(\ups^*)$ i.o. Thus, by \eqref{eqn_dist_scheme_A}, $\Ups_n$ moves from $N_{\delta_2}(\ups^*)$ to $\cD_b - N_{\delta_1}(\ups^*)$ i.o. 
Let $\overline{\Ups}^0(\cdot)$ be the usual linear interpolated trajectory of $\Ups_n$, i.e., 
\begin{align*}
    \overline{\Ups}^0(t_n) = \Ups_n, \mbox{ and } \overline{\Ups}^0(t) = \frac{t_{n+1}-t}{\epsilon_n}\Ups_n + \frac{t-t_{n}}{\epsilon_n}\Ups_{n+1} \mbox{ for } t \in (t_n, t_{n+1}).
\end{align*}
Then, there exists sequence $(l_j, r_j)$ such that (i) $\dots > r_j > l_j > r_{j-1} > l_{j-1}> \dots$, (ii) $r_j \to \infty$, (iii) $\overline{\Ups}^0(l_j) \in \partial N_{\delta_2}(\ups^*)$, $\overline{\Ups}^0(r_j) \in \partial N_{\delta_1}(\ups^*)$, and (iv) $\overline{\Ups}^0(t) \in cl(N_{\delta_1}(\ups^*)) -  N_{\delta_2}(\ups^*)$, for all $t\in [l_j, r_j]$. Consider the segments (one for each $j$) of $\overline{\Ups}^0(\cdot)$, i.e., consider  functions, $\q_j (t) := \overline{\Ups}^0(l_j+t)$ for any $ t\geq 0$;
observe by construction that for each $j,$ we have $\q_j (t) \in \left \{ \ups: \delta_2 <  d(\ups, \ups^*) \le \delta_1 \right \}$ for all $0 < t \le r_j - l_j$.


\textbf{Case (a):} Suppose there is a $T < \infty$ such that for some sub-sequence (call it $j$ again) $r_{j} - l_{j} \to T$. Now, consider a sub-sequence of $(\q_j(\cdot))$ which (again) converges to some solution of ODE, $\widetilde{\ups}(\cdot)$ uniformly over $[0,T]$.\footnote{The equicontinuity in extended sense can easily be extended to linear interpolated trajectories.} Then, $\widetilde{\ups}(0) \in \partial N_{\delta_2}(\ups^*)$ and $\widetilde{\ups}(T) \in \partial N_{\delta_1}(\ups^*)$. This contradicts \eqref{eqn_local_stability}. For $T = 0$, there is an obvious contradiction.

\textbf{Case (b):} If $r_{j} - l_{j} \to \infty$, then, $\widetilde{\ups}(0) \in \partial N_{\delta_2}(\ups^*)$ and $\widetilde{\ups}(t) \in cl(N_{\delta_1}(\ups^*)) -  N_{\delta_2}(\ups^*)$ for all $t > 0$. Then, it is a contradiction to $\ups^*$ being an attractor. 

In all, ${\Ups}_n \to \ups^*$; since $\ups^* \in \cA \cap \cD_I$ is arbitrary, we have ${\Ups}_n \to \cA \cap \cD_I$.

{\bf Step S:} 
Now consider $\ups^* \in \cR \cap \cD_I$. 
If $\nu_e < \infty$, i.e., in extinction sample paths, $\Ups_n \to {\bf 0}$ and we are done. For others, $\lim \inf_n \Pc_n > 0$ by Lemma \ref{lemma_sum_pop}.
Thus, with $\nu_e = \infty$ and $\ups^* \in \cR \cap \cD_I$, 
 by Definition \ref{defn_q_attractor}, the initial condition $\ups_0  \in {\mathbb S} (\ups^*)$ with  $\bc(\ups_0) = \bc (\ups^*)$.

Similar to step A, by exponential stability (\ref{a4}), one can show that \eqref{eqn_local_stability} follows for any ODE solution $\widetilde{\ups}(\cdot)$ when initial condition  $\widetilde{\ups}(0) \in N_{\delta_2}(\ups^*) \cap \mathbb{S}(\ups^*)$. Further, clearly \eqref{eqn_dist_ODE_A}-\eqref{eqn_dist_scheme_A} also hold for this case. Thus, $\Ups_n$ visits $N_{\delta_1}(\ups^*) \cap \mathbb{S}(\ups^*)$ i.o. 

Further, if for every $\delta_1 > 0$, $\Ups_n$ does not exit $N_{\delta_1}(\ups^*) \cap \mathbb{S}(\ups^*)$ i.o., then $\Ups_n \to \Ups^* \in \cR\cap \cD_I$. Otherwise, for every $\delta_2 > 0$, $\Ups_n$ visits $N_{\delta_2}(\ups^*) \cap \mathbb{S}(\ups^*)$ and for some $\delta_1 > 0$, $\Ups_n$ exits $N_{\delta_1}(\ups^*) \cap \mathbb{S}(\ups^*)$ i.o. 
\eop




\hide{

\begin{lemma}\label{lemma_nonauto_auto_ode}
For any $T>0$,
\begin{align}\label{eqn_nonauto_auto_ode}
    g(\om(\widetilde{\ups}^n(s), t_n+s), \widetilde{\ups}^n(s)) \to g_\infty(\ups^*(s)) \mbox{ for all } s \in [0, T], \mbox{ as } n \to \infty.
\end{align}
\end{lemma}
\begin{proof} Consider the following, with $\widetilde{\ups}^n(s) = \om(\widetilde{\ups}^n(s), t_n+s)$:
\begin{align*}
    \left|g(\om(\widetilde{\ups}^n(s), t_n+s), \widetilde{\ups}^n(s)) - g_\infty(\ups^*(s))\right| &
    = \left|g(\widetilde{\ups}^n(s), t_n+s) - g_\infty(\ups^*(s))\right|\\
    &      \hspace{-5.5cm}  
    \leq \left|g(\widetilde{\ups}^n(s), t_n+s) - g(\ups^*(s), t_n+s)\right| + \left|g(\ups^*(s), t_n+s) - g_\infty(\ups^*(s))\right|
\end{align*}
Again by \textbf{A}.5, the second term converges to $0$, as $n \to \infty$. For the first term, we illustrate the claim for $\pa$, with $\widetilde{\om}_n(s) = \om(\widetilde{\ups}^n(s), t_n+s)$ and $\om_n^*(s) = \om(\ups^*(s), t_n+s)$:

\vspace{-2mm}
{\small
\begin{align*}
    \left|g_\psi^a(\widetilde{\ups}^n(s), t_n+s) - g_\psi^a(\ups^*(s), t_n+s)\right| 
    &\leq \bigg|\bc \sum_{j \in \{x, y\}} \bigg(m_{xj}(\widetilde{\om}_n(s)) - m_{xj}(\om^*_n(s))\bigg) \\
    &\hspace{-4.7cm}+ (1-\bc)\sum_{j \in \{x, y\}} \bigg(m_{yj}(\widetilde{\om}_n(s)) - m_{yj}(\om^*_n(s))\bigg)  - \left((\widetilde{\psi}^a)^n(s) - (\pa)^*(s)\right) \bigg|\\
    &\hspace{-4.7cm}\leq \sum_{i, j} |m_{ij}(\widetilde{\om}_n(s)) - \minf_{ij}(\bstar)| + \sum_{i, j} |m_{ij}(\om^*_n(s)) - \minf_{ij}(\bstar)| + |(\widetilde{\psi}^a)^n(s) - (\pa)^*(s)|\\
    &\hspace{-4.7cm}\leq \frac{4}{(\tilde{s}^{c, n})^\alpha} + \frac{4}{(s^{c, *})^\alpha} + |(\widetilde{\psi}^a)^n(s) - (\pa)^*(s)| \leq \frac{8}{(\Delta \eta(t_n+s))^\alpha} + |(\widetilde{\psi}^a)^n(s) - (\pa)^*(s)|.
\end{align*}}By uniform convergence of $(\widetilde{\ups}^n)$ on each bounded interval, and   $\eta(t_n + s) \to \infty$, as $n \to \infty$, the claim holds.
\end{proof}}

\vspace{2mm}
\newcommand{\OL}{\Omega}
\noindent \textbf{Proof of Theorem \ref{thrm_attractors_beta}.}\label{proof_thrm2}  Recall $\bc(\ups) := \tc/\pc$.  Consider the initial condition $\ups(0) \in \cD_I$ with $\pc(0) = 0$, then ODE \eqref{eqn_ODE} simplifies to $\dot{\ups} = -\ups$, which clearly has a unique solution and further $\ups(t) \to \mathbf{0}$ as $t \to \infty$. We claim that ${\mathbf{0}} \in \cR$ as we next show that   with $\pc(0) > 0$, the solution $\ups$ converges to other equilibrium points. 

Let $\pc(0) > 0$, and  say without loss of generality, $\bc(\ups(0)) \in {\cal N}_i^{-}$ for some $i$. By Lemma \ref{lemma_psi_c_general}, $\pc(t) > 0$ for all $t \geq 0$, thus ODE \eqref{eqn_ODE} simplifies to $\dot{\ups} = \mathbf{h}(\bc(\ups)) - \ups$. Consider the following smooth ODE, with initial condition $\ups(0) $ (by (c), the right hand side given below is Lipschitz continuous):
\begin{align}\label{eqn_func_fli}
\begin{aligned}
    \dot{\ups} &= f_l^i(\bc) - \ups, \mbox{ where } \\
    f_l^i(x) := \mathbf{h}(x)1_{\{x < x^*_i\}\cap N_i^*} &+ \mathbf{h}_l^* 1_{\{x \geq x^*_i\}} + \mathbf{h}_l^o 1_{x \leq \Delta_l^i}, \mbox{ with } \\
    \mathbf{h}_l^* := \lim_{x_n \up x^*_i} \mathbf{h}(x_n), \ \mathbf{h}_l^o := \lim_{x_n \downarrow \Delta_l^i} \mathbf{h}(x_n), &\mbox{ and } \Delta_l^i := \inf\{ \bc(\ups) : \bc(\ups) \in {\cal N}_i^*\}.
\end{aligned}
\end{align}Then, by \cite[Theorem 1, sub-section 1.4, pp. 6]{piccinini2012ordinary}, the above smooth ODE has a unique solution, say $\ups^1(t)$. Let $\tau:= \inf\{ t : \bc(\ups^1(t)) = x^*_i\}$, then by Lemma \ref{lemma_tau_finite}, $\tau < \infty$. Observe that the solution of the original ODE \eqref{eqn_ODE}, with the same initial condition $\ups(0)$, coincides with $\ups^1(\cdot)$ for all $t < \tau$, as $\pc(t) > 0$ for all $t >0$ by Lemma \ref{lemma_psi_c_general} for such initial condition. Now, let  $\ups^\tau := \ups^1 (\tau)$ and observe $\beta^c (\ups^\tau)  = x_i^*$. Using similar logic, one can prove that $x_i^*$ is an attractor for ODE \eqref{eqn_beta_ode_simple}. - these kind of statements are required for $z$-ODE in prelim chapter. Further,  by uniqueness of the solutions of the smooth\footnote{The ODEs \eqref{eqn_func_fli} and \eqref{eqn_omega} are the two smooth ODEs.} ODEs, the solution of ODE \eqref{eqn_ODE} for $t > \tau$ is given by:
\begin{align}\label{eqn_ups2}
    \ups^2(t) = (\pc(t), x^*_i \pc(t), \pa(t), \ta(t)),
\end{align}where the three components of $\ups^2(\cdot)$, defined as $\OL(\cdot) := (\pc(\cdot), \pa(\cdot), \ta(\cdot))$ is the solution of the following initial value problem (IVP) for all $t \geq \tau$ (see \eqref{eqn_ODE}):
\begin{align}\label{eqn_omega}
\dot{\OL} = \mathbf{h}_i  - \OL, \mbox{ with } \OL (\tau) :=  \OL(\ups^*),  \mbox{ where constant, } \mathbf{h}_i := (h_{\psi}^c, h_{\psi}^a, h_{\theta}^a)|_{x^*_i}.
\end{align}Observe that $\bc(t) = x^*_i$ for all $t > \tau$ by (a). With this, $\ups(t) := \ups^1(t)1_{t < \tau} + \ups^2(t)1_{t > \tau}$ is the unique solution, which satisfies ODE \eqref{eqn_ODE}  for all $t \neq \tau$, and with initial condition $\ups(0)$. Thus, \eqref{eqn_ODE} satisfied \ref{a3}. Clearly from \eqref{eqn_omega}, 
$$
\ups(t) \to \mathbf{h}(x_i^*), \mbox{ where } \mathbf{h}(x_i^*) = (h^c_{\psi}, x_i^*h^c_{\psi}, h^a_{\psi}, h^a_{\theta})|_{x_i^*}.
$$
Similarly, one can show that $\ups(t) \to \mathbf{h}(x_i^*)$, if $\bc(\ups(0)) \in {\cal N}_i^{+}$. 

Thus, $\mathbf{h}(x_i^*)$ is an attractor for ODE \eqref{eqn_ODE}, with domain of attraction as $\cD_i := \{\ups\in \cD_I: \bc(\ups) \in  {\cal N}_i^*\} \cap \{\pc > 0\}.$ Since $x_i^* \in {\cal I}$ is arbitrary, $\cA = \{ \mathbf{h}(x_i^*) : x_i^* \in {\cal I}\}$, with corresponding domain of attraction as $\cD_{\cal A} =\cup_{1 \leq i \leq n} \cD_i$.  Also, ${\cal I}$ is an attractor for \eqref{eqn_beta_ode_simple}.

By hypothesis (b.i), any initial condition $\ups(0)$ with $\bc(\ups(0)) \in [0,1]-{\cal J}$ is already considered above. Now consider $\ups(0)$ with $\bc(\ups(0)) = y_i^* \in {\cal J}$, i.e., $\ups(0) \in {\mathbb S} (h(y_i^*))$. Then, the analysis follows as in \eqref{eqn_ups2}-\eqref{eqn_omega} to show that $\ups(t) \to \ups(y_i^*)$ as $t \to \infty$; the exponential convergence is clear from ODE \eqref{eqn_omega}. This proves that $\mathbf{h}(y_i^*)$ is a saddle point for ODE \eqref{eqn_ODE}.  Clearly, by (a), (b.ii)-(b.iii), $y_i^* \in {\cal J}$ is a saddle point for ODE \eqref{eqn_beta_ode_simple}. Hence, the theorem follows, as similar things are true for ${\mathbf 0}$.  \eop

\begin{lemma}\label{lemma_tau_finite}
The time $\tau$ defined in the proof of Theorem \ref{thrm_attractors_beta} is finite. 
\end{lemma}
\begin{proof}
By hypothesis (b),  $g_\beta(\cdot) > 0$ and continuous, for all $\bc \in {\cal N}_i^{-}$. Further,  $x_i^*$ is a point of discontinuity for $g_\beta$ and  $g_\beta(x_i^*) = 0$; thus $\bc(\mathbf{h}_l^*) = \lim_{x_n \up x_i^*} g_\beta (x_n) > 0$ (see \eqref{eqn_func_fli}), which implies, $\inf_{\{\bc \in  {\cal N}_i^{-}\}} g_\beta(\bc) > 0$. Observe $\tau$ is determined by $\bc$-component of $\ups^1(\cdot)$, the solution of  ODE \eqref{eqn_func_fli}. From \eqref{eqn_func_fli}, the latter is a continuous extension of the original ODE \eqref{eqn_ODE}, thus, the $\bc$-component of  the ODE \eqref{eqn_func_fli} can be uniformly lower bounded by $\inf_{\{\bc \in  {\cal N}_i^{-}\}} g_\beta(\bc) > 0$.
Thus, by Lemma \ref{lemma_general_ode_BPA}(a.ii), $\tau < \infty$.
\end{proof}

\begin{lemma}\label{lemma_general_ode_BPA}
Consider an initial value problem $\dot{z} = f(z, t)$, with $z(0) \in (z_0^l, z_0^u)$ where $f$ is a measurable function with finitely many discontinuities. 
\begin{enumerate}[label=(\alph*)]
    \item Say $f(z, t) > 0$, for all $z \in (z_0^l, z_0^u)$ and all $t$. Then:
    \begin{enumerate}[label=(\roman*)]
        \item  $z(\cdot)$ is an increasing function of $t$ till $\tau^u := \inf\{t : z(t) \geq z_0^u\}$.
        \item Say $f(z, t) > \delta$ for some $\delta > 0$, for all $z \in (z_0^l, z_0^u)$ and all $t$. Then, $\tau^u < \infty$.
    \end{enumerate}
    \item If $f(z, t) < 0$, for all $z \in (z_0^l, z_0^u)$ and all $t$, then $t \mapsto z(t)$ is a decreasing function till $\tau^l := \inf\{t : z(t) \leq z_0^l\}$, and if in addition $f(z, t) < -\delta$ for some $\delta > 0$, for all $z \in (z_0^l, z_0^u)$ and all $t$, then $\tau^l < \infty$.
\end{enumerate} 
\end{lemma}
\begin{proof}
We will provide the proof for part (a), and it can be done analogously for part (b). Contrary to the claim, let $\tau_1 < \tau_2 < \tau^u$ be two time points such that $z(\tau_1) \geq z(\tau_2)$, with  $z(\tau_1), z(\tau_2) \in (z_0^l, z_0^u)$. Then, we have:
$$
0 \geq  z(\tau_2) - z(\tau_1) = \int_{\tau_1}^{\tau_2} f(z(s), s) ds,
$$
which is a contradiction to the hypothesis. 
Now if possible, let $\tau_u = \infty$, then $z(t) < z_0^u$ for all $t$ and $t \mapsto z(t)$ is an increasing function (as proved before).
Further, since $z(t) = z(0) + \int_{0}^t f(z(s), s) ds > z(0) + t \delta $, there exists $T_\delta > 0$ such that $z(t) \geq z_0^u$ for all $t \geq T_\delta$, which contradicts $\tau^u = \infty$. 
\end{proof}

\begin{lemma}\label{lemma_psi_c_general}
Let \ref{a2} and \ref{a3} hold. Define \begin{align}\label{eqn_bound_m_general}
\begin{aligned}
    \leps &:= \inf\{\minf_{ix}(\bc) + \minf_{iy}(\bc): \bc \in [0,1], i \in \{x, y\}\}, \mbox{ and}\\
    \ueps &:= \sup\{\minf_{ix}(\bc) + \minf_{iy}(\bc): \bc \in [0,1], i \in \{x, y\}\}.
\end{aligned}
\end{align}For any $0 < \epsilon < \leps  - 1$,  define $A_\epsilon := [2\leps  - 1-\epsilon, 2\ueps -1 + \epsilon]$. In case, $\pc(0) \in int(A_\epsilon)$ (interior) for some $\epsilon > 0$, then $\pc(t) \in A_\epsilon$ for all $t\geq 0$.  
Thus, if  $\pc(0) > 0$, then, $\pc(t) > \pc(0) - \delta$ for all $t \geq 0$ and for any $\delta \in (0, \pc(0))$.
\end{lemma}

\begin{proof} Recall from \eqref{eqn_ODE}, ODE for $\pc$ is $\dot{\pc} = h_\psi^c(\bc) 1_{\pc > 0} - \pc$. Now, one can lower bound $h_\psi^c(\bc) - \pc$ as, for all $t$ (by \ref{a1} and \eqref{a2}):
\begin{align}
    h_\psi^c(\bc) - \pc
    &\geq 2\bc \leps + 2\big(1 -\bc \big) \leps - 1 - \pc = 2\leps - 1- \pc.
\end{align}It is easy to observe (by Weierstrass Theorem) that there exists a strict positive uniform lower bound $l_I$ for any closed interval $I \subset (0, 2\leps-1)$ as below:
\begin{eqnarray}\label{eqn_g_psi_general}
\dot{\pc} \geq 2\leps-1- \pc \ge  l_I > 0 \mbox{ for any }  \pc \in I \mbox{ and for all } t. 
\end{eqnarray}For the first part, consider $I = [\pc(0), 2\leps-1-\frac{\epsilon}{2}]$, where $\pc(0) \notin A_\epsilon$. Then, by \eqref{eqn_g_psi_general}, $\dot{\pc} > l_I 
$ for all $\pc \in I$ and all $t$. By Lemma \ref{lemma_general_ode_BPA}(a), we have $\tau^u:=\inf\{t : \pc(t) \geq 2\leps-1-\frac{\epsilon}{2}\} < \infty$, i.e., $\pc(\cdot)$ enters $A_{\epsilon}$ from the left. 

We will now explicitly show that $\pc(\cdot)$ can not exit $A_\epsilon$, once it enters/starts in it (set $\tau^u = 0$ when $\pc(0) \in int(K_\epsilon)$). In contrast, say $\pc$ leaves $A_\epsilon$ and to the left.  Observe  $\pc(\tau^u) > 2\leps-1- \epsilon$. 
For $\pc$ to exit $A_\epsilon$, by continuity of $\pc$ (and Intermediate Value Theorem, IVT), there exist $2\leps-1-\epsilon < \underline{\nu}< \overline{\nu}< 2\leps-1$ such that for some $t_2 > t_1 > \tau^u$,  $\pc(t_2) = \underline{\nu}$ and $\pc(t_1) = \overline{\nu}$. Then, by MVT, we have:
$$
\dot{\pc}(s) = \frac{\pc(t_2) - \pc(t_1)}{t_2-t_1}  = \frac{\underline{\nu}-\overline{\nu}}{t_2-t_1} < 0, 
$$
for some $s \in (t_1, t_2)$. This is a contradiction as $\dot{\pc}(t) > 0$ for $\pc \in (0, 2\leps-1)$ and any $t$.
Conclusively, ODE solution $\pc(\cdot)$ enters $A_{\epsilon}$ from  left when $\pc(0) < 2\leps-1-\epsilon$, and does not exit $A_\epsilon$ from  left.

Similarly from \eqref{eqn_ODE}, $\mathbf{h}_\psi^c(\bc) - \pc$ 
can be  upper bounded  as (by \ref{a1} and \eqref{a2}):
\begin{align}
    \mathbf{h}_\psi^c(\bc) - \pc
    &\leq 2\bc \ueps + 2\big(1 -\bc\big)\ueps -1 - \pc = 2\ueps - 1 - \pc,
\end{align}and $\dot{\pc} \leq   2\ueps - 1 - \pc \le  u_I < 0$ for all $t$ and for any  $\pc \in I $ where $I \subset (2\ueps - 1, \infty)$ is any closed interval. Then, applying similar arguments as above, one can show that $\pc(\cdot)$ enters and does not exit $A_\epsilon$ from/to  right as well.
\end{proof}

\noindent \textbf{Proof of Theorem \ref{thrmBPA}.}\label{proof_thrmBPA} We first study ODE \eqref{eqn_beta_ode_simple}, using which we then analyze ODE \eqref{eqn_ODE} or \eqref{ODE_BPA}. Observe by definition of $\minf_{xy}(\cdot)$, $\minf_{yx}(\cdot)$ in \ref{k2} that $0, 1$ are equilibrium points of ODE \eqref{eqn_beta_ode_simple}. Further,  \textit{$g_\beta(\bc)$ is convex or concave in only $(0, 1)$, respectively if $\minf \leq 0$ or $\geq 0$}, as can be seen from below (see  \ref{k2} for definitions):
\begin{align}\label{eqn_g_beta}
\begin{aligned}
    g_\beta(\bc) &= \left(-\bpam_{yx}^\infty + \bc \widetilde{m}^\infty - (\bc)^2 \minf \right)1_{\bc \in (0, 1)}, \mbox{ where}\\
\widetilde{m}^\infty &:= \bpam^\infty _{xx} + \bpam_{xy}^\infty 
 -\bpam_{yy}^\infty + \bpam^\infty_{yx}, \mbox{ and } \minf := \bpam^\infty _{xx} - \bpam^\infty_{yy}.
\end{aligned}
\end{align}

At first by Lemma \ref{lemma_psi_c_general}, $\mathbf{0}$ is a saddle point for ODE \eqref{eqn_ODE} and hence for \eqref{ODE_BPA}.
Now, let $\minf \geq 0$, and consider the following two sub-cases.  

\textbf{Sub-case 1:} $\bpam_{xy}^\infty > 0$ and $\bpam_{yx}^\infty > 0$. Since $g_\beta(\cdot)$ is continuous in $(0, 1)$:
\begin{align}\label{eqn_case1}
    g_\beta(0^+) = \lim_{\delta \to 0} g_\beta(\delta) = -\bpam_{yx}^\infty  < 0, \mbox{ and }
    g_\beta(1^-) = \lim_{\delta \to 0} g_\beta(1-\delta)  = \bpam_{xy}^\infty  > 0.
\end{align}Therefore, there exists  a unique zero of $g_\beta$, say $\bstar_r \in (0, 1)$. Further by concavity,  $g_\beta(\bc) < 0$ when $\bc < \bstar_r$ and $g_\beta(\bc) > 0$ when $\bc > \bstar_r$. Thus, the result follows for this case by Theorem \ref{thrm_attractors_beta} with $x_1^* = 0$, $x_2^* = 1$ and $y^* = \beta^*_r$. That is, $\{0,1\}$ is the attractor set,  $\{\bstar_r\}$ is the repeller set for ODE \eqref{eqn_beta_ode_simple}. Thus, $\cA = \{\mathbf{h}(0), \mathbf{h}(1)\}$ is the attractor set and $\cD = \{\mathbf{0}, \mathbf{h}(\bstar_r)\}$ is the saddle set for ODE \eqref{eqn_ODE}, with combined domain of attraction, $\cD$ as in (v) of the Theorem.  

\textbf{Sub-case 2:} $\bpam_{xy}^\infty > 0$ and $\bpam_{yx}^\infty = 0$. Observe $\bpam^\infty_{xx} < \bpam^\infty_{yy}$ is not possible here, as it would contradict $\minf \geq 0$. Thus, $\bpam^\infty_{xx} \geq \bpam^\infty_{yy}$. Therefore, for any $\beta  \in (0, 1)$, $g_\beta(\beta) = \beta(1-\beta)(\bpam^\infty_{xx} - \bpam^\infty_{yy}) + \beta \bpam^\infty_{xy} > 0$. Further, $g_\beta(1^-) > 0$, as in case 1. Thus, the result follows for this case as well by Theorem \ref{thrm_attractors_beta} with $x_1^* = 1$ and $y^* = 0$. 



This completes parts (i) and (ii) for the case when $\minf \geq 0$. Analogously, one can prove (i) and (ii) when $\minf \leq 0$. Then, the proof is complete using Theorem \ref{thrm_attractors_beta}.   
\eop



\hide{
\noindent \textbf{Part (b):} Observe that $g_\beta(\cdot)$ with constant $m_{ij}^\infty(\cdot)$ can be re-written (as in \eqref{eqn_g_beta}):
\begin{align}
    g_\beta(\bc) = \minf_{yx} + \bc \widetilde{m}^\infty - (\bc)^2 \minf, \mbox{ for all } \bc \in [0,1].
\end{align}
Unlike part (a), $0, 1$ are not equilibrium points for the ODE \eqref{eqn_beta_ODE}. Now, as in part (a), we provide the proof for $\minf \geq 0$, which implies $g_\beta(\bc)$ is a concave function for all $\bc \in [0,1]$. 

\textbf{(i)} Consider $\minf_{xy} > 0$. At first, let $\minf_{yx} > 0$. Observe that $g_\beta(0) = \minf_{yx} > 0$ and $g_\beta(1) = -\minf_{xy} < 0$. Thus, there exists a unique $\bstar \in (0, 1)$ such that $g_\beta(\bstar) = 0$, $g_\beta(\bc) > 0$ for all $\bc \in [0, \bstar)$ and $g_\beta(\bc) < 0$ for all $\bc \in (\bstar, 1]$. 
This implies that the function $t \mapsto \bc(t)$ is strictly increasing (decreasing) till $\tau^u := \inf\{t : \bc(t) \geq \bstar\}$, if $\bc(t_0) \in [0, \bstar)$ for some $t_0 \in [0, \tau_u)$ (respectively, till $\tau^l := \inf\{t : \bc(t) \leq \bstar\}$, if $\bc(t_0) \in (\bstar, 1]$ for some $t_0 \in [0, \tau_l)$). As a result, $\bstar$ is AS with domain of attraction as $[0,1]$. 

Secondly, let $\minf_{yx} = 0$ with $\minf_{xx} > \minf_{yy}$. Here, $g_\beta(0) = 0$, however $g_\beta(0^+) > 0$. Thus, $\bstar = \widetilde{m}^\infty/\minf \in (0, 1)$ is AS with domain of attraction as $(0,1]$. 

\textbf{(ii)} Consider $\minf_{xy} > 0$ and $\minf_{yx} = 0$ with $\minf_{xx} \leq \minf_{yy}$. Note that $g_\beta(\bc) < 0$ for all $\bc \in (0, 1]$ and $g_\beta(0) = 0$. Then, as in (b.i), $\bstar = 0$ is AS with domain of attraction as $[0,1]$.}

\noindent \textbf{Proof of Corollary \ref{corollary_BPA}.}\label{proof_cor_BPA}
Given limit mean functions as in \ref{k2}, the assumption \ref{a3} is guaranteed by Theorem \ref{thrmBPA}. We now prove the assumption \ref{a4}.

$\cA$ and $\cR$ are the attractor and saddle sets of ODE \eqref{ODE_BPA} respectively, with subset of the combined domain of attraction as $\cD_I$, as identified in Theorem \ref{thrmBPA}. Towards getting a compact sub-domain of $\cD_I$, as in \eqref{eqn_pi_n}, 
 from \eqref{Eqn_XnYnSnetc_general}, \eqref{eqn_dynamics_BPA} and \ref{k1}, one can bound $\Pa_n$:
\begin{align*}
    \Pa_n \leq \overline{\Psi}_n^a :=  \frac{1}{n}\left(\sum_{k=1}^{\min\{\nu_e, n\}} \left( \overline{\xi}_{xx, k} + \overline{\xi}_{yy, k} \right)1_{\{\Pc_k > 0\}} + s_0^c \right).
\end{align*}
As before, $\overline{\Psi}_n^a \to E[\overline{\xi}_{xx,1} + \overline{\xi}_{yy,1}]$ a.s. in survival paths and $\overline{\Psi}_n^a \to 0$ in extinction paths, as $n \to \infty$. Thus, $\cS :=  \cD_I \cap  \left\{\ups :  \pa \in [0, E[\overline{\xi}_{xx,1} + \overline{\xi}_{yy,1}] ]\right\}$ is the compact subset of $\cD_I$ and $p_{b} := P(\Ups_n \mbox{ visits } \cS \mbox{ i.o.}) = 1$. Hence, by Theorem \ref{thrmBPA} and Theorem \ref{thrm1}(ii), we have $\Ups_n \to \cA \cup \cR$ with probability $1$. \eop

    \chapter{For Chapter \ref{ch:journal2}}\label{appendix_journal2}

\noindent \textbf{Proof of Theorem \ref{thm1}}
The proof follows exactly as in \cite{agarwal2021new}, except for some changes due to unnatural deaths. Here, we directly mention the SA based scheme for the new process, and necessary details where ever required. 

From \eqref{evolve_x_up_time_prop}, the embedded process immediately after $n$-th death, when say an $x$-type individual $d$-dies, is given by:
\begin{equation}\label{evolve_x_up_gen_prop}
\begin{aligned}
C_{x, n} &= C_{x, n-1}  + \offs_{xx, d, n}(\Om_{n-1}) - 1, \ \ \  T_{x, n} = T_{x, n-1}  + \offs_{xx, d, n}(\Om_{n-1}), \\
C_{y,n} &= C_{y, n-1} + \offs_{xy, d, n}(\Om_{n-1}),  \ \ \ A_{y, n} = A_{y, n-1} + \offs_{xy, d, n}(\Om_{n-1}).
\end{aligned}
\end{equation}
The ratios in $\Ups_n$ can be re-written as (with $\epsilon_{n-1} := 1/n$):
\begin{align}\label{eqn_SA_scheme_prop}
\Ups_n &= \Ups_{n-1} + \frac{1}{n}\mathbf{L}_{n}, \mbox{ where } \mathbf{L}_{n} := (L_{n}^{\psi, c}, L_{n}^{\theta, c}, L_{n}^{\psi, a}, L_{n}^{\theta, a})^t, \mbox{ with}
\end{align}
\vspace{-2mm}
{\small 
\begin{align*}
\begin{aligned}
L_{n}^{\psi, c} &:=  \left\{ \sum_{d \in [d_x]} \left( H^x_{n,d} (\offs_{x, d, n}(\Om_{n-1})-1) \right) + \sum_{d \in [d_y]} \left( H^y_{n,d} (\offs_{y, d, n}(\Om_{n-1}) - 1)  \right)  \right\} 1_{\Pc_{n-1} > 0}   -  \Pc_{n-1}, \\
L_{n}^{\theta, c} &:= \left\{ \sum_{d \in [d_x]} \left( H^x_{n,d} (\offs_{xx, d, n}(\Om_{n-1})-1) \right) + \sum_{d \in [d_y]} \left( H^y_{n,d} \offs_{yx, d, n}(\Om_{n-1})  \right) \right\}1_{\Pc_{n-1} > 0} - \Tc_{n-1},  \\
L_{n}^{\psi, a} &:=  \bigg\{ \sum_{d \in [d_x]} \left( H^x_{n,d} \offs_{x, d, n}(\Om_{n-1}) \right) + \sum_{d \in [d_y]} \left( H^y_{n,d} \offs_{y, d, n}(\Om_{n-1})   \right)  \bigg\}1_{\Pc_{n-1} > 0}  - \Pa_{n-1}, \mbox{ and}\\
L_{n}^{\theta, a} &:= \bigg\{ \sum_{d \in [d_x]} \left( H^x_{n,d} \offs_{xx, d, n}(\Om_{n-1}) \right) + \sum_{d \in [d_y]} \left( H^y_{n,d} \offs_{yx, d, n}(\Om_{n-1})  \right) \bigg\}1_{\Pc_{n-1} > 0}  - \Ta_{n-1}, \mbox{ where}\\
\offs_{x, d, k} &:= \offs_{xx, d, k} + \offs_{xy, d, k}, \ \  
    \offs_{y, d, k} := \offs_{yy, d, k} + \offs_{yx, d, k},
\end{aligned}
\end{align*}}

\vspace{-0.45cm}
$H^x_{k,d} \in \{0, 1\}$ indicates that an $x$-type individual $d$-dies at $k$-th epoch such that:
$$
\sum_{d \in D_x} H^x_{k,d} \in \{0, 1\} \mbox{ and } \sum_{d \in D_y} H^y_{k,d} := 1 - \sum_{d \in D_x} H^x_{k,d}.
$$

Henceforth, the proof of part (i) has two major steps: (a) to construct a sequence of piece-wise constant interpolated trajectories for almost all sample-paths; (b) to prove that the designed trajectories are equicontinuous in extended sense. We will provide the proof in terms of  $\tc$-component of the vector $\ups$, when the proof for the remaining components goes through in exactly similar manner.

Define $\gna = (\rho_\psi^c, \rho_\theta^c, \rho_\psi^a, \rho_\theta^a)$ as the conditional expectation, $E[\mathbf{L}_n|\mathcal{F}_n] =: \gna(\Ups_n, t_n)$, with respect to the sigma algebra, ${\cal F}_n  := \sigma\{\Om_k : 1 \leq k < n \}$ (see \cite[(16)]{agarwal2021new}).
Let $\Ups^n(\cdot) := (\Psi^{n, c}(\cdot), \Theta^{n, c}(\cdot), \Psi^{n, a}(\cdot), \Theta^{n, a}(\cdot))$  be  the constant piece-wise interpolated trajectory defined as below (see \eqref{eqn_SA_scheme_prop}, and recall $t_n = \sum_{i=1}^n \epsilon_{i-1}$):
\begin{align}\label{eqn_interpolated_traj_1_prop}
    \begin{aligned}
        \Theta^{n, c}(t) &:= \Tc_n + \int_0^t g_\theta^c(\Ups^n(s)) ds +  \sum_{i=n}^{\eta(t_n+t)-1} \epsilon_i L_i^{\theta, c} - \int_0^t g_\theta^c(\Ups^n(s)) ds\\
    &=  \Tc_n + \int_0^t g_\theta^{c}(\Ups^n) ds + M^{n, \theta, c}(t) + R^{n, \theta, c}(t) + D^{n, \theta, c}(t), \mbox{ where}\\
    M^{n, \theta, c}(t) &:=  \sum_{i=n}^{\eta(t_n + t)-1}  \epsilon_i \left(L_i^{\theta, c} - \rho_\theta^{c}(\Ups_i, t_i)\right), \\
    R^{n, \theta, c}(t) &:=  \sum_{i=n}^{\eta(t_n + t)-1}\epsilon_i g_\theta^{c}(\Ups_i) - \int_0^t g_\theta^{c}(\Ups^n) ds,  \\
    D^{n, \theta, c}(t) &:= \sum_{i=n}^{\eta(t_n + t)-1} \epsilon_i D_i^{\theta, c}, \mbox{ where } D_i^{\theta, c}:=\rho_\theta^{c}(\Ups_i, t_i) - g_\theta^{c}(\Ups_i),
    \end{aligned}
\end{align}
$\Psi^{n, c}(t), \Psi^{n, a}(t)$ and $\Theta^{n, a}(t)$ are defined analogously. As in \cite{agarwal2021new}, the  extended equicontinuity can be proved for $M^{n, \theta, c}(\cdot)$, $R^{n, \theta, c}(\cdot)$. For, $D^{n, \theta, c}(\cdot)$ 
 the procedure again follows as in \cite{agarwal2021new}  when $S_n \to 0$; however for sample paths where $S_n \nto 0$, the  arguments for proving the equicontinuity for $D^{n, \theta, c}(\cdot)$ slightly changes as below:
\begin{align}\label{eqn_bound_Di}
\begin{aligned}
|D_i^{\theta, c}| &\leq  |f_{\beta}(\Om_i) (m_{xx}(\Om_i) -1) - f_{\beta}^\infty(\Bc_i) (m_{xx}^\infty(\Bc_i)-1)|\\
&\hspace{2cm}+ | (1-f_{\beta}(\Om_i)) m_{yx}(\Om_i) - (1-f_{\beta}^\infty(\Bc_i))m_{yx}^\infty(\Bc_i)|\\
&\leq |f_{\beta}(\Om_i) m_{xx}(\Om_i)  - f_{\beta}^\infty(\Bc_i) m_{xx}^\infty(\Bc_i)| + |f_{\beta}(\Om_i) - f_{\beta}^\infty(\Bc_i)|\\
&\hspace{2cm}+ |m_{yx}(\Om_i) - m_{yx}^\infty(\Bc_i)|  + |f_{\beta}(\Om_i) m_{yx}(\Om_i)  - f_{\beta}^\infty(\Bc_i) m_{yx}^\infty(\Bc_i)|
\end{aligned}
\end{align}
In the above, under \ref{a2_prop}, the third term is bounded above by $1/(S_i)^\alpha$. The second term can be bounded above as follows:
\begin{align*}
    |f_{\beta}(\Om_i) - f_{\beta}^\infty(\Bc_i)| &= \Bc_i \left|\frac{\sum_{d \in D_x} \lambda_{x, d}(\Om_i)}{d(\Om_i)} - \frac{\sum_{d \in D_x} \lambda_{x, d}^\infty(\Bc_i)}{d^\infty(\Bc_i)} \right | \\
    &\hspace{-2cm} = \Bc_i \left|\frac{\sum_{d \in D_x} \lambda_{x, d}(\Om_i)}{d(\Om_i)} - \frac{\sum_{d \in D_x}  \lambda_{x, d}^\infty(\Bc_i)}{d(\Om_i)} + \frac{\sum_{d \in D_x}  \lambda_{x, d}^\infty(\Bc_i)}{d(\Om_i)} - \frac{\sum_{d \in D_x} \lambda_{x, d}^\infty(\Bc_i)}{d^\infty(\Bc_i)} \right |\\
    &\hspace{-2cm}\leq \frac{\Bc_i}{d(\Om_i)} \sum_{d \in D_x} |\lambda_{x, d}(\Om_i) - \lambda_{x, d}^\infty(\Bc_i)| + \Bc_i\left|  \sum_{d \in D_x}\lambda_{x, d}^\infty(\Bc_i) \left( \frac{1}{d(\Om_i)} - \frac{1}{d^\infty(\Bc_i)}\right)\right|\\
    &\hspace{-2cm}\leq \frac{\Bc_i}{d(\Om_i)} \left(\frac{|D_x|}{(S_i^c)^\alpha} +  \frac{\left|  \sum_{d \in D_x}\lambda_{x, d}^\infty(\Bc_i)\right|}{d^\infty(\Bc_i)} \left|d^\infty(\Bc_i) - d(\Om_i) \right|  \right)\\
    &\hspace{-2cm}\leq \frac{\Bc_i}{d(\Om_i)} \left(\frac{|D_x|}{(S_i^c)^\alpha} +  \frac{\left|  \sum_{d \in D_x}\lambda_{x, d}^\infty(\Bc_i)\right|}{d^\infty(\Bc_i)} \frac{(|D_x| + |D_y|)}{(S_i^c)^\alpha}  \right) \\
    &\hspace{-2cm}\leq \frac{\Bc_i}{d(\Om_i)} \frac{|D_x| + |D_y|}{(S_i^c)^\alpha} \left(1 + \frac{\left|  \sum_{d \in D_x}\lambda_{x, d}^\infty(\Bc_i)\right|}{d^\infty(\Bc_i)}  \right)\\
    &\hspace{-2cm}\leq \frac{\Bc_i}{d(\Om_i)} \frac{|D_x| + |D_y|}{(S_i^c)^\alpha} \left(1 + \frac{1}{B^c_i} \right) \hspace{5mm}\left(\mbox{ since } d^\infty(\Bc_i) \geq B_c^i  \sum_{d \in D_x}\lambda_{x, d}^\infty(\Bc_i)\right)\\
    &\hspace{-2cm}= \frac{(|D_x| + |D_y|)(B_i^c + 1)}{d(\Om_i) (S_i^c)^\alpha} \leq \frac{2(|D_x| + |D_y|)}{d(\Om_i) (S_i^c)^\alpha}
\end{align*}
Define $\Delta_1 := \min\left\{\inf_{\Om} \lambda_{x, d}(\Om), \inf_{\Om} \lambda_{y, d}(\Om) \right\} > 0$, by \ref{a1_prop}. Then:
$$d(\Om_i) \geq \Bc_i \inf_{\Om} \lambda_{x, d}(\Om) + (1-\Bc_i) \inf_{\Om} \lambda_{y, d}(\Om) \geq \Delta_1.$$
Thus, we have:
\begin{align}\label{eqn_sec_term}
    |f_{\beta}(\Om_i) - f_{\beta}^\infty(\Bc_i)| &\leq \frac{2(|D_x| + |D_y|)}{(S_i^c)^\alpha} \frac{1}{\Delta_1}
\end{align}
The first term in \eqref{eqn_bound_Di} can be bounded as follows  under \ref{a2_prop} and \eqref{eqn_sec_term}:
\begin{align*}
    |f_{\beta}(\Om_i) m_{xx}(\Om_i) - f_{\beta}^\infty(\Bc_i) m_{xx}^\infty(\Bc_i)|  &\\
    &\hspace{-2cm}\leq |f_{\beta}(\Om_i)| |m_{xx}(\Om_i) - m_{xx}^\infty(\Bc_i)| + |m_{xx}^\infty(\Bc_i)| |f_{\beta}(\Om_i) - f_{\beta}^\infty(\Bc_i)|\\
    &\hspace{-2cm}\leq \frac{1}{(S_i)^\alpha} + \frac{2(|D_x| + |D_y|)}{(S_i^c)^\alpha} \frac{1}{\Delta_1} \left( E[\overline{\Gamma}] + \frac{1}{(S_i)^\alpha}\right).
\end{align*}
Similarly, the fourth term in \eqref{eqn_bound_Di} can be upper bounded as follows:
\begin{align*}
    |f_{\beta}(\Om_i) m_{yx}(\Om_i) - f_{\beta}^\infty(\Bc_i) m_{yx}^\infty(\Bc_i)| 
    &\leq \frac{1}{(S_i)^\alpha} + \frac{2(|D_x| + |D_y|)}{(S_i^c)^\alpha} \frac{1}{\Delta_1} \left( E[\overline{\Gamma}] + \frac{1}{(S_i)^\alpha}\right).
\end{align*}
Thus, $D_i^{\theta, c}$ can be upper bounded as follows for some $K <\infty$ (recall, $\alpha \geq 1$):
\begin{align*}
    D_i^{\theta, c} &\leq 2\left(\frac{1}{(S_i)^\alpha} + \frac{2(|D_x| + |D_y|)}{(S_i^c)^\alpha} \frac{1}{\Delta_1} \left( E[\overline{\Gamma}] + \frac{1}{(S_i)^\alpha}\right)\right) + \frac{1}{(S_i^c)^\alpha} + \frac{2(|D_x| + |D_y|)}{(S_i^c)^\alpha} \frac{1}{\Delta_1} \\
    &\leq \frac{K}{(S_i^c)^\alpha} \leq \frac{K}{S_i^c} = \frac{K}{\Pc_i \eta(t_i)} \leq \frac{K}{\Delta i}.
\end{align*}
This implies that, (recall $\epsilon_i = 1/(i+1)$ and $\alpha \geq 1$)
\begin{align*} 
|D^{n, \theta, c}(t)| = \left|\sum_{i= n}^{\eta(t_n + t) - 1}\epsilon_i D_i^{\theta, c}\right| \leq \sum_{i= n}^{\eta(t_n + t) - 1} \frac{K}{\Delta i (i+1)} \leq \sum_{i= n}^{\infty} \frac{K}{\Delta i (i+1)}, \mbox{ for any }t.
\end{align*}Thus, $D^{n, \theta, c}(t)$ uniformly converges to $0$ as $n \to \infty$.  In all, $(\Theta^{n, c}(\cdot))$ is equicontinuous in the extended sense. 

The proof of part (ii) follows exactly as in \cite{agarwal2021new}.  \eop



\noindent \textbf{Proof of Theorem \ref{thrm_beta_ODE_prop}}
Observe that each point $x_i^* \in {\cal I}$ can either be a point of dis-continuity or continuity for $g_\beta$.
In the former case, when $x_i^*$ is either an attractor or repeller of the ODE \eqref{eqn_beta_ode_simple_prop}, the result can be proved exactly as in \cite[Theorem 2.]{agarwal2021new}. In fact, when $x_i^*$ is a saddle point of the ODE \eqref{eqn_beta_ode_simple_prop}, the analysis can be easily extended similar to the case when $x_i^*$ is a repeller.

Now consider $x_i^* \in {\cal I}$ such that $g_\beta$ is continuous at $x_i^*$.
Let $\ups(0) \in \cD_I$ with $\pc(0) > 0$.  By \cite[Lemma 5.]{agarwal2021new}, $\pc(t) > 0$ for all $t \geq 0$, thus ODE \eqref{eqn_ODE_prop} simplifies to $\dot{\ups} = \mathbf{h}(\beta(\ups)) - \ups$. Now, we will prove the claim for different possibilities of $x^*$ as in the hypothesis separately. 
Firstly for all cases global solution exists because of Lipschtiz continuity.

\underline{Part (i)} Without loss of generality, let $\beta(0) \in {\cal N}_i^-$. Then, by \cite[Lemma 4(a)(i)]{agarwal2021new}, $\beta(t)$ increases to $x^*_i$ for all $t < \tau:=\inf\{t : \beta(t) = x^*_i\}$. If $t < \infty$, then $\beta(t) = x^*_i$ for all $t \geq \tau$ (as $x^*_i$ is an equilibrium point). Then, clearly, $\beta(t) \to x^*_i$ and $\ups(t) \to \mathbf{h}(x^*_i)$ as $t \to \infty$, as above.

Else say $\tau = \infty$;  then for every $\delta > 0$, there exists a $T_\delta < \infty$ (guaranteed as before by \cite[Lemma 4(a)(i)]{agarwal2021new} because by continuity the RHS of ODE can be uniformly bounded by non-zero values) such that:
\begin{align*}
x^*_i - \delta \leq \beta(t) \leq x^*_i + \delta \mbox{ for all }t \geq T_\delta.
\end{align*}
Thus, $\beta(t) \to x^*_i$ as $t \to \infty$. This also implies that:
\begin{align*}
\underline{\mathbf{h}}_\delta(x^*_i) - \ups &\leq \dot{\ups} \leq  \overline{\mathbf{h}}_\delta(x^*_i) - \ups \mbox{ for all }t \geq T_\delta, \mbox{ for}\\
\overline{\mathbf{h}}_\delta(x^*_i) := &  \sup_{x \in \overline{{\cal N}_{\delta}}(x^*_i)}\mathbf{h}(x) \mbox{ and } \underline{\mathbf{h}}_\delta(x^*_i) := \inf_{x \in \overline{{\cal N}_{\delta}}(x^*_i)}\mathbf{h}(x).
\end{align*}
By Comparison Theorem in \cite{piccinini2012ordinary} for ODEs having Lipschitz continuous right hand sides and using classical methods to derive the upper  and lower bounds, we get: 
$$
\underline{\mathbf{h}}_\delta(x^*_i)   + e^{-t + T_\delta}(
\ups(T_\delta) - \underline{\mathbf{h}}_\delta(x^*_i))
\leq \ups(t) \leq
\overline{\mathbf{h}}_\delta(x^*_i)  + e^{-t + T_\delta}( \ups(T_\delta) - \overline{\mathbf{h}}_\delta(x^*_i))
%
\mbox{ for all }t\geq T_\delta.
$$
Then clearly by considering limits $t\to \infty$ we have:
\begin{align*}
    \underline{\mathbf{h}}_\delta(x^*_i) &\leq \liminf_{t\to \infty} \ups(t) \leq \limsup_{t\to \infty} \ups(t) \leq \overline{\mathbf{h}}_\delta(x^*_i),
\end{align*}
and now letting $\delta \to 0$:
\begin{align*}
    \mathbf{h}(x^*_i) 
\leq \liminf_{t\to \infty} \ups(t) \leq \limsup_{t\to \infty} \ups(t) \leq  \mathbf{h}(x^*_i).
\end{align*}
Hence, $\ups(t) \to \mathbf{h}(x^*_i)$ as $t \to \infty$.

\underline{Part (ii)} If $\beta(0) = x^*_i$, then clearly $\beta(t) = x^*_i$ for all $t \geq 0$ and $\ups(t) \to \mathbf{h}(x^*_i)$ as $t \to \infty$. However if $\beta(0) \in {\cal N}_i^-$, then it can be shown as above that $\beta(t) \to y^* := \max\{y \in {\cal I}: y < x^*_i\}$. Similarly, if $\beta(0) \in {\cal N}_i^+$, then $\beta(t) \to y^* := \min\{y \in {\cal I}: y > x^*_i\}$. Thus, $x^*_i$ is a repeller for ODE \eqref{eqn_beta_ode_simple_prop} and $\mathbf{h}(x^*_i)$ is a saddle point for ODE \eqref{eqn_ODE_prop}.

\underline{Part (iii)} If $\beta(0) = x^*_i$, then clearly $\beta(t) = x^*_i$ for all $t \geq 0$ and $\ups(t) \to \mathbf{h}(x^*_i)$ as $t \to \infty$. Say $g(x) > 0$ for all $x \in {\cal N}_i^-$ and $g(x) > 0$ for all $x \in {\cal N}_i^+$.
Then, if $\beta(0) \in {\cal N}_i^-$, $\beta(t) \to x^*_i$, as shown for part 1. While if $\beta(0) \in {\cal N}_i^+$, then $\beta(t) \to y^* := \min\{y \in {\cal I}: y > x^*_i\}$, as shown for part 2. Thus, $x^*_i$ is a saddle point for ODE \eqref{eqn_beta_ode_simple_prop} and $\mathbf{h}(x^*_i)$ is a saddle point for ODE \eqref{eqn_ODE_prop}.

Lastly, consider the initial condition $\ups(0) \in \cD_I$ with $\pc(0) = 0$, then ODE \eqref{eqn_ODE_prop} simplifies to $\dot{\ups} = -\ups$, which clearly has  unique solution and $\ups(t) \to \mathbf{0}$ as $t \to \infty$. We have shown above that whenever $\pc(0) > 0$, $\ups(t) \nto \mathbf{0}$. Therefore, $\mathbf{0} \in \cR$. \eop


\noindent \textbf{Proof of Theorem \ref{thrm_BP_to_fake}}
 At first, observe that in view of the hypothesis regarding ${\cal F}$ and \eqref{eqn_lambda} with $\sum_i \mu_i = 1$, the assumption \ref{a1_prop} holds. Further, it is clear from \eqref{eqn_tag_wi}-\eqref{eqn_final_shares}, \eqref{eqn_mean_matrix} and \eqref{eqn_mean_matrix_red_coloring} that the assumption \ref{a2_prop} holds. 

We will now prove that $\cA^u_\beta \neq \emptyset$, which will then imply that the assumption \ref{a3_prop} holds,  by Theorem \ref{thrm_beta_ODE_prop}; this would complete Theorem \ref{thm1}(i). Towards proving the claim, note that (recall $\mu_2 > 0$):
\begin{align}\label{eqn_bound_g}
\begin{aligned}
    \underline{g}^u_\beta(\beta)  &< g_\beta^u(\beta) \leq \overline{g}^u_\beta(\beta) \mbox{ for all } \beta \in [0,1] \mbox{ where} \\
    \underline{g}^u_\beta(\beta) &:= \left(-\beta \mu_2 - \beta \mu_1 (1-\alpha_x^u \rho) + (1-\beta) \mu_1 \rho  \alpha_y^u \right) m_f\eta^u - \beta \mu_a m_f \eta_a, 
    \mbox{ and}\\
    \overline{g}^u_\beta(\beta) &:= \left(- \beta \mu_1 (1-\alpha_x^u \rho) + (1-\beta) \mu_1 \rho  \alpha_y^u  + \mu_2 (1-\beta)\right) m_f\eta^u - \beta \mu_a m_f \eta_a.
\end{aligned}
\end{align}
Now, $\underline{g}^u_\beta(0) = \mu_1 \rho  \alpha_y^u m_f\eta^u \geq 0$; thus, $g_\beta^u(0) > 0$. Further, $\overline{g}^u_\beta(1) = -\mu_1(1-\alpha_x^u\rho) m_f\eta^u - \mu_a m_f \eta_a \leq 0$; thus, $g_\beta^u(1) \leq 0$. Since $g_\beta^u(\beta)$ is a continuous function of $\beta$, therefore there exists at least one zero of $g_\beta^u$, say $\beta^{u, \infty}$ such that $g_\beta^u > 0$ in ${\cal N}_\epsilon^-(\beta^{u, \infty})$ and  $g_\beta^u < 0$ in ${\cal N}_\epsilon^+(\beta^{u, \infty})$; ${\cal N}_\epsilon^+(\beta^{u, \infty}) = \emptyset$ if $\beta^{u, \infty} = 1$. Then, by Theorem \ref{thrm_beta_ODE_prop}, $\beta^{u, \infty} \in \cA_\beta^u$; thus, $\cA_\beta^u \neq \emptyset$.

Since $\overline{g}^u_\beta(\beta)$ is a linear function such that $\overline{g}^u_\beta(0) > 0$ and (recall) $\overline{g}^u_\beta(1) \leq 0$, therefore, $\overline{\beta}^u \in (0,1]$, given in \eqref{eqn_beta_bar}, is the unique zero of $\overline{g}^u_\beta(\beta)$. Further, since $g_\beta^u \leq \overline{g}^u_\beta$ and $\overline{g}^u_\beta(\beta) < 0$ for all $\beta \in (\overline{\beta}^u, 1]$ when $\overline{\beta}^u < 1$, therefore, there exists no zero of $g_\beta^u$ in $(\overline{\beta}^u, 1]$; if $\overline{\beta}^u = 1$, then also, any zero of $g_\beta^u$ is atmost $1$. Thus, if at all, there is any zero of $g_\beta^u$, which can be an attractor or repeller or saddle point of \eqref{eqn_general_g_beta}, it is lesser than or equals to $\overline{\beta}^u$. Next, notice that there is a unique zero of the function $\underline{g}^u_\beta$, namely $\underline{\beta}^u \in (0,1)$, as given in \eqref{eqn_beta_bar}. Again using similar arguments as before, we get that $\beta^{u, \infty} > \underline{\beta}^u$. This proves \eqref{eqn_beta_bar}. 

Now, by Theorem \ref{thrm_beta_ODE_prop}, the attractor and saddle sets are as in the hypothesis with subset of the combined domain of attraction as $\cD_I$. 

We will now identify the compact sub-domain of $\cD_I$ for completing the proof using Theorem \ref{thm1}. From \ref{a1_prop} for our case, one can bound $\Pa_n$:
\begin{align*}
    0 \leq \Pa_n \leq \overline{\Psi}_n^a :=  \frac{1}{n}\left(\sum_{k=1}^{\min\{\nu_e, n\}} 2{\cal F}1_{\{\Pc_k > 0\}} + s_0^c \right).
\end{align*}
By strong law of large numbers, $\overline{\Psi}_n^a \to 2E[{\cal F}]$ a.s. in survival paths and $\overline{\Psi}_n^a \to 0$ in extinction paths, as $n \to \infty$. Thus, $\cS :=  \cD_I \cap  \left\{\ups :  \pa \in [0, 2E({\cal F}) ]\right\}$ is the compact subset of $\cD_I$ and $p_{b} := P(\Ups_n \mbox{ visits } \cS \mbox{ i.o.}) = 1$. Hence, by Theorem \ref{thm1}(ii), the claim holds. \eop


\noindent \textbf{Proof of Theorem \ref{thrm_unique_att}}
Let all parameters except $\kappa$ be fixed. 
Consider the case when $\nabla^u(\kappa, \kappa + \partial \kappa) = g^u_\beta(\beta^{\infty, u}(\kappa); \kappa+\partial \kappa) > 0$ for some $\partial \kappa > 0$. Since $g^u_\beta(\beta; \kappa+\partial \kappa)$ is either a convex or concave or linear function of $\beta$ with a unique zero in $(0,1)$, therefore, there exists a $\beta^{\infty, u}(\kappa + \partial \kappa) > \beta^{\infty, u}(\kappa)$ such that $g^u_\beta(\beta^{\infty, u}(\kappa + \partial \kappa); \kappa+\partial \kappa) = 0$. One can prove the claim similarly when $\nabla^u(\kappa, \kappa + \partial \kappa) < 0$. Lastly if $\nabla^u(\kappa, \kappa + \partial \kappa) = 0$, then again due to uniqueness, $\beta^{\infty, u}(\kappa + \partial \kappa) = \beta^{\infty, u}(\kappa)$. \eop


\noindent \textbf{Proof of Corollary \ref{corollary_ex_wm}}
We will first show that the function $g_\beta^{o, u}$ is either convex or concave or linear depending upon warning-specific and user-specific parameters. Towards this, note that for each $u$:
\begin{align}\label{eqn_convex_g_beta}
    \begin{aligned}
    \frac{d g^{o, u}_\beta(\beta)}{d \beta} 
    &= - (\mu_1 + \mu_2)m_f\eta^u + (\alpha_x^u - \alpha_y^u) (\mu_1 \rho + \mu_2 \omega(\beta)) m_f\eta^u  \\
    &\hspace{2cm}+ (\beta \alpha_x^u  + (1-\beta)\alpha_y^u)   \frac{bw \mu_2 m_f\eta^u}{(\beta + b (1-\beta))^2} - \mu_a m_f \eta_a\\
    \implies    \frac{d^2 g^{o, u}_\beta(\beta)}{d\beta^2}
    &= \frac{2 m_f\eta^u bw \mu_2}{(\beta + b (1-\beta))^3} \left( b \alpha_x^u - \alpha_y^u \right).
    \end{aligned}
\end{align} 
Thus, if $bw\mu_2(b \alpha_x^u - \alpha_y^u)= 0$ or $< 0$ or $>0$, then $g_\beta^{o, u}$ is a linear, concave or convex function respectively. From \eqref{eqn_beta_ODE_etac1}:
\begin{align*}
    g^{o, u}_\beta(0) &= 
    \left(\mu_1 \rho + \mu_2 \gamma\right)\alpha_y^u m_f\eta^u > 0, \mbox{ and }\\
    g^{o, u}_\beta(1) 
    &= - \bigg( \mu_1 m_f\eta^u \left(1-\alpha_x^u \rho\right) + \mu_2 (1 - \alpha_x^u(w+\gamma))m_f\eta^u + \mu_a m_f \eta_a \bigg) < 0;
\end{align*}the last inequality in above holds as $\alpha_x^u(w+\gamma) \leq 1$ for each $u$ and $\alpha_x^u \rho < \alpha_x^u < 1$. 
Therefore, there exists a unique $\beta^{o, \infty, u} \in (0,1)$ such that $g^u_\beta(\beta^{o, \infty, u}) = 0$, $g^u_\beta(\beta) > 0$ for all $\beta \in [0, \beta^{o, \infty, u})$ and $g^u_\beta(\beta) < 0$ for all $\beta \in (\beta^{o, \infty, u}, 1]$. This implies that for the ODE \eqref{eqn_general_g_beta}, $t \mapsto \beta^{ u}(t)$ is strictly increasing if $\beta^{u}(0) \in [0, \beta^{o, \infty, u})$ and strictly decreasing if $\beta^{u}(0) \in (\beta^{o, \infty, u}, 1]$. Thus, $\cA_\beta^{o,u} = \{\beta^{o, \infty, u}\}$ with  the domain of attraction as $[0,1]$. Lastly, observe that $g_\beta^{o, u}(\beta) \leq \overline{g}_\beta^u(\beta)$ for each $\beta \in [0,1]$, therefore, $\beta^{o, \infty, u} \leq \overline{\beta}^u$, as these two zeroes are unique zeroes of their respective functions (see \eqref{eqn_bound_g}).  \eop


\noindent \textbf{Proof of Corollary \ref{cor_limits_warning}}
Recall from Corollary \ref{corollary_ex_wm}, $g_\beta^{o, u}$ has a unique attractor, $\beta^{o, \infty, u} \in (0,1)$, for each $u\in\{R, F\}$. Observe further, $g^{o, u}_\beta(\beta^{o, \infty, u}(w); w) = 0$ and $g^u_\beta(\beta^{o, \infty, u}(b); b)  = 0$. Henceforth, the corollary will be proved using Theorem \ref{thrm_unique_att}. For any $\partial w > 0$ and $\partial b > 0$, we get:

\vspace{-6mm}
{\footnotesize
\begin{align*}
        \nabla^u(w, w+\partial w) &= g^{o, u}_\beta(\beta^{o, \infty, u}(w); w+\partial w) \\
        &= g^{o, u}_\beta(\beta^{o, \infty, u}(w); w) \\
        &\hspace{0.5cm}+ m_f\eta^u \mu_2 \bigg(\alpha_x^u \beta^{o, \infty, u}(w) + \alpha_y^u (1-\beta^{o, \infty, u}(w))\bigg) \left(\frac{\partial w \beta^{o, \infty, u}(w) }{\beta^{o, \infty, u}(w) + (1-\beta^{o, \infty, u}(w))b}\right) > 0 \mbox{ and}\\
        \nabla^u(b, b+\partial b) &= g^{o, u}_\beta(\beta^{o, \infty, u}(b); b+\partial b) \\
        &= g^{o, u}_\beta(\beta^{o, \infty, u}(b); b) \\
        &\hspace{0.5cm}- \partial b m_f\eta^u \mu_2 \frac{w \beta^{o, \infty, u}(b) \bigg(\alpha_x^u \beta^{o, \infty, u}(b) + \alpha_y^u (1-\beta^{o, \infty, u}(b)\bigg)}{\bigg(\beta^{o, \infty, u}(b) + (1-\beta^{o, \infty, u}(b))(b+\partial b)\bigg)\bigg(\beta^{o, \infty, u}(b) + (1-\beta^{o, \infty, u}(b))b\bigg)} < 0.
\end{align*}}
Thus, by Theorem \ref{thrm_unique_att}, $\beta^{o, \infty, u}(w, b)$ strictly increases with $w$ and strictly decreases with $b$ for any $u \in \{R, F\}$. \eop

\noindent \textbf{Proof of Theorem \ref{thrm_opt}} 
In this proof, we explicitly show the dependency of zeros of \eqref{eqn_beta_ODE_etac1} on design parameters $(w, b)$.

\noindent \underline{Part (i)} Consider a $\delta > 0$ such that $\beta^{o, \infty, R}(\overline{w}, 0) > \delta$. Then, $w \in [0, \overline{w}] = W_1 \cup W_2$, where $W_1 := \{w : \beta^{o, \infty, R}(w, 0) > \delta \}$ and $W_2 := \{w : \beta^{o, \infty, R}(w, 0) \leq \delta \}$. If $W_2 \neq \emptyset$, by Corollary \ref{corollary_ex_wm}, there exists a $\widetilde{w} > 0$ such that $\beta^{o, \infty, R}(\widetilde{w}, 0) = \delta$,  $W_1 = \{w : w > \widetilde{w} \}$, and $W_2 := \{w : w \leq \widetilde{w} \}$. The proof for case with $W_2 = \emptyset$ is trivially true once the other case is proved. Hence, consider $W_2 \neq \emptyset$.

Consider $w \in W_1$.  Then, by Corollary \ref{cor_limits_warning}, there exists a unique $b(w; \delta) > 0$ such that $\beta^{o, \infty, R}(w, b(w; \delta)) = \delta$ (i.e., the zero of $g_\beta^{o, F}$ equals $\delta$) and hence:
\begin{align}\label{eqn_bw}
\begin{aligned}
    b(w; \delta) &:= \left(\frac{\delta}{1-\delta}\right)\left(w p(\delta) - 1 \right), \mbox{ where}\\
    p(\delta) &:= \frac{\eta^R \mu_2(\delta \alpha_x^R + (1-\delta)\alpha_y^R)}{\delta ((\mu_1+\mu_2)\eta^R + \mu_a  \eta_a) - \eta^R (\mu_1 \rho + \mu_2 \gamma) (\delta \alpha_x^R + (1-\delta) \alpha_y^R)}.
\end{aligned}
\end{align}Thus, again by Corollary \ref{cor_limits_warning} and because $[0, \overline{w}] \cap W_1 = (\widetilde{w}, \overline{w}]$ (as said before):
    \begin{align}\label{eqn_opt_prob_reduced}
    \begin{aligned}
        \sup_{w \in [0, \overline{w}] \cap W_1; b \in [0, \infty); \beta^{o, \infty, R}(w, b) \leq \delta}\beta^{o, \infty, F}(w, b) = \sup_{w \in (\widetilde{w}, \overline{w}]}\beta^{o, \infty, F} (w, b(w; \delta)).
    \end{aligned}
    \end{align}

By Lemma \ref{cor_bw}, $\beta^{o, \infty, F}(w, b(w; \delta))$ strictly increases with $w$, for every $\delta > 0$. Then, the optimal value for the problem in \eqref{eqn_opt_prob_reduced} is given by:
\begin{align}\label{eqn_opt_prob_reducedW1}
    \begin{aligned}
 \sup_{w \in (\widetilde{w}, \overline{w}]}\beta^{o, \infty, F} (w, b(w; \delta)) = \beta^{o, \infty, F}(\overline{w}, b(\overline{w}; \delta)).
\end{aligned}
    \end{align}

Now, consider $w \in W_2$. Then, $\beta^{o, \infty, R}(w, 0) \leq \delta$. Further
by Corollary \ref{cor_limits_warning}, for any $w < \widetilde{w}$ and $b > 0$, we have:
\begin{equation*}
     \beta^{o, \infty, F}(\widetilde{w}, 0) > \beta^{o, \infty, F}({w}, 0) > \beta^{o, \infty, F}(w, b),  \mbox{ and } \beta^{o, \infty, F}(\widetilde{w}, 0) > \beta^{o, \infty, F}(\widetilde{w}, b).
\end{equation*}
Thus, we have:
\begin{align}\label{eqn_opt_prob_reduced_W2}
    \begin{aligned}
        \sup_{w \in [0, \overline{w}] \cap W_2; b \in [0, \infty); \beta^{o, \infty, R}(w, b) \leq \delta}\beta^{o, \infty, F}(w, b) = \beta^{o, \infty, F} (\widetilde{w}, 0).
    \end{aligned}
    \end{align}

In all, by \eqref{eqn_opt_prob_reducedW1}, \eqref{eqn_opt_prob_reduced_W2}, we have:
\begin{align}\label{eqn_opt_W1_W2}
    \sup_{w \in [0, \overline{w}]; b \in [0, \infty); \beta^{o, \infty, R}(w, b) \leq \delta}\beta^{o, \infty, F}(w, b)  =  \max\bigg\{\beta^{o, \infty, F}(\overline{w}, b(\overline{w}; \delta)), \beta^{o, \infty, F}(\widetilde{w}, 0)\bigg\}.
 \end{align}
 
Let us now consider a sequence of $w \down \widetilde{w}$ and observe $\frac{\partial b(w; \delta)}{\partial w} = \left(b(w; \delta) + \frac{\delta}{1-\delta}\right)\frac{1}{w} > 0$. Thus, $b(w; \delta)$ decreases as $w$ decreases. We claim that $\lim_{w \down \widetilde{w}} b(w; \delta) = 0$. Let us suppose on the contrary that the limit is positive; note that the limit can not be negative as $b(w; \delta) > 0$. By continuity of $b(w; \delta)$ with respect to $w$ (see \eqref{eqn_bw}), there exists a $w' < \widetilde{w}$ such that $b(w'; \delta) > 0$, and further $\beta^{o, \infty, R}(w', b(w'; \delta)) = \delta$, by definition of $b(w'; \delta)$. However, since $w' \in W_2$, we also have $\beta^{o, \infty, R}(w', 0) \leq \delta$, leading to a contradiction. Thus, the limit is $0$. 

Consider function $L(\beta; w) := \bigg(g_\beta^{o, F} (\beta(w, b(w))\bigg)^2$. Clearly this function is jointly  continuous and has a unique minimum at $\beta^{o, \infty, F}(w, b(w))$ for each $w$  (as it is the unique zero of $g_\beta^o (\cdot)$). Hence  
by Maximum Theorem: 
$$
\beta^{o, \infty, F}(w, b(w)) \to \beta^{o, \infty, F}(\widetilde{w}, 0), \mbox{ as } w \down \widetilde{w}$$
and further by Lemma \ref{cor_bw}:
$$
\beta^{o, \infty, F}(w, b(w)) \down \beta^{o, \infty, F}(\widetilde{w}, 0).
$$
Thus, $\beta^{o, \infty, F}(\widetilde{w}, 0) \leq \beta^{o, \infty, F}(w, b(w; \delta)) < \beta^{o, \infty, F}(\overline{w}, b(\overline{w}; \delta))$, where the last inequality is again due to Lemma \ref{cor_bw}. Conclusively, by \eqref{eqn_opt_W1_W2}, we get that $$
\sup_{w \in [0, \overline{w}]; b \in [0, \infty); \beta^{o, \infty, R}(w, b) \leq \delta}\beta^{o, \infty, F}(w, b) = \beta^{o, \infty, F}(\overline{w}, b(\overline{w}; \delta)).
$$

\noindent \underline{Part (ii)} Consider $\delta > 0$ such that $\beta^{o, \infty, R}(\overline{w}, 0) \leq \delta$.
Again, by Corollary \ref{cor_limits_warning}, for all $w \in [0, \overline{w}]$ and $b > 0$:
\begin{equation*}
     \beta^{o, \infty, F}(\overline{w}, 0) > \beta^{o, \infty, F}({w}, 0) > \beta^{o, \infty, F}(w, b),  \mbox{ and } \beta^{o, \infty, F}(\overline{w}, 0) > \beta^{o, \infty, F}(\overline{w}, b).
\end{equation*}Thus, the optimal value is achieved at $b = 0$ and $w = \overline{w}$, with $\beta^{o, \infty, R}(\overline{w}, 0) \leq \delta$. \eop

\begin{lemma}\label{cor_bw}
    The function $\beta^{o, \infty, F}(w, b(w; \delta))$ strictly increases with $w$, when $w < \overline{w}$, for every $\delta > 0$.
\end{lemma}
\begin{proof}
     Fix $w$ and $\partial w > 0$, we have (for simplicity, denote $\beta^{o, \infty, F}(w, b(w; \delta))$ by $\beta_\delta(w)$):

    \vspace{-4mm}
    {\small 
    \begin{align}\label{eqn_nabla_betaF}
    \begin{aligned}
        \nabla^F(w, w+\partial w; b) &= g^{o, F}_\beta(\beta_\delta(w); w+\partial w) - g^{o, F}_\beta(\beta_\delta(w); w)  \\
        &= m_f\eta^F \mu_2 \bigg(\alpha_x^F \beta_\delta(w) + \alpha_y^F (1-\beta_\delta(w))\bigg) \\
        &\hspace{1cm}\times\left( \left(\frac{(w + \partial w) \beta_\delta(w) }{\beta_\delta(w) + (1-\beta_\delta(w))b(w + \partial w; \delta)}\right) - \left(\frac{w \beta_\delta(w) }{\beta_\delta(w) + (1-\beta_\delta(w))b(w; \delta)}\right) \right)\\
        &= \frac{m_f\eta^F \mu_2 \beta_\delta(w)  \bigg(\alpha_x^F \beta_\delta(w) + \alpha_y^F (1-\beta_\delta(w))\bigg) }{\bigg(\beta_\delta(w) + (1-\beta_\delta(w))b(w + \partial w; \delta)\bigg) \bigg( \beta_\delta(w) + (1-\beta_\delta(w))b(w; \delta) \bigg)}\\
        &\hspace{1cm}\times\bigg( \partial w\bigg( (1-\beta_\delta(w)) b(\partial w; \delta) + \beta_\delta(w) \bigg)\bigg), \mbox{ where}
        \end{aligned}
        \end{align}}the last equality follows by simple algebra after substituting for $b(\cdot; \delta)$ from \eqref{eqn_bw}. Since $b(\cdot, \delta) > 0$ and  by Theorem \ref{thrm_BP_to_fake}, $\beta_\delta(w) \in (\underline{\beta}^F, \overline{\beta}^F] \subset [0,1]$, therefore,  $\nabla^F(w, w+\partial w; \delta) > 0$ for any $\delta > 0$. Thus, the proof follows by Theorem \ref{thrm_unique_att}. 
\end{proof}

\noindent \textbf{Proof of Corollary \ref{cor_beta_o_na}}
Consider any $\delta > 0$, $\mu_a \in (0, 1-\mu_1-\mu_2]$ and let $b^*, w^*$ be as in Theorem \ref{thrm_opt}.

\noindent \underline{Case 1: when $b^* > 0$:} Let $b^*(\mu_a 
= 0) =: b^*_0$. From \eqref{eqn_optimal_parameter_old}, observe that $b^*$ is a strictly decreasing function of $\mu_a$, therefore, $b^*_0 > b^* > 0$. Further, from \eqref{eqn_beta_ODE_etac1}, we have:
\begin{align}\label{eqn_g_beta_na}
\begin{aligned}
    g_\beta^{o, F}(\betana) &= g_\beta^{o, F}(\betana; \mu_a = 0) - \betana \mu_a m_f \eta_a \\
    &\hspace{2mm}+ \mu_2 m_f \eta^F \bigg(\betana \alpha_x^F + (1-\betana) \alpha_y^F \bigg) \left( \frac{w^*\betana}{\betana + (1-\betana) b^*} - \frac{w^*\betana}{\betana + (1-\betana) b^*_0} \right) \\
    &= 0 + \betana \mu_a m_f \eta_a \left( \frac{\mu_2 \eta^F w^*(1-\betana) (\betana \alpha_x^F + (1-\betana) \alpha_y^F)}{\bigg(\betana + (1-\betana) b^*_0 \bigg) \bigg( \betana + (1-\betana) b^* \bigg) }  \left( \frac{b^*_0 - b^*}{\eta_a \mu_a}\right) - 1 \right).
\end{aligned}
\end{align}
Define $p(\mu_a) := \delta ((\mu_1+\mu_2)\eta^R + \mu_a  \eta_a) - \eta^R (\mu_1 \rho + \mu_2 \gamma) (\delta \alpha_x^R + (1-\delta) \alpha_y^R)$. Then, by \eqref{eqn_optimal_parameter_old}, we have: 
\begin{align*}
\frac{b^*_0 - b^*}{\eta_a \mu_a} = \left(\frac{\delta^2}{1-\delta}\right) \left( \frac{w^* \eta^R \mu_2 (\delta \alpha_x^R + (1-\delta)\alpha_y^R)}{p(\mu_a) p(0)} \right).
\end{align*}
Substitute the above term in  \eqref{eqn_g_beta_na} and consider the following limit to analyse \eqref{eqn_g_beta_na}:

\vspace{-4mm}
{\footnotesize
\begin{align*}
\begin{aligned}
    \lim_{\delta \to 0} \left( \frac{\mu_2 \eta^F w^*(1-\betana) (\betana \alpha_x^F + (1-\betana) \alpha_y^F)}{\bigg(\betana + (1-\betana) b^*_0 \bigg) \bigg( \betana + (1-\betana) b^* \bigg) } \right) \lim_{\delta \to 0} \left(\frac{\delta^2}{1-\delta}\right) \lim_{\delta \to 0} \left( \frac{w^* \eta^R \mu_2 (\delta \alpha_x^R + (1-\delta)\alpha_y^R)}{p(\mu_a) p(0)} \right) - 1.
\end{aligned}
\end{align*}}In the above, the second limit is clearly $0$ and the rate of convergence is independent of other factors. The first and third limits are finite, and the respective terms can be upper-bounded independent of $\mu_a$ and other factors. Thus, the product of three limits is $0$, and the rate of convergence is uniform in $\mu_a$ and $b^*$, i.e., there exists a $\overline{\delta} > 0$ such that, for example for all $\delta \leq \overline{\delta} $ and $\mu_a > 0$:

\vspace{-4mm}
{\small
$$
 \left( \frac{\mu_2 \eta^F w^*(1-\betana) (\betana \alpha_x^F + (1-\betana) \alpha_y^F)}{\bigg(\betana + (1-\betana) b^*_0 \bigg) \bigg( \betana + (1-\betana) b^* \bigg) } \right)  \left(\frac{\delta^2}{1-\delta}\right)  \left( \frac{w^* \eta^R \mu_2 (\delta \alpha_x^R + (1-\delta)\alpha_y^R)}{p(\mu_a) p(0)} \right) - 1 < -\frac{1}{2}.
$$}
Thus, from \eqref{eqn_g_beta_na}, $g_\beta^{o, F}(\betana) < - \betana \mu_a m_f \eta_a/2 < 0$ for any $\mu_a > 0$ and all $\delta \leq \overline{\delta}$.

Recall from the proof of Corollary \ref{cor_limits_warning} that $g_\beta^{o, F}(\cdot)$ is either convex/concave/linear with a unique zero in $(0,1)$. Therefore, the unique zero of $g_\beta^{o, F}(\cdot; \mu_a)$, namely $\beta^o(\mu_a) < \betana$ for all $\delta \leq \overline{\delta}$ and  for all $\mu_a \in (0, 1-\mu_1-\mu_2]$.

\noindent \underline{Case 2: when $b^* = 0$:} Here, again $b^*(\mu_a 
= 0) =: b^*_0 > b^* = 0$.  Similar to \eqref{eqn_g_beta_na}, using \eqref{eqn_optimal_parameter_old}:

\vspace{-4mm}
{\tiny
\begin{align*}
    g_\beta^{o, F}(\betana) &= \betana \mu_a m_f \eta_a \left( \frac{\mu_2 \eta^F w^* (1-\betana) (\betana \alpha_x^F + (1-\betana) \alpha_y^F) b^*_0 }{\eta_a \mu_a  \betana \bigg(\betana + (1-\betana) b^*_0 \bigg)   } \left(\left(\frac{\delta}{1-\delta}\right) \left( \frac{w^* \eta^R \mu_2 (\delta \alpha_x^R + (1-\delta)\alpha_y^R)}{p(0)}  - 1 \right) \right) - 1 \right).
\end{align*}}
Hereafter, the proof follows as in Case 1. \eop

\noindent \textbf{Proof of Theorem \ref{corollary_ea_wm}}
We begin the proof for the fake post.

\noindent \underline{Part (i)} Consider
$0 < \mu_a \leq \min\{1-\mu_1-\mu_2, \Delta_a\}$. Then, by the definition of upper-bound $\Delta_a$ and \eqref{eqn_warning_ea}, $\alpha_x^F \omega^a(\betana) \leq 1$. Note from \eqref{eqn_warning_ea} that $\omega^a(\beta)$ is a strictly increasing function of $\beta$. Therefore, $\alpha_x^F \omega^a(\beta) \leq 1$ for all $\beta \leq \betana$, for given $\mu_a$.

This implies that for $\beta \le \betana$, we have $g_\beta^{a, F}(\beta) = g_\beta^{o, F}(\beta; \mu_a = 0)$ (see \eqref{eqn_g_F_eaWM}). Further, $\betana$ is a zero of $g_\beta^{a, F}$, as $g_\beta^{a, F}(\betana) = g_\beta^{o, F}(\betana; \mu_a = 0) = 0$. Furthermore, by uniqueness given in Corollary \ref{corollary_ex_wm}, $\betana$ is the unique zero of $g_\beta^{a, F}$ in $[0, \betana]$. Therefore, any $\beta^{a} \in {\cA}^{a, F}_\beta \cup \cR^{a, F}_\beta$ is  in $[\betana, 1]$.

\noindent \underline{Part (ii)} Consider $\mu_a > \Delta_a$. Then, the corresponding $\alpha_x^F \omega^a(\betana) > 1$.
Define the function $h(\beta) := \alpha_x^F \omega^a({\beta}) - 1$. It is easy to see that $h(0) < 0$, $h(1) > 0$ and $h(\cdot)$ is a strictly increasing function. Thus, there exists a unique zero of $h$, denoted by  $\widetilde{\beta} \in (0,1)$, i.e., $\alpha_x^F \omega^a(\widetilde{\beta}) = 1$. As $\beta \mapsto \omega^a (\beta)$ is strictly increasing, we further have $\alpha_x^F \omega^a(\beta) < 1$ for all $\beta < \widetilde{\beta} $; furthermore $\widetilde{\beta} < \betana$ as $\alpha_x^F \omega^a(\betana) > 1$. 

From \eqref{eqn_g_F_eaWM}, we have:

\vspace{-2mm}
{\small 
\begin{align}\label{eqn_cor3}
    g^{a, F}_\beta(\beta) &= g^{o, F}_\beta(\beta; \mu_a = 0)  \\
    &\hspace{-2mm}+ \mu_2 m_f \eta^F \left\{ \beta \bigg(\min\{1, \alpha_x^F \omega^a(\beta)\} - \alpha_x^F \omega^a(\beta) \bigg) + (1-\beta) \bigg( \min\{1, \alpha_y^F \omega^a(\beta)\} - \alpha_y^F\omega^a(\beta) \bigg) \right\}. \nonumber 
\end{align}}Thus, $g^{a, F}_\beta(\beta) < g^{o, F}_\beta(\beta; \mu_a = 0)$ if $1 < \alpha_{j}^F \omega^a(\beta)$ for some $j \in \{x, y\}$, and $g^{a, F}_\beta(\beta) = g^{o, F}_\beta(\beta; $ \\$\mu_a = 0)$ if $\alpha_{j}^F \omega^a(\beta) \leq 1$ for each $j \in \{x, y\}$. As a result, we have:

(a) for $\beta \in [0, \widetilde{\beta}]$, $g_\beta^{a, F}(\beta) = g_\beta^{o, F}(\beta; \mu_a = 0) > 0$, and 

(b) for $\beta \in [\betana, 1]$, $g_\beta^{a, F}(\beta) < g_\beta^{o, F}(\beta; \mu_a = 0) \leq 0$.

By Theorem \ref{thrm_BP_to_fake},  there exists at least one zero of $g_\beta^{a, F}$, say $\beta^a $ and by above arguments, $\beta^a \in (\widetilde{\beta}, \betana)$. We will now claim and show that $\beta^a > \beta^o$, but first observe that $\beta^o < \betana$ by Corollary \ref{cor_beta_o_na}. Towards this, note that for $\beta \in (\widetilde{\beta}, \betana)$, we have:

\vspace{-2mm}
{\footnotesize
\begin{align}
\begin{aligned}
    g^{a, F}_\beta(\beta) &= g^{o, F}_\beta(\beta) \\
    &\hspace{1cm}+ \mu_2 m_f \eta^F \left\{ \beta \bigg(\min\{1, \alpha_x^F \omega^a(\beta)\} - \alpha_x^F \omega(\beta) \bigg) + (1-\beta) \bigg( \min\{1, \alpha_y^F \omega^a(\beta)\} - \alpha_y^F\omega(\beta) \bigg) \right\}\\
    &= g^{o, F}_\beta(\beta) + \mu_2 m_f \eta^F \left\{ \beta \bigg(1 - \alpha_x^F \omega(\beta) \bigg) + (1-\beta) \bigg( \min\{1, \alpha_y^F \omega^a(\beta)\} - \alpha_y^F\omega(\beta) \bigg) \right\}.
\end{aligned}
\end{align}}
In the above, if $1 > \alpha_y^F \omega^a(\beta)$, then:

\vspace{-2mm}
{\small
\begin{align*}
    g^{a, F}_\beta(\beta) &= g^{o, F}_\beta(\beta) + \mu_2 m_f \eta^F \left\{ \beta \bigg(1 - \alpha_x^F \omega(\beta) \bigg) + (1-\beta)\alpha_y^F  \bigg( \omega^a(\beta) - \omega(\beta) \bigg) \right\}\\
    &= g^{o, F}_\beta(\beta) + \mu_2 m_f \eta^F \left\{ \beta \bigg(1 - \alpha_x^F \omega(\beta) \bigg) + (1-\beta)\alpha_y^F  \left( \frac{\beta \mu_a m_f \eta_a}{ \mu_2 m_f\eta^F \left(\beta \alpha_x^F + (1-\beta)\alpha_y^F\right) } \right) \right\}  \\
    &> g^{o, F}_\beta(\beta), 
\end{align*}}
as $ \omega(\beta)\alpha_x^F < 1$ for all $\beta \in [0,1)$. Further, $ \omega(\beta)\alpha_y^F < 1$ for all $\beta \in [0,1)$, hence even with $1 \leq  \alpha_y^F \omega^a(\beta)$, we have:
\begin{align*}
    g^{a, F}_\beta(\beta) &= g^{o, F}_\beta(\beta) + \mu_2 m_f \eta^F \left\{ \beta \bigg(1 - \alpha_x^F \omega(\beta) \bigg) + (1-\beta)  \bigg( 1- \alpha_y^F\omega(\beta) \bigg) \right\} > g^{o, F}_\beta(\beta).
\end{align*}
Now, for $\beta \in (\widetilde{\beta}, \beta^o]$, $g_\beta^{o, F}(\beta) \geq 0$, and thus, $g^{a, F}_\beta(\beta) > 0$.  
This completes the proof of the claim.

Now, consider the real post. By Theorem \ref{thrm_BP_to_fake}, $\cA^{a, R}_\beta \neq \emptyset$, therefore, there exists at least one zero of $g_\beta^{a, R}$, say $\beta^{a, R} \in (0,1)$. Now, using arguments as above:

\vspace{-2mm}
{\small
\begin{align*}
    g^{a, R}_\beta(\beta) &= g^{o, R}_\beta(\beta; \mu_a = 0) \\
    &\hspace{-1cm}+ \mu_2 m_f \eta^R \left\{ \beta \bigg(\min\{1, \alpha_x^R \omega^a(\beta)\} - \alpha_x^R \omega^a(\beta) \bigg) + (1-\beta) \bigg( \min\{1, \alpha_y^R \omega^a(\beta)\} - \alpha_y^R \omega^a(\beta) \bigg) \right\} \\
    &\hspace{2cm} + \beta \mu_a m_f \eta_a  \left( \frac{\eta^R}{\eta^F}\left( \frac{\beta \alpha_x^R + (1-\beta) \alpha_y^R}{\beta \alpha_x^F + (1-\beta) \alpha_y^F}\right) - 1  \right) < g^{o, R}_\beta(\beta; \mu_a = 0).
\end{align*}}
Thus, any zero of $g^{a, R}_\beta(\beta)$ is strictly less than the unique zero of $g^{o, R}_\beta(\beta; \mu_a = 0)$, i.e., $\beta^{a, R} < \beta^{o, R}(\mu_a = 0) \leq \delta$, for any $\beta^{a, R} \in \cA^{a, R}_\beta \cup \cR^{a, R}_\beta$ (see Theorem \ref{thrm_opt}). \eop

\noindent \textbf{Proof of Theorem \ref{corollary_eh_WM}}
We divide the proof in two cases.

\textbf{Case 1:} If $\overline{\zeta} < \frac{1}{\alpha_y^R \omega^a(\delta)}$. Then $\overline{\zeta}$ is the unique zero of $g_{\beta, {\zeta}}^{h, R}(\delta) = 0$.  Further, for any $\zeta' \in \left( \overline{\zeta}, \frac{1}{\alpha_y^R \omega^a(\delta)}\right)$, $g_{\beta, \zeta'}^{h, R}(\delta) > 0$. By \eqref{eqn_bound_g}, $g_{\beta, \zeta'}^{h, R}(1) < 0$.  Thus, there exists at least one zero of $g_{\beta, \zeta'}^{h, R}$ greater than $\delta$. Now, consider any $\zeta' \geq \frac{1}{\alpha_y^R \omega^a(\delta)}$. Since the function $\zeta \mapsto g_{\beta, {\zeta}}^{h, R}(\delta)$ is continuous, therefore, $g_{\beta, {\zeta'}}^{h, R}(\delta) > 0$ for such $\zeta$. Thus, again as before, there exists at least one zero of $g_{\beta, \zeta'}^{h, R}$ greater than $\delta$. Hence, any $\zeta$ satisfying the constraint in \eqref{eqn_opt_phi} is  less than or equals to $\overline{\zeta}$. Thus, the optimizer of \eqref{eqn_opt_phi} is $\zeta^* = \overline{\zeta}$.

\textbf{Case 2:} If $\overline{\zeta} \geq \frac{1}{\alpha_y^R \omega^a(\delta)}$. Then by monotonicity, for any $\zeta \ge \overline{\zeta} $:
\begin{align*}
g_{\beta, \zeta}^{h, R} (\beta) &\leq q_\zeta(\beta) := \bigg(-\beta \mu_2 - \beta \mu_1 (1-\alpha_x^R \rho) + (1-\beta) \mu_1 \rho  \alpha_y^R \\
&\hspace{4cm}+ \mu_2 \zeta \omega^a(\beta) \bigg(\beta  \alpha_x^R + (1-\beta)\alpha_y^R\bigg) \bigg) m_f\eta^R - \beta \mu_a m_f \eta_a.
\end{align*}
Thus for all such $\zeta$, $q_\zeta(\delta)$ is a strictly increasing function  of $\zeta$  with  $q_1(\delta) < 0$ (by Theorem \ref{corollary_ea_wm}) and $q_{\overline{\zeta}}(\delta) = 0$. Thus, $g_{\beta, \zeta}^{h, R} (\delta) \le q_\zeta(\delta) \leq 0$. 

Further, by strict monotonicity  of $\omega^a(\cdot)$ in $\beta$, 
we have $\zeta \alpha_y^R \omega^a(\beta) \geq 1$ for all $\beta > \delta$ whenever   $\zeta \geq \overline{\zeta}$. Thus, $g_{\beta, \zeta}^{h, R} (\beta) $ is linearly (strictly) decreasing in $\beta$, when $\beta > \delta$. As already proved  $g_{\beta, \zeta}^{h, R} (\delta) \leq 0$, and hence  $g_{\beta, \zeta}^{h, R} (\beta) < 0$ for all $\beta > \delta$. Hence, the feasibility condition of \eqref{eqn_opt_phi} is satisfied for any $\zeta \ge \overline{\zeta}$.

Let $\underline{\beta}^F > 0$ and/or $b > 0$. By definition of $\zeta^*$ in this case (the second row), we have:
$$
    \zeta^* \omega^a(\beta) \alpha_y^F = 1 \mbox{ for all } \beta \geq \underline{\beta}^F.
$$
Further,  $\min\{1,  \zeta^* \omega^a(\beta) \alpha_y^F\} = 1 $  for all $\beta > \underline{\beta}^F$,  when  $\zeta \geq \zeta^*$. Thus, the functions $g^{h, F}_{\beta, \zeta}(\beta) = g^{h, F}_{\beta, \zeta^*}(\beta)$ for all $\beta \geq \underline{\beta}^F$. Also, by Theorem \ref{thrm_BP_to_fake}, any  zero of $g^{h, F}_{\beta, \zeta}$ is larger than $\underline{\beta}^F$. Thus, $\left\{\beta: \beta \in \cA^{h, \zeta}_\beta \cup \cR^{h, \zeta}_\beta\right\} = \left\{\beta: \beta \in \cA^{h, \zeta^*}_\beta \cup \cR^{h, \zeta^*}_\beta\right\}$. Now, given any $\beta$, observe that $\zeta \mapsto g_{\beta, \zeta}^{h, F}(\beta)$ is an increasing  (actually non-decreasing) function. Thus, $\inf\left\{\beta: \beta \in \cA^{h, \zeta}_\beta \cup \cR^{h, \zeta}_\beta\right\}$ increases with $\zeta$. Conclusively,  we get that $\zeta^*$ is an  optimizer of \eqref{eqn_opt_phi}. 

Now, let $\underline{\beta}^F = 0$ and $b = 0$. Then, for all $\zeta \geq \overline{\zeta}$ and for all $\beta \in [0,1]$, we have (by \eqref{eqn_relation_parameters}):
\begin{align*}
    \zeta \alpha_j^F \omega^a(\beta) &= \zeta \alpha_j^F \left( \frac{1}{\alpha_x^F} + \frac{\beta \mu_a m_f \eta_a}{ \mu_2 m_f\eta^F \bigg(\beta \alpha_x^F + (1-\beta)\alpha_y^F\bigg) } \right) > \zeta \frac{\alpha_j^F}{\alpha_x^F} \geq \overline{\zeta} \frac{\alpha_j^F}{\alpha_x^F} \\
    &\geq \frac{\alpha_j^F}{\alpha_x^F} \frac{1}{\alpha_y^R \omega^a(\delta)} = \frac{\alpha_j^F}{\alpha_x^F} \frac{\alpha_x^F}{\alpha_y^R} > 1 \mbox{ for each } j \in \{x, y\}.
\end{align*}
Thus, $g_{\beta, \zeta}^{h, F}$ is linear in $\beta$ and independent of $\zeta$, for all $\zeta \geq \overline{\zeta}$. This implies that $\left\{\beta: \beta \in \cA^{h, \zeta}_\beta \cup \cR^{h, \zeta}_\beta\right\}$ is also independent of $\zeta$, for all $\zeta \geq \overline{\zeta}$. As before, \\ $\inf\left\{\beta: \beta \in \cA^{h, \zeta}_\beta \cup \cR^{h, \zeta}_\beta\right\}$ increases with $\zeta$. Thus, here, an optimizer of \eqref{eqn_opt_phi} is $\overline{\zeta}$. \eop
    \chapter{For Chapter \ref{ch:STPBP}} \label{Appendix_STPBP}

\begin{lemma}\label{lemma_existence}
There exists a unique, extended, continuous solution of the ODE \eqref{eqn_ODE_stpbp} over any finite time interval. 
\end{lemma}
\begin{proof}
At first, we show that if at all there exists a solution for the non-smooth, non-autonomous ODE \eqref{eqn_ODE_stpbp}, then, the ODE solution is bounded. Towards this, say ${\psi}^c(0) = {\psi}^a(0) = a_0$. If possible, say there exists a solution $\hat{\ups}(\cdot)$ for the ODE \eqref{eqn_ODE_stpbp} for all $t < T$ (for some fixed, finite $T$). Consider the ODE for $z(\cdot) := (z^c(\cdot), z^a(\cdot))$ (see assumption \ref{d1}):
\begin{align}\label{eqn_z_ode}
    \dot{z}^c = b_1 - 1,\mbox{ and } \dot{z}^a = b_1, \mbox{ with } E[\hat{\Gamma}_1] =: b_1.
\end{align}The RHS of the above ODE is smooth and autonomous.
Further, consider $\bar{t}$ such that $\eta(t) = n$ for all $t \leq \bar{t}$. Then, for $t \leq \bar{t}$, we have:
\begin{align}
    \dot{\pc} &= m(\pa n) - 1 - \pc, \mbox{ and } \dot{\pa} = m(\pa n) - \pa.
\end{align}It is clear that the RHS of the above ODE is Lipschitz continuous and autonomous, and further that $\dot{\ups } < \dot{z}$,  for all $t \leq \bar{t}$. Hence, if $z(0) = \ups(0)$, then, $\hat{\ups}(t) < z(t)$ for all $t \leq \bar{t}$ (e.g., see \cite[pp. 168]{piccinini2012ordinary}).

Thus,  \cite[Lemma 2.1(ii)]{singer2006bounding} is not true (while remaining hypotheses are true), and hence by \cite[Lemma 2.1(i)]{singer2006bounding}, we have the following:
\begin{align}\label{eqn_bound_ode}
    \hat{\psi}^a(t) &< z^a(t) 
    < a_0 + b_1 T =: \beta, \mbox{ and }
    \hat{\psi}^c(t) < z^c(t) 
    < \beta, \mbox{ for all } t \in [0, T]. 
\end{align}
Next, consider the set $\cd := [0, T] \times [-2\beta, 2\beta]^2$. Then, under assumption \ref{d1} and using \cite[Theorem 1.3, pp.  47]{coddington1955theory}, there exists 
a solution $\hat{\ups}(\cdot)$ for all $
t < \tau$, where 
$$
\tau:= \min\bigg\{\inf_t\{\hat{\psi}^c(t) \in \{-2\beta, 2\beta\}, \hat{\psi}^a(t) \in \{-2\beta, 2\beta\}\}, T\bigg\}.
$$
Lastly, note that $\tau = T$; this is so because $\tau < T$ contradicts \eqref{eqn_bound_ode} as $\ups \geq 0$.\footnote{One can prove that $\ups \geq 0$ using \eqref{eqn_ODE_stpbp} and simple lower-bounding arguments like above.} This proves the existence of the solution of the ODE for all $t \in [0, T]$. 

Next, the uniqueness of the solution holds by Lemma \ref{lemma_unique_ode_sol}. Lastly, the continuity for the solution follows by the integral representation of the solution, and because the RHS of the ODE \eqref{eqn_ODE_stpbp} can be bounded by $b_1 + \beta$. 
\end{proof}

\noindent \textbf{Proof of Theorem \ref{thrm1_stpbp}:} \textbf{Part (i)} The proof of this part follows closely as in \cite[Theorem 2.1, pp. 127]{kushner2003stochastic}, but the RHS of the ODE in our case  is only measurable. 
Let $n \geq 0$. Using \eqref{eqn_SA}, one can re-write $\Ups_n$ as:
\begin{align}\label{eqn_scheme2}
\begin{aligned}
 \Ups_{n+1} &= \Ups_n + \epsilon_n L_n, \mbox{ where } L_n := ( L_n^{ c}, L_n^{ a}), \mbox{ for}\\
 L_n^{c} &:=  \left[\offs_{n}(A_{n-1}) - 1 - \Pc_{n-1} \right ]1_{\Pc_{n-1} > 0}, \mbox{ and }
L_n^{ a} :=  \left [\offs_{ n}(A_{n-1})  - \Pa_{n-1} \right ]1_{\Pc_{n-1} > 0}.
\end{aligned}
 \end{align}
 
\textbf{Interpolated trajectory.} Let $\Ups^n(\cdot) = (\Psi^{n, c}(\cdot), \Psi^{n, a}(\cdot))$ be the piece-wise interpolated trajectory defined as (see \eqref{eqn_scheme2}):
\begin{align}\label{eqn_piecewise}
    \Ups^{n}(t) &= \Ups_n + \sum_{i=n+1}^{\eta(t_n +t)}(\Ups_i - \Ups_{i-1}) = \Ups_n + \sum_{i=n}^{\eta(t_n +t)-1}\epsilon_{i} L_{i}, \mbox{ for any }t \geq 0.
\end{align}Let  $\bar{g}(\Ups_n, n) := E[ L_n | \mathcal{F}_n]$, i.e., the conditional expectation of $L_n$ with respect to ${\cal F}_n  := \sigma\{\Ups_k : 1 \leq k < n \}$ and $\delta M_n :=  L_n - \bar{g}(\Ups_n, n)$. Then, \eqref{eqn_piecewise} can be re-written component-wise as (for each $k \in \{a, c\}$):
\begin{align}\label{eqn_linear2}
    \Psi^{n, k}(t) &= \Psi^k_n + \sum_{i=n}^{\eta(t_n +t)-1}\epsilon_{i} \left(\delta M_{i}^k + \bar{g}^k(\Ups_i, i)\right) \nonumber \\
    &= \Psi^k_n + \int_0^t \bar{g}^k(\Ups^n(s), n) ds + \varepsilon^{n, k}(t), \mbox{ where}\\
    \varepsilon^{n, k}(t) &:= M^{n, k}(t) + \rho^{n, k}(t) \mbox{ with }
    M^{n, k}(t) := \sum_{i=n}^{\eta(t_n +t)-1}\epsilon_i \delta M_i^k, \nonumber  \\ 
    \rho^{n, k}(t) &:=  \sum_{i=n}^{\eta(t_n +t)-1} \epsilon_{i} \bar{g}^k(\Ups_i, i) - \int_0^t \bar{g}^k(\Ups^n(s), n) ds. \nonumber
\end{align}It is important to note  that $\bar{g}(\Ups_n, n)$ is the RHS of the ODE \eqref{eqn_ODE_stpbp} (as $\eta(t_n) = n$). 

Next, we begin by proving that the BP trajectory (see \eqref{eqn_SA}) can be bounded (under assumption \ref{d1}) as follows:
\begin{align*}
0 \leq \Psi^{n, c}(0) = \Psi_n^c &\leq \frac{1}{n} \left(\sum_{k = 1}^n \Gamma_k(A_{k-1})1_{\Pc_{k-1} > 0} + \ax_0 \right) \leq \frac{1}{n} \left(\sum_{k = 1}^n \hat{\Gamma}_k + \ax_0 \right) := \hat{\Pi}_n.
\end{align*}By strong law of large numbers,  $\hat{\Pi}_n \to E[\hat{\Gamma}_1]$ a.s.
Consider any  such sample path ($\omega$). Then, for any $\epsilon > 0$, there exists $N_\epsilon(\omega)$ such that:
\begin{align}\label{eqn_bound}
\Psi^{n, c}(0) &\le \hat{\Pi}_n \leq M(\omega) \mbox{ for all } n, \mbox{ where } b_1 := E[\hat{\Gamma}_1], \mbox{ and} \nonumber \\
M(\omega) &:= \max\{\max\{\hat{\Pi}_i : 0 \leq i < N_\epsilon(\omega)\}, b_1 + \epsilon\}.
\end{align}

Now, we will prove that $M^n(\cdot) = (M^{n, c}(\cdot), M^{n, a}(\cdot))$ and $\rho^n(\cdot)$ individually converge to $0$ (as $n \to \infty$) uniformly on any bounded interval. It suffices to prove uniform convergence for sample paths $\omega \in \{\hat{\Pi}_n \to b_1\}$. 
We prove the claim for $\pc$-component, and it can proved analogously for the $\pa$-component as well. Henceforth, the convergence will be proved w.r.t. $n$, where ever not mentioned explicitly. 
 
Now, define $M_n^{c} := \sum_{i=0}^{n-1}\epsilon_i \delta M_i^{c}$. Then, it is easy to prove that $(M_n^{c})$ is a Martingale   with respect to $(\mathcal{F}_n)$. Thus, using Martingale inequality, 
for each $\mu > 0$ (as in \cite[Theorem 2.1, pp. 127]{kushner2003stochastic}), with $E_n(\cdot)$ denoting the expectation conditioned on $(\mathcal{F}_n)$: 
$$
P\left\{\sup_{m\leq j \leq n} |M_j^{ c} - M_m^{ c}| \geq \mu \right\} \leq \frac{E_n\left[\left(\sum_{i=m}^{n-1} \epsilon_i \delta M_i^{ c} \right)^2 \right]}{\mu^2}.
$$
Observe, $E\left[\delta M_i^{ c} \delta M_j^{ c}\right]  = 0$ for $i < j$. Let $O(\omega)$ be the upper-bound on the ODE solution for $t \in [0, T]$, see Lemma \ref{lemma_existence}. Then, from \eqref{eqn_ODE_stpbp} and \eqref{eqn_bound}, 
\begin{align}\label{eqn_g_bound}
    |\bar{g}^i(\ups(\cdot), \cdot)| < b_1 + 1 + O(\omega) \mbox{ for each } i \in \{a, c\}.
\end{align}Thus,  under \ref{d1}, $\sup_n E_n|L_n^{ c}- \bar{g}^c(\Ups_i, t_i)|^2 < K$ for some finite $K$.
Using this,  we have: 
\begin{align*}
    P\left\{\sup_{m\leq j \leq n} |M_j^{c} - M_m^{ c}| \geq \mu \right\} &\leq \frac{\sum_{i=m}^{n-1} \epsilon_i^2 E_n\left| \delta M_i^{ c} \right|^2}{\mu^2}  = \frac{\sum_{i=m}^{n-1} \epsilon_i^2 E_n\left| L_i^{ c} - \bar{g}^c(\Ups_i, i) \right|^2}{\mu^2} \leq \frac{K}{\mu^2} \sum_{i=m}^{\infty} \epsilon_i^2.
\end{align*}
By first letting $n \to \infty$ (and using continuity of probability), then, letting $m \to \infty$, for each $\mu > 0$, we have:
\begin{align}
    \lim_{m \to \infty} P\left\{\sup_{m\leq j } |M_j^{ c} - M_m^{ c}| \geq \mu \right\} &= 0. \label{eqn_equi_cont_M_stpbp}
\end{align}
Now, define the set $A_k := \lim_{m \to \infty} \sup_{m\leq j } |M_j^{c} - M_m^{ c}| < 1/k$. Then, by \eqref{eqn_equi_cont_M_stpbp} and continuity of probability for each $k > 0$,  $P(A_k) = 1$. We further restrict our attention to  sample paths  $\omega \in  N:= (\cap_k A_k) \cap \{\hat{\Pi}_n \to b_1 \}$. For any such $\omega$, using \eqref{eqn_linear2}:
\begin{align*}
&\sup_{t\geq 0}|M^{n, c}(t)| = \sup_{t \geq 0} \left|M^{ c}_{\eta(t_n+t)} - M^{c}_n \right|  
= \sup_{j \geq n} |M^{ c}_{j} - M^{c}_n|.
\end{align*}This implies:
\begin{align*}
\lim_{n \to \infty} \sup_{t \in [0, T]}|M^{n, c}(t)|  &\leq \lim_{n \to \infty} \sup_{t \in [0, T]}\left|\sum_{i=n}^{\eta(t_n+t)}\epsilon_i \delta M_i^k\right| \leq   \lim_{n \to \infty} \sup_{\eta(t_n+t) + 1 \geq n} |M^{ c}_{\eta(t_n+t) + 1} - M^{c}_n| < 1/k.
\end{align*}Letting $k \to \infty$, we get, $M^{n,  c}(\cdot) \to 0$ uniformly on each bounded interval. 

For $\rho^{n, c}(\cdot)$, note that for  $t = t_k - t_n$ $(k > n)$, $\rho^{n, c}(t) = 0$. Thus, under \eqref{eqn_g_bound}, for any $|t| \leq T$ (as $\epsilon_{\eta(t_n + t)} \le \epsilon_n$):
\begin{align*}
    |\rho^{n, c}(t)|
    &\leq  \int_{t_{\eta(t_n+t) } - t_n}^t \left|\bar{g}^c(\Ups^n(s), \eta_n) \right|ds 
    <    \epsilon_n (b_1+1+O).
\end{align*}Thus, $\rho^{n, c}(\cdot) \to 0$ uniformly on each bounded interval. 

\textbf{Part (ii)}
We construct this proof using Maximum Theorem, which provides parameterized continuity of the optimizers. We begin by constructing the required elements  (i.e., appropriate objective function and domains).

\textbf{Ingredients for Maximum Theorem.} Fix any $\omega \in N$.  Then, the interpolated trajectory $\Ups^n(\cdot)$ and the ODE solution $\hat{\ups}^n(\cdot)$ are bounded as (see \eqref{eqn_bound}):
$$
\sup_{t} \Ups^n(t) = \sup_{n} \Ups_n < 1.1 M(\omega), \mbox{ and } \sup_{t \in [0, T]} \hat{\ups}^n(t) < 1.1 O(\omega).
$$
With the norm \eqref{eqn_norm}, let $\cd^2$ be the Banach space of all those $\ups(\cdot)$ such that both $\pc, \pa$ are left continuous with right limits on $[0, T]$ and $||\ups||< \infty$. Further, let $\cd^2_B$ be the space of all those $\ups(\cdot) \in \cd^2$ such that $||\ups|| \leq C(\omega) := 1.1(M(\omega) + O(\omega))$. 
Define $\cd_p :=  \cd^2_B \times \mathbb{R}^2 \times \mathbb{R}$, and then, define the function $F(\ups; \varepsilon, u_0, \eta): \cd^2_B \times \cd_p \to \mathbb{R}$ as:
\begin{align*}
F(\ups; \varepsilon, u_0, \eta) &:= \sum_{i \in \{a, c\}} \int_0^T \bigg(\Psi^i(t) - h^i(\ups; \varepsilon, u_0, \eta)(t) \bigg)^2 dt,
\end{align*}where for any $t$, the function $h^i$ is defined as:
\begin{align}\label{eqn_h}
    h^i(\ups; \varepsilon, u_0, \eta)(t) &:= u_{0, i} + \int_0^t \bar{g}^i(\ups(s), \eta) ds + \varepsilon^i(t).
\end{align}
We prove the required continuity via the parametric continuity of the following optimization problem:
\begin{align}\label{eqn_opt}
F^*(\varepsilon, u_0, \eta) := \inf_{\ups\in \cd^2_B} F(\ups; \varepsilon, u_0, \eta)  \ \forall \ (\varepsilon, u_0, \eta) \in \cd_p.
\end{align}
It is clear that the minimizer ($\ups^*$) of \eqref{eqn_opt} is the fixed point of  the operator $\ups \mapsto h(\ups; \cdot, \cdot, \cdot)$, if one exists, and then, $F(\ups^*; \cdot, \cdot, \cdot) = 0$. Also, from \eqref{eqn_linear2}, $\Ups^n(\cdot)$ is the optimizer of \eqref{eqn_opt} at parameters $(\varepsilon, u_0, \eta) = (\varepsilon^n, \ups_n, n)$, by choice of $C(\omega)$ and domain $\cd^2_B$. Similarly, the ODE solution $\hat{\ups}^n(\cdot) \in \mbox{arginf}_{\ups\in \cd^2_B} F(\ups; 0, \ups_n, n)$, again by choice of $C(\omega)$ and domain $\cd^2_B$. We complete the remaining proof in two steps. 

$\bullet$ $\mathbf{F(\ups; \varepsilon, u, \eta)}$ \textbf{is jointly continuous}, i.e., 
if $||\ups^n - \ups|| \to 0$, $u_n \to u$, $\eta_n \to \eta$ and $||\varepsilon^n - \varepsilon||\to 0$, we have, $F(\ups^n; \epsilon^n, u_n, \eta_n) \to F(\ups; \epsilon, u, \eta)$. Recall from \eqref{eqn_g_bound}, $\bar{g}^i(\ups(\cdot), \eta_k) \leq b_1 + 1 + O(\omega)$ for each $i\in\{a,c\}$. Further, by assumption \ref{d2}, we have, $m(\ups^{a, n}(s)\eta_n) \to m(\ups^a(s)\eta)$.
This implies $\bar{g}(\ups^n(s), \eta_n) \to \bar{g}(\ups(s), \eta)$. Then, by applying bounded convergence theorem twice, we have the claim. 

$\bullet$ $\mathbf{\cd^2_B}$ \textbf{is weak-compact.}
Consider the projection, $p_s^i(\ups) := \ups^i(s)$, for each $i \in \{a, c\}$ and $s \in [0, T]$. For each $s, i$, we have,  $p_s^i(\cd_B^2) = [-C(\omega), C(\omega)]$, which are clearly compact. By Tychonoff's Theorem, 
$\cd_B^2$
is weak-compact under the well known product topology on $\cd^2$.

Thus, the parametric optimization problem in \eqref{eqn_opt} satisfies the hypothesis of Berge's maximum theorem (e.g., \cite{feinberg2014berges}). So, the set of optimizers defined by (for all  $(\varepsilon, u, \eta)\in \cd_p$):
\begin{align}\label{eqn_h_star}
{\cal H}^*(\varepsilon, u, \eta) &:= \mbox{arg inf}_{\ups \in \cd_B^2} F(\ups; \varepsilon, u, \eta) = \{\ups^*(\varepsilon, u, \eta)\} 
\end{align}
is upper semi-continuous correspondence on $\cd_p$. 

Next, define the set $\Theta \subset \cd_P$ such that
\begin{align}\label{eqn_Theta}
\Theta := \{(\varepsilon^n, \Ups_n,  n), (0, \Ups_n, n) \mbox{ for all } n \}.
\end{align}
By Lemma \ref{lemma_unique_ode_sol}, the optimizers are unique when restricted to $\Theta \subset \cd_p$.  Thus,  ${\cal H}^*$ of \eqref{eqn_h_star} is  continuous on $\Theta$, when viewed as a function. In other words, when arguments (particularly, $(\varepsilon^n, \ups_n, n)$ and $(0, \ups_n, n)$) of ${\cal H}^*$ are close-by,  then the corresponding values of ${\cal H}^*$ are also close-by. Also, by part (i), these arguments of ${\cal H}^*$ are closing-in, as $n \to \infty$.
\eop

\begin{lemma}\label{lemma_unique_ode_sol}
The optimizer of the problem in \eqref{eqn_opt} is unique if $F^* = 0$. This also implies, the solution for ODE \eqref{eqn_ODE_stpbp} is unique for any given initial condition over any bounded interval.
\end{lemma}
\begin{proof} If $F^*= 0$ for some parameter say $(\varepsilon, u_0, \eta)$, then by definition, any optimizer is a fixed point for $(h^a, h^c)$, see \eqref{eqn_h}. If possible, let $\ups_1$ and $\ups_2$ be two such distinct fixed points. Then, for each $i \in \{a, c\}$ and $j \in \{1, 2\}$, we have:
\begin{align*}
    \psi_j^i(t) &= u_0 + \int_0^t \bar{g}^i(\ups_j(s),  \eta) ds + \varepsilon^i(t) \mbox{ for any } t\geq 0.
\end{align*}
Let $k_m$ be the Lipschitz constant for function $m(\cdot)$ (see assumption \ref{d2}). Then, we have (see \eqref{eqn_ODE_stpbp}):
\begin{align}\label{eqn_pc}
    |\psi^{c}_1(t) - {\psi}^{c}_2(t)| &= \left|\int_0^t \bigg(\bar{g}^c(\ups_1(s), \eta) - \bar{g}^c(\ups_2(s), \eta) \bigg) ds \right| \nonumber\\
    &\leq \int_0^t \left|  m(\psi^{a}_1(s) \eta) - m(\psi^a_2(s) \eta)\right| ds + \int_0^t \left| \psi^{c}_1(s) - {\psi}^{c}_2(s)  \right| ds \nonumber \\
    &\leq  (k_m \eta +1)  \int_0^t u(s) ds.
\end{align}Similarly,
\begin{align}\label{eqn_pa}
    |\psi^a_1(t) - {\psi}^a_2(t)| 
    \leq (k_m \eta +1) \int_0^t  u(s) ds.
\end{align}
Define $u(s) := \max\{\left| \psi_1^i(s) - \psi_2^i(s) \right|: i \in \{a, c\}\}$ for each $s \geq 0$. Then, from \eqref{eqn_pc}, \eqref{eqn_pa}, we have:
\begin{align*}
    u(t) &\leq \sum_{i \in \{a, c\}}\left| \psi_1^{i}(s) - \psi_2^{i}(s) \right| \leq 2(k_m \eta+1) \int_0^t u(s) ds.
\end{align*}
Applying Gronwall inequality, we have $u(t) = 0$ for each $t \in [0, T]$. This implies, $||\psi_1^{a} - \psi_2^{a}|| = 0$ and $||\psi_1^{c} - \psi_2^{c}|| = 0$, i.e., $||\ups_1 - \ups_2|| = 0$.
\end{proof}
    \chapter{For Chapter \ref{ch:MFG}}\label{appendix_MFG}


\noindent \underline{\textbf{Proof of Lemma \ref{lemma_beta}:} }
The proof of this Lemma follows from \cite[Theorem 2.1, pp. 127]{kushner2003stochastic} under \textbf{(A)}, if further assumptions (A2.1)-(A2.5) of the cited Theorem hold, which we prove next. At first, observe $\sup_n E[L_{u, n}^2] < 2 < \infty$ (since from \eqref{eqn_dynamics_beta}, $\beta_{u, n} \leq 1$) for each $u$. Further, from \eqref{eqn_cond_exp}, there is no bias term as in the cited Theorem and $g_u(\cdot)$ is Lipschitz continuous. Lastly, $\sum_{i\geq 1} \epsilon_i^2 < \infty$ (for $\epsilon_i := 1/i$). \eop





\vspace{2mm}

\noindent \underline{\textbf{Proof of Theorem \ref{thrm_AI}:} } Consider $u \in \{R, F\}$. 
By hypothesis (B.i), (B.ii), $g_u(\beta_u^\eta) = 0$ for some $\beta_u^\eta$. Further, by hypothesis (B.iv), i.e., local stability, $g_u(\beta_u) > 0$ for all $\beta_u \in (\beta_u^\eta-\epsilon, \beta_u^\eta)$,  and $g_u(\beta_u) < 0$ for all $\beta_u \in (\beta_u^\eta, \beta_u^\eta + \epsilon)$, for some $\epsilon > 0$. Since $g_u(\cdot)$ is a continuous function with unique zero (see hypothesis (B.iii)), $g_u(\beta_u) > 0$ for all $\beta_u \in [0, \beta_u^\eta)$ and $g_u(\beta_u) < 0$ for all $\beta_u \in (\beta_u^\eta, 1]$. Thus, $t \mapsto \beta_u(t)$ is strictly increasing and decreasing, if $\beta_u(0) = \beta_u \in [0, \beta_u^\eta)$ and $(\beta_u^\eta, 1]$ respectively.  This implies that the assumption \textbf{(A)} of Lemma \ref{lemma_beta} is satisfied with $\cA_u = \{\beta_u^\eta\}$ and $\cD_u = [0,1]$ for each $u$. Thus, by Lemma \ref{lemma_beta}, $\beta_{u, k} \to \beta_u^\eta$ w.p. $1$. 

For the given $R, \gamma$, we now prove (a)-(c) for the game $\G(R, \gamma, \omega)$. By hypothesis (B.i), (B.ii) and above arguments, $P_{\bmu_\eta}(S; \theta, \delta) = 1$ (see \eqref{eqn_simplified_prob_success}). Also, from \eqref{eqn_util} and \eqref{eqn_R_gamma}, {\small$U(1, \bmu_\eta)$ $ = U(2, \bmu_\eta) > U(0, \bmu_\eta)$}. Thus, $\S(\bmu_\eta) = \mbox{arg max}_s U(s, \bmu)$, which by Definition \ref{defn_NE_MFG} implies that $\bmu_\eta$ is an AI-NE. Hence, part (a). 

Now, if possible, let $\bmu$ be another NE such that $P_{\bmu}(S; \theta, \delta) =: q \in [0,1]$. By \eqref{eqn_util}, {\small$U(1, \bmu) \geq U(0, \bmu)$}, as $Q_p \ge Q_{np}$. 

First consider the case with $U(1, \bmu) > U(0, \bmu)$. Thus, $0 \notin \S(\bmu)$, and hence, any $\bmu$ with $\bmu_0 > 0$ can not be a NE for this case (see Definition \ref{defn_NE_MFG}). For the rest, we divide the proof in two sub-cases:

$\bullet$ If $\S(\bmu) = \{1\}$, then $\bmu = \bmu_{1-\mua}$. By Lemma \ref{lemma_beta_only_type1}, $\beta_F(\bmu) < \theta_a(\bmu)$ and $\beta_R(\bmu) < \delta_a(\bmu)$. Thus, $\bmu$ can be a NE (if at all) when $q = 1-p$. Further, $\bmu$ being a NE implies that $U(1, \bmu) \geq U(2, \bmu)$, with utility function as in \eqref{eqn_util}. That is, $R, \gamma$ should satisfy the following relation:
$$
    R(1-p) (\gamma-1) \leq \ce, \mbox{ i.e., } g(\gamma) \leq 1, \mbox{ where}
$$
$g(\gamma) := (\gamma-(\gamma-1)(\eta+\mua)) (1-p)$. Observe, $\gamma \mapsto g(\gamma)$ is increasing and $g(\underline{\gamma}(\eta)) = 1$; thus, $g(\gamma) > 1$ for $\gamma$ given in hypothesis. This contradicts $\bmu$ being a  NE. 

$\bullet$ If $\S(\bmu) = \{2\}$ or $\{1, 2\}$, then $\bmu = \bmu_x$ for some $x\in [0, 1-\mua) - \{\eta\}$. One can show that $U(1, \bmu) > U(2, \bmu)$ (for any $q \in [0,1]$). Thus, $\S(\bmu) \not\subset \mbox{argmax}_s U(s, \bmu) = \{1\}$, which contradicts $\bmu$ being a  NE. 



Lastly, consider the case where $U(1, \bmu) = U(0, \bmu)$; this is possible only when $q = 0$ and $\cnp = \cp$. Clearly,  $\cp > \cp - \ce = U(2, \bmu)$. 
Thus, $2 \notin \S(\bmu)$. If $\bmu = \bmu_{1-\mua}$, from Lemma \ref{lemma_beta_only_type1}, $\beta_R(\bmu) < \delta_a(\bmu)$. This implies, $q \neq 0$, which is a contradiction. Another possibility for $\bmu$ is $(1-\mua, 0, 0)$ which can not be NE as $\beta_F(\bmu) = 0 = \theta_a(\bmu)$, leading to $q \geq p \neq 0$. The last possibility for $\bmu$ is $(x, 1-x-\mua, 0)$ for any $x \in (0, 1-\mua)$, for which we have:
\begin{align*}
\beta_R(\bmu) &= \alpha_R\left(\frac{1-x-\mua}{1-x} \right) < \delta \left(\frac{1-x-\mua}{1-x} \right) = \delta_a(\bmu).
\end{align*}Thus, $q \neq 0$; hence, any such $\bmu$ also can not be a NE.  \eop

\begin{lemma}\label{lemma_beta_only_type1}
For $\bmu = \bmu_{1-\mua}$,  $\beta_{u, k}(\bmu) \to \alpha_u(1-\mua)$ w.p. $1$, as $k \to \infty$, for each $u$. 
\end{lemma}
\begin{proof}
From \eqref{eqn_dynamics_beta}, $\beta_{u, k}$ can be re-written as follows:
\begin{align*}
    \beta_{u, k}(\bmu) &= \frac{\sum_{i=1}^k 1_{\{\mbox{tag for $u$-post} = F\}}}{k}.
\end{align*}
Thus, $\beta_{u, k}(\bmu) \to \alpha_u (1-\mua)$ w.p. $1$, as $k \to \infty$, by strong law of large numbers, and \eqref{eqn_cond_exp}.
\end{proof}

\noindent \underline{\textbf{Proof of Theorem \ref{thrm_response1}:}} The proof is in $3$ steps:

\noindent (a) $\bmu_\eta$ is an AI-NE such that $\beta_F^\eta \geq \widetilde{\theta}(1-\mua) \geq \theta(1-\mua)$ and $\beta_R^\eta < \delta_a = \delta_a(\bmu_\eta)$,

\noindent (b) by Theorem \ref{thrm_AI}, any $\bmu$ with $\mu_0 > 0$ or any $\bmu_x$ for $x \in [0, \eta) \cup \{1-\mua\}$ is not a NE,

\noindent (c) $\bmu_{x_\eta}$ can be the only other NE, if at all $x_\eta > \eta_{\widetilde \theta}^*$ and $\beta_F^{x_\eta} < \theta(1-\mua)$. 


Define $x_F := \frac{1-\mua - \frac{1}{cw\alpha_R(\Delta_R)^a}}{1-\alpha_F}$.
Define $\rho_F(x) := 1-(1-x-\mua)cw\alpha_R(\Delta_R)^a$ for $x \in (0,1)$. Then, $\rho_F(x) = 0$ for $x =  x_F(1-\alpha_F)$. Also, $\rho_F(x)$ is increasing in $x$. Therefore, $\rho_F(x) < 0$ for $x < x_F(1-\alpha_F)$ and $\rho_F(x) > 0$ for $x >  x_F(1-\alpha_F)$. 

Next define $\overline{\rho}_F(x) := \alpha_F x + 1 - x - \mua$ for $x \in (0,1)$. Then, $\overline{\rho}_F(x) = \frac{1}{cw\alpha_R(\Delta_R)^a}$ for $x = x_F$ and $\overline{\rho}_F(x)$ is decreasing in $x$. Therefore, $\overline{\rho}_F(x) \in {\cal R}_F:= \left\{y: y < \frac{1}{cw\alpha_R(\Delta_F)^a}\right\}$ for all $x > x_F$ and $\overline{\rho}_F(x) \in {\cal R}_F^c$ for all $x \leq x_F$.

In all, by above, $\overline{\rho}_F(x) \in {\cal R}_F^c$ for $x \leq x_F(1-\alpha_F)$. If not, $\rho_F(x) > 0$ for all $x > x_F(1-\alpha_F)$, and then by Lemma \ref{lemma_beta_traj}, both $\overline{\rho}_F(x)$ and $\frac{\alpha_F x}{\rho_F(x)}$ are in ${\cal R}_F^c$ for $x \in (x_F(1-\alpha_F), x_F]$; both $\overline{\rho}_F(x)$ and $\frac{\alpha_F x}{\rho_F(x)}$ are in ${\cal R}_F$ for $x > x_F$. Further, by Lemma \ref{lemma_beta_traj}, $\beta_F^x$ (the attractor of ODE \eqref{eqn_ode}) is given by:
\begin{align}\label{eqn_betaF}
\beta_F^x = 
\begin{cases}
 \overline{\rho}_F(x) \mbox{ if } x \in (0, x_F], \\
 \frac{\alpha_F x}{\rho_F(x)} \mbox{ if } x \in (x_F, 1).
\end{cases}
\end{align}

Similarly, again by Lemma \ref{lemma_beta_traj}, one can show that $\beta_R^x$ is:
\begin{align}\label{eqn_betaR}
\beta_R^x = 
\begin{cases}
 \overline{\rho}_R(x) \mbox{ if } x \in (0, x_R], \\
 \frac{\alpha_R x}{\rho_R(x)} \mbox{ if } x \in (x_R, 1),
\end{cases}
\end{align}
for $\overline{\rho}_R(x) := \alpha_R x + 1 - x - \mua$, $\rho_R(x) := 1-(1-x-\mua)cw\alpha_R$ and $x_R := \frac{1-\mua-\frac{1}{cw\alpha_R}}{1-\alpha_R}$.


 Observe that by the choice of $w$ as in Algorithm \ref{alg_AI}, 
\begin{align}\label{eqn_bound_cwalphaR}
    cw\alpha_R > \frac{1}{(\Delta_R)^a \widetilde{\theta} (1-\mua)}.
\end{align}
Thus, from \eqref{eqn_notations}, $\eta^*_{\widetilde{\theta}} < x_F$. Consider any $x \leq \eta^*_{\widetilde{\theta}}$. By \eqref{eqn_betaF}, 
$\beta_F^x = \overline{\rho}_F(x)$. Since $\beta_F^x$ strictly decreases with $x$, therefore, $\beta_F^x \geq \beta_F^{\eta^*_{\widetilde{\theta}}} = \widetilde{\theta}(1-\mua) \geq \theta(1-\mua)$, by Lemma \ref{lemma_theta_tilde}. 

Observe that $\eta = \overline{\eta} + \epsilon_2 \leq \eta^*_{\widetilde{\theta}}$ (see claim $1$ at the end of the proof for details), thus proving conditions (B.ii), (B.iii) and (B.iv) for $u = F$ of Theorem \ref{thrm_AI} for $\bmu_\eta$.


Again by the choice of $w, \epsilon_2$ and $\delta>\alpha_R$, $\eta - x_R > \overline{\eta}-x_R > 0$ (see claim $2$ at the end of the proof). Therefore, by \eqref{eqn_betaR}, $\beta_R^\eta = \frac{\alpha_R \eta}{{\rho}_R(\eta)}$. 
By the choice of $w$ and since $\beta_R^x$ strictly decreases as $x$ increases, we have, $\beta_R^\eta < \beta_R^{x_R} = 1/(cw \alpha_R) < \delta_a$. 
This proves conditions (B.i), (B.iii) and (B.iv) for $u = R$ of Theorem \ref{thrm_AI}  for $\bmu_\eta$. 

In all, by Theorem \ref{thrm_AI}, $\G(R, \gamma, \omega)$ is an AI game with $\bmu_\eta$ as a NE such that it achieves $(\widetilde{\theta}, \delta)$-success (i.e., $\td$ as $\widetilde{\theta} \geq \theta$ by Lemma \ref{lemma_theta_tilde} \TR{in \cite{arxiv}}{}); further, any $\bmu_x$ for $x\in [0, \eta) \cup \{1-\mua\}$  and any $\bmu$ with $\mu_0 > 0$ can not be a NE, by Theorem \ref{thrm_AI}. This complete steps (a) and (b).

Consider any $x \in (\eta, 1-\mua)$. Since $\beta_R^x$ decreases in $x$, $\beta_R^x < \beta_R^\eta < \delta_a$. This proves (b). Recall $\beta_F^x \geq \theta(1-\mua)$ for $x \in (\eta, \eta^*_{\widetilde{\theta}}]$. Thus:
$$
P_{\bmu_x}(S; \widetilde{\theta}, \delta) = 1 \mbox{ for each } x \in (\eta, \eta^*_{\widetilde{\theta}}].
$$
For the given $R, \gamma$ and chosen $x$, one can show that $U(1, \bmu_x) < U(2, \bmu_x)$. Thus, $\S(\bmu_x) = \{1,2\} \not\subset \mbox{argmax}_s u(s, \bmu) = \{2\}$; under Definition \ref{defn_NE_MFG}, $\bmu_x$ is not a NE. In fact, if $P_{\bmu_x}(S;\theta, \delta) = 1$ for some $x \in (\eta^*_{\widetilde{\theta}}, 1-\mua)$, then again using above arguments, one can show that $\bmu_x$ is not a NE. 

Recall $\beta_R^x < \delta_a$ for each $x \in (\eta, 1-\mua)$. Further by definition of $x_\eta$, $U(1, \bmu_x) = U(2, \bmu_x)$ only for $x = x_\eta$ with $
P_{\bmu_x}(S; \theta, \delta) = 1-p$; further, by \eqref{eqn_util}, $x_\eta$ is the only such possible $x$. Thus, by Definition \ref{defn_NE_MFG}, $\bmu_{x_\eta}$ is a NE, but not AI-NE, if at all $x_\eta > \eta^*_{\widetilde{\theta}}$ and $\beta_F^{x_\eta} < \theta(1-\mua)$. This completes step (c). 

Now, we will prove the sub-claims made above. 

\noindent \underline{Claim 1:} $\eta^*_{\widetilde{\theta}}>\overline{\eta}$. Let us consider the difference:
\begin{align*}
        \eta^*_{\widetilde{\theta}} - \overline{\eta} &=  \eta^*_{\widetilde{\theta}} - \frac{\delta_a((1-\mua)cw\alpha_R - 1)}{cw\alpha_R \delta_a - \alpha_R}\\
        &= \frac{cw\alpha_R   \delta_a (\eta^*_{\widetilde{\theta}} - (1-\mua))  - \eta^*_{\widetilde{\theta}} \alpha_R + \delta_a }{cw\alpha_R \delta_a - \alpha_R}\\
        &=  \frac{-(1 - \mua -  \eta^*_{\widetilde{\theta}})  \delta_a\left( cw\alpha_R   - \frac{- \eta^*_{\widetilde{\theta}} \alpha_R + \delta_a}{\delta_a (1 - \mua  - \eta^*_{\widetilde{\theta}} )} \right) }{cw\alpha_R \delta_a - \alpha_R}
\end{align*}
Now, $1 - \mua >  \eta^*_{\widetilde{\theta}} $. Further, by the choice of $w$, $cw\alpha_R < \fone$, 
which implies the numerator in $\eta^*_{\widetilde{\theta}} - \overline{\eta}$ is strictly positive. Furthermore, we have $cw\alpha_R > \frac{1}{1-\mua} > \frac{\alpha_R}{\delta(1-\mua)}$. This implies that the denominator  in $\eta^*_{\widetilde{\theta}} - \overline{\eta}$  is strictly positive. Therefore, $\eta^*_{\widetilde{\theta}}>\overline{\eta}$.

\noindent \underline{Claim 2:} $\eta > x_R$

\vspace{-4mm}
{\small
\begin{align*}
    \eta - x_R &> \overline{\eta} - x_R \\
    &\hspace{-6mm}= \frac{\delta_a((1-\mua)cw\alpha_R - 1)}{cw\alpha_R \delta_a - \alpha_R} - \frac{1}{1-\alpha_R} \left( 1-\mua-\frac{1}{cw\alpha_R}\right)\\
    &\hspace{-6mm}= ((1-\mua)cw\alpha_R - 1) \left( \frac{\delta_a }{cw\alpha_R \delta_a - \alpha_R} - \frac{1}{cw\alpha_R(1-\alpha_R)} \right)\\
    &\hspace{-6mm}= \frac{((1-\mua)cw\alpha_R-1) (1-cw\alpha_R\delta_a)}{cw\alpha_R(1-\alpha_R) (cw\alpha_R\delta_a - \alpha_R)} > 0 \mbox{ (by choice of $w$)}.  
\end{align*}} \eop

\begin{lemma}\label{lemma_beta_traj}
    Define $\overline{\rho}_u := \alpha_u\eta + 1 - \eta - \eta_a$ and $\rho_u := 1 - (1-\eta-\eta_a)cw\alpha_R(\Delta_u)^a$ for $\eta \in (0, 1-\eta_a)$ and $u \in \{R, F\}$. Consider the regime, ${\cal R}_u := \left\{x: x < \frac{1}{cw\alpha_R(\Delta_u)^a}\right\}$. Then, for the response function given in \eqref{eqn_response1}, the following statements are true:
    \begin{enumerate}
        \item if $\rho_u \leq 0$, then $\overline{\rho}_u \in {\cal R}_u^c$, and the attractors of ODE \eqref{eqn_ode}, ${\cal A}_u = \{\overline{\rho}_u\}$; 
        \item if $\rho_u > 0$, then $\overline{\rho}_u \in {\cal R}_u^c$ if and only if $\frac{\alpha_u\eta}{\rho_u} \in {\cal R}_u^c$. Further, if $\overline{\rho}_u \in {\cal R}_u$ then ${\cal A}_u = \{\frac{\alpha_u\eta}{\rho_u}\}$, while if $\overline{\rho}_u \in {\cal R}^c_u$ then ${\cal A}_u = \{\overline{\rho}_u\}$.
    \end{enumerate}
\end{lemma}
\begin{proof}
    At first, let $\rho_u \leq 0$. Then, by definition of $\rho_u$, $1-\eta-\eta_a \geq \frac{1}{cw\alpha_R(\Delta_u)^a}$. Since $\eta > 0$, therefore, 
    $$
    \overline{\rho}_u \geq \alpha_u \eta + \frac{1}{cw\alpha_R(\Delta_u)^a} > \frac{1}{cw\alpha_R(\Delta_u)^a} \implies \overline{\rho}_u \in {\cal R}_u^c.
    $$ 
    
    The ODE \eqref{eqn_ode} can be written in the simplified form as follows (recall \\ $r(\alpha_u, \omega(\beta_u)) = \min\{1, cw\alpha_R(\Delta_u)^a\beta_u\}$):
\begin{align}\label{eqn_simplified_ODE}
\dot{\beta_u} = 
\begin{cases}
 \overline{\rho}_u- \beta_u, \mbox{ if } r(\alpha_u, \omega(\beta_u)) = 1, \mbox{ i.e., } \beta_u \in {\cal R}_u^c, \\
 \alpha_u \eta - \rho_u\beta_u, \mbox{ if } r(\alpha_u, \omega(\beta_u)) < 1, \mbox{ i.e., } \beta_u \in {\cal R}_u.
\end{cases}
\end{align}
Clearly, the RHS of the above ODE is piecewise linear, and hence the solution $\beta_u(\cdot)$ exists.

Now, say $\beta_u(0) \in {\cal R}_u$. Then, $\dot{\beta_u} > 0$, thus, $\beta_u(t)$ increases with $t$. This implies the existence of $\tau <\infty$ such that $cw\alpha_R(\Delta_u)^a\beta_u(\tau) = 1$. Then, the solution of the ODE for all $t\geq \tau$ is:
\begin{align}\label{eqn_beta_sol}
    \beta_u(t) = \overline{\rho}_u + e^{-t+\tau}(\beta_u(\tau) - \overline{\rho}_u).
\end{align}
The above solution holds for all $t\geq \tau$ as $r(\alpha_u, \omega(\beta_u(t))) = 1$ for all $t\geq \tau$; towards this, observe that $\beta_u(\tau) \leq \overline{\rho}_u$ (since $\overline{\rho}_u \in {\cal R}_u^c$), therefore, $t \mapsto \beta_u$ is an increasing function. Hence, from \eqref{eqn_beta_sol}, $\beta_u(t) \to \overline{\rho}_u$ as $t \to \infty$. On the contrary if $\beta_u(0) \in {\cal R}_u^c$, i.e., $r(\alpha_u, \omega(\beta_u(0))) = 1$, then for all $t\geq 0$ (check $r(\alpha_u, \omega(\beta_u(t))) = 1$ for all $t\geq 0$):
\begin{align}\label{eqn_beta_sol_2}
    \beta_u(t) = \overline{\rho}_u + e^{-t}(\beta_u(0) - \overline{\rho}_u).
\end{align}
From above, $\beta_u(t) \to \overline{\rho}_u$.

Now, let $\rho_u > 0$. Then, by definitions, we have:
\begin{align*}
    \frac{\alpha_u\eta}{\rho_u} \in {\cal R}_u^c &\iff cw\alpha_R(\Delta_u)^a  \frac{\alpha_u\eta}{\rho_u} \geq 1 \iff cw\alpha_R(\Delta_u)^a\alpha_u\eta \geq \rho_u\\
    &\iff cw\alpha_R(\Delta_u)^a\overline{\rho}_u \geq 1 \iff \overline{\rho}_u \in {\cal R}_u^c.
\end{align*}
We will now derive ${\cal A}_u$ for the case when $\overline{\rho}_u \in {\cal R}_u$, and ${\cal A}_u$ can be derived analogously for the complementary case. As before, say $\beta_u(0) \in {\cal R}_u$. Then initially the $\beta_u$-ODE is:
$$
    \dot{\beta_u} = \alpha_u \eta - \rho_u \beta_u.
$$
Thus, the solution of the above ODE is:
\begin{align*}
    \beta_u(t) = \frac{\alpha_u\eta}{\rho_u} + e^{-\rho_u t}\left(\beta_u(0) - \frac{\alpha_u\eta}{\rho_u} \right) \mbox{ for all } t \geq 0.
\end{align*}
Clearly, $\beta_u(t) \to \frac{\alpha_u\eta}{\rho_u}$.
If $\beta_u(0) \in {\cal R}_u^c$, then as previously, the ODE solution is given by \eqref{eqn_beta_sol_2} for all $t < \tau$, where $\tau := \inf\{t : r(\alpha_u, \omega(\beta_u(t))) < 1\}$. For $t \geq \tau$, the solution is:
$$
    \beta_u(t) = \frac{\alpha_u\eta}{\rho_u} + e^{-\rho_u(t-\tau)}\left( \beta_u(\tau) - \frac{\alpha_u\eta}{\rho_u} \right).
$$
Then, $\beta_u(t) \to \frac{\alpha_u\eta}{\rho_u}$.
\end{proof}

\noindent \underline{\textbf{Proof of Theorem \ref{thrm_perf_x_eta}:}} From the definition of $x_\eta$, note that $x_\eta > \eta$. From  \eqref{eqn_betaR}, $\beta_R^x$ is decreasing in $x$, and thus, by (a) in the proof of Theorem \ref{thrm_response1}, $\beta_R^{x_\eta} <  \beta_R^\eta \leq \delta_a$. By \eqref{eqn_betaR}:


\noindent $\bullet$ if $x_\eta \in (\eta^*_{\widetilde{\theta}}, x_F]$, $\beta_F^{x_\eta} \geq \beta_F^{x_F} = \overline{\rho}_F(x_F) = \frac{1}{cw\alpha_R (\Delta_R)^a}$;

\noindent $\bullet$ if $x_\eta \in (x_F, 1-\mua)$, $\beta_F^{x_\eta} \geq \beta_F^{1-\mua} = \alpha_F(1-\mua)$. \eop

\begin{lemma}\label{lemma_theta_tilde}
    For notations as in Algorithm \ref{alg_AI} when $K_\delta \geq 0$, ${\widetilde \theta} \geq \theta$.
\end{lemma}
\begin{proof}
    If $\theta > f(\theta, \delta)$, then ${\widetilde \theta} = \theta$, and we are done. Else, define, $g(x) := x - f(x, \delta)$ for $x \in \mathbb{R}$; observe $g(\theta) \leq 0$. Using the definitions in \eqref{eqn_notations}, re-write $g(x)$ as:
    \begin{align*}
        g(x) &= \left( \frac{1-\mua}{1-\alpha_F}\right) \frac{p(x)}{t(x)}, \mbox{ where}\\
        t(x) &:= (\Delta_R)^a\left(\delta_a - \frac{(1-x)(1-\mua)\alpha_R}{1-\alpha_F} \right) \mbox{ and}\\
        p(x) &:= Ax^2+\kappa x+C, \mbox{ for } A := (\Delta_R)^a \alpha_R \mbox{ and } C := \delta \alpha_F.
    \end{align*}
    Observe that $t(\cdot)$ is strictly increasing and $t(\theta^*) = 0$ for $\theta^* := 1 - \frac{\delta(1-\alpha_F)}{\alpha_R}$; note $\theta^* < 1$. Thus, $t(x) < 0$ for $x < \theta^*$ and $t(x) > 0$ for $x > \theta^*$. Also, $p(\cdot)$ is convex function such that $p(0) > 0$ and $p(1) > 0$ (recall $(\Delta_R)^a > 1$). Thus, there exists $\theta_1 = \frac{-\kappa - \sqrt{K_\delta}}{2A}, \theta_2 = \frac{-\kappa + \sqrt{K_\delta}}{2A}$ such that $p(\theta_1) = p(\theta_2) = 0$, provided $K_\delta \geq 0$. By convexity, $p(1) > 0$ and $p(0) > 0$, either both $\theta_1, \theta_2$ are above $1$, or below $0$, or are in $(0,1)$. With the above notations, 
    $$
        \widetilde{\theta} = \min\{ \max\{\theta^*, \theta_2\} + \epsilon , 1\}, \mbox{ with } \epsilon > \max \{0, \theta - \theta_2\}.
    $$

    If ${\widetilde \theta} = 1$, then clearly $\theta \leq {\widetilde \theta} = 1$. If $g(\theta) = 0$, then, $p(\theta) = 0$. Thus, either $\theta = \theta_1$ or $\theta = \theta_2$. Therefore, ${\widetilde \theta} = \max\{\theta_2, \theta^*\} + \epsilon > \theta_2 \geq \theta$. Else if $g(\theta) < 0$, then we will prove the claim for three cases separately. 

    Case 1: If $\theta > \theta^*$. Then $t(\theta) > 0$. Also, $g(\theta) < 0$, therefore, $p(\theta) < 0$. Thus, $\theta_1, \theta_2 \in (0,1)$ and $\theta \in (\theta_1, \theta_2)$.     By definition of ${\widetilde \theta}$, in this case,
    ${\widetilde \theta} = \theta_2 + \epsilon > \theta_2 > \theta$. 

    Case 2: If $\theta < \theta^*$. Then $t(\theta) < 0$, and $g(\theta) < 0$. Thus, $p(\theta) > 0$, which implies  either $\theta < \theta_1$ or $\theta > \theta_2$ (by convexity, $p(x) < 0$ for $x \in (\theta_1, \theta_2)$). Again by definition of ${\widetilde \theta}$, in this case we have:

    (i) ${\widetilde \theta} = \max\{\theta^*, \theta_2\} + \epsilon > \max\{\theta^*, \theta_2\} > \theta$ if $\theta < \theta_1$, or 

    (ii) ${\widetilde \theta} = \max\{\theta^*, \theta_2\} + \epsilon > \theta^* + \theta - \theta_2 > \theta$ if $\theta > \theta_2$.

    Case 3: If $\theta = \theta^*$. Then $t(\theta) = 0$ and $g(\theta) < 0$. Thus, $p(\theta) < 0$, and the claim follows as in case 1. 
\end{proof}

\begin{lemma}\label{lemma_feasibility_w}
    Under the hypothesis of Theorem \ref{thrm_response1} and for notations as in Algorithm \ref{alg_AI} when $K_\delta \geq 0$, the choice of $w$ is feasible.
\end{lemma}
\begin{proof}
    We are given $w$ such that:
\begin{align*}
cw\alpha_R &\in \bigg(\frac{1}{1-\mua} \max\left\{1, \frac{1}{(\Delta_R)^a \widetilde{\theta}}\right\}, \min\left\{  \frac{1}{\delta_a}, \fone \right\}  \bigg).
\end{align*}
We will show that the above interval is not empty. 

\noindent (i) If $\frac{1}{\delta_a} \leq \fone$, then:

$\bullet$ $\frac{1}{1-\mua} < \frac{1}{\delta_a} = \frac{1}{\delta(1-\mua)}$ since $\delta < 1$.

$\bullet$ under hypothesis of Theorem \ref{thrm_response1} and by Lemma \ref{lemma_theta_tilde}, $\widetilde{\theta} (\Delta_R)^a \geq \theta (\Delta_R)^a > \delta$. Thus, $\frac{1}{\widetilde{\theta} (\Delta_R)^a (1-\mua)} < \frac{1}{\delta_a}$.

\noindent (ii) If $\frac{1}{\delta_a} > \fone$, then:

$\bullet$ recall that $\alpha_R < \delta$ and $\eta^*_{\widetilde{\theta}} < 1-\mua$, therefore, 
\begin{align*}
    0 = \delta(1-\mua) - \delta_a &< \eta^*_{\widetilde{\theta}} (\delta-\alpha_R)\\
    \implies \delta(1-\mua - \eta^*_{\widetilde{\theta}} ) &< \delta_a - \eta^*_{\widetilde{\theta}} \alpha_R\\
    \implies \frac{1}{1-\mua} &< \fone.
\end{align*}

$\bullet$ lastly, $\frac{1}{(\Delta_R)^a \widetilde{\theta} (1-\mua)} < \fone$ if:
\begin{align*}
   \frac{1}{(\Delta_R)^a \widetilde{\theta}} &< \frac{\delta_a - \eta^*_{\widetilde{\theta}} \alpha_R}{(1-\mua-\eta^*_{\widetilde{\theta}} )\delta}, \mbox{ i.e., if}\\
   \widetilde{\theta} (\Delta_R)^a (\delta_a - \eta^*_{\widetilde{\theta}}  \alpha_R) &> \delta_a - \eta^*_{\widetilde{\theta}}  \delta, \mbox{ i.e., if}\\
    \eta^*_{\widetilde{\theta}} (\delta - \alpha_R   (\Delta_R)^a \widetilde{\theta} ) &> \delta_a (1-  (\Delta_R)^a\widetilde{\theta}), \mbox{ i.e., if}\\
    \left(\frac{1-\widetilde{\theta}}{1-\alpha_F}\right) (\delta - \alpha_R   (\Delta_R)^a \widetilde{\theta} )  &> \delta (1-  (\Delta_R)^a\widetilde{\theta}), \mbox{ i.e., if}\\
     p(\widetilde{\theta}) &>0,
 \end{align*}
 for $p(\cdot)$ defined in the proof of Lemma \ref{lemma_theta_tilde}. 
Recall from the proof of Lemma \ref{lemma_theta_tilde} that the two zeroes, $\theta_1, \theta_2$, of convex function $p(\cdot)$ are either above $1$, or below $0$, or are in $(0,1)$. Further, $p(0) > 0$ and $p(1) > 0$. In the first two cases, $p(x) > 0$ for all $x \in [0,1]$; thus $p(\widetilde{\theta}) > 0$. In the last case, by definition of $\widetilde{\theta}$, $\widetilde{\theta} > \theta_2$, thus $p(\widetilde{\theta}) > 0$.
\end{proof}

\end{appendices}

\newpage
\setlength{\parskip}{5mm}
\titlespacing{\chapter}{0cm}{0mm}{0mm}
\titleformat{\chapter}[display]
  {\normalfont\huge\bfseries}
  {\chaptertitlename\ \thechapter}{20pt}{\Huge}

\bibliographystyle{unsrt}
\bibliography{References/mybibfile}


\chapter*{List of Publications}
\label{ch:pub}
\addcontentsline{toc}{chapter}{\nameref{ch:pub}}

   
\noindent \textbf{International Journals}
\begin{enumerate}
    \item \textbf{Agarwal, Khushboo}, and Veeraruna Kavitha. "Robust fake-post detection against real-coloring adversaries." Performance Evaluation 162 (2023): 102372. 

    *The paper was presented at IFIP Performance 2023 and awarded with Best Student Paper Award. A shorter version is also published as a conference paper. 
    \item \textbf{Agarwal, Khushboo}, and Veeraruna Kavitha. "Total-Current Population dependent Branching Processes: Analysis via Stochastic Approximation." (submitted)
    \item Kapsikar, Suyog, Indrajit Saha, \textbf{Khushboo Agarwal}, Veeraruna Kavitha, and Quanyan Zhu. "Controlling fake news by collective tagging: A branching process analysis." IEEE Control Systems Letters 5, no. 6 (2020): 2108-2113.

    *The paper was presented at American Control Conference 2020. 
\end{enumerate}

\noindent \textbf{International Conferences}
  \begin{enumerate}
  \item \textbf{Agarwal, Khushboo}, and Veeraruna Kavitha. "Co-Virality of Competing Content over OSNs?." 2021 IFIP Networking Conference (IFIP Networking). IEEE, 2021.  

  \item \textbf{Agarwal, Khushboo}, and Veeraruna Kavitha. "Saturated total-population dependent branching process and viral markets." 2022 IEEE 61st Conference on Decision and Control (CDC). IEEE, 2022.

  \item \textbf{Agarwal, Khushboo}, and Veeraruna Kavitha. "Single-out fake posts: participation game and its design." (Accepted at 2023 American Control Conference (ACC))

  \item Singh, Vartika, \textbf{Khushboo Agarwal}, Shubham, and Veeraruna Kavitha. "Evolutionary Vaccination Games with premature vaccines to combat ongoing deadly pandemic." EAI VALUETOOLS, 2021.
\end{enumerate}
\newpage
\setlength{\parskip}{5mm}
\titlespacing{\chapter}{0cm}{0mm}{0mm}
\titleformat{\chapter}[display]
  {\normalfont\huge\bfseries \centering}
  {\chaptertitlename\ \thechapter}{20pt}{\Huge}

\chapter*{Acknowledgments}
\label{ch:Acknowledgments}
\addcontentsline{toc}{chapter}{\nameref{ch:Acknowledgments}}
\vspace{10mm}
First and foremost, I thank Maa (goddess) Durga for giving me the opportunity of a lifetime to do PhD at the prestigious institute, IIT Bombay. I feel blessed that she provided me with the strength and all I needed to pursue this journey. 

Next, I would like to extend my deepest gratitude to my PhD supervisor, \textit{Prof. Veeraruna Kavitha}. From my first conversation with her to this day, when I write my word of thanks, she has always been encouraging and supportive. My greatest learning from her is to always start with the basics of the problem and then proceed further. I sought only to learn more about mathematics and how to do research from her. But she has also taught me how to be a stronger person who can aim higher and persevere to achieve them. In short, I can only say that she has guided me in ways no one can, and I am forever indebted to her for this. 

I sincerely acknowledge the constructive feedback of my
research progress committee (RPC) members, \textit{Prof. N. Hemachandra} and \textit{Prof. Koushik Saha}. I would also like to thank Prof. K.S. Mallikarjuna Rao, Prof. Jayendran Venkateswaran, Prof. Narayan Rangaraj, Prof. Ashutosh Mahajan, Prof. Vishnu Narayanan, Prof. Manjesh Hanawal and Prof. P Balamurugan who taught me how to see mathematics through the lens of the applied world, which was the main reason for joining the department. 

I wish to thank \textit{Prof. N Hemachandra} for advising me to start solving the problem with all I know and simultaneously learn more about the topic. I will also always remember \textit{Prof. Mallikarjuna}'s words that knowing likes or dislikes is equally helpful in sorting the options. I am forever grateful to \textit{Prof. Jayendran} for his extreme support as the Head of the department. 

I am also thankful to Dr. M Venkateswararao K for informing me about the research internship at INRIA, France, and, more importantly, to \textit{Prof. Manjesh} and \textit{Prof. Narayan} for allowing me to take the opportunity. I am truly grateful to \textit{Prof. Bruno Tuffin} and \textit{Prof. Patrick Maill{\'e}} for giving me my first exposure to the research world. 

I am also thankful to the Ministry of Education (MoE) and the Prime Minister’s Research Fellowship (PMRF), India, for the financial support. I am also grateful to Mr. Abasaheb Molavane, Mr. Amlesh Kumar, Mr. Siddhartha Salve, and Mr. Pramod Pawar for helping with administrative queries and formalities.

As a PMRF candidate, I interacted with \textit{Prof. Amiya Bhowmick}, and I am thankful to him for his kind words and patiently answering all queries. In the process, I enjoyed teaching students at ICT, Mumbai, which has motivated me to pursue a teaching career.

I thoroughly enjoyed the learning sessions with the `Stochastic group'. I am thankful to all the friends I made, starting from my school days to the new ones I made in IITB, for uplifting me in one way or another. To name a few, I wish to thank Harshita, Kushagra, Karnica, Krishna, my batchmates (Amlu, Gullu, Adsuuuu, Vanessa, Akul, Bijendra, Sheshadev), Mustafa, Shubham and Sambit. I also want to thank Vartika, for being such a wonderful friend, listening to my silly concerns and guiding me many times. I feel truly blessed to have Sanket, who feels like family and makes it all seem easier at times. All my seniors, particularly Richa Di, Chhavi Di, Aanchal Di, Tejas Di, Indrajit and Puja Di, have been very supportive. The fun times at the PhD lab with Prem, Mustafa, Simran, Madhu, Akshay, Anand and others are memorable.

Last but not least, I can not help but feel grateful for the most important people without whom this journey would have been impossible. My father is the one who suggested IEOR to me first. His constantly encouraging words have been a support like none other. My mother always showered me with endless love. Also, my younger brother (yet acting like an elder at times) always loved me and helped with tech- or code-related queries. I am also thankful to all my family members who always praised and encouraged me to pursue this career. Lastly, I wish my grandparents (nana and nani) were here today to bless me as they always did.

\end{document}